%% file: numeric-rho_memoir_style.tex
\documentclass{memo-l}

\usepackage[all]{xy}
\usepackage{amsmath,amsthm,amsfonts,latexsym,amscd,amssymb,enumerate}
\usepackage[numeric]{amsrefs}
\usepackage[pdftex,pdftitle={Numeric rho invariants},pdftex]{hyperref}

\newtheorem{theorem}{Theorem}[chapter]
\newtheorem{lemma}[theorem]{Lemma}
\newtheorem{corollary}[theorem]{Corollary}
\newtheorem{proposition}[theorem]{Proposition}

\theoremstyle{definition}
\newtheorem{definition}[theorem]{Definition}
\newtheorem{example}[theorem]{Example}

\newtheorem{set-up}[theorem]{Geometric set-up}

\theoremstyle{remark}
\newtheorem{remark}[theorem]{Remark}

\newcommand{\xrightarrowdbl}[2][]{%
  \xrightarrow[#1]{#2}\mathrel{\mkern-14mu}\rightarrow}

\newcommand\NN{\mathbb N}
\newcommand\QQ{\mathbb Q}
\newcommand\RR{\mathbb R}
\newcommand\ZZ{\mathbb Z}

\newcommand\pa{\partial}
\newcommand\Ker{\operatorname{Ker}}

\DeclareMathOperator{\Tr}{Tr}

\newcommand\B{\mathcal B}
\DeclareMathOperator{\Ch}{Ch}
\DeclareMathOperator{\Id}{Id}

\numberwithin{section}{chapter}
\numberwithin{equation}{chapter}

\makeindex

\DeclareMathOperator{\Ph}{Ph}
\newcommand{\PosConc}{\mathcal{P}^+}
\newcommand{\MF}{\mathrm{MF}}
\newcommand{\Pdo}{\mathrm{\Psi}}

\newcommand{\reals}{\mathbb{R}}
\newcommand{\complexs}{\mathbb{C}}
\newcommand{\naturals}{\mathbb{N}}
\newcommand{\integers}{\mathbb{Z}}
\newcommand{\rationals}{\mathbb{Q}}

\newcommand{\K}{\mathbb{K}}
\newcommand{\KK}{\mathbb{K}}

\DeclareMathOperator{\id}{id}
\newcommand{\boundary}[1]{\partial#1}

\newcommand{\abs}[1]{\left\lvert#1\right\rvert} %absolute value
\newcommand{\norm}[1]{\left\lVert#1\right\rVert}

\newcommand{\tensor}{\otimes}
\newcommand{\into}{\hookrightarrow}
\newcommand{\onto}{\twoheadrightarrow}
\newcommand{\iso}{\cong}

\newcommand{\disjointunion}{\sqcup}
\newcommand{\subgroup}{\leq}

\newcommand{\superset}{\supset}

\DeclareMathOperator{\Aut}{Aut}
\DeclareMathOperator{\Inn}{Inn}
\DeclareMathOperator{\Out}{Out}
\DeclareMathOperator{\Diffeo}{Diffeo}
\DeclareMathOperator{\supp}{supp}   %support
\DeclareMathOperator{\im}{im}      %image
\DeclareMathOperator{\vol}{vol}    %Volume
\DeclareMathOperator{\ord}{ord}    %order
\DeclareMathOperator{\End}{End}    %Endomorphisms
\DeclareMathOperator{\Hom}{Hom}    %Homomorphisms

\DeclareMathOperator{\spin}{spin}

\DeclareMathOperator{\Pos}{Pos} 
\DeclareMathOperator{\Riem}{Riem}

 \DeclareMathOperator{\tr}{tr}
\DeclareMathOperator{\TR}{TR}
\DeclareMathOperator{\coker}{coker}

\DeclareMathOperator{\ind}{ind}
\DeclareMathOperator{\Ind}{Ind}
\DeclareMathOperator{\sgn}{sgn}
\DeclareMathOperator*{\dirlim}{dirlim}
\DeclareMathOperator*{\invlim}{invlim}
\DeclareMathOperator{\rank}{rank}
\DeclareMathOperator{\vrank}{vrank}

\newcommand{\innerprod}[1]{\langle #1 \rangle}

\newcommand\SG{{\bf {\rm S}}}

\newcommand\Di{D\kern-6pt/}
\newcommand\cDi{{\mathcal D}\kern-6pt/}
\newcommand\spi{S\kern-6pt/}
\newcommand \cspi{\Sp\kern-6pt/}

\newcommand\CC{\mathbb C}
\newcommand\E{\mathcal E}

\newcommand\V{\mathcal V}
\newcommand\C{\mathcal C}

\newcommand\tM{\widetilde{M}}
\newcommand\tN{\widetilde{N}}
\newcommand\tE{\widetilde{E}}

\begin{document}

\frontmatter

\title{Mapping analytic surgery to homology, higher rho  numbers and metrics of positive scalar curvature}

%    Remove any unused author tags.

%    author one information
\author{Paolo Piazza}
\address{Dipartimento di Matematica, Sapienza Univerit\`a di Roma, Piazzale
  Aldo Moro 5, 00185 Roma, Italy}
%\curraddr{}
\email{paolo.piazza@uniroma1.it}
%\thanks{}

%    author two information
\author{Thomas Schick}
\address{Mathematisches Institut, Universit\"at G\"ottingen, Bunsenstr.~3,
  D-37073 G\"ottingen, Germany}
%\curraddr{}
\email{thomas.schick@math.uni-goettingen.de}
%\thanks{}

%    author three information
\author{Vito Zenobi}
\address{Istituto Nazionale di Alta Matematica (INdAM),  00185 Rome, Italy}
%\curraddr{}
\email{zenobi@altamatematica.it}
%\thanks{}

%    \date is required; it is the date received by the editor.
\date{xx.xx.xxxx}

%    Recognition of the 2010 edition of the Mathematics Subject
%    Classification requires a version of amsbook.cls from July 2009
%    or later.  If "2010" is not recognized, please upgrade.


%\dedicatory{Dedication text (use \\[2pt] for line break if necessary)}

\begin{abstract}
  	Let $\Gamma$ be a finitely 
	generated discrete group and let ${\tM}$  be a Galois
        $\Gamma$-covering of a smooth compact manifold  $M$. Let $u\colon M\to B\Gamma$ be the associated classifying map. Finally, let 
	$\SG_*^\Gamma ({\tM})$ be the analytic structure group, a K-theory group
        appearing in the Higson-Roe analytic surgery sequence $\cdots\to
          \SG_*^\Gamma({\tM})\to K_*(M)\to K_*(C_{red}^*\Gamma)\to\cdots$. 
	Under suitable assumptions on the group $\Gamma$ we construct two pairings,  
        first between $\SG^\Gamma_*(\widetilde
          M)$ and the delocalized part of
          the cyclic cohomology of $\CC\Gamma$ and
          secondly between $\SG^\Gamma_*(\tM)$ and the relative
          cohomology $H^{*-1}(M\to B\Gamma)$.
        Both are compatible with known pairings associated with the other terms in the
          Higson-Roe sequence.
	In particular, we define higher rho numbers associated to the rho class 
	$\rho(\widetilde{D})\in \SG_*^\Gamma ({\tM})$ of an invertible $\Gamma$-equivariant Dirac type operator
	on ${\tM}$.
	This applies, for example,    to the rho class $\rho(g)$ of a positive scalar curvature metric $g$ on $M$, if $M$ is spin.
	Regarding the first pairing, we establish in fact a more general result. Indeed, for an arbitrary 
	discrete group $\Gamma$   we prove  that it is possible to map the whole Higson-Roe analytic surgery sequence
to the  long exact sequence in even/odd-graded noncommutative de
        Rham homology,
     $\xymatrix{
		\cdots\ar[r]^(.4){j_*}& H_{*-1} (\mathcal{A}\Gamma)\ar[r]^{\iota}& H^{del}_{*-1} (\mathcal{A}\Gamma)\ar[r]^{\delta}& H^{e}_{*} (\mathcal{A}\Gamma)\ar[r]^{j_*}&\cdots
	}$
	with $\mathcal{A}\Gamma$ any dense holomorphically closed
	subalgebra  of $C^*_{red} \Gamma$. Here,
	$ H_{*}^{del} (\mathcal{A}\Gamma)$ and 
	$H_{*} ^{e}(\mathcal{A}\Gamma)$ denote versions of the delocalized homology and the homology localized
	at the identity element, respectively. 
In particular, we prove that there exists a  group homomorphism
	   $\SG_*^\Gamma({\tM})\to H^{del}_{*-1} (\mathcal{A}\Gamma)$. We then establish that under additional assumptions 
	   on $\Gamma$ 
	     there exists a
	   pairing between the delocalized part of
          the cyclic cohomology of $\CC\Gamma$ and
$H_{*}^{del} (\mathcal{A}\Gamma)$.\\
We give a detailed study of the action of the diffeomorphism group
on the analytic surgery sequence and on its homological counterpart, proving in particular 
precise formulae for the behaviour of the higher rho numbers under the action of the diffeomorphism
group. We use these formulae  in order to establish new lower bounds on the
size of the moduli space of metrics 
of positive scalar curvature.
\end{abstract}

\maketitle

\tableofcontents

%    Include unnumbered chapters (preface, acknowledgments, etc.) here.
%\include{}

\mainmatter
%    Include main chapters here.
\include{numeric_rho_intro}
\include{numeric_rho_sec2}
\include{numeric_rho_sec3}

\include{numeric_rho_sec4-15}
%\include

\appendix
%    Include appendix "chapters" here.
%\include{}

\backmatter
%    Bibliography styles amsplain or author-year (using natbib) are
%    also acceptable.

\include{numeric_rho_biblio}
%\bibliographystyle{amsalpha}
%\bibliography{}
%    See note above about multiple indexes.
\printindex

\end{document}

%% file: numeric_rho_intro.tex
\chapter{Introduction}

Let $(M,g)$ be a connected smooth compact Riemannian manifold without boundary. We denote by  $\Gamma$
the fundamental group $\pi_1 (M)$ and by  $\widetilde{M}$ 
the associated universal cover. For the time being we assume that $M$ is even dimensional.
Let $E=E^+\oplus E^-\to M$ be a $\ZZ_2$-graded bundle of Clifford modules and 
$$D= \begin{pmatrix}0&D^-\\
    		D^+&0\end{pmatrix}\colon C^\infty (M,E)\to C^\infty (M,E)$$
a self-adjoint generalized Dirac operator. Relevant examples are given by the signature operator
if we assume  that $M$ is orientable, and  the spin Dirac operator 
if, in addition, $M$ is spin. We can lift all the data defining $D$ to $\widetilde{M}$ and obtain in this 
way a $\Gamma$-equivariant Dirac operator $\widetilde{D}$.

One of the most basic invariant that can be associated to $D$ is the Fredholm index of $D^+$:
\begin{equation*}
\begin{aligned}
{\rm ind} (D^+)=&\dim \ker D^+ - \dim \coker D^+\\
=& \dim \ker D^+ - \dim \ker D^-\\
=& \Tr  \Pi^+ - \Tr \Pi^-
\end{aligned}
\end{equation*}
with $\Pi^\pm$ the orthogonal projections onto $ \ker D^\pm $. The Atiyah-Singer index theorem, one of the milestones of modern
mathematics, provides a geometric formula for this analytic invariant. 
One very useful expression for the index of $D^+$ is given in terms of the remainders associated to a parametrix $Q$:
if $$D^+ Q=\Id - S_-\,\quad Q D^+=\Id - S_+\quad\text{with}\quad 
S_\pm \in \Psi^{-\infty}(M,E^\pm),$$
where $\Psi^{-\infty}(M,E^{\pm})$ denotes the algebra of smoothing
pseudodifferential operators, then
\begin{equation}\label{calderon}
{\rm ind} (D^+)=\Tr S_+^N - \Tr S_-^N \quad \forall N\geq 1.
\end{equation}

We can also consider the $\Gamma$-index associated to $\widetilde{D}$, as in  \cite{Atiyah}. In order to define it we consider the
von Neumann algebra $\mathcal{A}_\Gamma$ of all bounded $\Gamma$-equivariant operators on $L^2 (\tM,\tE)$; the orthogonal projections onto $\Ker \widetilde{D}^\pm$, denoted $\Pi_\Gamma^\pm$, are elements in $\mathcal{A}_\Gamma$. The von Neumann
algebra $\mathcal{A}_\Gamma$
comes equipped with a canonical trace $\Tr_\Gamma$ and the orthogonal projections $\Pi_\Gamma^\pm$
are trace class with respect to this von Neumann trace. It thus  makes sense to define
\begin{equation*}
{\rm ind}_{\Gamma}  (\widetilde{D}^+):= \Tr_\Gamma  \Pi^+_\Gamma - \Tr \Pi^-_\Gamma
\end{equation*}
and a fundamental result of Atiyah asserts that 
\begin{equation*}
{\rm ind} (D^+)={\rm ind}_{\Gamma}  (\widetilde{D}^+).
\end{equation*}
A major step in the proof of this equality consists in showing that we can express the $\Gamma$-index of $\widetilde{D}^+$ 
in terms of the remainders associated to a suitable $\Gamma$-equivariant parametrix $\widetilde{Q}$ of $\Gamma$-compact support,
where 
a  $\Gamma$-equivariant pseudodifferential operator is defined to be of
$\Gamma$-compact support if 
the support of its Schwartz kernel  has compact image when projected in
$(\widetilde{M}\times \tM)/\Gamma$, where we divide by the diagonal action.
We denote the smoothing $\Gamma$-equivariant operators of $\Gamma$-compact support by $\Psi^{-\infty}_{\Gamma,c}(\tM,\tE)$
and recall  that if 
 $\widetilde{S}\in \Psi^{-\infty}_{\Gamma,c}(\tM,\tE)$ then $\widetilde{S}$
 has finite $\Gamma$-trace computed by
 $$\Tr_\Gamma (\widetilde{S})=\int_{\mathcal{F}} \tr_p K_{\widetilde{S}} (p,p) \,{\rm dvol}$$
 with $K_{\widetilde{S}}$ the smooth integral kernel of $\widetilde S$ and $\mathcal{F}$ a fundamental domain for the action of $\Gamma$ on $\tM$.
 
 Going back to the $\Gamma$-index, it is shown by Atiyah in \cite{Atiyah} that for every parametrix $\widetilde Q$
 such that 
 $$\widetilde{D}^+ \widetilde{Q}=\Id - \widetilde{S}_-\,\quad \widetilde{Q} \widetilde{D}^+=\Id - \widetilde{S}_+\quad\text{with}\quad 
 \widetilde{S}_\pm \in \Psi^{-\infty}_{\Gamma,c}(\tM,\tE^\pm)$$ one can
 compute the $\Gamma$-index by
\begin{equation}\label{calderon-gamma}
{\rm ind}_{\Gamma}  (\widetilde{D}^+)=\Tr_\Gamma \widetilde{S}_+^N - \Tr_\Gamma \widetilde{S}_-^N \quad \forall N\geq 1.
\end{equation}

We can recover these invariants using  $K$-theory. Let $Q$ be a parametrix for $D$ with remainders $S_\pm\in\Psi^{-\infty}(M,E^\pm)$.
Consider the $2\times 2$ matrix
$$P:= \left(\begin{array}{cc} S_{+}^2 & S_{+}  (I+S_{+}) Q\\ S_{-} D^+ &
I-S_{-}^2 \end{array} \right).
$$
This is an idempotent matrix over the unitalization of $\Psi^{-\infty}(M,E)$;
also $$e_1:=\left( \begin{array}{cc} 0 & 0 \\ 0&1_{E^-}
\end{array} \right)$$
is an idempotent matrix in the unitalization and we define the index class associated to $D$ as
\begin{equation}\label{ind-cs}\operatorname{Ind} (D):= [P] - [e_1]\in K_0
  (\Psi^{-\infty}(M,E)).
\end{equation}
This definition comes from the short exact sequence of algebras
$$0\to \Psi^{-\infty}(M,E)\to \Psi^0 (M,E)\to \Psi^0 (M,E)/\Psi^{-\infty}(M,E)\to 0$$
and the associated long exact sequence in K-theory: the K-theory class \eqref{ind-cs} is nothing but
the boundary map
applied to the $K_1$-class given by the bounded transform of $D$ (which is indeed invertible in the quotient).
See \cite{ConnesMoscovici} for details. The  functional analytic trace $\Tr$
defines a cyclic cocycle of degree $0$ on the algebra  $ \Psi^{-\infty} (M,E)$
and from the pairing of K-theory with cyclic cohomology we have a well-defined number
$\langle \Ind (D), [\Tr] \rangle$; from the expression of $\Ind (D)$ and formula \eqref{calderon} we deduce
$$\langle \Ind (D), [\Tr] \rangle= \Tr S_+^2 - \Tr S_-^2=\ind (D^+).$$
Similarly, we have an index class $$\Ind_{\Gamma,c} (\widetilde{D}^+):= [\widetilde{P}]-[e_1]\in K_0 (\Psi^{-\infty}_{\Gamma,c} (\tM,\tE))$$
defined using a parametrix of $\Gamma$-compact support for $\widetilde{D}$; the $\Gamma$-trace $\Tr_\Gamma$ defines an element
in $HC^0 (\Psi^{-\infty}_{\Gamma,c} (\tM,\tE))$  and we get
$$\langle \Ind_{\Gamma,c} (\widetilde{D}), [\Tr_\Gamma] \rangle= \Tr \widetilde{S}_+^2 - \Tr \widetilde{S}_-^2=\ind_\Gamma (D^+)$$
where we have now used \eqref{calderon-gamma}.

\medskip
Summarizing:  we have recovered the index  by pairing the  index {\it class} 
$\Ind (D)\in K_0 (\Psi^{-\infty}(M,E))$
with a 
0-degree cyclic cocycle $[\Tr]\in HC^0 (\Psi^{-\infty}(M,E))$ and we have recovered 
the $\Gamma$-index by pairing the index class $\Ind_{\Gamma,c}
 (\widetilde{D})\in K_0 (\Psi^{-\infty}_{\Gamma,c} (\tM,\tE))$ with the 
 0-degree cyclic cocycle $[\Tr_\Gamma]\in HC^0 (\Psi^{-\infty}_{\Gamma,c} (\tM,\tE))$.

\medskip
We are led naturally to define {\it higher indices} by pairing the index class
\begin{equation*}
  \Ind_{\Gamma,c} (\widetilde{D})\in K_0 (\Psi^{-\infty}_{\Gamma,c} (\tM,\tE))
\end{equation*}
with {\it higher cyclic cocycles} for the algebra $\Psi^{-\infty}_{\Gamma,c} (\tM,\tE)$, that is with elements
in $HC^{{\rm even}} ( \Psi^{-\infty}_{\Gamma,c} (\tM,\tE))$.
One reason to pass directly to the index class associated to $\widetilde{D}$ and disregard the index class associated to 
$D$ is that the K-theory and the cyclic cohomology of $\Psi^{-\infty}_{\Gamma,c} (\tM,\tE)$ are much richer than the 
corresponding groups for $ \Psi^{-\infty} (M,E)$.
For example, elements in $H^k (\Gamma,\CC)$
define in a natural way elements in $HC^k (\CC\Gamma)$ which in turn define elements in $HC^k (\Psi^{-\infty}_{\Gamma,c} (\tM,\tE))$. 
The Connes-Moscovici higher index theorem
\cite{ConnesMoscovici} provides a geometric formula for these higher indices; this is of great interest both intrinsically and for geometric applications.

Another important reason to consider operators on $\tM$ is the following:  for geometric applications, one is eventually  interested in the image of these index classes in the K-theory of the $C^*$-closures of these algebra. 
Now, the $C^*$-closure
of $\Psi^{-\infty}(M,E)$ in $\mathcal{B}(L^2 (M,E))$ is equal to the compact operators, whose K-theory is well known, equal to $\ZZ$ in even degree
and equal to 0 in odd degree; this means that the {\it $C^*$-index class} on $M$ retains essentially the same information as the numeric  index.
The K-theory of the $C^*$-closure of $\Psi^{-\infty}_{\Gamma,c} (\tM,\tE)$ in $\mathcal{B}(L^2 (\tM,\tE))$ is, on the other hand,  much richer, as it is  isomorphic to the K-theory of $C^*_{red} \Gamma$,
the reduced $C^*$-algebra of $\Gamma$.
The $C^*$-closure of $\Psi^{-\infty}_{\Gamma,c} (\tM,\tE)$ in $\mathcal{B}(L^2 (\tM,\tE))$
  is called the {\it Roe algebra} of $\tM$ and it is denoted by $C^* (\tM,\tE)^\Gamma$; 
we denote the index class in the Roe algebra of $\tM$ by $\Ind_{\Gamma}(\widetilde{D})$. Thus
$$\Ind_\Gamma (\widetilde{D})\in K_0 (C^* (\tM,\tE)^\Gamma)=K_0
(C^*_{red}\Gamma). $$

At this point, it is important to point out  that the higher indices we have defined  are {\it always} defined for the compactly supported index class
$\Ind_{\Gamma,c}(\widetilde{D})$ and computable through
the Connes-Moscovici formula.
However, the most relevant geometric properties of the index cannot be derived 
from this index class but only from the higher $C^*$-indices, that is higher
indices 
associated to $\Ind_\Gamma (\widetilde{D})\in K_0 (C^*
(\tM,\tE)^\Gamma)$. This refers in particular to vanishing results: the higher
$C^*$-index of the Dirac operator of a spin manifold vanishes if the scalar
curvature is positive. Similarly the difference of the higher $C^*$-indices of
the signature operators of two homotopy equivalent manifolds vanish. The
corresponding statements in general are not known to be true for the compactly supported
index class $\Ind_{\Gamma,c}(\widetilde{D})$.  The vanishing results follow
from the fact that positivity implies invertibility in $C^*$-algebras (but
this is not true for general $*$-algebras).

To obtain numerical $C^*$-algebraic indices, we 
need an additional argument (and possibly additional hypothesis) ensuring that the relevant pairing extends
from the algebra $\Psi^{-\infty}_{\Gamma,c} (\tM,\tE)$ to its $C^*$-closure, that is, to the Roe algebra $C^* (\tM,\tE)^\Gamma$.  
Equivalently, if we consider the index class as an element  in $K_0 (C^*_{red}\Gamma)$ we shall need additional arguments
and possibly additional hypothesis ensuring that we can {\it extend} cyclic cocycles for the algebra $\CC\Gamma$ to the $C^*$-algebra
$C^*_{red}\Gamma$.
We shall come back to this crucial
point later in this introduction.

\medskip
To give an example of this circle of ideas let us consider the flat $n$-torus $T^n$, with $n=2k$.
We know that on a spin manifold the index of the spin Dirac operator, which by
the Atiyah-Singer index theorem is the $\widehat{A}$-genus,  is an obstruction 
to the existence of a metric of positive scalar curvature. The
$\widehat{A}$-genus of the flat torus vanishes (because its tangent bundle is trivial)
and so we find ourselves without valuable information as to whether $T^n$ admits or not a metrics of positive scalar curvature.
Now $\pi_1 (T^n)=\ZZ^n$, $C^*_{red}\ZZ^n= C(\widehat{T}^n)$, with $\widehat{T}^n$ the dual torus
and $H^* (\ZZ^n,\CC)\simeq \Lambda^* \CC^n$. For this $C^*$-algebra it is indeed possible to extend the cyclic cocycles
coming from $H^* (\ZZ^n,\CC)$ and a direct application of the Connes-Moscovici
higher index formula shows that the higher index of the spin Dirac operator
associated to the generator of $H^n (T^n)$ is equal to a non-zero multiple of the volume of $T^n$ and thus different from zero.
This implies of course that the index class $\Ind_{\ZZ^n} (\widetilde{D}^+)$ is different from 0  and we conclude that $T^n$ does not admit
a metric of positive scalar curvature; this conclusion has been made possible
by passing from the lower index $\ind (D^+)$ to a higher index.

\medskip
So far, we have restricted ourselves to the even dimensional case, but it is also possible to introduce the higher index class 
in the odd dimensional case, obtaining
$$\Ind_{\Gamma,c} (\widetilde{D})\in K_1 (\Psi^{-\infty}_{\Gamma,c} (\tM,\tE)).
$$
Higher numerical indices are then obtained by pairing with $HC^{{\rm odd}} ( \Psi^{-\infty}_{\Gamma,c} (\tM,\tE))$.
As in the even dimensional case, it is then even more important to consider
$$\Ind_\Gamma (\widetilde{D})\in K_1 (C^* (\tM,\tE)^\Gamma)$$
and the possibility of extracting numerical invariants out of it which is indeed granted, as for even dimensions,
 under additional hypothesis on $\Gamma$. We shall explain this momentarily.

\medskip
Despite their many interesting applications, the index invariants we have introduced will
be inadequate whenever we have a geometric situation in which the index classes vanish.
This is the case, for example, if $D$ is the spin Dirac operator associated to a metric of positive scalar curvature
or if $D$ is the signature operator on the disjoint union of two homotopy
equivalent manifolds with opposite orientation.
In geometric questions involving, for example, the diffeomorphism type of homotopy equivalent manifolds
or the moduli space of metrics of positive scalar curvature $\mathcal{R}^+ (M)/{\rm Diffeo}(M)$,
one is led to
consider {\it secondary}  invariants, the rho invariants associated to $D$ and $\widetilde{D}$.
In their basic version, analogous to the numeric indices $\ind(D^+)$ and $\ind_\Gamma (\widetilde{D}^+)$, these are defined 
for a self-adjoint Dirac operator and we list the three most important versions:
\begin{itemize}
\item the Atiyah-Patodi-Singer rho invariant $\rho_{\alpha-\beta}(D)$ associated to two 
unitary representations $\alpha, \beta\colon \pi_1 (M)\to U(\ell)$ as
introduced in \cite{APS2};
\item the Cheeger-Gromov $L^2$-rho invariant $\rho_{(2)} (\widetilde{D})$ of \cite{CheegerGromov};
\item the delocalized eta invariant of Lott $\eta_{\langle x \rangle} (\widetilde{D})$
associated, for the time being, to a finite conjugacy class $\langle x \rangle$ in $\Gamma=\pi_1 (M)$
and for $\widetilde{D}$ assumed here to be invertible, as constructed in \cite{Lott_delocalized}.

\end{itemize}
Results of Botvinnik-Gilkey \cite{BotvinnikGilkey_eta}, Chang-Weinberger
\cite{ChangWeinberger}, Piazza-Schick \cite{PiazzaSchick_PJM} and subsequent work show that these invariants
are extremely useful in the geometric questions we have presented above, that is,  the problem of distinguishing the 
diffeomorphism type
of two homotopically equivalent manifolds or the problem of distinguishing
components of the moduli space of metrics of positive scalar curvature,
i.e.~elements of $\pi_0 (\mathcal{R}^+ (M)/{\rm Diffeo}(M))$.

We now come to a fundamental question, which is the analogue of the passage from the numeric {\it lower} index invariants $\ind (D^+)$
and $\ind_\Gamma (\widetilde{D}^+)$ to the {\it higher} indices associated to the  index classes $$\Ind_{\Gamma,c} (\widetilde{D})\in K_* (\Psi^{-\infty}_{\Gamma,c} (\tM,\tE))\quad\text{and}\quad \Ind_\Gamma (\widetilde{D})\in K_* (C^* (\tM,\tE)^\Gamma)$$ 
and the cyclic cocycles for the algebra $\CC\Gamma$:

\medskip
\noindent
{\bf Question:}  {\it can we define a rho class in the K-theory of a suitable $C^*$-algebra and produce the lower 
rho invariants introduced above as well as new, higher numeric rho invariants by pairing this rho class with suitable cyclic
cocycles?}

\medskip
Notice that these are really two distinct questions. We ask whether it is possible to

\begin{enumerate}
\item produce a rho class in the K-theory of a suitable $C^*$-algebra;
\item pair this rho class with suitable cyclic cocycles so as to define and compute higher rho invariants.
\end{enumerate}

The answer to the first question is positive and is fundamental work of Higson and Roe.
Let us expunge the vector bundle $E$ from the notation and let us denote the Roe $C^*$-algebra by
$C^* (\tM)^\Gamma$.
In their seminal paper \cite{HigsonRoe1}  Higson and Roe construct a long exact sequence of K-theory groups
\begin{equation}\label{hr-intro}
\cdots \to  K_{*+1}(C^*(\tM)^\Gamma)\to K_{*+1}(D^*(\tM)^\Gamma)\to
K_{*+1}(D^*(\tM)^\Gamma/C^*(\tM)^\Gamma)\to \cdots
\end{equation}
We call this sequence the \emph{Higson-Roe exact sequence} or the
\emph{Higson-Roe analytic surgery sequence}. {We shall explain in a moment why the word surgery 
is employed in this context.}

\medskip
This sequence organizes and captures in a very elegant way the interplay of  primary versus secondary
index  information of Dirac operators defined on $M$ and  $\tM$. 
As already explained, first of all we have  a canonical isomorphism
\begin{equation*}
K_{*}(C^*(\tM)^\Gamma)\iso
K_{*}(C^*_{red}\Gamma)\,.
\end{equation*}
Next we have that
\begin{equation}\label{eq:Paschke}
  K_{*+1}(D^*(\tM)^\Gamma/C^*(\tM)^\Gamma)\iso K_* (M)\,,
\end{equation}
the K-homology of $M$. Here, $K$-homology is the homology theory dual to
topological K-theory, and the isomorphism \eqref{eq:Paschke}  is called {\it Paschke duality}. 
Moreover,  under these isomorphisms the connecting homomorphism of the long exact sequence \eqref{hr-intro}, namely
$$K_{*+1}(D^*(\tM)^\Gamma/C^*(\tM)^\Gamma)\to K_*(C^*(\widetilde{
M})^\Gamma),$$ becomes the Baum-Connes assembly map $$K_*(M)=K^\Gamma_*(\tM)\xrightarrow{\mu}
K_*(C^*_{red}\Gamma).$$ Setting 
$$\SG^\Gamma_* ({\tM}):= K_{*+1}(D^*(\tM)^\Gamma),$$
we can rewrite the Higson-Roe surgery sequence as
\begin{equation*}
\cdots \to  K_{*+1}(C^*_{red}\Gamma)\xrightarrow{s} 
\SG^\Gamma_* ({\tM})\xrightarrow{c} K_{*}(M)\xrightarrow{\mu}  K_{*}(C^*_{red}\Gamma)\to
\cdots
\end{equation*}

\smallskip
We know that the group $K_*(C^*_{red}\Gamma)$ is the home of index invariants associated to $\Gamma$-equivariant geometric operators
$\tM$. 

\smallskip
The group $\SG^\Gamma_* ({\tM})$ is called the \emph{analytic structure group}
by Higson and Roe and it is the home of the K-theoretic secondary invariants we are looking for.
In particular, if
$M$ is a spin manifold with a positive scalar curvature metric $g$, one obtains
a secondary invariant $\rho(g)\in \SG^\Gamma_{\dim M} ({\tM})$, called the rho class of $g$, which
captures information about the positive scalar curvature metric $g$, and
allows often to distinguish different bordism classes of such metrics.
Similarly, if $N\xrightarrow{f} M$ is an oriented homotopy equivalence, then 
we have a rho class $\rho(f)\in  \SG^\Gamma_{\dim M} ({\tM})$ and this class 
can distinguish triples $N^\prime \xrightarrow{f^\prime} M$ and $N\xrightarrow{f} M$
that are not h-cobordant or, put it differently, triples that define distinct elements in the
manifold structure set $\mathcal{S} (M)$.

\medskip
More generally,  the rho class can be defined for any $L^2$-invertible $\Gamma$-equivariant
Dirac operator $\widetilde{D}$ on ${\tM}$, $\rho(\widetilde{D})\in  \SG^\Gamma_{\dim M} ({\tM})$;
in fact, we can allow for invertible smoothing perturbations of such operators, see \cite{PiazzaSchick_sig}.

\medskip
The construction of the K-theoretic rho classes follows the following
abstract principle, which is explained in more detail and with more examples in
  \cite{Schick_ICM}.
\begin{enumerate}
\item The geometry gives as a fundamental class $F\in K_*(M)$. In the example
  of a spin manifold, $F=[D]$ is indeed the K-homology fundamental class of the
  Dirac operator $D$ (or, formulated in topological terms, of the 
  spin structure). 
\item The map $\mu$ is an index map. It maps the fundamental class $F$ to the
  K-theoretic index class $\Ind_\Gamma(\widetilde D)\in K_* (C^*_{red}(\Gamma))$. In the example of a spin
  manifold, this is indeed the higher index class of the Dirac operator. In
  case of an oriented manifold, we consider the index class of the signature
  operator. 
\item We have the following observation: by the exactness of
  the Higson-Roe sequence, $\Ind_\Gamma(\widetilde D)=0$ if and only if there
  are classes $a\in \SG^\Gamma_* ({\tM})$ with $c(a)=F$. However, the
  vanishing of $\Ind_\Gamma(\widetilde D)$ by itself does not yet determine a
  specific lift $a$ of the fundamental class. 
\item {What is needed, additionally, is a \emph{geometric reason}} for the
  vanishing of the index. In favourable situations, as we will describe them
  now, the geometry not only will imply the vanishing of the index class, but
  will give rise to a specific lift $\rho\in \SG^\Gamma_*(\tM)$ of $F$
  (i.e.~$c(\rho)=F$). 
\item This arises as follows: in our definition of the Higson-Roe exact
  sequence, $K_*(M)$ is defined as the K-theory of the quotient  $D^*(\tM)^\Gamma/C^*(\tM)^\Gamma$ of two
  $C^*$-algebras. The fundamental class is explicitly given by
  an element $F$ in 
  this quotient algebra, which is indeed obtained from the underlying
  geometric operator (of Dirac type). Note that, to define a K-theory class, the
  element $F$ has to be invertible (to define a $K_1$-class) or idempotent (to
  define a $K_0$-class), and this property for general geometric situation is
  only satisfied in $D^*(\tM)^\Gamma/C^*(\tM)^\Gamma$.

 However, if one has additional special geometry (in particular, positive
 scalar curvature in the case of the Dirac operator), then the operator $F$
 already is invertible/an idempotent in the
  algebra $D^*(\tM)^\Gamma$ and therefore defines a class in its K-theory. Slightly more
  generally, the special geometry (in particular,
  a homotopy equivalence for the signature operator on the disjoint union of
  the two homotopy equivalent manifolds) gives a canonical
  perturbation of $F$ which is invertible/an idempotent in
   $D^*(\tM)^\Gamma$. 

  In any case, this invertible/idempotent now defines a class in the
  K-theory of  $D^*(\tM)^\Gamma$, i.e.~in $\SG^\Gamma_*(\tM)$. By definition, this is the
  $\rho$-class we want to construct. By its very construction, it is a lift of
  the fundamental class $F$ and it depends on and therefore potentially
  contains information about the geometric reason for the vanishing of the
  index class.
\end{enumerate}

As already mentioned, there are two main domains where the Higson-Roe exact
sequence is used: signature invariants of orientable manifolds, as part of surgery
theory to classify smooth manifolds, and classifications of metrics of
positive scalar curvature.
The first domain is organized in the celebrated \emph{surgery exact sequence} of 
Browder, Novikov, Sullivan and Wall, and the
connection between topology and analysis is made systematic by the construction of a map of exact sequences,
from the surgery exact sequence in topology to the Higson-Roe analytic surgery  sequence. This explains the word surgery in the
exact sequence of Higson and Roe. This is
carried out originally in \cites{HigsonRoe1,HigsonRoe2,HigsonRoe3} using
analytic homological algebra. A more direct construction of this map, based on 
certain higher 
index theorems,  is given in
\cite{PiazzaSchick_sig}.

This second approach is in fact inspired by the precursor
work \cite{PiazzaSchick_psc}, where the Stolz exact sequence for positive
scalar curvature metrics (compare \cite[{Proposition 1.27}]{PiazzaSchick_psc}) is mapped to the Higson-Roe exact
sequence. The Stolz exact sequence is the analogue, in the world of positive scalar curvature metrics,
of the surgery exact sequence in differential topology;
it  organizes questions about
existence and (bordism) classification of Riemannian metrics of positive scalar
curvature in a long exact sequence of groups, but where the main terms are
somewhat hard to analyze directly.\\
The results of Piazza and Schick about the Stolz surgery sequence \cite{PiazzaSchick_psc} 
were established in odd dimensions.
Subsequently, 
Xie and Yu gave a different treatment of these results, valid in all
dimensions and based on Yu's localization algebras (see \cite{Yu_localization}, \cite{XY-advances}).
The even dimensional case of the results of Piazza-Schick was later treated by
Zenobi in \cite{zenobi-JTA}, using suspension. Yet  different treatments of these results were given by Zeidler in
\cite{zeidler-JTOP} and by Zenobi in \cite{Zenobi}.

\medskip
Both for the surgery sequence in topology and for the Stolz sequence for positive scalar curvature metrics,
the map to the Higson-Roe exact sequence allows to apply all
the tools from $C^*$-algebraic index theory and $C^*$-algebra K-theory to
obtain information about the geometrically defined sequences. 
In this toolbox
we find, in particular, homological methods.

This has a long history for primary index invariants, as we have seen 
earlier in this introduction, thanks in particular to the seminal work of Connes and Moscovici on the higher index formula and its 
applications to the Novikov conjecture
\cite{ConnesMoscovici}. As already hinted to, this geometric application  is made
possible only if we are able to {\it extend} the cyclic cocycles associated to elements in $H^* (\Gamma,\CC)$ from
$\CC\Gamma$ to $C^*_{red}\Gamma$, defining in this way higher $C^*$-indices.
This crucial step is achieved   by Connes and Moscovici
under the additional assumption that   $\Gamma$ is Gromov hyperbolic; the proof of the extendability property for
cyclic cocycles defined by group cocycles of Gromov hyperbolic groups is in fact rather subtle.

\smallskip
Lott, inspired by the work of Bismut on the family index theorem,
gave a different treatment of the Connes-Moscovici theorem, both for defining the higher indices
and for computing them. This approach employs in a crucial way the Chern character 
homomorphism of Karoubi \cite{Karoubi}, from  $K_*(C^*_{red}\Gamma)\equiv K_* (\mathcal{A}\Gamma)$
to the noncommutative de Rham homology  groups $H_* (\mathcal{A}\Gamma)$, with $\mathcal{A}\Gamma$
a suitable smooth subalgebra of $C^*_{red}\Gamma$. 
Here we recall that the noncommutative de Rham homology  groups
of an algebra are subgroups of its cyclic homology groups.
Lott establishes an explicit  formula  for the Chern character
of the index class
and this formula is valid without any assumption on $\Gamma$. 
Then, in a second
stage, one can  construct pairings with group cocycles of $\Gamma$
by employing a classic result of Burghelea which identifies the group
cohomology $H^*(\Gamma,\CC)$ with a subgroup of the cyclic cohomology of $\CC\Gamma$.
 This way,
one obtains numerical invariants for the K-theory classes. For these
pairings, one  needs assumptions on the group like Connes and Moscovici, for
example that $\Gamma$ is Gromov hyperbolic or of polynomial growth.

\smallskip
We emphasise the following conclusion from the above discussion: by  
mapping the index class first to noncommutative de Rham homology and {\it then} pairing the result with cyclic cohomology (this second step 
under additional assumption on $\Gamma$)
it has been possible to
define  {\it higher  indices} for a $\Gamma$-equivariant  Dirac operator on ${\tM}$. We shall soon make use of this 
fundamental observation.
\medskip

Let us go back to secondary invariants. Having explained how Higson and Roe solved 
in a positive way the existence of a rho class, thus giving a positive answer to the first question raised above,
we can now 
concentrate on the second question,  which  we reformulate in the following more precise way:

\medskip
\noindent
{\bf Question:}
	{\it can one define {\it higher rho numbers}  associated  to the rho class 
	$\rho(\widetilde{D})\in \SG^\Gamma_{\dim M} ({\tM})$ of an invertible $\Gamma$-equivariant Dirac operator?}

\medskip
\noindent
We shall give a positive answer to this question in two  ways. First, following Burghelea
  \cite{Burghelea}, denote by $HC^*(\mathbb{C}\Gamma;\langle
  x\rangle)$ the factor of the cyclic cohomology of
  $\mathbb{C}\Gamma$ supported on a conjugacy class $\langle x\rangle$ of
  $\Gamma$. If $\Gamma$ is Gromov hyperbolic, we prove that there is a pairing  
\begin{equation}\label{intro:pairing-del}
(\cdot,\cdot)\,:\,  HC^* (\mathbb{C}\Gamma;\langle x\rangle)\times
\SG^\Gamma_{*} ({\tM})\rightarrow \mathbb{C}\quad\forall x\ne e.
\end{equation}
A construction similar to this pairing is carried out independently in 
\cite{ChenWangXieYu}. We shall compare more thoroughly the two papers at the end of this introduction.

As a second answer to our basic question we also prove in this paper that 
under different assumptions on $\Gamma$, satisfied for example by groups that are Gromov hyperbolic 
or of polynomial growth,
there is a well defined pairing
\begin{equation}\label{intro:pairing-rel}
(\cdot,\cdot)_{\rm rel}\,:\,  H^*(M\xrightarrow{u} B\Gamma)\times  \SG^\Gamma_{*} ({\tM})\rightarrow \mathbb{C}\,,
\end{equation}
where $H^*(M\xrightarrow{u} B\Gamma)$ denotes the relative singular homology of the map $u\colon M\to B\Gamma$. 

\medskip
We shall now put these two results in context. The first result should be thought of as the analogue of the
Connes-Moscovici pairing but done here for the group $\SG^\Gamma_{*} ({\tM})$,
which is the recipient of secondary invariants, instead of $K_{*}(C^*_{red}\Gamma)$,
which is  the recipient of primary invariants.\\
Following  Lott's approach to the Connes-Moscovici
higher index theorem, {\it we shall in fact solve a much more general problem}. Indeed, without any additional
assumption on the group $\Gamma$ we shall map the
whole Higson-Roe analytic surgery sequence to a  long exact sequence in non-commutative 
de Rham homology:
\begin{equation}\label{intro:dR-sequence}
\xymatrix{
		\cdots\ar[r]& K_{*-1}(C^*_{red}\Gamma)\ar[r]^{s}\ar[d]^{\Ch_\Gamma}& \SG^\Gamma_{*} ({\tM}) \ar[r]^{c}\ar[d]^{\Ch_{\Gamma}^{del}}& K_*^{\Gamma}(\tM)\ar[r]\ar[d]^{\Ch^e_\Gamma}&\cdots\\
	\cdots\ar[r]^(.4){j_*}& H_{*-1}(\mathcal{A}\Gamma)\ar[r]^{\iota}& H_{*-1} ^{del}(\mathcal{A}\Gamma)\ar[r]^{\delta}& H_{*}^{e}(\mathcal{A}\Gamma)\ar[r]^{j_*}&\cdots}
\end{equation}
with $\mathcal{A}\Gamma$ any dense holomorphically closed
subalgebra  of $C^*_{red} \Gamma$.
In a second stage, building on work of Puschnigg and also Meyer, we pair the ``delocalized parts'' of the cyclic
cohomology of $\mathbb{C}\Gamma$, $HC^*_{\rm del} (\mathbb{C}\Gamma ;\langle x\rangle)$, with elements in $H_{*}^{del} (\mathcal{A}\Gamma)$, the ``delocalized part'' of the
noncommutative
de Rham homology of a (suitable) smooth subalgebra $\mathcal{A}\Gamma$ of $C^*_{red}\Gamma$.
It is only at this point
that we need to make additional assumptions on the group $\Gamma$, for example that $\Gamma$ is Gromov hyperbolic.
In this case the smooth subalgebra $\mathcal{A}\Gamma$ is the one constructed by Puschnigg in \cite{Puschnigg}. Thus, under this additional assumption we show  that there is a well defined pairing
\begin{equation}\label{intro:H-pairing}
\langle \cdot, \cdot \rangle \,: HC^* (\mathbb{C}\Gamma;\langle x\rangle)\times H_{*}^{del} (\mathcal{A}\Gamma)
\rightarrow \mathbb{C}\quad\forall e\ne x\in\Gamma
\end{equation}
and we can therefore define  the higher rho number associated to the rho class  $\rho (\widetilde{D})$ and a class $\tau \in  
HC^*_{{\rm del}} (\mathbb{C}\Gamma;\langle x\rangle)$ as

\begin{equation}\label{intro:pairing-del-bis}
\rho_\tau (\widetilde{D}):= \langle \tau, \Ch^{del} (\rho (\widetilde{D})) \rangle.
\end{equation}
Notice that the existence of the pairing \eqref{intro:H-pairing} has an
intrinsic interest that goes beyond the problem of defining higher rho numbers.

As an easy application of our analysis we prove that by pairing the rho class  $\rho (\widetilde{D})$ with the delocalized trace associated to a non-trivial conjugacy class
$\langle g\rangle$, 
$$\tau_{\langle g\rangle} (\sum_\gamma \alpha_\gamma \gamma):= \sum _{\gamma\in \langle g\rangle } \alpha_{\gamma}\,,$$
we obtain precisely the delocalized eta invariant of Lott (\cite{Lott2},\cite{Puschnigg}):
\begin{equation}\label{intro:pairing-del-ter}
\langle \tau_{\langle g\rangle } , \Ch^{del} (\rho (\widetilde{D})) \rangle = \eta_{\langle g\rangle }  (\widetilde{D})\,.
\end{equation}
Put it differently,
\begin{equation}\label{intro:pairing-del-quater} \eta_{\langle g\rangle }  (\widetilde{D})= 
\rho_{\tau_{\langle g\rangle }} (\widetilde{D})
\end{equation}
which expresses the delocalized eta invariant of Lott as a rho number obtained from the general rho class 
$\rho (\widetilde{D})$. This result is also discussed in \cite{ChenWangXieYu}.

  Predecessors of our work are \cite{HigsonRoe4}, where the relative eta
  invariants $\rho_{\alpha-\beta} (D)$ of Atiyah-Patodi-Singer associated to a virtual representation of
  $\Gamma$ of dimension $0$ are obtained  as pairings of the rho class in 
  $\SG^\Gamma_*(\tM)$ with a trace, i.e.~a cyclic cohomology
  class of degree $0$. 
 For recent interesting results on the Atiyah-Patodi-Singer rho invariant 
  see  \cite{AAS1, AAS2}.
  Similarly, \cite{BenameurRoy} identifies the Cheeger-Gromov
  $L^2$-rho invariant $\rho_{(2)}(\widetilde{D})$ through a pairing of the rho class in $\SG^\Gamma_*(\tM)$  with a
  trace (the difference of the canonical $L^2$-trace and the trivial
  representation), i.e.~a cyclic cocycle of degree $0$. 

  These cocycles a priori are defined for a version of the Higson-Roe sequence
  where one uses a slightly larger completion of $\complexs\Gamma$ than
  $C^*_{red}\Gamma$. The maximal completion $C^*_{max}\Gamma$ can of course
  also be used. For groups where the canonical map $C^*_{max}\Gamma\to
  C^*_{red}\Gamma$ induces an isomorphism in K-theory, this fits directly in
  the context of our paper and is a very special case (with cocycles of degree
  $0$ which are defined on the whole $C^*$-algebra instead of a smooth
  subalgebra).
  
  We stress the consequence that  thanks to the present work and these  
  contributions of Higson-Roe and Benameur-Roy, the three
  lower rho invariants introduced earlier in this introduction, that is
  $$\rho_{\alpha-\beta} (D),\quad \rho_{(2)}(\widetilde{D}),\quad \eta_{<x>}(\widetilde{D})$$
  have all been obtained through the rho class $\rho (\widetilde{D})\in \SG^\Gamma_*(\tM)$
  and its paring with suitable 0-cocycles.

More generally, using our results,
the delocalized Atiyah-Patodi-Singer index theorem in K-theory \cites{PiazzaSchick_psc,PiazzaSchick_sig,Zenobi} and the higher Atiyah-Patodi-Singer index theorem in noncommutative de Rham homology \cite{LeichtnamPiazzaMemoires},
we are able to show (if $M$ is a boundary) that
\begin{equation}\label{intro:equality-with-Lott}
\Ch^{del} (\rho (\widetilde{D}))=-\frac{1}{2}\varrho_{{\rm Lott}} (\widetilde{D})\quad\text{in}\quad H^{del}_{*} (\mathcal{A}\Gamma),
\end{equation}
where on the right hand side the higher rho invariant of Lott appears. This solves a question raised by Lott in
\cite[Remark 4.11, (3)]{Lott2}. 
Notice that the compatibility with Lott's treatment is also tackled  in
\cite[Section 8]{ChenWangXieYu}  where it is established
 for the  corresponding higher rho 
numbers  (under an assumption of polynomial growth of the group, but without
the assumption that $M$ is a boundary).

Contributions on higher rho numbers associated to 2-cocycles are given in
 \cite{BarcenasZeidler}; see also \cite{AW}.

\smallskip
The mapping of the Higson-Roe analytic surgery sequence to a homological sequence was a program that
had been already proposed in \cite{PiazzaSchick_psc}. One might wonder why it took so long to achieve it. The problem 
is that the algebra $D^* ({\tM})^\Gamma$ is ``very large''; what made the difference for the present work
is the description given by the third author of the surgery sequence through the adiabatic groupoid \cite{Zenobi},
involving much smaller algebras.
Here we use a slight variation of this description, see
\cite{Zenobi_compare}. In \cite{Zenobi_compare} the compatibility
between the Higson-Roe  approach and the approach via groupoids is carefully explained.

\smallskip
The mapping of the Higson-Roe surgery sequence to homology already appears in the work of Deeley and Goffeng, 
see  \cite{DeeleyGoffeng3}. However, their solution passes through the geometric realisation of the Higson-Roe sequence, where geometric means \emph{\`a la} Baum-Douglas; more precisely, in the work of Deeley-Goffeng one 
maps out of the Higson-Roe surgery sequence by   inverting the isomorphisms going from the geometric 
realization of the Higson-Roe sequence to the true Higson-Roe surgery sequence. Inverting explicitly these isomorphisms,
especially for the part involving $\SG^\Gamma_{*} ({\tM})$, is a difficult problem.
Our Chern characters, on the other hand,  are computable by direct calculations. It remains an interesting problem
to show the full compatibility of the two results.

\smallskip
We finally come to the second paring, the one with the relative cohomology $H^* (M\xrightarrow{u} B\Gamma)$.
There is certainly a pairing, obtained via the Chern character
\begin{equation}\label{eq:CM-localized}
H^* (M)\times K_* (M)\to \CC\,.
\end{equation}
Let us identify $ H^* (\Gamma)= H^* (B\Gamma)$.
As already pointed out, if 
$\Gamma$ is  Gromov hyperbolic or of polynomial growth,
we  have  the Connes-Moscovici pairing, rewritten as 
\begin{equation}\label{eq:CM-higher1}
H^* (B\Gamma)\times K_* (C^*_{red}\Gamma)\to \CC.
\end{equation}
Writing the Higson-Roe analytic surgery sequence 
with Paschke duality
\begin{equation*}
\cdots \to  
\SG^\Gamma_* ({\tM})\to K^\Gamma_{*}(\tM)\xrightarrow{\mu}  K_{*}(C^*_{red}\Gamma)\to  \SG^\Gamma_{*-1}({\tM})\to
\cdots
\end{equation*}
and comparing it with the long exact sequence 
$$ \cdots  \leftarrow  H^{n+1} (M\xrightarrow{u} B\Gamma) \leftarrow H^n (M)\leftarrow H^n (B\Gamma)\leftarrow  H^{n} (M\xrightarrow{u} B\Gamma)\leftarrow  \cdots$$
we understand  that $\SG^\Gamma_* ({\tM})$ should be paired 
with $H^{n} (M\xrightarrow{u} B\Gamma)$ with $*=n \pmod 2$. This is precisely what we achieve in the second part of this
article.
Crucial for this result is an Alexander-Spanier description of $H^{*} (M\xrightarrow{u} B\Gamma)$, the new description by Zenobi 
of the group $\SG^\Gamma_* ({\tM})$,  together with arguments directly inspired by the work 
of Connes and Moscovici. We show that the pairing \eqref{intro:pairing-rel}
exists if $\Gamma$ satisfies the Rapid Decay (RD) condition and has a combing
of polynomial growth (for example, if $\Gamma$ is Gromov hyperbolic). 
Similar results through different, more topological, techniques have been discussed in 
\cite{WXY}. We describe in Section \ref{sect:application-via-cyclic} how the two approaches are related.

\smallskip
  Finally, we use the pairings for explicit geometric applications. We 
  concentrate on the study of the space $\Riem^+(M)$ of metrics of positive scalar
  curvature. It is well known that for a closed spin manifold $M$ of dimension
  $n$ which admits a
  metric $g_0$ of positive scalar curvature, the higher rho invariant valued
  in the structure group $\SG^\Gamma_n ({\tM})$ can be used to
  distinguish components in the space $\Riem^+(M)$. The Higson-Roe exact sequence, combined with surgery
  methods for the construction of new positive scalar curvature metrics, has been used in
  a quite general context to show that $\pi_0(\Riem^+(M))$ is indeed  rich, in
  particular infinite. A somewhat exhaustive culmination of this development is
  \cite{XieYuZeidler}. Of course, one can now go on to distinguish higher rho
  classes in $\SG^\Gamma_n ({\tM})$ also using our pairings with
  cyclic cocycles.

  This is indeed beneficial because it allows to obtain additional
  information. In particular, we can often control more explicitly how the
  resulting numeric rho invariants change if one pulls back the metric $g$ by
  a diffeomorphism $\psi\in\Diffeo(M)$.

  This way, we manage to show in numerous cases that the known elements of
  $\pi_0(\Riem^+(M))$ are also in different components of the moduli space,
  i.e.~that $\pi_0(\Riem^+(M)/\Diffeo(M))$ is infinite.

  As a specific example, we list the following theorem:
  \begin{theorem}
    Let $M$ be a connected closed spin manifold of dimension $n\ge 5$ with fundamental group
    $\Gamma$. Let $g_0$ be a Riemannian metric of positive scalar curvature on
    $M$. Assume that $\Gamma$ is word hyperbolic or of polynomial growth.

    Assume one of the two conditions:  for some $k<n$,
      $k\equiv n+1\pmod{4}$ 
    \begin{itemize}
    \item $e\ne\gamma_0\in\Gamma$ is an element of finite order, and its
      centralizer $\Gamma_{\gamma_0}$ in $\Gamma$ satisfies that
      $H^k(\Gamma_{\gamma_0},\complexs)$ is one-dimensional
    \item The relative cohomology group
      $H^k(M\xrightarrow{u}B\Gamma;\complexs)$ is $1$-dimensional, where $u$
      is a classifying map for the universal covering, i.e.~an isomorphism on
      $\pi_1$. 
   \end{itemize}
  Then the moduli space of metrics of positive scalar curvature
  $\Riem^+(M)/\Diffeo(M)$ has infinitely many components.
  \end{theorem}
 
 \noindent
Notice that we get more general results and more structural information on the moduli
  space of metrics of positive scalar curvature in Sections
  \ref{sec:moduli_space}, \ref{sec:appl_real_K}, \ref{sect:application-via-cyclic} and 
  \ref{sec:rel_cohom_geometric}.

\medskip
\noindent
{\bf The paper is organized as follows.} In Section \ref{section2} and Section
\ref{section3} we recall the alternative description, due to the third author, 
of the Higson-Roe surgery sequence and the realization of the  rho class of an invertible Dirac operator in this new 
context. This new realisation of the surgery sequence employs $\Gamma$-equivariant
pseudodifferential operators on $\tM$ in an essential way. It also relies
heavily on the use of K-theory for relative algebras, i.e.~algebra morphisms
$f\colon A\to B$. In Section \ref{section4}  we recall the definition of
Karoubi's Chern character. We develop the theory of non-commutative de Rham
homology for relative algebras, and we extend the Chern character to the relative case. In Section \ref{section5}
we discuss results of Lott on  pseudodifferential operators on $\tM$ with coefficients $\Omega_* (\mathbb{C}\Gamma)$.
Section \ref{section6} is devoted to one of the main results of this paper,
the mapping of the analytic surgery sequence to homology. 
 As already 
remarked, this mapping is valid for any finitely generated discrete group $\Gamma$. In this section we also
compare the delocalized Chern character of the rho class associated to an invertible Dirac operator $\widetilde{D}$
to Lott's class 
$\varrho_{{\rm Lott}} (\widetilde{D})$
 and show that at least if $M$ is a boundary $\Ch^{del} (\rho (\widetilde{D}))=-\frac{1}{2}\varrho_{{\rm Lott}} (\widetilde{D})
 \in H^{del}_*(\mathcal{A}\Gamma)$.
 In Section \ref{sec:hyperbolic_pairing}, assuming the group $\Gamma$ to be
 Gromov hyperbolic or of polynomial growth, we finally define the higher rho numbers 
associated to classes in $HC^*_{{\rm del}} (\mathbb{C}\Gamma;\langle x\rangle)$. 
In Section \ref{section8} we discuss carefully how to realize the relative cohomology groups $H^* (M\xrightarrow{u} B\Gamma)$
in terms of locally zero $\Gamma$-equivariant Alexander-Spanier cocycles on $\tM$; we use this description 
and the pseudodifferential realization of $\SG^\Gamma_* ({\tM})$ in order to define, under suitable assumptions
on $\Gamma$, the pairing $H^* (M\xrightarrow{u} B\Gamma)\times \SG^\Gamma_* ({\tM})\to \mathbb{C}$. 

In Section \ref{sect:higher-aps} 
we provide  explicit formulae for the higher delocalized Atiyah-Patodi-Singer 
indices of a Galois covering with boundary $\widetilde{W}$, expressing these indices in terms of the higher rho numbers of 
the boundary $\partial \widetilde{W}$.

In Section \ref{sec:naturality}, we give general results on the action
of the diffeomorphism group $\Diffeo(M)$ on the whole diagram
\eqref{intro:dR-sequence} and the resulting naturality of the parings like
 \eqref{intro:pairing-rel} and \eqref{intro:pairing-del-bis}. We give explicit geometric applications of this discussion 
 in Sections \ref{sec:appl_real_K},  \ref{sect:application-via-cyclic}
and  \ref{sec:rel_cohom_geometric}.

\bigskip
As this article is rather long, we would like to give at this point a {\bf summary of our main results}:

\begin{enumerate}
	\item the development of a relative Chern character in the setting of relative non-commutative de Rham homology, see Section \ref{chern-nchomology};
	\item the mapping of the analytic surgery exact sequence to non-commutative de Rham homology 
		without any assumption on the fundamental group $\Gamma$, see Theorem \ref{commutativity-chern}; 
	\item the isomorphism $HC^*_{pol}(\CC\Gamma; \langle x\rangle)\cong HC^*(\CC\Gamma; \langle x\rangle) $ for $\Gamma$ hyperbolic, which says that the cyclic cohomology of the group ring can be calculated using cochains of polynomial growth along each conjugacy class,  see Corollary \ref{corol:periodic_pol}; 
	\item  the construction of the pairing
$\langle \cdot, \cdot \rangle\colon HC^* (\mathbb{C}\Gamma;\langle x\rangle)\times H_{*}^{del} (\mathcal{A}\Gamma)
\rightarrow \mathbb{C}$, $\quad\forall \langle e\rangle\ne \langle x \rangle
\in \langle \Gamma \rangle$ for $\Gamma$ hyperbolic, where $\mathcal{A}\Gamma$
is the Puschnigg algebra;
\item 
the construction of the pairing $HC^* (\CC\Gamma;\langle x
          \rangle)\times \SG_*^\Gamma ({\tM})\longrightarrow \CC$ for $\Gamma$ hyperbolic,
         see Theorem \ref{pairing-hyperbolic}, and  the corresponding
         definition of higher rho numbers given in Definition \ref{higher-rho-number};
	\item the description of $H^*(M\xrightarrow{u} B\Gamma)$ using
          suitable Alexander-Spanier cocycles,
	via the isomorphism $H^*_{AS,0,\Gamma}(\tM)^\Gamma\xrightarrow{\iso}
	H^*(M\xrightarrow{u} B\Gamma)$, see Proposition \ref{Higson-lemma}; 
\item  the construction of the pairing $H^* (M\xrightarrow{u} B\Gamma)\times \SG_*^\Gamma
	({\tM})\longrightarrow \CC$ for discrete groups with property RD and
        with a combing of polynomial growth, see Theorem
        \ref{pairing-relative}, and the corresponding definition of higher rho
        numbers given in Definition \ref{higher-rho-number-AS};
	\item  the treatment of delocalized higher Atiyah-Patodi-Singer index formulae 
	 associated to elements in $HC^*(\CC\Gamma; \langle x\rangle) $ and  $H^* (M\xrightarrow{u} B\Gamma)$;
	\item a detailed discussion about the action of the
	diffeomorphism group of $M$ on the analytic surgery sequence, the related 
	de Rham noncommutative homology sequence and the compatibility between these two actions;
	\item the discussion of several geometric applications of higher rho numbers, associated either
	to delocalized cyclic cocycles $HC^* (\CC\Gamma;\langle x 
          \rangle)$ or to elements of the relative cohomology  group $H^*(M\xrightarrow{u} B\Gamma)$,
	in problems connected with the classification  of metrics of positive
        scalar curvature, showing in particular many new cases in which the
        moduli space of such metrics has infinitely many
        components. More precisely, we provide strong lower bounds on
        the rank of the moduli space in terms of delocalized and relative
        group cohomology. This rank is defined via the group of coinvariants
        for the usual group structure on the space of concordance classes of
        metrics of positive scalar curvature.
\end{enumerate}

The results of this paper relative to hyperbolic groups
were announced in a seminar at Fudan University in January 2019. At the same time
the article by Chen, Wang, Xie and Yu 
\cite{ChenWangXieYu} appeared, giving, among other things, a different
treatment of some of the same results, most notably the well-definedness
of the higher rho numbers associated to elements in $HC^* (\CC\Gamma,\langle x \rangle)$.
 The two articles have a non-trivial intersection but also a non-trivial symmetric difference. For example, 
 in \cite{ChenWangXieYu}, building on  \cite{XieYu}, the convergence of Lott's delocalized eta invariant 
 is established under very general
 assumptions; we don't discuss at all this problem in this paper.
 On the other hand, items (2),  (6), (7) and (10) in the summary given above are  results 
	treated here and not in \cite{ChenWangXieYu}. Moreover, we believe that our
         results in (3) and (4) are more precise than those in \cite[Section 5]{ChenWangXieYu} and
        derived in a systematic and complete way using homological algebra
        methods throughout.
        
        Notice that the techniques employed for the common results are
completely different. Recently the results of \cite{ChenWangXieYu} have been
extended to the case of discrete groups of polynomial growth, see \cite{John}. In this second version of our paper
we observed that our arguments work equally well for the (easier) case of groups of polynomial growth.

\section*{Acknowledgements.}

We are indebted to Nigel Higson for suggesting the pairing \eqref{intro:pairing-rel}
and Proposition \ref{Higson-lemma} (which we therefore call ``Higson's Lemma''). We thank
Michael  Puschnigg for very useful discussions about hyperbolic groups.
We also had interesting correspondence and discussions with John Lott, Ralf Meyer and  
Ryszard Nest and we take this opportunity to thank them all. Finally, we thank the referee
of this article for useful comments and suggestions.

Part of this work was performed while the second author was visiting Sapienza
Universit\`a di Roma, with a  3-months visiting professorship. 
There were also  several short-time visits of Schick and Zenobi to Sapienza. We are grateful to Sapienza
Universit\`a di Roma for the financial support that made all these visits possible.
The first and third author have been supported by the DFG-SPP 2026 ``Geometry at Infinity''.

%%% Local Variables:
%%% mode: LaTeX
%%% TeX-master: "numeric-rho_memoir_style"
%%% End:

%% file: numeric_rho_sec2.tex
\chapter{The relation between the structure algebra $D^*(\tM)^\Gamma$ and
  $\Pdo^0_\Gamma (\tM)$}\label{section2}

Let  $\Psi^0_{\Gamma,c} (\tM)$ be the algebra of $\Gamma$-equivariant
pseudodifferential operators of order $0$
on $\tM$ with Schwartz kernel of $\Gamma$-compact support. Let $\Pdo^0_\Gamma(\tM)$ denote the $C^*$-closure of $\Psi^0_{\Gamma,c} (\tM)$
  in the bounded operators of $L^2 (\tM)$. Recall the algebra $D^*_c (\tM)^\Gamma$  of finite propagation 
  pseudolocal bounded operators on $L^2 (\tM)$ and its $C^*$-closure $D^*
  (\tM)^\Gamma$. We don't repeat the details of the definitions of these
    algebras, as we are actually not using them.
Directly from the definition of $D^*_c (\tM)^\Gamma$ and the properties of pseudodifferential
operators, we understand that if $P\in \Psi^0_{\Gamma,c} (\tM)$
then $P\in  D^*_c (\tM)^\Gamma$. One might  then wonder what is the precise relationship
between  the algebra  $\Pdo^0_\Gamma(\tM)$ and the algebra   $D^* (\tM)^\Gamma$, at least as far as K-theory is concerned. The next section, based on  \cite{Zenobi_compare},
gives an answer to this question.

\medskip
We shall consider the mapping cone C*-algebra of a *-homomorphism $\varphi\colon A\to B$. We will  use  both notations $K_*(A\xrightarrow{\varphi} B)$  and $K_*(\mathfrak{\varphi})$ for denoting the K-theory of this mapping cone. This is also called   the relative K-theory of $\varphi$.
In the realization of $K_{*} (\varphi)$
   as in  \cite[Section 2.3]{AAS1}, elements of $K_0(\mathfrak{\varphi})$ are given by homotopy classes of triples $[p_0,p_1;q_t]$ with $p_0,p_1$ projections over $ A$ and  $q_t$ a path  of projections over $B$ such that $\varphi(p_i)=q_i$ for $i=0,1$. Furthermore, elements in $K_1(\varphi)$ are couples $[u,f]$ with $u$ a unitary over $A$ and $f$ a path of unitaries over $B$ joining $\varphi(u)$ and the identity.

\section{Realizing the 
surgery sequence with pseudodifferential operators.}\label{subsect:realizing}
Consider the classic short exact sequence  of pseudodifferential operators 
$$ 0\to C^*_{red}(\widetilde{M}\times_\Gamma\widetilde{M})\rightarrow
\Pdo^0_{\Gamma}({\tM})\xrightarrow{\sigma} C(S^*M) 
\to 0$$
where we recall that $\Pdo^0_\Gamma(\tM)$ denotes the $C^*$-closure of $\Psi^0_{\Gamma,c} (\tM)$
  in the bounded operators of $L^2 (\tM)$, $\sigma$ is the principal symbol
  map and $C^*_{red}(\widetilde{M}\times_\Gamma\widetilde{M})$ is the C*-closure of $\Psi^{-\infty}_{\Gamma,c}(\tM)$, the ideal of smoothing $\Gamma$-equivariant operators with $\Gamma$-compact support. Consider the natural inclusion
$C(M)\xrightarrow{\mathfrak{m}}\Pdo^0_{\Gamma}({\tM})$, with $\mathfrak{m}$  denoting the map that associates to $f\in C(M)$ the multiplication operator by  the equivariant lift $ 
   \tilde{f}$ of $f$.
This induces a long exact sequence in K-theory for the mapping cone $C^*$-algebras: 
\begin{multline}\label{Pdo-mc}
\cdots\xrightarrow{\partial}  K_*(0\hookrightarrow
C^*_{red}(\widetilde{M}\times_\Gamma\widetilde{M}))\xrightarrow{i_*}
K_*(C(M)\\
\xrightarrow{\mathfrak{m}}
\Pdo^0_{\Gamma}({\tM}))\xrightarrow{\sigma_*}
K_*(C(M)\xrightarrow{\mathfrak{\pi^*}} C(S^*M))\xrightarrow{\partial}\cdots	
% \xymatrix{\cdots\ar[r]^(.25){\partial}& K_*(0\hookrightarrow C^*_{red}(\widetilde{M}\times_\Gamma\widetilde{M}))\ar[r]^(.45){i_*}& K_*(C(M)\xrightarrow{\mathfrak{m}}\Pdo^0_{\Gamma}({\tM}))\ar[r]^(.45){\sigma_*}& K_*(C(M)\xrightarrow{\mathfrak{\pi^*}} C(S^*M))\ar[r]^(.75){\partial}&\cdots	
% }.
\end{multline}
where $\pi\colon S^*M\to M$ is the bundle projection of the cosphere bundle. Observe that 
$K_*(0\hookrightarrow C^*_{red}(\widetilde{M}\times_\Gamma\widetilde{M}))$ is nothing but 
$K_{*}(C^*_{red}(\tM\times_\Gamma\tM)\otimes C_0(0,1))$.
\\
In \cite{Zenobi_compare} the author proves the existence of a canonical
isomorphism of long exact sequences
  \begin{equation}\label{ases}
{\tiny  \xymatrix{\to K_{*+1}(C^*(\tM)^\Gamma)\ar[r]\ar[d]^\cong&K_{*+1}(D^*(\tM)^\Gamma)\ar[r]\ar[d]^\cong&K_{*+1}(D^*(\tM)^\Gamma/C^*(\tM)^\Gamma)\ar[d]^\cong\to\\
  \to	 K_{*}(C^*_{red}(\tM\times_\Gamma\tM)\otimes
        C_0(0,1))\ar[r]&K_{*}(C(M)\xrightarrow{\mathfrak{m}}
        \Pdo^0_\Gamma(\tM))\ar[r]&K_{*}(C(M)\xrightarrow{\pi^*} C(S^*M) )\to}
      }
    \end{equation}
   
We remark  that the mapping cone $C^*$-algebra $\C_{\pi^*}$, that is $C(M)\xrightarrow{\pi^*} C(S^*M)$, is isomorphic to $C_0(T^*M)$. Indeed recall that, if $f\colon X\to Y$ is a continuous map of topological spaces, then the mapping cone of $f$ is the quotient of
 $X\times[0,1]\sqcup Y$  by the relation $(x,1)\sim f(x)$ and $(x,0)\sim(x',0)$. It is easy to see that the algebra of continuous functions on the mapping cone of $f$ vanishing at $[X\times\{0\}]$ is isomorphic to the mapping cone C*-algebra of $f^*\colon C(Y)\to C(X)$. Since the mapping cone of $\pi$ is the one-point compactification of $T^*X$, one has that the mapping cone C*-algebra associated to $\pi^*$ is isomorphic to $ C_0(T^*M)$ as desired.

  \medskip 

Recall that to every $C^*$-algebra homomorphism we obtain an
    associated long exact K-theory sequence using the mapping cone, which we
    call the pair sequence. For $\mathfrak{m}\colon C(M)\to
    \Psi_\Gamma^0({\tM})$ this gives, using the abbreviation
    $K_*(\mathfrak{m})$ for the K-theory of the mapping cone introduced above
    \begin{multline*}
      \cdots\to K_*(C(M)\otimes
      C_0(0,1))\xrightarrow{\mathfrak{m}\otimes\id_*}
      K_*(\Pdo^0_\Gamma(\tM)\otimes C_0(0,1))\\
      \to
      K_*(\mathfrak{m})\xrightarrow{(ev_0)_*} K_*(C(M))\to \cdots
    % \xymatrix{\cdots\ar[r]&K_*(C(M)\otimes C_0(0,1))\ar[r]^{\mathfrak{m}\otimes\id_*}& K_*(\Pdo^0_\Gamma(\tM)\otimes C_0(0,1))\ar[r] &K_*(\mathfrak{m})\ar[r]^(.4){(ev_0)_*}& K_*(C(M))\ar[r]&\cdots}
\end{multline*}
where the map $(ev_0)_*$ is given by the following  composition

\[
\xymatrix{K_*(\mathfrak{m})\ar[r] &K_*(\pi^*)\ar[r]_(.4){\cong} & K_*(C_0(T^*M))\ar[r]^{j_*} & K_*(C(D^*M))\ar[r]^{i_*}_(.5){\cong} & K_*(C(M)) }.
\]
 Here, $\pi^*\colon C(M)\to C(S^*M)$ is the pull-back map through the bundle
 projection, $j\colon T^*M\to D^*M $ is given by a diffeomorphism between the
 cotangent bundle and the open codisk bundle, and $i\colon M\to D^*M$ is the inclusion.
 If the Euler characteristic of $M$ is zero, then $j^*$ is the 0 map and, consequently, so is $(ev_0)_*$.\\
Let us prove  that $j_*=0$ if $\chi (M)=0$: we consider the following exact sequence
 \begin{equation}\label{cotangent-sequence}
   \xymatrix{\cdots\ar[r]&K_*(C_0(T^*M))\ar[r]^{j_*}& K_*(C(D^*M))\ar[r] &K_*(C(S^*M))\ar[r]& \cdots}.
 \end{equation}
 Then it is easy to see that its boundary map is given by the composition of the suspension isomorphism $K_*(C(S^*M))\to K_{*+1}(C_0(T^*M\setminus M))$ and the inclusion $\iota \colon C_0(T^*M\setminus M)\to C_0(T^*M)$.
 Indeed,  if the Euler characteristic of $M$ is zero, the cotangent bundle of
 $M$ admits a nowhere vanishing section $\xi$. Then take a diffeomorphism $\varphi$ between $T^*M$ and the open disc bundle $\mathring{D}^*M$ and a diffeomorphism $\psi$ between $\mathring{D}^*M$ and $U_\xi$, a tubular neighbourhood of the image of $\xi$.
 The morphism  $h=\psi\circ\varphi$ induces a morphism
 $h_*\colon C_0(T^*M)\to C_0(T^*M\setminus M)$ such that
 $\iota_*\circ h_*=\mathrm{Id}\colon K_*(C_0(T^*M))\to K_*(C_0(T^*M))$. This means that $\iota_*$ is surjective and, by exactness of \eqref{cotangent-sequence}, that $j_*$ is zero.\\ Summarizing: if the Euler characteristic is zero
 then $j_*=0$ and consequently $(ev_0)_*=0$.\\
This tells us that 
 $K_*(\Pdo^0_\Gamma(\tM))\to K_{*+1}(\mathfrak{m})$ is surjective and then any class in $K_{*+1}(\mathfrak{m})$  can be lifted, up to Bott periodicity, to $K_{*}(\Pdo^0_\Gamma(\tM))$.
 Observe that this is always true for $*+1=\dim M$ odd. 
Thus, when the Euler characteristic of $M$ is zero, we have the following  surjective morphism of exact sequences:

    \begin{equation}\label{surjective}
{\tiny    \xymatrix{   	\to
      K_{*+1}(C^*_{red}(\tM\times_\Gamma\tM))\ar[r]\ar@{->>}[d]&K_{*+1}(
      \Psi^0_{\Gamma} (\tM))\ar[r]\ar@{->>}[d]&K_{*+1}(C(S^*M))\ar@{->>}[d]\to \\
 	\to K_{*}(C^*_{red}(\tM\times_\Gamma\tM)\otimes
        C_0(0,1))\ar[r]&K_{*}(C(M)\xrightarrow{\mathfrak{m}}
        \Pdo^0_\Gamma(\tM))\ar[r]&K_{*}(C(M)\xrightarrow{\pi^*} C(S^*M) )\to }
        }
    \end{equation}
We shall use this information in order to simplify some of our formulae.

%%% Local Variables:
%%% mode: LaTeX
%%% TeX-master: "numeric-rho_memoir_style"
%%% End:

%% file: numeric_rho_sec3.tex
  \chapter{Rho classes}\label{section3}
    
    \section{Analytic properties of $\pi_{\geq }(\widetilde {D})$}

Assume that $\widetilde {D}$ is a $\Gamma$-invariant Dirac type operator on $\widetilde {M}$ with coefficient in a Clifford bundle $E$, which will be omitted from the notation all along this section.  Assume that $\widetilde{D}$ is
$L^2$-invertible. Let $\pi_\geq\colon \RR\to \RR$ be the characteristic
function of $[0,\infty)$. Functional calculus then allows us to define
$\pi_{\ge }(\widetilde {D})$ as a bounded operator on $L^2(\widetilde {M})$. 
We shall need more precise properties of this operator and these are given in the next two propositions.
\begin{proposition}
{The bounded operator $\pi_{\ge}(\widetilde {D})$ is an element in}
the $C^*$-closure $\Pdo^0_\Gamma(\widetilde {M})$ of $\Pdo^0_{\Gamma,c}(\widetilde {M})$.
\end{proposition}
\begin{proof}
First recall that there is a chopping function $f\colon \reals\to [-1,1]$ which is odd,
    monotonically increasing, smooth, ``quasi-Schwarz'', with $\lim_{t\to\infty} f(t)=1$ and
    such that its Fourier transform has compact support. Note that this
    Fourier transform has a singularity at zero and is smooth
    outside.
 Then $f(\widetilde {D})$ is a pseudodifferential operator of order zero, which is
    $\Gamma$-compactly supported because of finite propagation speed and the
    compact support of $\hat f$.
By the mapping properties of $f$, $\pi_{\ge}( \widetilde {D})=\pi_{\ge}(f(\widetilde {D}))$. 
      Because
    $C^*$-algebras are closed under continuous functional calculus and because
    $f(\widetilde {D})$ is invertible, $\pi_{\ge}(f(\widetilde {D}))$ is contained in  the $C^*$-closure
    of $\Pdo^0_{\Gamma,c}(\widetilde {M})$.

\end{proof}

\begin{proposition}\label{rho-smooth}
 If $\widetilde {D}$ is $L^2$-invertible, then   $\pi_{\ge}(\widetilde {D})$ belongs to  any holomorphically closed *-algebra $\mathcal{A}$
 such that  
 $$\Pdo^0_{\Gamma, c}(\widetilde {M})\subset \mathcal{A}\subset \Pdo^0_{\Gamma} (\widetilde {M}).$$
\end{proposition}

\begin{proof}
With the choice of $f$ as in the previous proof,
so that  $f(\widetilde {D})\in \Pdo^0_{\Gamma, c}(\widetilde {M})$, we observe that  $\pi_{\ge}(f(\widetilde {D}))$
is obtained by holomorphic   functional calculus; indeed for any $T\in \Pdo^0_{\Gamma, c} (\widetilde {M})$
invertible in   $\Pdo^0_{\Gamma} (\widetilde {M})$ we have 
$$\pi_{\ge} (T)= \frac{( T\circ (\sqrt{T^* T})^{-1})+1}{2}\,.$$ It then  follows that $\pi_{\ge}(f(\widetilde {D}))$ belongs to any algebra closed under holomorphic functional calculus which contains $ \Pdo^0_{\Gamma, c}(\widetilde {M})$.
\end{proof}

\begin{remark}
	The same reasoning applies to the operator defined by the sign
        function, $\mathrm{sign} (\widetilde {D})=2\pi_{\geq}(\tilde D) -
          1$.
\end{remark}

    \section{Rho classes and  K-theory of pseudodifferential operators}\label{rho-classes}
  
    We want to give the new realisation of the rho classes in the  setting
    of relative K-theory, see \cite{Zenobi_compare}. Concretely, cycles
      for $K_0(A\xrightarrow{f} B)$ are triples $(p_0,p_1,q_t)$ where $p_0,p_1$ are
      idempotent matrices over $A$ and $(q_t)_{t\in [0,1]}$ is a continuous loop of
      idempotent matrices in $B$ from $f(p_0)$ to $f(p_1)$. Cycles for
      $K_1(A\xrightarrow{f}B)$ are of the form $(u,v_t)$ where $u$ is an
      invertible matrix over $A$ and $(v_z)_{z\in S^1\subset\complexs}$ is a loop of
      invertible matrices in $B$ with $v_1=f(u)$.
   Let $\tM$ be a Galois $\Gamma$-covering of a closed smooth manifold $M$. Let $\widetilde{D}$ be an $L^2$-invertible $\Gamma$-equivariant Dirac operator on $\tM$.
    \begin{definition}\label{rho-class}
    	\begin{itemize}
    		\item If  $\dim M$ is even, then $\varrho(\widetilde{D})$ is
                  defined as the class in $K_0(\mathfrak{m})$ given
                  by
    		\begin{equation}\label{classmc}
    		\left[\begin{pmatrix}1_+&0\\
    		0&0\end{pmatrix}, \begin{pmatrix}0&0\\
    		0&1_-\end{pmatrix};\begin{pmatrix}\cos^2\left(\frac{\pi}{2}t\right)1_+&\cos\left(\frac{\pi}{2}t\right)\sin\left(\frac{\pi}{2}t\right)
                \sgn(\widetilde{D})_-\\\cos\left(\frac{\pi}{2}t\right)\sin\left(\frac{\pi}{2}t\right)
                \sgn(\widetilde{D})_+&\sin^2\left(\frac{\pi}{2}t\right)1_-\end{pmatrix} \right]  
    		\end{equation}
    		where $1_+$ is the projection which defines the
                  $C(M)$-module $C(M,E_+)$ of continuous sections of the even
                  part $E_+$ of the Clifford bundle $E$, entering into the
                  definition of $\widetilde{D}$, and $1_-$ is defined
                  correspondingly using the odd part of $E$.
    		\item
    		If $\dim \tM$ is odd then $\varrho(\widetilde{D})$ is defined as
                the class in $K_1(\mathfrak{m})$   given by
                $[1,1+\pi_{\geq}(\widetilde{D})(z-1)]$, with $z\in
                  S^1\subset \mathbb{C}$.
    	\end{itemize}
    \end{definition}
    \begin{remark}\label{rem:classic_def_of_rho_D}
      It is proved in \cite{Zenobi_compare} that these definitions are
      compatible with the original ones, as given in \cite{PiazzaSchick_psc},
      through the isomorphism \eqref{ases}.  Notice that in the odd
      dimensional case we know that
      $K_1(\Pdo^0_\Gamma(\tM)\otimes C_0(0,1))\to K_1(\mathfrak{m})$ is
      surjective. In fact, the rho class we have just defined is the image,
      up to suspension isomorphism, of the class in $K_0(\Pdo^0_\Gamma(\tM))$
      given by $[\pi_{\geq}]$. Put it differently,
      $[\pi_{\geq}]\in K_0(\Pdo^0_\Gamma(\tM))$ is a natural lift of
      $\varrho(\widetilde{D})\in K_1(\mathfrak{m})$. In the sequel, we shall be
      interested in pairing $\varrho(\widetilde{D})\in K_1(\mathfrak{m})$ with
      suitable delocalized cyclic cocycles. Since the kernel of
      $K_1(\Pdo^0_\Gamma(\tM)\otimes C_0(0,1))\to K_1(\mathfrak{m})$ is
      $K_0(C(M))$ (up to suspension isomorphism) and since this group pairs
      trivially with such delocalized cocycles, we conclude that in the odd
      dimensional case we can focus solely on $K_0(\Pdo^0_\Gamma(\tM))$ and on
      $[\pi_{\geq}]\in K_0(\Pdo^0_\Gamma(\tM))$.
\end{remark}

\begin{definition}
Let $g$ be a $\Gamma$-equivariant metric with positive scalar curvature on a
spin $\Gamma$-covering $\tM$ of $M$. Then we will denote by $\varrho(g)\in
K_*(\mathfrak{m})$ the rho class associated to the spin Dirac operator
$\widetilde{D}$ defined on the spinor bundle. Recall that it is
$L^2$-invertible because the scalar curvature is bounded below by a positive constant.
\end{definition}

%%% Local Variables:
%%% mode: LaTeX
%%% TeX-master: "numeric-rho_memoir_style"
%%% End:

%% file: numeric_rho_sec4-15.tex
\chapter{Non-commutative de Rham homology and Chern characters}\label{ncderham}

\section{Non-commutative de Rham homology}\label{section4}

For the material in this section we refer the reader to \cite{Karoubi}. We recall the standard definitions which lead to the de Rham
cohomology of a Fr\'echet algebra, the target of the Chern character from
the K-theory of the algebra.

Let $A$ be a Fr\'{e}chet algebra, 
not necessarily unital. Then we denote by $A_+$ its
unitalization, i.e.~$A_+:=A\oplus \complexs$ with multiplication such that
$1\in\complexs$ is the unit of $A_+$. Note that, if $A$ used to have a
unit $e\in A$, $e\in A_+$ is a central idempotent and we have the algebra
decomposition $A_+= eA\oplus (1-e)\complexs$. We always have the split
inclusion $\complexs\into A_+$ of the projection
$A_+=A\oplus\complexs\to\complexs$ which is an algebra homomorphism.

A Fr\'{e}chet differential graded
algebra, shortly a DGA, $(\Omega_\bullet, d)$, with $\Omega_\bullet=\bigoplus\Omega_i$,
is said to be a DGA over $A$ if $\Omega_0=A$.
Let us denote by $\Omega_{\bullet}^+$ the unitalization of $\Omega_{\bullet}$, it is a DGA over $A_+$ given by $\Omega_\bullet^+:=\Omega_\bullet\oplus \Omega_\bullet(\complexs)$.
Because
$\Omega_\bullet(\complexs)$ is isomorphic to $\complexs$ concentrated in
degree $0$, the inclusion $\complexs\to A_+$ always induces  a split injection
of DGAs $\Omega_*(\complexs)\into \Omega_\bullet^+$.

\begin{example}
	
	Given a unital algebra $B$ over $\complexs$ with unit $e$, the universal differential
	graded algebra $\Omega_\bullet(B)$ is defined as
	$\Omega_n(B)=B\otimes\bar B^{\otimes n}$ with $\bar B=B/\complexs e$ the
	quotient by the multiples of the identity. The differential sends $b\in B$
	to $db:=1\otimes[b]\in \bar B$. The product is defined by concatenation and the
	Leibniz rule. Elements are linear combinations of the form
	$b_0db_1\dots db_n$.
	
	If $A$ is a unital Fr\'echet algebra, we denote by $\widehat{\Omega}_\bullet(A)$
	the corresponding degree-wise completed object, constructed using the projective
	tensor product.
\end{example}

\begin{definition}
	Let us consider a DGA $\Omega_{\bullet}$ over a (Fr\'{e}chet) algebra $A$, not necessarily unital. Let $\Omega_{ab,\bullet}$ denote the vector space quotient of $\Omega_\bullet$ by the closure of the subspace generated by the graded commutators $\overline{[\Omega_\bullet,\Omega_\bullet]}$.
	Notice that $d$ induces a differential on $\Omega_{ab,\bullet}$ since
	$d[\omega_i,\omega_j]= [d\omega_i,\omega_j]+(-1)^i[\omega_i,d\omega_j]$  for
	$\omega_i\in \Omega_i$, $\omega_j\in\Omega_j$.
	Define the de Rham homology $H_*(\Omega_\bullet,d)$ as the $i$-th homology group of the complex
	$(\Omega_{ab,\bullet},d)$.
\end{definition}
\begin{remark}	To be able to apply standard homological
	algebra, we use algebraic homology, i.e.~we quotient by the image of the
	differential, and not its closure. Here, our conventions differ from part
	of the literature. However, there is the canonical map from algebraic homology to
	the version where one divides by closures, and our invariants are lifts of the previously constructed
	ones (where applicable), and thus contain a priori even more information.
	By mapping to the quotient by closures we obtain the usual invariants.
\end{remark}

\begin{definition}\label{universal-deRham}
	We define the \emph{universal de Rham homology} of $A$
	\begin{equation*}
	H_*(A):= \coker\left(H_*(\widehat\Omega_\bullet(\complexs))\to H_*(\widehat\Omega_\bullet(A_+))\right).
	\end{equation*}
	Moreover, 
	for a DGA $\Omega_{\bullet}$ over a Fr\'{e}chet algebra $A$, we define
	the \emph{augmented homology} $$
	H^+_*(\Omega_\bullet):=\coker(H_*(\complexs)\to H_*(\Omega_\bullet^+)),$$ which is
	canonically a summand of $H_*(\Omega_\bullet^+)$. 
\end{definition}

\begin{remark}\label{augmented}
	Notice that, if $\Omega_\bullet$ is a DGA over $A$,
	the canonical map
	$ H^+_*(\Omega_\bullet)\to H_*(\Omega_\bullet)$ is
	an isomorphism.
	Because of this fact we will freely use the non augmented version also when, to make sense of a situation, the augmentation seems  necessary.
	Furthermore, observe that, by universality, we always have a map of complexes
	$\theta\colon\widehat{\Omega}_\bullet(A_+)\to \Omega^+_\bullet$ for any
	Fr\'echet DGA $\Omega^+_\bullet$ over $A_+$ compatible with the map
	from $\Omega_\bullet(\complexs)$. Consequently, one always
	has a map $$\theta_*\colon H_*(A)\to   H^+_*(\Omega_\bullet)\cong H_*(\Omega_\bullet).$$
\end{remark}

Now consider a homomorphism of Fr\'{e}chet algebras $\varphi\colon A\to B$. We
will work with \emph{relative} non-commutative de Rham homology in this
context and work out the basics for this now. We denote $\varphi_+\colon
A_+\to B_+$ the canonical unital homomorphism of the unitalizations.

\begin{definition}
	A \emph{pair of DGAs over $\varphi$ }is given by a homomorphism of DGAs $\Phi_\bullet\colon \Omega^A_\bullet\to\Omega^B_\bullet$ with  $(\Omega^A, d^A)$ a DGA over $A$, $(\Omega^B,d^B)$ a DGA over $B$ and $\Phi_0=\varphi$.
	The \emph{universal pair over $\varphi$} is given by the 
	homomorphism of DGAs $\Omega_\bullet(\varphi_+)\colon
	\widehat\Omega_\bullet(A_+)\to \widehat\Omega_\bullet(B_+)$, obtained
	by the universal property of $\widehat\Omega_\bullet(A_+)$.
\end{definition}

By general principles, we can associate to a pair over $\varphi$ the mapping cone complex
$$(\Omega^A\xrightarrow{\Phi}\Omega^B)_{ab,\bullet}:=\Omega^A_{ab,\bullet+1}\oplus\Omega^B_{ab,\bullet}$$ which is equipped with the differential
$$
d^{\Phi}:=\begin{pmatrix}
d^A& 0 \\ -\Phi_{\bullet+1}& d^B
\end{pmatrix}.
$$
\begin{definition}
	Define the \emph{relative   homology} of the pair $\Phi_\bullet\colon \Omega^A_\bullet\to\Omega^B_\bullet$ as the homology of the mapping cone complex and denote it by $H_*(\Phi)$.
	The \emph{universal relative de Rham homology} of
	$\varphi$ is the homology of the mapping cone complex of $\Omega_\bullet(\varphi_+)$ and is denoted by $H_*(\varphi)$.
\end{definition}

It is a standard fact in homological algebra that a pair $\Phi_\bullet\colon \Omega^A_\bullet\to\Omega^B_\bullet$ over $\varphi\colon A\to B$ gives rise to a long exact sequence of homology groups 

\begin{equation}\label{les-homology}\xymatrix{\cdots\ar[r]& H_*(\Omega_\bullet^B)\ar[r]& H_*(\Phi)\ar[r]&H_{*+1}(\Omega_\bullet^A)\ar[r]^{\Phi_*}&H_{*+1}(\Omega_\bullet^B)\ar[r]&\cdots}
\end{equation}
Moreover, by naturality, we can associate to $\varphi$ a long exact sequence of universal de Rham homology groups 
\begin{equation}\label{ules-homology}\xymatrix{\cdots\ar[r]& H_*(B)\ar[r]& H_*(\varphi)\ar[r]&H_{*+1}(A)\ar[r]^{\varphi_*}&{H}_{*+1}(B)\ar[r]&\cdots}
\end{equation}
Observe  that if  $\Phi_\bullet^+\colon \Omega_\bullet^{A,+}\to \Omega_\bullet^{B,+}$ is the augmentation of $\Phi_\bullet$, then $H_*(\Phi^+)\cong H_*(\Phi)$. 
Thanks to this fact and to Remark \ref{augmented}, there always exists a canonical natural mapping from  \eqref{ules-homology} to \eqref{les-homology}.

\begin{remark}
	In contrast to some of the literature, for what concern the universal de Rham homology, we use here the unitalization and augmented versions.
	This has the advantage to be
	applicable to non-unital algebras. Moreover, it is perfectly in line with
	the definition of K-theory. 
	Finally, we argue that
	${H}_*(A^+)$ is more appropriate also because  it canonically maps as a subgroup into cylic homology, the other popular
	homology theory for algebras, but it only works in the   augmented setting.
\end{remark}

\section{Chern characters}\label{chern-nchomology}
Let $(\Omega_\bullet, d)$ be a DGA over $A$. We are going to recall the
definition of the Chern character
\begin{equation}
	\Ch\colon K_0(A)\to  H_{\mathrm{even}}(\Omega_\bullet).
\end{equation}
All along this section we do not want to care about the fact that $A$ is unital or not. Even if the general definition  $K_0(A)=\coker(K_0(\complexs)\to K_0(A_+))$ for K-theory groups would consequently imply that we use the augmented version of the homology groups, we will instead employ the non augmented notation, bearing in mind that  Remark \ref{augmented} prevents us from any risk  of  inconsistency.

For this, consider a class $[p]\in K_0(A)$, where $p$ is an idempotent
$n\times n$ matrix over $A$, and denote by $E$ the finitely generated projective module $p A^n$.
Then $\nabla:=p\circ d\circ p\colon  E\to E\otimes_{A}\Omega_1$ is a connection  and its curvature $\Theta:= \nabla^2$ is an  $A$-linear endomorphism of $E$ with coefficients in $\Omega_2$ explicitly given by $pdpdp$.

\begin{definition}
	The degree $2k$ part of the Chern character of $[p]\in K_0(A)$
	associated to the DGA $(\Omega_\bullet, d)$ is defined as the
	class  $$\Ch_{2k}([p]):=\frac{1}{k!}\Tr(\Theta^k)\in  H_{2k}(\Omega_\bullet). $$
\end{definition}

We want to define a relative version of the Chern character. To this end we need to prove a transgression formula for the absolute Chern character.
Let $p_t$ be a smooth path of idempotents in $M_n(A)$, with $t\in [0,1]$.
We regard it as an idempotent $\tilde{p}$ over $C^{\infty}[0,1]\widehat\otimes A$.
Moreover, if $(\Omega_*, d)$ is a Fr\'echet DGA over $A$, then we can
canonically associate to it the following  Fr\'echet DGA over
$C^{\infty}[0,1]\widehat\otimes A$ 
$$\left((C^{\infty}[0,1]\oplus C^{\infty}[0,1]dt)\widehat\otimes
\Omega_*,\;D:=dt\wedge\frac{\partial}{\partial t}\otimes\id + \id\otimes d\right).$$
The connection $\tilde{p}\circ D\circ\tilde{p}$ is equal to 
\begin{equation}
	p_t\circ d\circ p_t+ p_t\circ dt\wedge\frac{\partial}{\partial t} \circ p_t
\end{equation}
and its curvature $(\tilde{p}\circ D\circ\tilde{p})^2$ is equal to 
$$\Theta_t+ \Xi_t\wedge dt$$ for some path of 1-forms $\Xi_t$ and with $\Theta_t$ being the curvature of the connection $p_t\circ d\circ p_t$.
\begin{lemma}
	The path $\Xi_t$ of $A$-linear morphisms with coefficient in $\Omega_1$  is given by $p_tdp_t\dot{p}_t+p_t\dot{p}_tdp_t$.
\end{lemma}
\begin{proof}
	First recall that the equality $p_t=p^2_t$ implies that $\dot{p}_t(1-p_t)=p_t\dot{p}_t$. It follows that $p_t\dot{p}_t p_t=0$ and that $0= d(p_t\dot{p}_t p_t)=d(p_t\dot{p}_t) p_t+p_t\dot{p}_t dp_t$.
	By a direct calculation we see that 
	\begin{equation}
		\begin{split}
			(\tilde{p}\circ D\circ\tilde{p})^2-\Theta_t&= p_t\circ d\circ p_t\circ dt\wedge\frac{\partial}{\partial t}\circ p_t+p_t\circ dt\wedge\frac{\partial}{\partial t} \circ p_t\circ d \circ p_t=\\
			&= (p_t\circ d\circ p_t\circ \frac{\partial}{\partial t}\circ p_t+p_t\circ \frac{\partial}{\partial t} \circ p_t\circ d \circ p_t)\wedge dt.
		\end{split}
	\end{equation}
	Now apply the coefficient of $dt$ in the last line to an element
	$x_t\in \tilde{p}(C^{\infty}[0,1]\otimes A_+^n)$. Then
	\begin{equation}
		\begin{split}
			p_td(p_t&\frac{\partial}{\partial t}(p_tx_t))+p_t \frac{\partial}{\partial t}(p_td(p_tx_t))\\
			=&p_td(p_t(\dot{p}_t x_t+p_t\dot{x}_t))+p_t \frac{\partial}{\partial t}(p_tdp_tx_t+p_tdx_t)\\
			=&p_tdp_t(\dot{p}_tx_t+p_t\dot{x}_t)+
                           p_td(\dot{p}_tx_t+p_t\dot{x}_t)\\
                     &+
			p_t(\dot{p}_tdp_t x_t+p_td\dot{p}_t x_t+ p_tdp_t\dot{x}_t+ \dot{p}_tdx_t+p_td\dot{x}_t)\\
			=&(p_tdp_t\dot{p}_t+p_t\dot{p}_tdp_t)x_t
		\end{split}
	\end{equation}
	and the result follows.
\end{proof}

\noindent
It follows that the $2k$-degree component of the Chern character of $\tilde{p}$ is of the following form
\begin{equation}\label{ch-rel}
	\Ch_{2k}(\tilde{p})= \Ch_{2k}(p_t)+ \overline{\Ch}_{2k-1}(p_t)\wedge dt.
\end{equation}
{We call $\overline{\Ch}_{2k-1}(p_t)$ in the above formula the {\it transgression Chern character}.} 

\begin{lemma}\label{lem:transgress}
	{The transgression Chern character  is explicitly given by }
	\begin{equation}
		\overline{\Ch}_{2k-1}(p_t)=
		\frac{1}{k!}\Tr\left(\Theta_t^{k-1}\Xi_t\right)\label{eq:trans_ch}
	\end{equation}
	and satisfies the transgression formula 
	\begin{equation}
		\label{transgression-formula}
		\frac{\partial}{\partial t}\Ch_{2k}(p_t)=- d\overline{\Ch}_{2k-1}(p_t).
	\end{equation}
	It is functorial for algebra maps $f\colon A\to B$ covered by a DGA
	map $F\colon \Omega^A_\bullet\to \Omega^B_\bullet$ over $f\colon
	A\to B$.
\end{lemma}
\begin{proof}
	The first statement is given by a direct calculation of
	$\Ch_{2k}(\tilde{p})$ as follows
	\begin{equation*}
		\begin{split}
			\frac{1}{k!}\Tr&\left((\Theta_t+\Xi_t\wedge dt)^k\right)\\
			&=\frac{1}{k!}\Tr\left(\Theta^k_t+\sum_{j=0}^{k-1}\Theta_t^j(\Xi_t\wedge dt)\Theta^{k-j-1}\right)\\
			&=\frac{1}{k!}\Tr(\Theta^k_t)+\frac{1}{(k-1)!}\Tr(\Xi_t\Theta_t^{k-1})\wedge dt.
		\end{split}
	\end{equation*}
	Here we used the fact that $\Tr$ is linear and zero on graded commutators.
	Moreover, since $\Ch_{2k}(\tilde{p})$ is closed with respect to the
	differential $D$ and $\Ch_{2k}(p_t)$ is closed with respect to $d$,
	applying $D$ to \eqref{ch-rel} we obtain
	\eqref{transgression-formula}. Functoriality follows directly from
	\eqref{eq:trans_ch}.
\end{proof}

\begin{corollary}
	For the smooth path of idempotents $p_t$ we have
	\begin{equation}\label{ch-int}
		\Ch_{2k}(p_0)-\Ch_{2k}(p_1)= d\int_0^1\overline{\Ch}_{2k-1}(p_t) dt.
	\end{equation}
\end{corollary}
Now we are ready to define the relative Chern character. Recall that the relative K-group $K_0(\varphi)$ is defined via homotopy classes of triples $(p_0,p_1;q_t)$ where
$p_0$ and $p_1$ are idempotents over $A_+$ and $q_t$ is a path of idempotents
over $B_+$ such that $q_0=\varphi(p_0)$ and $q_1=\varphi(p_1)$. The equality
\eqref{ch-int} 
allows to give the following definition.
\begin{definition}\label{ch-rel-even}
	The degree $2k-1$ part of the relative Chern character of
	$[p_0,p_1;q_t]\in K_0(\varphi)$ associated to a pair of DGAs
	$\Phi_\bullet\colon \Omega^A_\bullet\to\Omega^B_\bullet$ over
	$\varphi$ is  defined as
	\begin{equation}
		\Ch^{rel}_{2k-1}([p_0,p_1;q_t]):=
		\left(\Ch_{2k}(p_0)-\Ch_{2k}(p_1),\int_0^1
		\overline{\Ch}_{2k-1}(q_t)\;dt\right)\in H_{2k-1}(\Phi).
	\end{equation}
	
\end{definition}

\begin{proposition}\label{chern-well-defined}
	The relative Chern character of a class in K-theory does not depend on the choice of a representative and it is therefore well defined.
\end{proposition}
\begin{proof}
	Let $(p_0^s,p_1^s;q_t^s)$ be a smooth path of relative K-cycles for $\varphi$.
	We can see $q_t^s$ as an idempotent $\tilde{q}$ over $C^{\infty}[0,1]^2\otimes B_+$.
	Consider the Fr\'echet DGA $\left(C^\infty[0,1]^2\otimes(1\oplus dt\oplus ds\oplus dt\wedge ds)\right)\otimes \Omega^B_\bullet$ with differential 
	$$D:=ds\wedge\frac{\partial}{\partial s}+dt\wedge\frac{\partial}{\partial t}+d^B.$$
	Let $E$ be the module associated to $\tilde{q}$ and let us equip $E$ with the connection $\tilde{q}D\tilde{q}$. Its curvature is given by
	\[
	\tilde{\Theta}=(q_t^sds\wedge\frac{\partial}{\partial s}q_t^s+q_t^sdt\wedge\frac{\partial}{\partial t}q_t^s+q_t^sd^Bq_t^s)^2
	\]
	which is the sum of four terms
	\[
	\overline{\Theta}= \Theta_{t,s}+ \Xi_{t,s}\wedge dt+ \Sigma_{t,s}\wedge ds+ \Psi_{t,s}\wedge dt\wedge ds.
	\]
	
	Then the degree  $2k$ component of the Chern character of $\tilde{q}$ is given by
	\begin{equation}\label{secondary-transgression}
		\Ch_{2k}(\tilde{q})= \Ch_{2k}(q^s_t)+ \overline{\Ch}_{2k-1}^t(q^s_t)\wedge dt+ \overline{\Ch}_{2k-1}^s(q^s_t)\wedge ds + \overline{\overline{\Ch}}_{2k-2}(q_t^s)\wedge dt\wedge ds
	\end{equation}
	where $\overline{\Ch}_{2k-1}^t$ and $\overline{\Ch}_{2k-1}^s$  are the transgression Chern characters in \ref{ch-rel} with respect to the parameter $t$ and $s$ respectively, whereas $\overline{\overline{\Ch}}_{2k-2}(q_t^s)$ is a  form in $\Omega_{2k-2}^B$.
	Now, applying $D$ to \eqref{secondary-transgression}, we obtain that $0$ on the left side is equal to
	\begin{equation}
		\begin{split}
                  0= &D\Ch_{2k}(\tilde q) \\
                  =&\left(ds\wedge\frac{\partial}{\partial
				s}+dt\wedge\frac{\partial}{\partial t}\right)\Ch_{2k}(q_t^s)+ d^B\left(\overline{\Ch}_{2k-1}^t(q^s_t)\wedge dt+ \overline{\Ch}_{2k-1}^s(q^s_t)\wedge ds\right)\\
			&+\left(-\frac{\partial}{\partial s}\overline{\Ch}_{2k-1}^t(q^s_t)+\frac{\partial}{\partial t}\overline{\Ch}_{2k-1}^s(q^s_t)+ d^B\overline{\overline{\Ch}}_{2k-2}(q_t^s)\right)\wedge dt\wedge ds.
		\end{split}
	\end{equation}
	In particular, $\frac{\partial}{\partial s}\overline{\Ch}_{2k-1}^t(q^s_t)+\frac{\partial}{\partial t}\overline{\Ch}_{2k-1}^s(q^s_t)+ d^B\overline{\overline{\Ch}}_{2k-2}(q_t^s)=0$.
	So we have that 
	\begin{equation}
		\begin{split}
                  \int_0^1&\left(\overline{\Ch}^t_{2k-1}(q_t^0)-\overline{\Ch}^t_{2k-1}(q_t^1)\right)dt\\
                  &=\int_0^1\int_0^1\frac{\partial}{\partial s}\overline{\Ch}^t_{2k-1}(q_t^s)dt\wedge ds\\
			&=-\int_0^1\int_0^1 \left(\frac{\partial}{\partial
                          t}\overline{\Ch}_{2k-1}^s(q^s_t)+
                          d^B\overline{\overline{\Ch}}_{2k-1}(q_t^s)\right)dt\wedge
                          ds\\
                  &=-\int_0^1\left(\overline{\Ch}_{2k-1}^s(q^s_0)-\overline{\Ch}_{2k-1}^s(q^s_1)\right)ds+ d^B\omega\\
			&= -\Phi_{2k-1}\left(\int_0^1\left(\overline{\Ch}_{2k-1}^s(p^s_0)-\overline{\Ch}_{2k-1}^s(p^s_1)\right)ds\right)+ d^B\omega
		\end{split}
	\end{equation}
	for some $\omega\in \Omega^B_{2k-1}$.
	This implies that 
	\begin{equation}
		\begin{split}
			\Ch^{rel}_{2k}&([p_0^0,p_1^0;q_t^0])-\Ch^{rel}_{2k}([p_0^1,p_1^1;q_t^1])\\
			&= \Big((\Ch_{2k}(p^0_0)-
                          \Ch_{2k}(p^1_0))-(\Ch_{2k}(p_1^0)-\Ch_{2k}(p^1_1)),\\
                  &\qquad\qquad\qquad\qquad\int_0^1( \overline{\Ch}_{2k-1}(q^0_t)-\overline{\Ch}_{2k-1}(q^1_t))dt\Big)\\
                                      &=\Big(d^A\int_0^1\left(\overline{\Ch}_{2k-1}^s(p^s_0)-\overline{\Ch}_{2k-1}^s(p^s_1)\right)ds,\\
                  &\qquad\qquad-\Phi_{2k-1}\left(\int_0^1\left(\overline{\Ch}_{2k-1}^s(p^s_0)-\overline{\Ch}_{2k-1}^s(p^s_1)\right)ds\right)+ d^B\omega\Big)
		\end{split}
	\end{equation}
	where we have applied \eqref{ch-int} to the first entry.
	The last term is a boundary in the mapping cone complex. This,  together with the fact that evidently $\Ch_{2k}^{rel}$ is unchanged under stabilizations, concludes the proof.
\end{proof}
\begin{proposition}\label{rel-functoriality}
	Let us consider the following commutative square of $DGA$ morphisms $$\xymatrix{\Omega^{A}_\bullet\ar[r]^{\Phi_\bullet}\ar[d]^{\alpha_\bullet}&
		\Omega^{B}_\bullet\ar[d]^{\beta_\bullet}\\
		\Omega^{A'}_\bullet\ar[r]^{\Phi'_\bullet}&\Omega^{B'}_\bullet}$$
Let $(\alpha_0,\beta_0)_*\colon K_
0(A\to B)\to K_0(A'\to B')$ and  $(\alpha,\beta)_*\colon H_*(\Phi)\to H_*(\Phi')$ denote the map induced in relative K-theory and relative non-commutative homology, respectively. Then we have that the relative Chern character is functorial, namely $$(\alpha,\beta)_*\Ch^{rel}_*([p_0,p_1;q_t])=\Ch^{rel}_*\left((\alpha_0,\beta_0)_*[p_0,p_1;q_t]\right)\in H_*(\Phi').$$
\end{proposition}
\begin{proof}
	This is a straightforward consequence of the second part of Lemma \ref{lem:transgress}.
\end{proof}
\begin{remark}\label{LMP}
	Notice that the definition of $\overline{\Ch}(p_t)$ is coherent with formula \cite[(1.46)]{LMP}, since up to commutators 
	\begin{equation}
		\begin{split}
			\Xi_t\Theta^{k-1}&=(p_tdp_t\dot{p}_t+p_t\dot{p}_tdp_t)(p_tdp_tdp_t)^{k-1}=\\
			&=(p_tdp_tdp_t)^{k-1}p_tdp_t\dot{p}_t+p_t\dot{p}_tdp_t(p_tdp_tdp_t)^{k-1}=\\
			&=(p_tdp_tdp_t)^{k}\dot{p}_t+p_t\dot{p}_t(1-p_t)(dp_tdp_t)^{k-1}dp_t=\\
			&=(\dot{p}_tp_t+p_t\dot{p}_t(1-p_t))(dp_tdp_t)^{k-1}dp_t=\\
			&=((1-p_t)(p_t\dot{p}_t-\dot{p}_t)+p_t\dot{p_t}) dp_t(dp_tdp_t)^{k-1}=\\
			&=((p_t-1)\dot{p}_t+p_t\dot{p_t}) dp_t(dp_tdp_t)^{k-1}=\\
			&=(2p_t-1)\dot{p}_t dp_t (dp_tdp_t)^{k-1}
		\end{split}
	\end{equation}
	which is exactly the integrand in   \cite[(1.46)]{LMP}. 
\end{remark}

Let us now consider the odd case. Let $[u]$ be a class in $K_1(A)$. By suspension, it corresponds to a class $[p_t]\in K_0(A\widehat\otimes C^\infty_0(0,1))$. Then we have the following definition.
\begin{definition}
	The degree $2k+1$ part of the Chern character of $[u]\in K_1(A)$
	associated to the DGA $\Omega^A_\bullet$ over $A$ is defined as the class 
	\begin{equation}
		\Ch_{2k+1}([u]):=\int_0^1\overline{\Ch}_{2k+1}(p_t)\;dt\in  H_{2k+1}(\Omega^A_\bullet).
	\end{equation}
\end{definition}

Of course, there is also a direct description of $\Ch_{2k+1}([u])$ in terms
	of the invertible element $u$. We use suspension because it is more
	convenient for our treatment of relative K-theory. We record an alternative description in the
        following lemma, proved e.g.~in \cite[Proposition 1.170]{MoscoviciWu}.
  \begin{lemma}\label{lem:odd_Chern_direct}
    Let $[u]\in K_1(A)$ be represented by an $n\times n$ invertible matrix
    $u$, and let $\Omega_\bullet^A$ be a DGA over $A$. Assume that $u$ remains
    invertible in $\Omega_\bullet^A$. Then we have
    \begin{equation}\label{eq:Ch_odd_univ_explicit}
      \Ch_{2k+1}([u]) = c_k \Tr(u^{-1}du(du^{-1}du)^{k})= c_k \Tr((u^{-1}du)^{2k+1}) \in H_{2k+1}(\Omega^A_\bullet),
    \end{equation}
    where $c_k$ is a universal constant.
  \end{lemma}
  \begin{proof}
    Strictly speaking, \cite{MoscoviciWu} works with the cyclic bicomplex
    instead of non-commutative de Rham homology. The computations, however,
    translate directly. The universal constant can be derived from the proof
    of \cite{MoscoviciWu}, which is a bit cumbersome because of the various
    normalizations chosen which one has to keep track of.
  \end{proof}

Let now $[u;v_t]$ be a relative class in $K_1(\varphi)$; thus $v_t$ is a
path of invertible elements from $\varphi(u)$ to the identity. Then
by suspension, it corresponds to a class $[p^0_{t}, p_{t}^1;
q_{t}^s]$ in the relative K-theory of $\varphi\otimes
\mathrm{id}\colon A\widehat\otimes C^{\infty}_0(0,1)\to
B\widehat\otimes C^{\infty}_0(0,1)$, where these projectors are
explicitly given in \cite[p.~322]{MoscoviciWu}.

\begin{definition}\label{ch-rel-odd}
	The degree $2k$ part of the relative Chern character of $[u, v_t]\in K_1(\varphi)$
	associated to a pair of DGAs $\Phi_\bullet\colon \Omega^A_\bullet\to\Omega^B_\bullet$ over $\varphi$ is  defined as the relative class
	\begin{multline}
		\Ch^{rel}_{2k}([u;v_t]):=
		\left(\int_0^1\overline{\Ch}_{2k+1}(p^0_t)\;dt-\int_0^1\overline{\Ch}_{2k+1}(p^1_t)\;dt,\right.\\
                \left.  \int_0^1\int_0^1
		\overline{\overline{\Ch}}_{2k}(q_{t,s})\wedge ds\wedge dt\right)\in H_{2k}(\Phi).
	\end{multline}
\end{definition}
The fact that the odd absolute and relative Chern characters are well-defined
and functorial is proved exactly as Proposition \ref{chern-well-defined}.

We note that relative and absolute Chern character are compatible to each other in the
following sense:
\begin{proposition}\label{prop:chern_and_boundary}
	For an algebra homomorphism $\varphi\colon A\to B$ and a homomorphism of DGAs
	$\Phi\colon \Omega_\bullet^A\to \Omega_\bullet^B$ covering $\varphi$, the
	following diagram is commutative:
	\begin{equation*}
		\xymatrix{
			K_*(\varphi)\ar[r]\ar[d]^{\Ch^{rel}_{2k-1+*}}& K_*(A)\ar[d]^{\Ch_{2k+*}}\\
			H_{2k-1+*}(\Phi) \ar[r]^{\partial}&  H_{2k+*}(\Omega^A_\bullet)
		}
	\end{equation*}
\end{proposition}
\begin{proof}
	This is a direct consequence of the definitions and holds on the cycle
	level: the natural map $K_*(\varphi)\to K_*(A)$ is given by projection onto the
	$A$-summand of a cycle, and the same applies to the boundary map
	$H_{2k-1+*}(\Phi)\to H_{2k+*}(\Omega^A_\bullet)$, and the $A$-summand of the
	relative Chern character is precisely the absolute $A$-Chern character.
\end{proof}

\begin{definition}
	We define the \emph{total Chern character} $K_*(A)\to H_{[*]}(\Omega_\bullet^A)$
	as the product of the $\Ch_{*+2k}$. Here, we write $H_{[*]}(\Omega^A_\bullet):=
	\prod_{k\in\integers} H_{*+2k}(\Omega^A_\bullet)$, homology made
	$\integers/2$-periodic, and correspondingly for the relative Chern
	character $K_*(\varphi)\to H_{[*+1]}(\Phi)$. 
\end{definition}

Finally, we will need the following result. 
\begin{lemma}\label{chern-bott}
	Let $\beta$ be an idempotent matrix over the unitalization of $ C_0((0,1)\times (0,1))$
	representing  the standard generator of $K_0(C_0((0,1)\times(0,1)))$.  Then we have the identity
	\[
	\int_0^1\int_0^1\overline{\overline{\Ch}}_{2k-2}(p\otimes \beta
	)dt\wedge ds= \Ch_{2k}(p)\in 	H^+_{2k}(A) \qquad\text{	for
		all }[p]\in K_0(A).
	\]
\end{lemma}

\begin{proof}
Choose $\beta\in
		C^\infty([0,1]\times[0,1])$ which is constant near the boundary. It
		is a standard fact that then $\int_0^1\int_0^1\Tr (\beta d\beta d\beta)
		= 1$.
	Then one has the following calculations
	\begin{equation*}
		\begin{split}
			\Ch_{2k}(p)&=	\Ch_{2k}(p)\otimes 1\\
			&= 	\Ch_{2k}(p)\otimes \int_0^1\int_0^1\Tr(\beta d\beta d\beta) = \int_0^1\int_0^1\Ch_{2k}(p)\otimes \Tr(\beta d\beta d\beta) \\
			&= \int_0^1\int_0^1 \Ch_{2k}(p\otimes \beta) =\int_0^1\int_0^1 \overline{\overline{\Ch}}_{2k-2}(p\otimes \beta) dt\wedge ds.
		\end{split}
	\end{equation*}
	In the last equality we used formula \eqref{secondary-transgression}.
	For the  last but one equality we observe that $$\Ch_{2k}(p\otimes\beta)=\Tr(p\otimes \beta d(p\otimes\beta)d(p\otimes\beta))^k=\Tr(p\otimes \beta (dp\otimes\beta+p\otimes d\beta)^k); $$ thus the equality follows because all the terms in the integral which involve higher powers of $d\beta d\beta$ disappear. 
\end{proof}

\section{Excision in non-commutative de Rham homology}

We have introduced relative de Rham cohomology for a DGA map which, by
construction, appears as third term in an associated long exact sequence in
non-commutative de Rham homology. Excision results concern the question under
which conditions an absolute non-commutative de Rham homology can be used
instead of the relative one, and we discuss some aspects of this briefly in
this section.

Let
\begin{equation}\label{eq:DGA_ext}
0\to \Omega^I\xrightarrow{\iota} \Omega^A\xrightarrow{q} \Omega^Q\to 0
\end{equation}
be an extension of Fr\`echet DGAs. A natural question now is when the
cohomology of $\Omega^I$ can be used instead of the (shifted) relative
cohomology of $q$ in an associated long exact sequence in homology.

A first guess might be that this should always be the case and just be based
on general homological algebra. There is one subtlety, however: before taking
homology of a chain complex, we pass to the quotient modulo commutators, and
this process in general will not preserve exactness. The relevant notion is
the following.

\begin{definition}\label{def:excisive}
  The DGA $\Omega^I$ is called \emph{split excisive} if $(\Omega^I)^2=\Omega^I$.
In the Fr\'{e}chet world we define that it is split excisive if $\overline{(\Omega^I)^2}=\Omega^I$.
\end{definition}

\begin{lemma}\label{lem:deRham_excision}
  Assume that $\Omega^I$ is split excisive and that the short exact sequence
  \eqref{eq:DGA_ext} has a DGA split $\alpha\colon \Omega^Q\to \Omega^A$. Then
  we have a long exact sequence in 
  non-commutative de Rham homology
  \begin{equation}\label{excisive}
   \cdots \to H_{*+1}(\Omega^Q)\xrightarrow{\partial}
    H_*(\Omega^I)\xrightarrow{\iota_*} H_*(\Omega^A) \xrightarrow{q_*}
    H_*(\Omega^Q) \xrightarrow{\delta}\cdots .
  \end{equation}

  In general, define $H_*(\Omega^I\subset\Omega^A)$ as the homology of the
  chain complex $\Omega^I/(\Omega^I\cap \overline{[\Omega^A,\Omega^A]})$,
  i.e.~we divide by all elements in $\Omega^I$ which are (limits of sums of) commutators
  of elements in $\Omega^A$. As the differential preserves commutators and
  preserves $\Omega^I$, this is indeed a chain complex. We get a long exact
  sequence in 
  non-commutative de Rham homology
  \begin{equation}\label{non-excisive}
    \cdots\to H_{*+1}(\Omega^Q)\xrightarrow{\partial}
    H_*(\Omega^I\subset \Omega^A)\xrightarrow{\iota_*} H_*(\Omega^A) \xrightarrow{q_*}
    H_*(\Omega^Q) \xrightarrow{\delta}\cdots.
  \end{equation}
  There is a canonical map $H_*(\Omega^I)\to H_*(\Omega^I\subset \Omega^A)$
  induced by the canonical projection of chain complexes defining the homology
  groups. 
\end{lemma}

  \begin{remark}
    Here, the notation $H_*(\Omega^I\subset\Omega^A)$ is \emph{not} supposed
    to indicate a relative homology group, it is rather meant to be the
    corrected replacement of $H_*(\Omega^I)$ which is needed for excision,
    i.e.~to obtain the exact sequence \eqref{non-excisive} as replacement of
    the desired \eqref{excisive} which holds only if $\Omega^I$ is excisive.
    The relative homology would be denoted $H_*(\Omega^I\to \Omega^A)$.
  \end{remark}

\begin{proof}
  The short exact sequence $0\to \Omega^I\to \Omega^A\to \Omega^Q$ induces by
  elementary considerations an
  exact sequence
  \begin{equation*}
   \Omega^I_{ab}\to  \Omega^A_{ab}\to \Omega^Q_{ab}\to 0.
  \end{equation*}
  However, the first map will in general not be injective: precisely those
  elements in $\Omega^I$ which are commutators in $\Omega^A$ are mapped to
  zero. We obtain the desired short exact sequence of chain complexes
  \begin{equation*}
    0\to \Omega^I/(\Omega^I\cap \overline{[\Omega^A,\Omega^A]}) \to
    \Omega^A_{ab}\to \Omega^Q_{ab}\to 0
  \end{equation*}
  and homological algebra gives the long exact sequence \eqref{non-excisive}.

  For the split excisive case, we simply have to check that
  $[\Omega^I,\Omega^I]=\Omega^I\cap [\Omega^A,\Omega^A]$. Write
  ${\Omega^Q}':=\alpha(\Omega^Q)\subset \Omega^A$, this is a sub-DGA and
  $\Omega^A=\Omega^I+{\Omega^Q}'$ with $\Omega^I\cap{\Omega^Q}'=\{0\}$. Note that
  this is not a direct sum of DGAs in general, as we might have non-trivial
  mixed products, but it is a direct sum of Fr\`echet spaces.

  We have
  \begin{equation*}
    [\Omega^A,\Omega^A]=\underbrace{[\Omega^I,\Omega^I]}_{\subset
      \Omega^I} + \underbrace{[\Omega^I,{\Omega^Q}']}_{\subset \Omega^I} +
    \underbrace{[{\Omega^Q}',{\Omega^Q}']}_{\subset {\Omega^Q}'}.
  \end{equation*}
  Therefore $\Omega^I\cap \overline{[\Omega^A,\Omega^A]} =
  \overline{[\Omega^I,\Omega^I]} + \overline{[\Omega^I,\Omega^A]}$ because $\Omega^I\cap
  {\Omega^Q}'=\{0\}$.

 Finally, using the assumption that $\Omega^I$ is Fr\'{e}chet split excisive, we can
  for $a\in\Omega^I$ and $T\in\Omega^A$ choose $u_i,v_i\in\Omega^I$ for $i\in \NN$, with $a=\sum _{i\in \NN}u_iv_i$
  and obtain (in case $u_i,v_i,T$ are all of even degree)
  \begin{equation*}
    \begin{split}
      [a,T] & =\sum_{i\in\NN}(u_iv_iT - Tu_iv_i)\\
      &=\sum_{i\in\NN}\Bigl(
      (u_i(v_iT)-(v_iT)u_i) + (v_i(Tu_i)-(Tu_i)v_i)\Bigr)\in \overline{[\Omega^I,\Omega^I]}.
    \end{split}
  \end{equation*}
Also in the case of general degree, an easy computation shows that the signs for the graded
commutator work out right and we get the desired inclusion
$[\Omega^I,\Omega^A]\subset [\Omega^I,\Omega^I]$ ($\overline{[\Omega^I,\Omega^A]}\subset\overline{ [\Omega^I,\Omega^I]}$), resulting in the
identification 
\begin{equation*}
  \Omega_{I,ab}\xrightarrow{\iso}
  \Omega^I/(\Omega^I\cap\overline{[\Omega^A,\Omega^A]})
\end{equation*}
 and therefore $H_*(\Omega^I)\xrightarrow{\iso}
 H_*(\Omega^I\subset\Omega^A)$.

\end{proof}

\section{Splitting constructions in relative K-theory and homology}

In this section, we discuss excision style results for relative K-theory
and for de Rham cohomology in the special case of
split short exact sequences. This will be needed in Section \ref{section6} for our
construction of the Chern character map from analytic surgery to homology,
i.e.~for the main construction of this paper.

Let us fix the following algebraic setting: $A$ is a Fr\'echet algebra, dense
and holomorphically closed in a C*-algebra $\mathbf{A}$; moreover let
$I\subset A$ be a closed ideal and $B\subset A$ a closed subalgebra such that
$I\cap B=\{0\}$. 

\begin{remark}\label{split}
	Because of the condition $I\cap B=\{0\}$ we have the split  exact
        sequence (with obvious maps)
	\[
	\xymatrix{0\ar[r]& I\ar[r]& I+B\ar[r] & B\ar[r]\ar@/_1pc/[l]_{\alpha}& 0},
	\] 
	which gives an isomorphism $K_*(I)\cong K_*(I+B)/\alpha_*(K_*(B))$.
\end{remark}

Let $\partial\colon K_*(B\to A/I)\to K_{*+1}(0\to I)$ be the boundary map associated to the following short exact sequence
\[
\xymatrix{0\ar[r]&(0\to I)\ar[r]& (B\to A)\ar[r] & (B\to A/I)\ar[r]& 0},
\]
where $(B\to A)$ and $(B\to A/I)$ denote the mapping cone algebras and $(0\to I)$ is a way of writing the suspension of $I$ as a mapping cone.
\begin{lemma}\label{iso-mapping-cones}
	The following diagram is commutative with vertical isomorphisms
	\begin{equation}
{\small	\xymatrix{\cdots \ar[r] K_{*-1}(I) \ar[r]^{\iota_*\circ S}\ar[d]_S & K_*(B\to A)\ar[r]\ar[d]_{\mathrm{id}}&K_*(B+I\to A)\ar[r]^(.6){\partial'}\ar[d]^\cong_{q_*}& K_*(I)\ar[d]_S\ar[r]^S& \cdots\\
		\cdots\ar[r] K_*(0\to I) \ar[r]^{\iota_*} & K_*(B\to A)\ar[r]&
                K_*(B\to A/I)\ar[r]^\partial & K_{*+1}(0\to I)\ar[r]&\cdots}
              }
	\end{equation}
	where $q\colon (B+I\to A)\to (B\to A/I)$ is the quotient  by $(I\to I)$, $S$ is the suspension isomorphism and $\partial'$ is the following composition
\begin{equation}\label{eq:def_deltaprime}
	\xymatrix{K_*(B+I\to A)\ar[r]&
          K_*(B+I)\ar[r]&K_*(B+I)/\alpha_*(K_*(B))&
          K_*(I)\ar[l]_(.3){\cong}}.        
\end{equation}
Consequently the top row is a long exact sequence.
\end{lemma}

The top row could be interpreted as the long exact K-theory sequence of
a hypothetical extension $0\to (B\to A)\to (B+I\to A) \to (I\to 0)\to
0$. Of course, as $B$ is \emph{not} assumed to be an ideal in $B+I$,
this is not an extension of $C^*$-algebras. Nonetheless, the lemma uses
the split $\alpha$ to obtain  the associated K-theory
sequence, and gives explicitly the relevant maps.

\begin{proof}
	We only need to check the commutativity of the third square and this  consists in chasing the following diagram.
	\[\small
	\xymatrix{
		K_*(B+I\to A)\ar@/^'2pc/[rrr]^{\partial'}\ar[r]\ar[rdd]_{i_*}\ar[ddd]^{\iso}_{q_*}& K_*(B+I)\ar[r]\ar[d]^S&K_*(B+I)/\alpha_*(K_*(B))\ar[dl]_{\lambda}^\cong& K_*(I)\ar[ddd]^S\ar[l]_(.3){\cong}\\
		&K_{*+1}\left(\substack{B\\ B+I}\right)& K_{*+1}\left(\substack{0\to A\\I \to A}\right)\ar[dl]^\cong\ar[d]^\cong\ar[ur]_{\mu}^\cong&\\
		&  K_{*+1}\left(\substack{B\to A\\B+I\to A}\right)\ar[u]_{\iso}\ar[d]_\cong& K_{*+1}\left(\substack{0\to A\\0\to A/I}\right)\ar[dl]_\cong &\\
		K_*(B\to A/I)\ar[r]\ar@/_2pc/[rrr]_{\partial}& K_{*+1}\left(\substack{B\to A\\B\to A/I}\right)& & K_{*+1}(0\to I)\ar[ll]\ar[lu]_\nu^\cong
	}
	\]
	Here $K_{*+1}\begin{pmatrix}B\to A\\ B\to A/I\end{pmatrix}$ denotes
	the K-theory of the double mapping cone $((B\to A)\to (B\to A/I))$
	(the same holds for the other similar K-groups). The map $i_*$ is
	given by the composition of the suspension isomorphism  and the
	natural inclusion; $\lambda $ is induced by the obvious map
	$K_{*}(B+I)\to K_{*+1}(B\to B+I)$ and it is well defined on the
	quotient by $\alpha_*K_*(B)$ because its restriction to
	$\alpha_*K_*(B)$ factors through $K_{*+1}(B\to B)=0$;   $\mu$ denotes
	the quotient by $\begin{pmatrix}0\to A\\ 0\to A\end{pmatrix}$ composed
	with the Bott periodicity isomorphism; $\nu$ is induced by the inclusion of $I$
	into $A$. The only part which is not clearly commutative is the right
	hand-side trapezoid, but this can be easily proved using \cite[Lemma
	3.6]{Zenobi}. We leave the details to the reader, knowing that $A$,
	$B$ and $C$ in the notation of \cite[Lemma 3.6]{Zenobi} are here
	represented by $\begin{pmatrix}0\to A\\ I\to A\end{pmatrix}$,
	$\begin{pmatrix}0\to A\\ 0\to A/I\end{pmatrix}$ and $(0\to I)$ and
	that $\partial_B$ is the isomorphism given by the composition of the
	Bott periodicity isomorphism and the inverse in K-theory of the
        inclusion of $I$
	in $(A\to A/I)$ and $\partial_C$ is just the suspension isomorphism.
\end{proof}

\bigskip
We now pass to the corresponding situation in non-commutative de Rham
homology. For this, assume that $\Omega^A$ is a (Fr\`echet) DGA over $A$
containing a closed DG
ideal $\Omega^I$ over $I$ and a closed sub-DGA $\Omega^B$ over $B$. Moreover,
assume that $\Omega^I\cap\Omega^B=\{0\}$. Due to the fact that $\Omega^I$ is a
DG ideal, $\Omega^I+\Omega^B$ is a DGA (over $I+B$).
As for K-theory in Remark \ref{split} the obvious inclusion defines a 
split $\alpha$ as DGA map of the sequence 
  \begin{equation}
\xymatrix{0\ar[r]& \Omega^I\ar[r]& \Omega^I+\Omega^B\ar[r] &
  \Omega^B\ar[r]\ar@/_1pc/[l]^{\alpha}& 0}.\label{eq:split_DGA_seq}
\end{equation}

For the following, we assume that $\Omega^I$ is split excisive as in
Definition \ref{def:excisive}, which will be the case whenever we apply the
general theory we are developing here.

\begin{remark}
  The following discussion easily generalizes to the case that $\Omega^I$ is
  not split excisive, just replacing $H_*(\Omega^I)$ by
  $H_*(\Omega^I\subset\Omega^I+\Omega^B)$ throughout. We decided to state it
  with the extra assumption because we will only apply this special case, and
  the notation then is a bit less cluttered.
\end{remark}

\begin{proposition}
Adopt the situation described above. Let $\Omega^I$ be split excisive.
 The split $\alpha$ of \eqref{eq:split_DGA_seq} gives rise to a long exact sequence 
 \begin{multline}\label{delta'-homology}
		\cdots\to  H_{[*-1]}(\Omega^I) \to H_{[*-1]}(\Omega^B\to
                \Omega^A)\to\\  H_{[*-1]}(\Omega^I+\Omega^B\to
                \Omega^A)\xrightarrow{\delta'} H_{[*]}(\Omega^I) \to \cdots 
% \xymatrix{		\cdots\ar[r] & H_{[*-1]}(\Omega^I) \ar[r] &H_{[*-1]}(\Omega^B\to \Omega^A)\ar[r]	  &  H_{[*-1]}(\Omega^I+\Omega^B\to \Omega^A)\ar[r]^(.65){\delta'}& H_{[*]}(\Omega^I) \ar[r]&\cdots }
\end{multline}
 where the  boundary map
  $\delta'\colon H_{*-1}(\Omega^I+\Omega^B\to \Omega^A)\to H_*(\Omega^I)$ is 
  defined by
  \begin{multline}\label{delta'}
    H_{*-1}(\Omega^I+\Omega^B\to \Omega^A)\to H_*(\Omega^B+\Omega^I)\\
      \to H_*(\Omega^B+\Omega^I)/\alpha_*(H_*(\Omega^B))\xleftarrow{\cong} H_*(\Omega^I).
    % \xymatrix{H_{*-1}(\Omega^I+\Omega^B\to \Omega^A)\ar[r]& H_*(\Omega^B+\Omega^I)\ar[r]&H_*(\Omega^B+\Omega^I)/\alpha_*(H_*(\Omega^B))& H_*(\Omega^I)\ar[l]_(.3){\cong}}.
  \end{multline}
\end{proposition}
\begin{proof}
Since $(\Omega^I+\Omega^B)_{ab}=(\Omega^I)_{ab}+(\Omega^B)_{ab}$ by split
  excisiveness, we have the following exact sequence of complexes
	\begin{equation}\label{eq:cone_seq}
	\xymatrix{0\ar[r]& (\Omega^B\to \Omega^A)_{ab}\ar[r]& (\Omega^I+\Omega^B\to \Omega^A)_{ab}\ar[r]^(.6){q}&(\Omega^I\to0)_{ab}\ar[r] & 0}.
	\end{equation}
	This gives rise to the long exact sequence	\begin{equation}\label{eq:h_cone_seq}
	\xymatrix{\cdots\ar[r]& H_*(\Omega^B\to \Omega^A)\ar[r]& H_*(\Omega^I+\Omega^B\to \Omega^A)\ar[r]^(.6){q_*}&H_*(\Omega^I\to0)\ar[r]^(.6){\delta} & \cdots},
	\end{equation}
	where $\delta$ is induced by the composition of the suspension $(\Omega^I\to 0)_\bullet\xrightarrow{S}(0\to \Omega^I)_{\bullet+1}$ and the inclusion of $(0\to\Omega^I)$ into $(\Omega^B\to\Omega^A)$.

Let us prove that the following square commutes.
	\[
	\xymatrix{
		H_*(\Omega^I+\Omega^B\to \Omega^A)\ar[r]^(.6){q_*}\ar[d]^=&H_*(\Omega^I\to0)\ar[d]^S\\
		H_*(\Omega^I+\Omega^B\to \Omega^A)\ar[r]^(.6){\delta'}&H_{*+1}(\Omega^I)}
	\]	
 Consider a class $[b+i, a]\in 	H_*(\Omega^I+\Omega^B\to \Omega^A)$, 
 then by definition $\delta'[b+i, a]=[i]\in H_{*+1}(\Omega^I) $ which is exactly the suspension of $q_*[b+i, a]=[i,0]\in H_*(\Omega^I\to0)$. This proves the commutativity of the square and implies the exactness of  \eqref{delta'-homology}.
\end{proof}

\begin{lemma}\label{lem:dprimnat}
  In the situation described so far, the Chern character is compatible with
  the boundary maps $\partial'$ of \eqref{eq:def_deltaprime} and $\delta'$ of \eqref{delta'}:
  \begin{equation*}
  \xymatrix{ K_{*}(I+B\to A)\ar[r]^(.6){\partial'}\ar[d]^{\Ch^{rel}}& K_*(I)\ar[d]^{\Ch}\\
  	  H_{[*-1]}(\Omega^I+\Omega^B\to\Omega^A)\ar[r]^(.7){\delta'}& H_{[*]}(\Omega^I)}
  \end{equation*}
\end{lemma}
\begin{proof}
  Thanks to Proposition \ref{rel-functoriality}, the Chern character is natural for maps of algebras covered
  by maps of DGAs and, by Proposition \ref{prop:chern_and_boundary}, for the boundary map in the usual pair sequence. We just
  have to apply this naturality to the individual maps whose
  composition defines $\partial'$ and $\delta'$, see respectively \eqref{eq:def_deltaprime} and \eqref{delta'}.
\end{proof}

 	\begin{remark}\label{mc-chern-commutativity}
 		It is clear that if the following square is commutative
 		\begin{equation}\label{sq}\xymatrix{A\ar[r]^\alpha\ar[d]^\varphi&A'\ar[d]^{\varphi'}\\ B\ar[r]^\beta& B'}\end{equation}  then from the definition of the relative Chern character we have that also the following square is commutative
 		\[
 		\xymatrix{K_*(A\xrightarrow{\varphi}B)\ar[r]^{(\alpha, \beta)_*}\ar[d]^{\Ch^{rel}}& K_*(A'\xrightarrow{\varphi'}B')\ar[d]^{\Ch^{rel}}\\
 			H_*(A\to B)\ar[r]^{(\alpha,\beta)_*}& H_*(A'\to B')}
 		\]
 		namely the relative Chern character is functorial with respect to commutative squares as \eqref{sq}.
 	\end{remark}
 By the functoriality of the Chern character and its compatibility with the splitting given by $\alpha$, we have the following result.

\begin{theorem}\label{functoriality-chern}
	Let $\Omega^I$ be split excisive.
	The following diagram is commutative
	\begin{equation}\label{eq:functor}\small
	\xymatrix{ K_{*+1}(I) \ar[d]^{\Ch} \ar[r]^{\iota_*\circ S} & K_*(B\to A)\ar[r]\ar[d]_{\Ch^{rel}}&K_*(B+I\to A)\ar[d]_{\Ch^{rel}}\ar[r]^{\partial'}& K_*(I)\ar[d]_\Ch\\
		 H_{[*+1]}(\Omega^I) \ar[r] &H_{[*+1]}(\Omega^B\to \Omega^A)\ar[r]	  &  H_{[*+1]}(\Omega^I+\Omega^B\to \Omega^A)\ar[r]^(.65){\delta'}& H_{[*]}(\Omega^I)  }
	\end{equation}
\end{theorem}
\begin{proof}    
	For the left square, consider the diagram
	\begin{equation}\label{eq:iS}
	\xymatrix{     K_{*-1}(I)\ar[r]^S_\iso\ar[d]^\Ch&K_*(0\to I)\ar[r]^{\iota_*}\ar[d]^{\Ch^{rel}}&  K_*(B\to A) \ar[d]^{\Ch^{rel}}\\
		H_{[*+1]}(\Omega^I) \ar[r]^(.4)\equiv&H_{[*+1]}(0\to \Omega^I) \ar[r]^(.45){\iota_*}&
		H_{[*+1]}(\Omega^B\to \Omega^A)
	}
	\end{equation}
	For $*$ even, $\Ch\colon K_1(I)\to H_{odd}(\Omega^I)$ is defined via the
	suspension isomorphism and $\overline{\Ch}_{odd}$, which also enters in the definition of
	$\Ch^{rel}$. In this case, it is an immediate consequence of the definitions
	that the left square of \eqref{eq:iS} commutes. If $*$ is odd, by definition
	$\Ch^{rel}\circ S\colon K_0(I)\to H_{ev}(0\to \Omega^I)$ is computed from the double suspension. Explicitly,
	if $p$ is a projector over $I$, then
        \begin{equation*}
          \Ch^{rel}\circ S(p) =
          [0,\int_0^1\int_0^1 \overline{\overline{\Ch}}_{ev}(p\tensor\beta)\,dt\wedge
          ds] \in H_{{ev}}(0\to\Omega^I),
        \end{equation*}
        where $\beta$ is the Bott projector
	as in Lemma \ref{chern-bott}. By this lemma, this class equals $\Ch(p)$
	under the identification $H_*(\Omega^I)\equiv H_*(0\to
	\Omega^I)$. Therefore, the left square of 
	\eqref{eq:iS} commutes in all degrees. The right square of \eqref{eq:iS} commutes
	because of naturality of the relative Chern character, Proposition
	\ref{rel-functoriality}. This also gives the commutativity of the middle square
	of \eqref{eq:functor}.
Moreover, \eqref{eq:iS} is a factorisation of the left square of
	\eqref{eq:functor} which therefore also commutes. Finally, the right square of \eqref{eq:functor} commutes because of Lemma \ref{lem:dprimnat}.
\end{proof}

\chapter{Pseudodifferential operators and smooth  subalgebras}\label{section5}

  The index theoretic approach to secondary invariants e.g.~of metrics of
  positive scalar curvature requires the detailed work with operators derived
  from the Dirac operator. These live in various calculi of
  pseudodifferential operators and the properties of these calculi play an important
  role. For primary invariants, it turns out that the main aspects are
  captured by smoothing operators. However, this is different for the
  secondary invariants we want to study. Our setup leads us to
  deal with certain operators of order zero.

  To capture the full information, in particular provided by the fundamental
  group, the 
  efficient way to proceed is to enhance the bundles by twisting with the 
  Mishchenko bundle and associated
  bundles. Equivalently, one can consider $\Gamma$-equivariant objects on the
  universal covering. This is perhaps less conceptual, but allows for explicit
  computations and fine tuned control due to the concreteness.

In this section we shall give a detailed account of classical facts about pseudodifferential calculus in the Mishchenko-Fomenko setting  and
holomorphically closed algebras of pseudodifferential operators, mostly
following the work of John Lott \cite{Lott1}. 

\section{Mishchenko bundle translated to universal covering}\label{sec:translations_to_invariant}

We want to work rather explicitly with the space of sections of the Mishchenko
bundle and with certain classes of pseudodifferential operators on it. For
this, it is useful and necessary to consider a couple of identifications.
We use the following setup:
\begin{itemize}
  \item a compact smooth manifold $M$ without boundary and a complex vector bundle $E$ over $M$;
  \item a $\Gamma$-principal bundle ${\tM}\to M$, where $\Gamma$ acts from
  the right and we denote  by 	$R_\gamma\colon {\tM}\to {\tM}$ the (right) action map for $\gamma\in \Gamma$; this induces a natural action on $\widetilde{E}$,  the equivariant lift of $E$ to ${\tM}$;
  \item dually, this action defines a      right $\CC\Gamma$-module structure on $C_c^{\infty}(\tM, \widetilde{E})$ given by $f\cdot \gamma:= R^*_{\gamma^{-1}}f$;
  \item the Mishchenko bundle
$$\V_{\complexs \Gamma}:=\widetilde
  M\times_\Gamma\complexs\Gamma$$ is defined as the bundle associated to the right action of $\Gamma$ on $\widetilde{M}$ and the left multiplication action of $\Gamma$ on $\CC\Gamma$: it is then defined as the quotient  $\widetilde{M}\times \CC\Gamma/\sim$, where $(R_{\gamma^{-1}}(\tilde{m}), \lambda)\sim (\tilde{m}, \gamma \lambda)$. This gives a flat bundle of
  $\complexs\Gamma$-right modules (fiberwise free of rank one) on $M$. More generally if
  $\Gamma$ acts on an algebra $\mathcal{A}$  by right $\mathcal{A}$-module
  automorphisms (e.g. consider $C^*_r\Gamma$ acted upon by left multiplication), we will denote by ${\V}_{\mathcal{A}}$ the bundle  $\widetilde
  M\times_\Gamma\mathcal{A}$ of $\mathcal{A}$-right modules as before. 
  
  An important special case of this is when $\mathcal{A}$ is a DGA over $\complexs\Gamma$,
  e.g.~$\Omega_\bullet(\complexs\Gamma)$. Then $\complexs\Gamma$-linear operators between
  $C^\infty(M;E\tensor\V_{\CC\Gamma})$ and $C^\infty(M;E\tensor
  \V_{\mathcal{A}})$ extend uniquely to $\mathcal{A}$-linear
  operators on $C^\infty(M;E\tensor\V_{\mathcal{A}})=
  C^\infty(M;E\tensor \V_{\CC\Gamma}\tensor_{\complexs\Gamma}\mathcal{A})$. This is
  treated in detail in Section \ref{Pdo-forms}.
  \item let  $\underline{\CC\Gamma}$ be the product bundle
    $\tM\times \CC\Gamma$; $\Gamma$ acts from the left on
    $C^\infty({\tM};\widetilde{E}\otimes\underline{\CC\Gamma})$ 
    via
    $\gamma\cdot \sum _{g\in \Gamma}f_g\,g :=\sum_{g\in
        \Gamma}R^*_{\gamma} f_g\,\gamma g $. This commutes with the
      right $\CC\Gamma$-action which just multiplies on the right in the
      argument. Let
          $\mathcal{E}^{\MF}_{\CC\Gamma}({\tM};\widetilde{E})$ denote  the $\Gamma$-invariant sections
          $C^\infty({\tM};\tE\otimes\underline{\CC\Gamma})^\Gamma$, it is a standard fact that  this right $\CC\Gamma$-module identifies with $
            C^\infty(M;E\tensor \V_{\CC\Gamma})$.
   
\end{itemize}

\begin{lemma}\label{lem:Lott} 
	There is a canonical identification of $\CC\Gamma$-modules
	\begin{equation*}
	L\colon    C^\infty_c({\tM};\widetilde{E})\xrightarrow{\iso} C^\infty(\tM;\widetilde{E}\otimes\underline{\CC\Gamma})^\Gamma \xrightarrow{\iso}
	C^\infty(M; E\otimes\V_{\complexs\Gamma}). 
	\end{equation*}
	The second isomorphism is a tautological one, just rewriting what a section
	of an associated bundle is in explicit terms.
	The first isomorphism is given by 
	\[
	s\mapsto \sum_{g\in \Gamma}(R^*_{g}s)\, g.
	\]
\end{lemma}
\begin{proof}
	The proof of this lemma is exactly as the proof of \cite[Proposition 5]{Lott1}, but in fact easier since it deals with $\CC\Gamma$ in place of the more complicated algebra $\mathcal{B}^\omega$, the algebra of function of exponential rapid decay.
\end{proof}
We shall now pass to pseudodifferential operators, concentrating first on the
smoothing ones.

  \begin{definition}
    The smoothing operators
    $\Psi_{\CC\Gamma}^{-\infty}(M,\V_{\complexs\Gamma})$ on
    the algebraic Mishchenko bundle $\V_{\complexs\Gamma}$ are the
    smooth sections of the bundle
  $$\Hom_{\complexs\Gamma}(pr_2^*\V_{\complexs\Gamma},pr_1^*\V_{\complexs\Gamma})\rightarrow
  M\times M.$$ As kernel functions, they act in the usual way on the space of
  sections $C^\infty(M;\V_{\complexs\Gamma})$ as right
  $\complexs\Gamma$-linear operators.
\end{definition}

  It will be necessary for us to work very explicitly with this bundle and its
  sections and their action on sections of
  $\V_{\complexs\Gamma}$. This is based on the following lemma.
  \begin{lemma}\label{lem:top_bottom_basic}
  	Let us consider $\mathrm{End}_{\CC\Gamma}(\CC\Gamma)$ where we use the right $\complexs\Gamma$-module structure on $\complexs\Gamma$.
    The bundle
    $\Hom_{\complexs\Gamma}(pr_2^*\V_{\complexs\Gamma},pr_1^*\V_{\complexs\Gamma})$
    over $M\times M$ is an associated bundle for the obvious right action of
    $\Gamma\times\Gamma$ on ${\tM}\times {\tM}$ and the left
    representation given by
    $\Gamma\times\Gamma\to \End_{\complexs\Gamma}(\complexs\Gamma)$
    \[(\gamma_1,\gamma_2)\cdot \phi= L_{\gamma_1}\circ\phi\circ
      L_{\gamma_2}^{-1}.\] Recalling that
    $\End_{\complexs\Gamma}(\complexs\Gamma)\iso\complexs\Gamma$ by means of
    the map $\phi\mapsto \phi(e)$, the action simply becomes
    $$(\gamma_1,\gamma_2)\cdot g= \gamma_1g\gamma_2^{-1} \text{ for }
    g\in\complexs\Gamma.$$
  In particular, we have the following identifications
    \begin{equation*}
      \begin{split}
        \Psi_{\CC\Gamma}^{-\infty}(M,{\V}_{\complexs\Gamma}) &\equiv
                                                               C^\infty(M\times M;\Hom(pr_2^*\V,pr_1^*{\V}))\\
        &\cong
        C^\infty(\tM\times\tM
          , \End_{\complexs\Gamma}(\complexs\Gamma))^{\Gamma\times\Gamma}\\
        &\cong
        C^\infty({\tM}\times\tM,\complexs\Gamma)^{\Gamma\times\Gamma}. 
      \end{split}
    \end{equation*}
    If
    $k_1,k_2\in C^\infty({\tM}\times\tM,\complexs\Gamma)^{\Gamma\times\Gamma}$ with
    $k_i=\sum_{g\in \Gamma}(k_i)_g\,g$ and
    $f\in C^\infty({\tM},\complexs \Gamma)^\Gamma$ is of the form
    $\sum_{g\in \Gamma}f_g\, g$ (all sums are locally finite), the product and the
    action are given by
    \begin{equation*}
      \begin{split}
        (k_1*k_2)_g(x,z) &= \sum_{h\in \Gamma}\int_{{\tM}/\Gamma} (k_1)_{gh^{-1}}(x,y)\cdot (k_2)_h(y,z)\,dy,\\
        (k_1 f)_g(x) &= \sum_{h\in\Gamma}\int_{{\tM}/\Gamma} (k_1)_{gh^{-1}}(x,y)f_h(y)\,dy,
      \end{split}
    \end{equation*}
    where we use that the expressions, as functions of $y$, are
    $\Gamma$-invariant by the invariance properties we assume. Of course, the
    integral over ${\tM}/\Gamma$ can also be written as an integral over a
    fundamental domain.
  \end{lemma}
  \begin{proof}
    We observe that
    $\Hom_{\complexs\Gamma}(\mathcal{E},\mathcal{F})=
    \mathcal{F}\tensor_{\complexs\Gamma} \mathcal{E}^*$ for two free finitely
    generated right $\complexs\Gamma$-modules, where
    $\mathcal{E}^*=\Hom_{\complexs\Gamma}(\mathcal{E},\complexs\Gamma)$. For
    us, $\mathcal{E}=\complexs\Gamma$ which has an additional $\Gamma$-action
    (by multiplication from the left), and $\gamma\in\Gamma$ acts on the left
    on $\mathcal{E}^*$ via $(\gamma\cdot f)(x):= f(\gamma^{-1} x)$ for
    $x\in \mathcal{E}$ and $f\in\mathcal{E}^*$. Under the identification
    $\complexs\Gamma\iso \Hom_{\complexs\Gamma}(\complexs
    \Gamma,\complexs\Gamma)$ this action becomes multiplication by
    $\gamma^{-1}$ on the right and this finally gives the description of the
    associated bundle as claimed.  The rest of the proposition follows from this description.
  \end{proof}

  Let us also fix the following identification of algebras
  \begin{equation}\label{def:Gamma_pseudos}
  \Pdo^{-\infty}_{\Gamma,c}({\tM})\cong C^\infty_c({\tM}\times{\tM})^\Gamma,
  \end{equation} 
  with composition on the left and convolution product on the right.
It identifies a smoothing $\Gamma$-equivariant operator of $\Gamma$-compact support with  its kernel, which is a smooth
  $\Gamma$-equivariant function on ${\tM}\times{\tM}$ with respect to the
  diagonal action of $\Gamma$ and with $\Gamma$-compact support.

\begin{proposition} \cite[Proposition 6]{Lott1} 
		We have the following  isomorphism of *-algebras
	\begin{equation}\label{lott-iso}
          \begin{split}
            C^\infty_{c}({\tM}\times{\tM})^\Gamma & \xrightarrow{\iso} C^\infty(\tM\times\tM,\complexs\Gamma)^{\Gamma\times\Gamma}
          \\
            k & \mapsto \sum_{\gamma\in\Gamma} R_{(\gamma,e)}^*k\,\gamma
          \end{split}
	\end{equation}
	where $R^*_{g,h}k(x,y):=k(xg,yh)$.
	Consequently 
	             $$\Psi^{-\infty}_{\Gamma,c}(\tM) \stackrel{
	             	 \eqref{def:Gamma_pseudos}}{\iso}
	             C^\infty_{c}({\tM}\times{\tM})^\Gamma \stackrel{\eqref{lott-iso}}{\iso} C^\infty(\tM\times\tM,\complexs\Gamma)^{\Gamma\times\Gamma}
	             \stackrel{ \ref{lem:top_bottom_basic}}{\iso} \Psi^{-\infty}_{\CC\Gamma}(M,\V_{\complexs\Gamma}). $$
	             	This composition is given by $Ad_L$, the adjoint of the canonical isomorphism $$L\colon C^\infty_c({\tM})\to C^\infty({\tM},\CC\Gamma)^\Gamma.$$
\end{proposition}

\noindent
All of this applies in a straightforward way if we add another auxiliary
$\Gamma$-equivariant vector bundle $\tE$, with algebra of smoothing operators
$\Psi^{-\infty}_{\Gamma,c}(\tM,\tE)$, which is isomorphic to $ \Psi^{-\infty}_{\CC\Gamma}(M,E\tensor
\V_{\complexs\Gamma})$.

\medskip
\noindent Moreover, passing to completions, $L$ clearly extends to an
  isomorphism of Hilbert $C^*_{red}\Gamma$-modules and the isomorphism \eqref{lott-iso} extends to an isomorphism 
of C*-algebras 
\begin{equation}\label{ad-l}
Ad_L\colon C^* ({\tM}; \tE)^\Gamma \to  \mathbb{K}(\mathcal{E}^{\MF}_{C^*_{red}\Gamma}({\tM}, {\tE}))
\end{equation}
where $\mathcal{E}^{\MF}_{C^*_{red}\Gamma}({\tM}; {\tE})$ 
is the Hilbert-$C^*_{red}\Gamma$-module of the measurable
sections of $E\otimes{\V}_{C^*_{red}\Gamma}$ with finite
  $C^*_{red}\Gamma$-valued norm, and $C^* ({\tM}; \tE)^\Gamma$ is the Roe algebra associated to the $\Gamma$-equivariant $C_0({\tM})$-module $L^2({\tM};{\tE})$.

Let $\Pdo_{\Gamma,c}^0({\tM},{\tE})$ be the *-algebra of
$\Gamma$-equivariant  pseudodifferential operators of order 0 on $C^\infty({\tM};\tE)$ of $\Gamma$-compact support. 
We can extend $Ad_L$ to an isomorphism between the multiplier algebra of $C^* ({\tM}; \tE))^\Gamma$ and the multiplier algebra of $ \mathbb{K}(\mathcal{E}^{\MF}_{C^*_{red}\Gamma}({\tM}; {\tE}))$,
which is $ \mathbb{B}(\mathcal{E}^{\MF}_{C^*_{red}\Gamma}({\tM};{\tE}))$.
This extension maps isomorphically the  *-algebra  $\Pdo_{\CC\Gamma}^{0}({\tM},{\tE})$, which
is certainly contained in the multiplier algebra of $C^* ({\tM}; \tE)^\Gamma$,
to its image in $ \mathbb{B}(\mathcal{E}^{\MF}_{C^*_{red}\Gamma}({\tM},{\tE}))$.

\begin{definition}\label{def:o0_pseudos}
Denote by $\Pdo^0_{\CC\Gamma}({\tM},{\tE})$ the image of $\Pdo_{\Gamma,c}^0({\tM}, {\tE})$  in
\begin{equation*}
  \mathbb{B}(\mathcal{E}^{\MF}_{C^*_{red}\Gamma}({\tM};{\tE}))
\end{equation*}
through the extension of $Ad_L$.
\end{definition}

\smallskip
\noindent
Generalizing \eqref{lott-iso}, elements in $\Pdo^0_{\CC\Gamma}({\tM},{\tE})$ are
locally finite sums of the form
$$
\sum_{\gamma\in \Gamma} R^*_{(\gamma,e)}T(x,y) \, \gamma$$
where $T(x,y)$ is the pseudodifferential kernel of an operator in
$\Pdo_{\Gamma,c}^0({\tM},{\tE})$ and $\gamma$ acts on the $\CC\Gamma$ part of
the fiber by left multiplication;
thus $R^*_{(\gamma,e)}T(x,y) $ has pseudodifferential singularities on the $(\gamma,e)$-translate of the diagonal of ${\tM}\times{\tM}$.

\begin{remark}\label{rem:special_MF}
Note that the image of   $\Pdo_{\Gamma,c}^{0}({\tM},{\tE})$
under $Ad_L$ is a proper subalgebra of the Mishchenko-Fomenko algebra
$\Pdo^0_{\MF}(M,E\otimes{\V}_{\CC\Gamma})$. This is why we use the
notation involving $\tM$: to stress that this remains pseudodifferential when
interpreted as a lifted operator on the $\Gamma$-covering.
For example, let us consider the multiplication by a central element
$\gamma\in \Gamma$; this is certainly a pseudodifferential operator of order $0$
in the Mishchenko-Fomenko calculus. An easy computation shows that this operator is the image under
$Ad_L$ of the translation operator $R^*_{\gamma^{-1}}$ which is not an element in $\Psi^0_{\CC\Gamma}({\tM},{\tE})$.
\end{remark}

\section[Pseudodifferential operators with coefficients in forms]{Pseudodifferential operators with coefficients in differential forms on $\CC \Gamma$}
\label{Pdo-forms}
All we did in the previous section can be extended to the situation where coefficients in non-commutative differential forms are taken into account.
Recall $\Omega_\bullet(\CC\Gamma)$,  the universal
  differential 
  algebra associated to $\CC\Gamma$. We write $e$ for the neutral element of
  $\Gamma$ (the unit of $\complexs\Gamma$). As a vector space over $\CC$, $\Omega_\bullet(\CC\Gamma)$ has basis
$\{g_0dg_1\dots dg_k\mid g_0,\dots,g_k\in \Gamma\}$ and the multiplication is given by the following formula
\begin{multline*}
  (g_0dg_1\dots dg_k)\cdot (g_{k+1}dg_{k+2}\dots dg_{n})\\
  :=\sum_{j=1}^k(-1)^{n-j}g_0dg_1\dots d(g_jg_{j+1})\dots dg_n + (-1)^ng_0g_1dg_2\dots dg_n,
\end{multline*}
while the differential is given by
\[
d(g_0dg_1\dots dg_k)= dg_0dg_1\dots dg_k
\]
First we fix some notation:
\begin{itemize}
	\item if $\omega=g_1\otimes\dots\otimes g_n$, then we set $d\omega:=dg_1 dg_2\dots dg_n\in \Omega_\bullet(\CC\Gamma)$ and $\pi(\omega):=g_1g_2\dots g_n\in \Gamma$.
	\item  Observe that the elements of the form $\pi(\omega)^{-1}d\omega$
          constitute a basis of the left $\CC\Gamma$-module
          $\Omega_\bullet(\CC\Gamma)$.
	\item Let   $\mathcal{E}^{\MF}_{\Omega_\bullet(\CC\Gamma)}({\tM}
	; {\tE})$ denote the  right $\Omega_\bullet(\CC\Gamma)$-module
        $C^\infty({\tM};{\tE}\otimes
        \Omega_\bullet(\CC\Gamma))^\Gamma$. 
\end{itemize}

\begin{definition}\label{MF-Omega}
Let us define a subalgebra
$\Pdo^0_{\Omega_\bullet(\CC\Gamma)}({\tM},{\tE})$ of the right
$\Omega_\bullet(\CC\Gamma)$-linear operators on the module
$\mathcal{E}^{\MF}_{\Omega_\bullet(\CC\Gamma)}({\tM};{\tE})$ in the
following way, extending Definition \ref{def:o0_pseudos}.
Elements in $\Pdo^0_{\Omega_\bullet(\CC\Gamma)}({\tM},{\tE})$ are
given by sums that are  locally finite over $\lambda\in\Gamma$ and finite over
  basic forms $\omega=g_1\tensor\dots\tensor g_n$ of the following type
\[
\mathbf{T}=\sum_{\lambda,\omega} R^*_{(\lambda,e)}T_\omega\otimes\lambda\pi(\omega)^{-1}d\omega
\]
where $T_\omega$ are the distributional kernels of $\Gamma$-equivariant
 0-order pseudodifferential operators on
${\tM}$ of $\Gamma$-compact support and where $\lambda\pi(\omega)^{-1}d\omega$ acts by left multiplication in
the $\Omega_\bullet(\CC\Gamma)$ tensor factor in
$C^\infty({\tM},{\tE}\otimes \Omega_\bullet(\CC\Gamma))^\Gamma$.
The form degree in the expression for $\mathbf{T}$ defines a grading
  by non-negative integers which makes
  $\Psi^0_{\Omega_\bullet(\CC\Gamma)}(\tM,\tE)$ a graded algebra. Its degree
  $0$ part can be identified with $\Psi^0_{\CC\Gamma}(\tM,\tE)$  as it follows 
  from the expression for the elements in Definition \ref{def:o0_pseudos}.
  This is also explained by the identification of Remark \ref{rem:lift_Omega_lin} below.
\end{definition}

  Under the identifications in Section \ref{sec:translations_to_invariant} of $\Gamma$-invariant objects on $\tM$ and
  sections of bundles on $M$, $\Psi^0_{\Omega_\bullet(\CC\Gamma)}(\tM,\tE)$ acts on the
right  $\Omega_\bullet(\CC\Gamma)$-module of sections $C^\infty(M;E\tensor
\V_{\Omega_\bullet(\CC\Gamma)})=\mathcal{E}^{\MF}_{\Omega_\bullet(\CC\Gamma)}(\tM,\E)$,
as operators of order $0$ in the 
Mishchenko-Fomenko calculus. But just as in the case without differential form
coefficients in Remark \ref{rem:special_MF}, these are very special operators
with constraints on the singularities, as explained
in the following remark. This motivates our notation.

\begin{remark}\label{locality}
 Given an element $\mathbf{T}$ in $
 \Pdo^0_{\Omega_\bullet(\CC\Gamma)}({\tM},{\tE})$, let us highlight that the
 only coefficients $R^*_{(\lambda,e)}T_{\omega}$ with pseudodifferential
 singularities on the diagonal of ${\tM}\times{\tM}$ are those with
   $\lambda=e$. 
   Using  Definition \ref{def:deloc} 
  in  Section \ref{sec:deloc_group_sequence} below we thus see  
  that the
 only coefficients $R^*_{(\lambda,e)}T_{\omega}$ with pseudodifferential
 singularities on the diagonal of ${\tM}\times{\tM}$ are those with 
   $\lambda\pi(\omega)^{-1}d\omega\in \Omega^{e}_\bullet(\CC\Gamma)$.
  \end{remark}

  \begin{remark}\label{rem:lift_Omega_lin}
    The restriction to the sections of the degree $0$ part of the forms gives
    a canonical bijection
    \begin{multline}\label{eq:lift_Omega_lin}
      \Hom_{\Omega_\bullet(\CC\Gamma)}(\mathcal{E}^{\MF}_{\Omega_\bullet(\CC\Gamma)}(\tM,\tE),\mathcal{E}^{\MF}_{\Omega_\bullet(\CC\Gamma)}(\tM,\tE))\\
      \xrightarrow{\iso}       \Hom_{\CC\Gamma}(\mathcal{E}^{\MF}_{\CC\Gamma}(\tM,\tE),\mathcal{E}^{\MF}_{\Omega_\bullet(\CC\Gamma)}(\tM,\tE)),
    \end{multline}
    using that $\E^{\MF}_{\Omega_\bullet(\CC\Gamma)}(\tM,\tE)=
    C^\infty(M;E\tensor
    \V_{\CC\Gamma}\tensor_{\CC\Gamma}\Omega_\bullet(\CC\Gamma))=C^\infty(M;E\tensor
    \V)\tensor_{\CC\Gamma}\Omega_\bullet(\CC\Gamma)$, so that the
    $\CC\Gamma$-linear homomorphisms of the right hand side of
    \eqref{eq:lift_Omega_lin} extend uniquely to
    $\Omega_\bullet(\CC\Gamma)$-linear homomorphisms of the left hand side. 
  \end{remark}

\section[Mishchenko-Fomenko caluli]{Mishchenko-Fomenko calculi associated to dense holomorphically closed subalgebras of $C^*_{red} \Gamma$} \label{Pdo-smooth-algebras}

If now $\mathcal{A}\Gamma$ is a dense and holomorphically closed *-subalgebra
of $C^*_{red}\Gamma$ which contains $\CC\Gamma$, then $\Pdo_{\MF}^{-\infty}(M,
E\otimes{\V}_{\mathcal{A}\Gamma})$ is dense and holomorphically closed
in
\begin{equation*}
  \Pdo^{-\infty}_{\MF}(M, E\otimes{\V}_{C^*_{red} \Gamma})\subset
  \mathbb{K}(\mathcal{E}_{C^*_{red}\Gamma}^{{\MF}}({\tM},
  {\tE})).
\end{equation*}
This follows from \cite[Section 6]{Lott3}.
Of course, this implies that $\Pdo^{-\infty}_{\MF}(M, E\otimes{\V}_{\mathcal{A}\Gamma})$ is dense and holomorphically closed in $\mathbb{K}(\mathcal{E}_{C^*_{red}\Gamma}^{{\MF}}({\tM}, {\tE}))$.

\smallskip
\noindent
Recall  $\Pdo^0_{\CC\Gamma}({\tM}, {\tE})$ and 
$\Pdo^0_{\Omega_\bullet(\CC\Gamma)}({\tM},{\tE})$ of Definitions \ref{def:o0_pseudos}
and \ref{MF-Omega}.

\begin{definition}
	We define $\Pdo^0_{\mathcal{A}\Gamma}({\tM}, {\tE})$ as the *-subalgebra 
	\[
	\Pdo^0_{\CC\Gamma}({\tM}, {\tE})+ \Pdo^{-\infty}_{\MF}(M, E\otimes{\V}_{\mathcal{A}\Gamma})\subset \Pdo^0_{\MF}(M, E\otimes{\V}_{C^*_{red}\Gamma}).
	\]  Analogously, let
        $\Pdo^0_{\widehat{\Omega}_\bullet(\mathcal{A}\Gamma)}({\tM},
        {\tE})$ denote the *-subalgebra
	\[
	\Pdo^0_{\Omega_\bullet(\CC\Gamma)}({\tM}, {\tE})+ \Pdo^{-\infty}_{\MF}(M, E\otimes{\V}_{\widehat{\Omega}_\bullet(\mathcal{A}\Gamma)})\subset \Pdo^0_{\MF}(M, E\otimes{\V}_{\widehat{\Omega}_\bullet(C^*_{red}\Gamma)}).
	\] 
\end{definition}

\begin{remark}
	Recall that the notation $\widehat{\Omega}_\bullet$ for the universal differential algebra of a Fr\'{e}chet algebra $\mathcal{A}$ means that we are using  the projective tensor product in its definition, see for instance \cite[Section 2]{Lott1}.
\end{remark}

\begin{proposition}\label{prop:MF-holo}
	The *-subalgebra
        $Ad_L^{-1}(\Pdo^{0}_{\mathcal{A}\Gamma}({\tM},
        {\tE}))$ is dense and holomorphically closed in
        $\Pdo^0_\Gamma({\tM},{\tE})$, 
where we recall that $\Pdo^0_\Gamma(\widetilde {M},{\tE})$ 
        is  the $C^*$-closure of $\Pdo^0_{\Gamma,c}(\tM,{\tE})$. The corresponding result is true for operators with coefficients in forms.
\end{proposition}

\begin{proof}
	Consider the algebras $\mathcal{J}:=C^*({\tM},L^2({\tM},{\tE}))^\Gamma$, $\mathcal{B}:=\Pdo^0_\Gamma({\tM},{\tE})$, $\mathcal{I}:=Ad_L^{-1} (\Pdo^{-\infty}(M, E\otimes{\V}_{\mathcal{A}\Gamma})) $ and $\mathcal{A}:=Ad_L^{-1} (\Pdo^{0}_{\mathcal{A}\Gamma}({\tM}, {\tE}))$. All the hypotheses of \cite[Theorem 4.2]{LMN} are verified and the result follows.
	Notice that \cite[Equation (4.4)]{LMN} has a misprint. It should
        contain $\mathcal{A}$ instead of $\mathcal{B}$.
\end{proof}

\chapter{Mapping analytic surgery to homology}\label{section6}

\section{The delocalized homology exact sequence}
\label{sec:deloc_group_sequence}

``Delocalized'' group (co)homology is an important theme for numeric secondary
	index invariants, but a bit subtle for the relevant case of completions of the
	group algebra. 
	This is already treated in \cite[Section 4.1]{DeeleyGoffeng3}. Let us  recall the
	basic definitions.

Let $\mathcal{A}\Gamma$ a Fr\'{e}chet *-algebra completion of $\CC\Gamma$.
The inclusion $\CC\Gamma\hookrightarrow\mathcal{A}\Gamma$ induces the
following morphism of DGAs (recall that $\widehat\Omega_\bullet$ is used for completions)
\[
j\colon \Omega_\bullet(\CC\Gamma)\to \widehat{\Omega}_\bullet(\mathcal{A}\Gamma)
\]
and we will keep the same notation for the morphism induced between the
abelianizations. 
\begin{definition}\label{def:deloc}
	Let us denote by $\Omega_\bullet^{e}(\CC\Gamma)$
	the sub-DGA of $\Omega_\bullet(\CC\Gamma)$ generated by those forms 
	$g_0dg_1\dots dg_n$ such that $g_0g_1\dots g_n= e\in \Gamma$.
	Moreover, denote by  $\Omega^{del}_\bullet(\CC\Gamma)$ the complementary
	space to $\Omega^{e}_\bullet(\CC\Gamma)$ in $\Omega_\bullet(\CC\Gamma)$ spanned by
	$g_0dg_1\dots g_k$  with $g_0g_1\dots g_k\neq e\in \Gamma$.
	
\end{definition}

We define the closed subcomplex
$(\widehat{\Omega}^{e}_\bullet(\mathcal{A}\Gamma)_{ab},d)$ of
$(\widehat{\Omega}_\bullet(\mathcal{A}\Gamma)_{ab},d)$ as the closure of
$j(\Omega_\bullet^{e}(\CC\Gamma)_{ab})$ and we denote the associated
homology group by $H_*^{e}(\mathcal{A}\Gamma)$.
Consider the resulting short exact sequence of complexes
\begin{equation}\label{deloc-complexes}
\xymatrix{ 0\ar[r]& \widehat{\Omega}^{e}_\bullet(\mathcal{A}\Gamma)_{ab}\ar[r]^{\hat{j}}& \widehat{\Omega}_\bullet(\mathcal{A}\Gamma)_{ab}\ar[r]^q& \widehat{\Omega}^{del}_\bullet(\mathcal{A}\Gamma)_{ab}\ar[r]& 0 }
\end{equation}
where $ 	\widehat{\Omega}^{del}_\bullet(\mathcal{A}\Gamma)_{ab}:=\widehat{\Omega}_\bullet(\mathcal{A}\Gamma)_{ab}/\widehat{\Omega}_{*}^{e}(\mathcal{A}\Gamma)_{ab}$.
We then have the associated long exact sequence of homology
groups
\begin{equation}\label{dRgamma}
\xymatrix{\dots\ar[r]& H_*(\mathcal{A}\Gamma)\ar[r]&
	H_*^{del}(\mathcal{A}\Gamma)\ar[r]^{\delta_\Gamma}& H_{*+1}^{e}(\mathcal{A}\Gamma)\ar[r]& H_{*-1}(\mathcal{A}\Gamma)\ar[r]&\dots }
\end{equation}
Whereas $\Omega^{e}_\bullet(\complexs\Gamma)_{ab}$ is a direct summand of
$\Omega_\bullet(\complexs\Gamma)_{ab}$, it is not clear and in general false
that the completion
$\widehat{\Omega}^{e}_\bullet(\mathcal{A}\Gamma)_{ab}$ is a chain complex direct summand
of $\widehat{\Omega}_\bullet(\mathcal{A}\Gamma)_{ab}$. Therefore, in general we
can't expect that \eqref{dRgamma} splits into short exact sequences. To talk of
the ``localized'' and ``delocalized'' parts of the homology is therefore not
really appropriate (although frequently done).

\section{Lott's connection and delocalized traces}\label{section6.1}

\label{lifting-delocalized}
Let $h\colon {\tM}\to [0,1]$ be a cut-off function for the $\Gamma$-action on ${\tM}$, namely a smooth function such that 
  \begin{equation}\label{eq:cutoff}
\sum_{g\in \Gamma}R^*_g h(x)=1 \quad \forall x\in {\tM}.
\end{equation}
Notice that, since the action of $\Gamma$ is cocompact, $h$ can be chosen with compact support.
\begin{definition}\label{lott-connection}
	Lott's connection $\nabla^{\mathrm{Lott}}\colon \mathcal{E}^{\MF}_{\CC\Gamma}({\tM}, {\tE})\to \mathcal{E}^{\MF}_{\Omega_1(\CC\Gamma)}({\tM}, {\tE})$ is defined by
	\begin{equation}\label{def-lott-connection}
	\nabla^{\mathrm{Lott}}\left(\sum_{\lambda\in \Gamma}R^*_\lambda f\,\lambda\right):=\sum_{\lambda, g\in \Gamma}R^*_{\lambda g^{-1}} h \cdot R^*_\lambda f\otimes\lambda g^{-1}dg.
	\end{equation}
\end{definition}
\begin{remark}
 Notice that $\nabla^{\mathrm{Lott}}$ is an operator, but not a
    pseudodifferential operator: it involves pulling back with
    diffeomorphisms. Moreover, the map is not $\complexs\Gamma$ linear but
    acts as a connection for the multiplication by $\complexs\Gamma$.
   
  \smallskip
  \noindent
  We check that this definition is coherent with the original definition
  of Lott \cite[Equation (41)]{Lott1}, which is given by
  \begin{equation}\label{lott-version}
  \sum_{g\in\Gamma}h R^*_{g}f\otimes dg \qquad\text{for }f \in C_c^{\infty}({\tM}, {\tE}). 
  \end{equation}
  In Definition \ref{lott-connection} we have written  the expression
  \eqref{lott-version} in our basis $\{g^{-1}dg\}$ of
  $\Omega_1(\CC\Gamma)$ instead of the basis $\{dg\}$.
  Indeed 
 $$ \sum_{g\in\Gamma}h\cdot R^*_{g}f\otimes dg= \sum_{g\in\Gamma}h\cdot R^*_{g}f\otimes gg^{-1} dg= \sum_{g\in\Gamma}(R^*_{g^{-1}}h)\cdot f\otimes g^{-1} dg, $$
 finally, applying $L$ of Lemma \ref{lem:Lott} we obtain \eqref{def-lott-connection}.
   It is shown in
  \cite[Proposition 9]{Lott1} that $\nabla^{\mathrm{Lott}}$  is indeed a
  connection.
\end{remark}
By standard arguments we can extend Lott's connection to
$\mathcal{E}^{\MF}_{\Omega_\bullet(\CC\Gamma)}({\tM}, {\tE})$ and we keep
the notation $\nabla^{\mathrm{Lott}}$ for the extended connection.

The curvature of this connection is given by
\begin{equation}\label{lott-curvature}
\Theta:= (\nabla^\mathrm{{Lott}})^2= \sum_{g_1, g_2}R^*_{(g_1g_2)^{-1}}h\cdot
R^*_{g_2^{-1}}h \otimes (g_1g_2)^{-1}dg_1dg_2\end{equation}
which is indeed pseudodifferential and $\complexs\Gamma$
  linear. Therefore, by conjugation through $L$, it gives an
element in $ \Hom_{\CC\Gamma}(\mathcal{E}^{\MF}_{\CC\Gamma}({\tM},
{\tE}),\mathcal{E}^{\MF}_{\Omega_2(\CC\Gamma)}({\tM}, {\tE}))$ which by
Remark \ref{rem:lift_Omega_lin} we can
also interpret as an element of
$\Hom_{\Omega_\bullet(\CC\Gamma)}(\mathcal{E}^{\MF}_{\Omega_\bullet(\CC\Gamma)}(\tM,\tE),\mathcal{E}^{\MF}_{\Omega_\bullet(\CC\Gamma)}(\tM,\tE))$.

  \begin{lemma}
  One has naturally a connection on
    $\Pdo^0_{\Omega_\bullet(\CC\Gamma)}({\tM}, {\tE})$ defined by
    \begin{equation}\label{overline-nabla}
      \overline{\nabla}^{\mathrm{Lott}}(\mathbf{T})=[\nabla^{\mathrm{Lott}},\mathbf{T}] := \nabla^{\mathrm{Lott}}\mathbf{T}-(-1)^{\deg\mathbf{T}}\mathbf{T} \nabla^{\mathrm{Lott}}.
    \end{equation}
  \end{lemma}
  \begin{proof}
    The connection property is obvious, as always if something is defined as a
    commutator with a fixed endomorphism: if
    $\mathbf{T},\mathbf{S}\in \Pdo^0_{\Omega_\bullet(\CC\Gamma)}({\tM}, {\tE})$
   then
   \begin{equation*}
   \begin{split}
     \overline{\nabla}^{\mathrm{Lott}}(\mathbf{TS})  =&
     \nabla^{\mathrm{Lott}}\mathbf{TS} -
     (-1)^{\deg(\mathbf{TS})}\mathbf{TS}\nabla^{\mathrm{Lott}}\\
     =&\nabla^{\mathrm{Lott}}\mathbf{TS} -
     (-1)^{\deg(\mathbf{T})}\mathbf{T}\nabla^{\mathrm{Lott}}\mathbf{S} +
        (-1)^{\deg(\mathbf{T})}\mathbf{T}\nabla^{\mathrm{Lott}}\mathbf{S}\\
     &-
     (-1)^{\deg(\mathbf{T})}(-1)^{\deg(\mathbf{S})}\mathbf{TS}\nabla^{\mathrm{Lott}}\\
     =& ( \overline{\nabla}^{\mathrm{Lott}}(\mathbf{T}))\mathbf{S} +
     (-1)^{\deg(\mathbf{T})} \mathbf{T}   (
     \overline{\nabla}^{\mathrm{Lott}}(\mathbf{S})) .
   \end{split}
 \end{equation*}

 Next we have to prove that indeed
      $\overline{\nabla}^{\mathrm{Lott}}(T)\in
      \Psi^0_{\Omega_\bullet(\complexs\Gamma)}(\tM,\tE)$, which is not
      obvious. First, observe that the
      $\Omega_\bullet(\complexs\Gamma)$-linearity of $\mathbf{T}$ and the
      connection property of $\nabla$ imply that
      $\overline{\nabla}(\mathbf{T})$ is again $\Omega_\bullet(\complexs\Gamma)$-linear:
      \begin{equation*}
        \begin{split}
          \overline{\nabla}(\mathbf{T})(f\alpha) &=
          \nabla(\mathbf{T}(f\alpha))- (-1)^{\deg(\mathbf{T})}
                                                   \mathbf{T}(\nabla(f\alpha))\\
          &=         \nabla(\mathbf{T}(f)\alpha)- (-1)^{\deg(\mathbf{T})}
         ( \mathbf{T}(\nabla(f)\alpha) + (-1)^{\deg(f)}  \mathbf{T}(fd\alpha)
         )\\
         &=          \nabla(\mathbf{T}(f))\alpha +
         (-1)^{\deg(\mathrm{T})+\deg(f)}\mathbf{T}fd\alpha - (-1)^{\deg(\mathbf{T})}
           \mathbf{T}(\nabla(f))\alpha\\
          &= (-1)^{\deg(\mathbf{T})+\deg(f)}
          \mathbf{T}(f)d\alpha\\
          &= \overline\nabla(\mathbf{T})(f)\alpha.
  \end{split}
      \end{equation*}

     To show that $\overline\nabla(\mathbf{T})$ indeed is pseudodifferential,
     let us follow Definition \ref{MF-Omega} and explicitly take 
      \begin{equation*}
      \mathbf{T}=\sum_{\gamma\in\Gamma,\omega}
      R^*_{(\gamma,e)}T_\omega(x,y)\gamma \pi(\omega)^{-1} d\omega,
    \end{equation*}
  with
      $\deg (\mathbf{T})=\deg(\omega)$. Moreover, fix $f
      =\sum_{\lambda\in\Gamma} R^*_\lambda f\lambda\in
      \mathcal{E}^{\MF}_{\Omega_\bullet(\CC\Gamma)}(\tM,\tE)$ of form degree
      $0$ (this determines the general case by
      $\Omega_\bullet(\complexs\Gamma)$-linearity). Then
      \begin{equation}\label{eq:nabla_bar}
        {\small
        \begin{split}
          &\overline{\nabla}^{\mathrm{Lott}}(\mathbf{T})(f) \\
          &=
          \nabla(\mathbf{T})(f) - (-1)^{\deg(\omega)} \mathbf{T}(\nabla f) \\
           &= \nabla\left(\sum_{\gamma,\omega,\lambda} \int_{\tM/\Gamma} dy\; 
          R^*_{\gamma,e}\mathbf{T}_\omega(x,y)R^*_\lambda f(y) \gamma
          \pi(\omega)^{-1} d\omega \cdot \lambda\right) \\
          & - (-1)^{\deg(\omega)}
          \sum_{\gamma,\omega} \int_{\tM/\Gamma}dy\;
          R^*_{(\gamma,e)}\mathbf{T}_\omega(x,y)\gamma\pi(\omega)^{-1}d\omega
          \left(\sum_{\lambda, g} R^*_{\lambda g^{-1}}h(y) R^*_\lambda f(y) \lambda
          g^{-1}dg\right) \\
  & = \sum_{g,\gamma,\omega,\lambda} \int_{\tM/\Gamma} dy\;
          R^*_{\gamma g^{-1}}h(x) 
          R^*_{\gamma,e}\mathbf{T}_\omega(x,y)R^*_\lambda f(y) \gamma g^{-1}dg
          \cdot 
          \pi(\omega)^{-1} d\omega \cdot \lambda\\
          &+ \sum_{\gamma,\omega,\lambda} \int_{\tM/\Gamma} dy\; 
          R^*_{\gamma,e}\mathbf{T}_\omega(x,y)R^*_\lambda f(y) \gamma\cdot
         d( \pi(\omega)^{-1} d\omega \cdot \lambda)\\
          & - (-1)^{\deg(\omega)}
          \sum_{\lambda, g,\gamma,\omega} \int_{\tM/\Gamma}dy\;
          R^*_{(\gamma,e)}\mathbf{T}_\omega(x,y) R^*_{\lambda g^{-1}}h(y) R^*_\lambda f(y) \gamma\pi(\omega)^{-1}d\omega
          \lambda
          g^{-1}dg\\
         &\stackrel{\tilde g=g\lambda^{-1}}{=} \sum_{g,\gamma,\omega,\lambda} \int_{\tM/\Gamma} dy\;
          R^*_{\gamma g^{-1}}h(x) 
          R^*_{\gamma,e}\mathbf{T}_\omega(x,y)R^*_\lambda f(y) \gamma g^{-1}dg
          \cdot 
          \pi(\omega)^{-1} d\omega \cdot \lambda\\
          &+ \sum_{\gamma,\omega,\lambda} \int_{\tM/\Gamma} dy\; 
          R^*_{\gamma,e}\mathbf{T}_\omega(x,y)R^*_\lambda f(y) \gamma\cdot
         d( \pi(\omega)^{-1}) d\omega \cdot \lambda\\
          &+ (-1)^{\deg(\omega)}\sum_{\gamma,\omega,\lambda} \int_{\tM/\Gamma} dy\; 
          R^*_{\gamma,e}\mathbf{T}_\omega(x,y)R^*_\lambda f(y) \gamma\cdot
          \pi(\omega)^{-1} d\omega \cdot d\lambda)\\
          & - (-1)^{\deg(\omega)}
          \sum_{\lambda, \tilde g,\gamma,\omega} \int_{\tM/\Gamma}dy\;
          R^*_{(\gamma,e)}\mathbf{T}_\omega(x,y) R^*_{ \tilde g^{-1}}h(y) R^*_\lambda f(y) \gamma\pi(\omega)^{-1}d\omega
          \tilde g^{-1}d(\tilde g \lambda)%\\
             \end{split} }
      \end{equation}
      \begin{equation*}
        {\small \begin{split}     &= \sum_{g,\gamma,\omega,\lambda} \int_{\tM/\Gamma} dy\; \gamma
          R^*_{\gamma,e} ( R^*_{g^{-1}}h(x)g^{-1}dg  
         \mathbf{T}_\omega(x,y)\pi(\omega)^{-1} d\omega ) \cdot R^*_\lambda f(y)  
           \lambda\\
          &+ \sum_{\gamma,\omega,\lambda} \int_{\tM/\Gamma} dy\; \gamma
          R^*_{\gamma,e}\mathbf{T}_\omega(x,y)\cdot
         d( \pi(\omega)^{-1}) d\omega \cdot R^*_\lambda f(y) \lambda\\
          &+ (-1)^{\deg(\omega)}\sum_{\gamma,\omega,\lambda} \int_{\tM/\Gamma} dy\; 
          R^*_{\gamma,e}\mathbf{T}_\omega(x,y)R^*_\lambda f(y) \gamma\cdot
          \pi(\omega)^{-1} d\omega \cdot d\lambda\\
          & - (-1)^{\deg(\omega)}
          \sum_{\lambda, g,\gamma,\omega} \int_{\tM/\Gamma}dy\;  \gamma
          R^*_{(\gamma,e)}\mathbf{T}_\omega(x,y)\pi(\omega)^{-1}d\omega R^*_{
            g^{-1}}h(y)  g^{-1}d g \cdot R^*_\lambda f(y) \lambda\\
          & - (-1)^{\deg(\omega)}
          \sum_{\lambda, \gamma,\omega} \int_{\tM/\Gamma}dy\;
          R^*_{(\gamma,e)}\mathbf{T}_\omega(x,y)  R^*_\lambda f(y) \gamma\cdot
          \pi(\omega)^{-1}d\omega
        \underbrace{ \sum_{g}R^*_{  g^{-1}}h(y)  g^{-1} g}_{=1} \cdot d
        \lambda\\
       & =\sum_{\gamma,\omega,\lambda} \int_{\tM/\Gamma} dy\; \gamma
          R^*_{\gamma,e} \left([ \sum_{g} R^*_{g^{-1}}h(x)g^{-1}dg,  
         \mathbf{T}_\omega(x,y)\pi(\omega)^{-1} d\omega ]_s +
         \mathbf{T}_\omega(x,y) \underbrace{\pi(\tilde\omega)}_{=1}
         d\tilde\omega\right)\\
       &\qquad\qquad\qquad\hfill\cdot R^*_\lambda f(y)
           \lambda\\
                \end{split} }
              \end{equation*}
              where the last equality uses ${\tilde\omega=
                (\pi(\omega)^{-1},\omega)}$.               
Here, if $\omega=(g_1,\dots,g_{\deg(\omega)})$ we write
$\tilde\omega:=((g_1\dots g_{\deg(\omega)})^{-1},g_0,\dots,g_{\deg(\omega)})$,
therefore $\pi(\tilde\omega)=1$. We also write $[\cdot,\cdot]_s$ for the
supercommutator, here with respect to the form degree.

From the explicit expression we immediately observe that
$\overline\nabla(\mathbf{T})$ belongs to
$\Psi^0_{\Omega_\bullet(\complexs\Gamma)}(\tM,\tE)$, i.e.~is
pseudodifferential with the right invariance properties. Note that, if the
form degree of $\mathbf{T}$ is $0$ the expression simplifies and the main term
is the commutator of $\mathbf{T}$ and the multiplication operators
$R^*_{g^{-1}}h$, which is  what is obtained in \cite[Equation
      (53)]{Lott1} for a special operator $\mathbf{T}$ (actually for a Dirac
      type operator which has positive degree, but the whole construction
      extends to this case anyway, compare Proposition
      \ref{prop:trace_on_all_pseudos}). Notice that Lott's
      sign conventions must  be 
      modified for right modules, as done in \cite{LeichtnamPiazzaMemoires}.
  \end{proof}

\noindent
It is immediate to see $(\overline{\nabla}^{\mathrm{Lott}})^2(\mathbf{T})=\Theta\mathbf{T}-\mathbf{T}\Theta$ and that $\overline{\nabla}^{\mathrm{Lott}}(\Theta)=0$.

\begin{remark}\label{connection-frechet}
Let now $\mathcal{A}\Gamma$ be any Fr\'{e}chet completion of $\CC\Gamma$ contained in $C^*_{red}\Gamma$.
The proof of \cite[Proposition 9]{Lott1} applies using  $\mathcal{A}\Gamma$,
therefore $\nabla^{\mathrm{Lott}}$ is a well-defined connection on
$\mathcal{E}^{\MF}_{\mathcal{A}\Gamma}({\tM}, {\tE})$ and
everything we have seen so far also holds in this setting.
\end{remark}
\medskip

\begin{definition}

In \cite[Section II]{Lott1},
Lott defines a $\widehat{\Omega}_\bullet(\mathcal{A}\Gamma)_{ab}$-valued trace  
\begin{equation}\label{eq:Lott_trace}
  \begin{split}
    \mathrm{TR}\colon
    \Pdo^{-\infty}_{\widehat{\Omega}_\bullet(\mathcal{A}\Gamma)}({\tM},{\tE})&\to
    \widehat{\Omega}_\bullet(\mathcal{A}\Gamma)_{ab}\\ \mathbf{T}&\mapsto
    \sum_{\lambda,\omega}\int_{\mathcal{F}}\Tr(R^*_{(\lambda,e)}T_\omega(\tilde{x},\tilde{x}))d\mathrm{vol}(\tilde{x})
    \, \lambda\pi(\omega)^{-1}d\omega,
  \end{split}
\end{equation}
where $\mathbf{T}=\sum_{\lambda,\omega}R^*_{(\lambda,e)}T_\omega \otimes \lambda\pi(\omega)^{-1}d\omega $ with $T_\omega$ smoothing $\Gamma$-equivariant kernels on ${\tM}$. Here, $\mathcal{F}$ is a fundamental domain for the action of $\Gamma$ on ${\tM}$ and $\Tr$ is the fiberwise trace of vector bundle homomorphisms.
\end{definition}

In this section, we will define in a similar fashion a delocalized trace on
\begin{equation*}
  \Pdo^0_{\widehat{\Omega}_\bullet(\mathcal{A}\Gamma)}({\tM},{\tE}).
\end{equation*}
Recall    
from Section \ref{sec:deloc_group_sequence} that, as a vector space,
$\Omega_\bullet(\CC\Gamma)$ decomposes into the direct sum
$\Omega_\bullet^{e}(\CC\Gamma)\oplus \Omega_\bullet^{del}(\CC\Gamma)$. In
turn, always as a vector space, Definition \ref{MF-Omega}
  shows immediately that $\Pdo^{0}_{\Omega_\bullet(\CC\Gamma)}({\tM},{\tE})$ splits into the direct sum $\Pdo^0_{\Omega^{e}_\bullet(\CC\Gamma)}({\tM},{\tE})\oplus\Pdo^{0}_{\Omega_\bullet^{del}(\CC\Gamma)}({\tM},{\tE})$.
Then we have a projection
\begin{equation*}
  \pi_{del}\colon \Pdo^{0}_{\Omega_\bullet(\CC\Gamma)}({\tM},{\tE})\to \Pdo^{0}_{\Omega_\bullet^{del}(\CC\Gamma)}({\tM},{\tE}).
\end{equation*}
Recall that, thanks to Remark \ref{locality},  if we consider an element in the image of $\pi_{del}$ then the restriction to the diagonal of its kernel is smooth. Hence  we can give the following definition.
\begin{definition}\label{tr-del}Let $\mathbf{T}$ be an element of $\Pdo^{0}_{\Omega_*(\CC\Gamma)}({\tM},{\tE})$. Define
	$$\mathrm{TR}_0^{del}(\mathbf{T}):=\int_{\tilde{x}\in\mathcal{F}}\Tr\left(\pi_{del}(\mathbf{T})(\widetilde{x},\widetilde{x})\right)d\mathrm{vol}(\widetilde{x})\in \Omega^{del}_\bullet(\CC\Gamma)_{ab},$$
	where 
	$\mathcal{F}$ is a fundamental domain for the action of $\Gamma$ on ${\tM}$.
\end{definition}
Unrolling the definition, we obtain
\begin{equation}\label{eq:Lott_del_tr}
\mathrm{TR}_0^{del}(\mathbf{T})=\sum_{\lambda\neq e,\omega}\int_{\mathcal{F}}\Tr(R^*_{(\lambda,e)}T_\omega(\tilde{x},\tilde{x}))d\mathrm{vol}(\tilde{x}) \, \lambda\pi(\omega)^{-1}d\omega,
\end{equation}
where it is important to notice that the sum is finite because all operators
are given by locally finite sums.
Furthermore, as for $\mathrm{TR}$, it is straightforward to see that $\mathrm{TR}_0^{del}$ is a trace and that Definition \ref{tr-del} does not depend on the choice of the fundamental domain $\mathcal{F}$. 

\smallskip
We now want to extend the definition of the delocalized trace to $\Pdo^0_{\widehat{\Omega}_\bullet(\mathcal{A}\Gamma)}({\tM},{\tE})$. To this end,
let us consider a smooth function $\chi$ on ${\tM}\times{\tM}$
such that it is equal to 1 in a neighborhood of the diagonal and is properly supported, namely the projections $\pi_1, \pi_2\colon {\tM}\times{\tM}\to {\tM}$ are proper on the support of $\chi$. Then an element $\mathbf{T}=\sum_{\lambda,\omega}R^*_{(\lambda,e)}T_\omega \otimes \lambda\pi(\omega)^{-1}d\omega\in \Pdo^0_{\widehat{\Omega}_\bullet(\mathcal{A}\Gamma)}({\tM},{\tE})$ decomposes as $\mathbf{T}=\mathbf{T}_{0}+\mathbf{T}_{\infty}$, where
$$
\mathbf{T}_{0}=\sum_{\lambda,\omega}R^*_{(\lambda,e)}(\chi T_\omega)
\otimes \lambda\pi(\omega)^{-1}d\omega,\quad \mathbf{T}_\infty= \mathbf{T}-\mathbf{T}_0 .
$$
Observe that by Definition \ref{MF-Omega}, since $\chi T_\omega$ is $\Gamma$-compactly supported,
\begin{equation*}
  \mathbf{T}_{0}\in \Pdo^0_{\Omega_\bullet(\CC\Gamma)}({\tM},{\tE}) \text{ and } \mathbf{T}_{\infty}\in \Pdo^{-\infty}_{\widehat{\Omega}_\bullet(\mathcal{A}\Gamma)}({\tM},{\tE}).
\end{equation*}

\begin{definition}\label{def-tr-del}
We define $\mathrm{TR}^{del}\colon \Pdo^0_{\widehat{\Omega}_\bullet(\mathcal{A}\Gamma)}({\tM},{\tE})\to \widehat{\Omega}^{del}_\bullet(\mathcal{A}\Gamma)_{ab}$
for $\mathbf{T}=\mathbf{T}_0+\mathbf{T}_\infty$ decomposed as
  above by
$$
\mathrm{TR}^{del}(\mathbf{T}):= \mathrm{TR}_0^{del}(\mathbf{T}_{0})+ \mathrm{TR}(\mathbf{T}_{\infty}).
$$
\end{definition}
\noindent 
\begin{remark}\label{tr-del-lift}
Let us prove that Definition \ref{def-tr-del} is independent of the choice
  of the cutoff function $\chi$. Indeed, assume $\chi'$ is a second cutoff
  function and let $\mathbf{T}=\mathbf{T}'_0+\mathbf{T}'_\infty$ be the
  corresponding decomposition. Then $\mathbf{T}_0-\mathbf{T}_0'$ vanishes
  in a neighborhood of the diagonal and therefore
  $\mathbf{T}_0-\mathbf{T}'_0\in
  \Pdo^{-\infty}_{\widehat{\Omega}^+_\bullet(\mathcal{A}\Gamma)}({\tM},{\tE})$
  and the summand of $\lambda=e$ 
  in formula \eqref{eq:Lott_trace} for  $\TR(\mathbf{T}_0-\mathbf{T}_0')$
  is zero, so that with Equation \eqref{eq:Lott_del_tr} 
  $\TR(\mathbf{T}_0-\mathbf{T}'_0)=\TR_0^{del}(\mathbf{T}_0-\mathbf{T}'_0)$
  and finally, 
  \begin{equation*}
 {\small   \begin{split}
      \TR_0^{del}(\mathbf{T}_0)+\TR(\mathbf{T}_\infty) - \TR_0^{del}(\mathbf{T}'_0)
      -\TR(\mathbf{T}'_\infty) &= \TR^{del}_0(\mathbf{T}_0-\mathbf{T}'_0) +
      \TR(\mathbf{T}_\infty-\mathbf{T}'_\infty)\\
     & = \TR(\mathbf{T}_0-\mathbf{T}'_0
       + \mathbf{T}_\infty-\mathbf{T}'_\infty) \\
                               &                            = \Tr(0)=0.
    \end{split} }
  \end{equation*}

 Moreover, notice that $\mathrm{TR}^{del}$ is a well-defined map with values in $\widehat{\Omega}^{del}_\bullet(\mathcal{A}\Gamma)$, but it is a trace only if we pass to the quotient $\widehat{\Omega}^{del}_\bullet(\mathcal{A}\Gamma)_{ab}$.
\end{remark}

\begin{proposition}\label{deloc-trace-commutator}
	If $\mathbf{T}\in\Pdo^{0}_{\Omega_\bullet(\CC\Gamma)}({\tM},{\tE})$, then
	\[
	d\mathrm{TR}_0^{del}(\mathbf{T})=\mathrm{TR}_0^{del}([\nabla^{\mathrm{Lott}},\mathbf{T}])\in\Omega^{del}_\bullet(\CC\Gamma)_{ab}.
	\]
\end{proposition}
\begin{proof}
	Let $\mathbf{T}$ be given by the sum $\sum_{\omega,\lambda} R^*_{(\lambda,e)}T_\omega\lambda\pi(\omega)^{-1}d\omega$.
	The direct calculation of \eqref{eq:nabla_bar}  gives that
          $[\nabla^{\mathrm{Lott}},\mathbf{T}]$ is equal to
	\begin{equation*}
	\begin{split}
	&\sum_{\lambda,\omega,\gamma}R^*_{\lambda\gamma^{-1}}
        h\,R^*_{(\lambda,e)}T_\omega\, \lambda\gamma^{-1}d\gamma\cdot
        \pi(\omega)^{-1}d\omega\\
        &+\sum_{\lambda,\omega}R^*_{(\lambda,e)}T_{\omega}\,\lambda
        d\pi(\omega)^{-1}d\omega-\sum_{\lambda,\omega,\gamma}(-1)^{deg( d\omega)}R^*_{(\lambda,e)}T_\omega R^*_{\gamma^{-1}} h\, \lambda\pi(\omega)^{-1}d\omega\cdot \gamma^{-1}d\gamma.
        \end{split}
	\end{equation*}
	Passing to the Mishchenko-Fomenko context and applying $\mathrm{TR}^{del}$ we obtain
	\begin{equation}\label{formula-comm}\small
	\begin{split}
          \mathrm{TR}_0^{del}&([\nabla^{\mathrm{Lott}},\mathbf{T}])\\
          =&	\int_{x\in\mathcal{F}}\sum_{\lambda\neq e,\omega,\gamma}\Tr\left(R^*_{\lambda\gamma^{-1}} h(x)\,R^*_{(\lambda,e)} T_\omega(x,x)\right)\, \lambda\gamma^{-1}d\gamma\cdot \pi(\omega)^{-1}d\omega\\
	&+\int_{x\in\mathcal{F}}\sum_{\lambda\neq e,\omega}\Tr\left(R^*_{(\lambda,e)} T_{\omega}(x,x)\right)\,\lambda d\pi(\omega)^{-1}d\omega\\
	&-\int_{x\in\mathcal{F}}\sum_{\lambda\neq e,\omega,\gamma}(-1)^{deg( d\omega)}\Tr\left(R^*_{(\lambda,e)} T_\omega(x,x) R^*_{\gamma^{-1}} h(x)\right)\, \lambda\pi(\omega)^{-1}d\omega\cdot \gamma^{-1}d\gamma.
	\end{split}
	\end{equation}
	Remember that this expression takes place in $\Omega^{del}_\bullet(\CC\Gamma)_{ab}$, where
	\begin{equation*}
	\begin{split}
	\lambda\pi(\omega)^{-1}d\omega\cdot \gamma^{-1}d\gamma&=\pi(\omega)^{-1}d\omega\cdot \gamma^{-1}d\gamma\cdot\lambda\\&= \pi(\omega)^{-1}d\omega\cdot (\gamma^{-1}d(\gamma\lambda)-d\lambda)\\&=(-1)^{deg( d\omega)}(\gamma^{-1}d(\gamma\lambda)-d\lambda)\cdot \pi(\omega)^{-1}d\omega.
	\end{split}
	\end{equation*}
Therefore the third term of \ref{formula-comm} is equal to
	\begin{multline*}
	-\int_{x\in\mathcal{F}}\sum_{\lambda\neq
          e,\omega,\gamma}\Tr\left(R^*_{\gamma^{-1}}
          h(x)R^*_{(\lambda,e)}T_\omega(x,x)\right)\gamma^{-1}d(\gamma\lambda)\cdot\pi(\omega)^{-1}d\omega\\
        +\sum_{\lambda\neq e,\omega}\Tr\left(R^*_{(\lambda,e)}T_\omega(x,x)\right)d\lambda\cdot\pi(\omega)^{-1}d\omega.
	\end{multline*}
	In turn, after changing the first summation over $\gamma$ to the
        summation over $\mu=\gamma\lambda$, this is equal to
	\begin{multline}\label{formula-comm2}
	-\int_{x\in \mathcal{F}}\sum_{\lambda\neq
          e,\omega,\mu}\Tr\left(R^*_{\lambda\mu^{-1}}h(x)R^*_{(\lambda,e)}T_\omega(x,x)\right)\lambda\mu^{-1}d\mu\cdot\pi(\omega)^{-1}d\omega\\
        +\sum_{\lambda\neq e,\omega}\Tr\left(R^*_{(\lambda,e)}T_\omega(x,x)\right)d\lambda\cdot\pi(\omega)^{-1}d\omega.
	\end{multline}
	Now, observing that the first term in \eqref{formula-comm2} is equal to the opposite of the first term in \eqref{formula-comm}, it follows  that \eqref{formula-comm} becomes
	\begin{equation}
	\begin{split}
	\mathrm{TR}_0^{del}([\nabla^{\mathrm{Lott}},\mathbf{T}])=&\int_{x\in\mathcal{F}}\sum_{\lambda\neq
                                                                   e,\omega}\Tr\left(R^*_{(\lambda,e)}
                                                                   T_{\omega}(x,x)\right)\,\lambda
                                                                   d\pi(\omega)^{-1}d\omega\\
          &+\sum_{\lambda\neq e,\omega}\Tr\left(R^*_{(\lambda,e)}T_\omega(x,x)\right)d\lambda\cdot\pi(\omega)^{-1}d\omega=\\
	=&	\int_{x\in\mathcal{F}}\sum_{\omega,\lambda\neq e}\Tr\left(R^*_{(\lambda,e)} T_{\omega}(x,x)\right)\,d(\lambda\cdot\pi(\omega)^{-1}d\omega)=\\
	=&d\mathrm{TR}_0^{del}(\mathbf{T}).
	\end{split}
	\end{equation}
Notice that, due to the fact that the support of the kernel of $T_\omega$ is $\Gamma$-compact, the sums over $\lambda$ are finite.
\end{proof}

\begin{proposition}
		Let  $\mathbf{T}\in\Pdo^{-\infty}_{\widehat{\Omega}_\bullet(\mathcal{A}\Gamma)}({\tM},{\tE})$, then
		\[
		d\mathrm{TR}(\mathbf{T})=\mathrm{TR}([\nabla^{\mathrm{Lott}},\mathbf{T}])\in\widehat{\Omega}_\bullet(\mathcal{A}\Gamma)_{ab}.
		\]
\end{proposition}
\begin{proof}
	The proof is analogous to the one of Proposition \ref{deloc-trace-commutator}. 
Observe that, due to Remark \ref{connection-frechet}, in view of the previous
	algebraic manipulations,  the fact that Fr\'{e}chet algebras are involved is not a problem here.
\end{proof}

\begin{corollary}
	If $ \mathbf{T}\in\Pdo^{0}_{\widehat{\Omega}_\bullet(\mathcal{A}\Gamma)}({\tM},{\tE})$, then 
			\[
			d\mathrm{TR}^{del}(\mathbf{T})=\mathrm{TR}^{del}([\nabla^{\mathrm{Lott}},\mathbf{T}])\in\widehat{\Omega}^{del}_\bullet(\mathcal{A}\Gamma)_{ab}.
			\]
\end{corollary}

\section[From the analytic surgery sequence to non-commutative de Rham]{Mapping the analytic surgery sequence to noncommutative de Rham homology}

In this section we shall use the tools developed in Section \ref{ncderham} to map the analytic surgery exact sequence to the non-commutative de Rham homology exact sequence \eqref{dRgamma}. We shall use ideas of Connes, Gorokhovsky and Lott, see \cites{Connes, Goro-Lott}.

Notice that the pair $\left(\Pdo^{0}_{\Omega_\bullet(\mathcal{A}\Gamma)}({\tM},{\tE}),\overline{\nabla}^{\mathrm{Lott}}\right)$ is not a complex, because the square of Lott's connection $\overline{\nabla}^{\mathrm{Lott}}$ is not zero.
Nevertheless, the triple  
\[
\left(
\Pdo^{0}_{\Omega_\bullet(\mathcal{A}\Gamma)}({\tM},{\tE}),\overline{\nabla}^{\mathrm{Lott}},\Theta\right)
\]
does verify the hypotheses of  \cite[Lemma 9, III.3]{Connes}, which
produces in a canonical way the following  complex. 
\begin{definition}\label{connes-lott}
	Let us define 
	$(\overline{\Omega}_\bullet(\Pdo^{0}_{\mathcal{A}\Gamma}({\tM},{\tE})),
        d)$ as the following DGA  over $\Pdo^{0}_{\mathcal{A}\Gamma}({\tM},{\tE})$:
	\begin{itemize}
		\item as a vector space,
                  $\overline{\Omega}_\bullet(\Pdo^{0}_{\mathcal{A}\Gamma}({\tM},{\tE}))$
                  is  
	\[
	\Pdo^{0}_{  \widehat{ \Omega}_\bullet(\mathcal{A}\Gamma)}({\tM},{\tE})\oplus
        X\cdot\Pdo^{0}_{  \widehat{ \Omega}_\bullet(\mathcal{A}\Gamma)}({\tM},{\tE})
        \oplus
        \Pdo^{0}_{  \widehat{ \Omega}_\bullet(\mathcal{A}\Gamma)}({\tM},{\tE})\cdot X \oplus X\cdot \Pdo^{0}_{  \widehat{ \Omega}_\bullet(\mathcal{A}\Gamma)}({\tM},{\tE})\cdot X
	\]
	where $X$  is a degree 1 auxiliary variable. 
	\item The algebra structure
          is defined by setting $X^2=\Theta$ and $\mathbf{T}_1 X\mathbf{T}_2=0$ for all $\mathbf{T}_i\in\Psi^0_{\Omega_\bullet(\mathcal{A}\Gamma)}({\tM},{\tE})$.
		\item The differential $d$ is defined by $dX=0$ and $d\mathbf{T}=\overline{\nabla}^{\mathrm{Lott}}\mathbf{T}+ X\mathbf{T}+(-1)^{\mathrm{deg}\mathbf{T}}\mathbf{T}X$. One easily checks that $d^2=0$.
	\end{itemize}

  The same construction applies to
  $\left(\Pdo^{-\infty}_{\Omega_\bullet(\mathcal{A}\Gamma)}({\tM},{\tE}),\overline{\nabla}^{\mathrm{Lott}}\right)$
  and we obtain a subcomplex
  $\overline{\Omega}_\bullet(\Pdo^{-\infty}_{\mathcal{A}\Gamma}({\tM},{\tE}))$
  of
  $\overline{\Omega}_\bullet(\Pdo^{0}_{\mathcal{A}\Gamma}({\tM},{\tE}))$.
\end{definition}
\medskip
Now we want to extend form valued traces to these new complexes.
\begin{definition}
Let us consider an element of $\overline{\Omega}_\bullet(\Pdo^{0}_{\mathcal{A}\Gamma}({\tM},{\tE}))$, it is of the form $\mathbf{T}_{11}+ \mathbf{T}_{12}X+X \mathbf{T}_{21} + X\mathbf{T}_{22}X$. Then set
\begin{multline}\label{TR}
\overline{\mathrm{TR}}^{del}(\mathbf{T}_{11}+ \mathbf{T}_{12}X+X
\mathbf{T}_{21} + X\mathbf{T}_{22}X)\\
:= 	\mathrm{TR}^{del}( \mathbf{T}_{11})-(-1)^{\mathrm{deg}\mathbf{T}_{22}}	\mathrm{TR}^{del}(\mathbf{T}_{22}\Theta)\in   \widehat{ \Omega}^{del}_\bullet(\mathcal{A}\Gamma)_{ab}.
\end{multline}
The extension $\overline{\mathrm{TR}}\colon
\overline{\Omega}_\bullet(\Pdo^{-\infty}_{\mathcal{A}\Gamma}({\tM},{\tE}))\to
\widehat{\Omega}_\bullet(\mathcal{A}\Gamma)_{ab}$  of $\mathrm{TR}$ is defined
analogously.
\end{definition}

\begin{lemma}\label{lem:TR_as_map_of_complexes}
	The following map, defined as in \eqref{TR}, $$\overline{\mathrm{TR}}^{del}\colon
        \overline{\Omega}_\bullet(\Pdo^{0}_{\mathcal{A}\Gamma}({\tM},{\tE}))_{ab}\to
       \widehat{ \Omega}^{del}_\bullet(\mathcal{A}\Gamma)_{ab}$$  is a morphism of complexes. The same is true for $\overline{\mathrm{TR}}\colon \overline{\Omega}_\bullet(\Pdo^{-\infty}_{\mathcal{A}\Gamma}({\tM},{\tE}))_{ab}\to   \widehat{ \Omega}_\bullet(\mathcal{A}\Gamma)_{ab}$.
\end{lemma}
\begin{proof}
Let 	$\mathbf{T}_{11}+ \mathbf{T}_{12}X+X \mathbf{T}_{21} +
X\mathbf{T}_{22}X$ be an element of $\overline{\Omega}_\bullet(\Pdo^{0}_{\mathcal{A}\Gamma}({\tM},{\tE}))$, then by Proposition \ref{deloc-trace-commutator} we have
\begin{equation}
\begin{split}
\overline{\mathrm{TR}}^{del}&\left(d(\mathbf{T}_{11}+ \mathbf{T}_{12}X+X \mathbf{T}_{21} + X\mathbf{T}_{22}X)\right)=\\
&=\mathrm{TR}^{del}\left([\nabla^{\mathrm{Lott}},\mathbf{T}_{11}]\right)-(-1)^{\mathrm{deg}\mathbf{T}_{22}}  \mathrm{TR}^{del}\left([\nabla^{\mathrm{Lott}},\mathbf{T}_{22}]\Theta\right)=\\
&=\mathrm{TR}^{del}\left([\nabla^{\mathrm{Lott}},\mathbf{T}_{11}]\right)-(-1)^{\mathrm{deg}\mathbf{T}_{22}} \mathrm{TR}^{del}\left([\nabla^{\mathrm{Lott}},\mathbf{T}_{22}\Theta]\right)=\\
&=d\mathrm{TR}^{del}\left(\mathbf{T}_{11}\right)-(-1)^{\mathrm{deg}\mathbf{T}_{22}} d\mathrm{TR}^{del}\left(\mathbf{T}_{22}\Theta\right)=\\
&=d\overline{\mathrm{TR}}^{del}\left(\mathbf{T}_{11}+ \mathbf{T}_{12}X+X \mathbf{T}_{21} + X\mathbf{T}_{22}X\right).
\end{split}
\end{equation}
and hence  the lemma follows. The proof for $\overline{\mathrm{TR}}$ is identical.
\end{proof}

So far, we have used the algebras of smoothing operators
$\Psi^{-\infty}_{\complexs\Gamma}(\tM,\tE)$ and of pseudodifferential operators
  of order $0$, $\Psi^0_{\complexs\Gamma}(\tM,\tE)$, and versions with
  coefficients in non-commutative differential forms over $\complexs\Gamma$,
  as well as Fr\'echet completions $\mathcal{A}\Gamma$, and also the
  $X$-extended DGA $\overline\Omega_\bullet(\Psi^0_{\mathcal{A}\Gamma}(\tM,\tE))$.

  Observe that we can of course also pass to the versions of
  pseudodifferential operators of arbitrary, possibly positive, degree. These
  are not contained in bounded operators on $L^2$-sections (and the associated
  $C^*$-algebra), but the whole constructions make sense in the context of
  abstract algebras and even Fr\'echet algebras.
  
  \begin{proposition}\label{prop:trace_on_all_pseudos}
    The connection $\overline{\nabla}^{\mathrm{Lott}}$ and the delocalized
      trace $\mathrm{TR}^{del}$ and all the other constructions and results of
      Section \ref{section6} derived so far extend to the versions of
      pseudodifferential operators $\Psi^{+\infty}$ of arbitrary order.
      In particular, $\mathrm{TR}^{del}\colon
      \Psi^{+\infty}_{\hat\Omega_\bullet(\mathcal{A}\Gamma)}(\tM,\tE)\to
      \widehat\Omega_\bullet^{del}(\mathcal{A}\Gamma)_{ab}$ has the trace
        property. 
  \end{proposition}
  \begin{proof}
    This follows with exactly the same arguments as in the case of operators of
    order $0$. The crucial point is that pseudodifferential operators of
    arbitrary order have Schwarz kernels which are smooth outside the
    diagonal. The singularities on the diagonal are ``worse'' for operators of
    positive order than for operators of order $0$, but the diagonal anyway is
    cut off in the relevant constructions
  \end{proof}

\begin{lemma}
	Let $\overline{\Omega}_\bullet(C^\infty(M),E)$ be the subcomplex of
        $\overline{\Omega}_\bullet(\Pdo^{0}_{\mathcal{A}\Gamma}({\tM},{\tE}))$
        generated by $C^\infty(M,E)$ (acting as multiplication
          operators), then we have that 
	\begin{enumerate}
		\item $\overline{\Omega}_*(C^\infty(M,E))\cap \overline{\Omega}_*(\Pdo^{-\infty}_{\mathcal{A}\Gamma}({\tM},{\tE}))=\{0\}$,
		\item $\overline{\TR}^{del}$ sends $\overline{\Omega}_*(C^\infty(M,E))$ to zero.
	\end{enumerate}
\end{lemma}

\begin{proof}
First observe that $\overline{\nabla}^{\mathrm{Lott}}f$ is zero for any smooth
function $f$ on $M$. This implies that the subcomplex generated by
  $C^\infty(M)$ consists of linear combinations of products of elements of
  $C^\infty(M)$ and $X$. Therefore, the support of any element in
$\overline{\Omega}_*(C^\infty(M,E))$ is contained in the diagonal of
${\tM}\times{\tM}$. On the other hand, a smoothing operator with support on
the diagonal is zero. Hence, the assertions follow.
\end{proof}

\begin{lemma}\label{split-excisive-smoothing}
The DGA  $\overline{\Omega}_\bullet(\Pdo^{-\infty}_{\mathcal{A}\Gamma}({\tM},{\tE}))$ is Fr\'{e}chet split excisive. 
\end{lemma}

\begin{proof}
First let us consider the algebra of smoothing operators $\Pdo^{-\infty}(M)$. Taking the Schwartz kernel gives an  isomorphism with $C^\infty(M\times M)$, where the composition is given by convolution. 
	A way of seeing this convolution is the following one. Consider:
	\begin{itemize}
		\item the smooth inclusion $i\colon M\times M\times M\to (M\times M)\times  (M\times M) $ given by $(x,z,y)\mapsto (x,z,z,y)$;
		\item the smooth projection $p\colon M\times M\times M\to M\times M$ given by $(x,z,y)\mapsto (x,y)$.
	\end{itemize}
	After fixing a Riemannian measure $\mu$ on $M$, the convolution of two smoothing kernels $k_1$ and $k_2$ is given by the following formula
	$$k_1*k_2= p_*(i^*(k_1\otimes k_2))$$
where $i^*$ is the restriction induced by the pull-back and $p_*$ is the integration along the fibers of the surjection.
 
Now, let us fix $k\in C^\infty(M\times M)$. If we put $\overline{k}(x,z,y):= (\mathrm{Vol}(M))^{-1}k(x,y)$, then $p_*(\overline{k})=k$.
Moreover, since  $i$ is injective, $i^*\colon C^\infty(M\times M\times M\times M)\to C^\infty(M\times M\times M)$ is surjective.
By using the fact that $C^\infty(M\times M\times M\times M)\cong
C^\infty(M\times M)\hat{\otimes} C^\infty(M\times M)$ we find then $k_i^1,
k_i^2\in C^\infty(M\times M)$ for $i\in \NN$ such that $\overline{k}=\sum_i k_i^1\otimes k_i^2$. 
Hence we finally have that 
$$k=\sum_{i\in \NN}k_i^1*k_i^2\in C^\infty(M\times M) \quad \forall k\in C^\infty(M\times M).$$

Now we consider the algebra $\Pdo^\infty_{\widehat{\Omega}_\bullet(\mathcal{A}\Gamma)}({\tM},
{\tE})$ of smoothing operators with coefficients in forms. As explained in Lemma \ref{lem:top_bottom_basic} these are given by smooth functions on $C^\infty(M\times M, pr_1^*F\otimes pr_2^*F^*)$, for a certain bundle of modules $F$ over $M$. In this case the argument is similar to the previous one, just by using the fact that the composition  $(pr_1^*F\otimes pr_2^*F^*)\otimes (pr_1^*F\otimes pr_2^*F^*)\to pr_1^*F\otimes pr_2^*F^*$ is surjective.
Finally, by using this argument on each summand of  $\overline{\Omega}_\bullet(\Pdo^{-\infty}_{\mathcal{A}\Gamma}({\tM},{\tE}))$, we obtain the desired result. 
\end{proof}

\medskip

We are now in the position to define our maps from the  analytic surgery exact sequence in Section
\ref{subsect:realizing},
\begin{multline}\label{Pdo-mc-bis}
\dots \xrightarrow{\partial} K_*(0\hookrightarrow
C^*_{red}(\widetilde{M}\times_\Gamma\widetilde{M}))\xrightarrow{i_*}
K_*(C(M)\xrightarrow{\mathfrak{m}}\Pdo^0_{\Gamma}({\tM}))\\
\xrightarrow{\sigma_*} K_*(C(M)\xrightarrow{\mathfrak{\pi^*}} C(S^*M))\xrightarrow{\partial}\cdots
% \xymatrix{ \to K_*(0\hookrightarrow C^*_{red}(\widetilde{M}\times_\Gamma\widetilde{M}))\ar[r]^(.45){i_*}& K_*(C(M)\xrightarrow{\mathfrak{m}}\Pdo^0_{\Gamma}({\tM}))\ar[r]^(.45){\sigma_*}& K_*(C(M)\xrightarrow{\mathfrak{\pi^*}} C(S^*M))\to
% }
\end{multline}
 to 
 \begin{equation}\label{dRgamma-bis}
\xymatrix{\dots\ar[r]& H_*(\mathcal{A}\Gamma)\ar[r]& H_*^{del}(\mathcal{A}\Gamma)\ar[r]^{\delta_\Gamma}& H_{*+1}^{e}(\mathcal{A}\Gamma)\ar[r]& H_{*+1}(\mathcal{A}\Gamma)\ar[r]&\dots }
\end{equation}
This will employ  Definitions \ref{ch-rel-even} and \ref{ch-rel-odd} for the relative Chern characters. 
Notice that we have expunged the vector bundles from the notation, given that from the K-theoretic point of view they are negligible. 

\medskip
We  first observe that there is a canonical isomorphism of long exact sequences 
{\tiny\[
	\xymatrix{\cdots K_*(0\to
          \Pdo^{-\infty}_{\mathcal{A}\Gamma}({\tM}))\ar[r]^(.45){i_*}\ar[d]&
          K_*(C^\infty(M)\xrightarrow{\mathfrak{m}}\Pdo^0_{\mathcal{A}\Gamma}({\tM}))\ar[r]^(.45){\sigma_*}\ar[d]&
          K_*(C^{\infty}(M)\to
          \Pdo^0_{\mathcal{A}\Gamma}({\tM})/\Pdo^{-\infty}_{\mathcal{A}\Gamma}({\tM}))\ar[d]
          \\
           K_*(0\to C^*_{red}(\widetilde{M}
            \times_\Gamma\widetilde{M}))\ar[r]^(.45){i_*}&
            K_*(C(M)\xrightarrow{\mathfrak{m}}\Pdo^0_{\Gamma}({\tM}))\ar[r]^(.45){\sigma_*}&
            K_*(C(M)\xrightarrow{\mathfrak{\pi^*}} C(S^*M))\cdots }
\]
}
The first two vertical arrows are isomorphism because they are induced by
inclusions of dense holomorphically closed subalgebras. The third homomorphism
is induced by the symbol map and it is an isomorphism because of the Five Lemma.

We can think of the exact sequence
\begin{multline*}
  \cdots\xrightarrow{\partial} K_*(0\to
  \Pdo^{-\infty}_{\mathcal{A}\Gamma}({\tM}))\xrightarrow{i_*}K_*(C^\infty(M)\xrightarrow{\mathfrak{m}}\Pdo^0_{\mathcal{A}\Gamma}({\tM}))\\
  \xrightarrow{\sigma_*}K_*(C^{\infty}(M)\to \Pdo^0_{\mathcal{A}\Gamma}({\tM})/\Pdo^{-\infty}_{\mathcal{A}\Gamma}({\tM}))\xrightarrow{\partial}\cdots
\end{multline*}
as a smooth version of the analytic surgery sequence. It is this smooth
  version that will be mapped to non-commutative de Rham homology.

\begin{definition}\label{ch-c-main}
\strut	  

\noindent
\begin{enumerate}
		
\item Let $x$ be a class in $K_*(0\to \Pdo^{-\infty}_{\mathcal{A}\Gamma}({\tM}))$.
Then  define 
\[
\Ch_\Gamma(x):=\overline{\TR}_{[*-1]}\left(\Ch^{rel}(x)\right)\in H_{[*-1]}(\mathcal{A}\Gamma). 
\]
Here 
\begin{equation}\label{eq:Ch_rel}
\Ch^{rel}\colon K_*(0\to
\Pdo^{-\infty}_{\mathcal{A}\Gamma}({\tM}))\to H_{[*-1]}\left(0\to
  \overline{\Omega}(\Pdo^{-\infty}_{\mathcal{A}\Gamma}({\tM}))\right),
\end{equation}
and $\overline{\TR}_*$ is a compact notation for
$$(0,\overline{\TR})_*\colon H_{*}\left(0\to \overline{\Omega}(\Pdo^{-\infty}_{\mathcal{A}\Gamma}({\tM}))\right)\to H_{*}(\mathcal{A}\Gamma).$$

\item Let $y$ be a class in $K_*(C^{\infty}(M)\to  \Pdo^{0}_{\mathcal{A}\Gamma}({\tM}))$. Then define 
\[
\Ch^{del}_\Gamma(y):= \overline{\TR}^{del}_{[*-1]}\left(\Ch^{rel}(y)\right)\in H^{del}_{[*-1]}(\mathcal{A}\Gamma).
\]
Here $$\Ch^{rel}\colon K_*(C^{\infty}(M)\to  \Pdo^{0}_{\mathcal{A}\Gamma}({\tM}))\to H_{[*-1]}\left(\overline{\Omega}(C^{\infty}(M))\to \overline{\Omega}(\Pdo^{0}_{\mathcal{A}\Gamma}({\tM}))\right)$$ 
and $\overline{\TR}^{del}_*$ is a compact notation for 
$$
(0, \overline{\TR}^{del})_*\colon H_*\left(\overline{\Omega}(C^{\infty}(M))\to \overline{\Omega}(\Pdo^{0}_{\mathcal{A}\Gamma}({\tM}))\right)\to H^{del}_{*}(\mathcal{A}\Gamma).
$$

\item
 Let $z$ be a class in $K_*(C^\infty(M)\to \Pdo^0_{\mathcal{A}\Gamma}({\tM})/\Pdo^{-\infty}_{\mathcal{A}\Gamma}({\tM}))$. Recall from Lemma \ref{iso-mapping-cones} that $$q_*\colon K_*(C^\infty(M)+ \Pdo^{-\infty}_{\mathcal{A}\Gamma}({\tM})\to  \Pdo^{0}_{\mathcal{A}\Gamma}({\tM}))\to K_*(C^\infty(M)\to \Pdo^0_{\mathcal{A}\Gamma}({\tM})/\Pdo^{-\infty}_{\mathcal{A}\Gamma}({\tM}))$$ is an isomorphism. Then define 
\[
\Ch^e_\Gamma(z):=J_*^{-1}\circ(\overline{\TR}, \overline{\TR}^{del})_{[*-1]}\left(\Ch^{rel}(q_*^{-1}(z))\right)\in H^{e}_{[*]}(\mathcal{A}\Gamma).
\]
Here
\begin{multline*}
  \Ch^{rel}\colon
  K_*(C^{\infty}(M)+\Pdo^{-\infty}_{\mathcal{A}\Gamma}({\tM})\to
  \Pdo^{0}_{\mathcal{A}\Gamma}({\tM}))\\
  \to
  H_*\left(\overline{\Omega}(C^{\infty}(M))+
    \overline{\Omega}(\Pdo^{-\infty}_{\mathcal{A}\Gamma}({\tM}))
    \to \overline{\Omega}(\Pdo^{0}_{\mathcal{A}\Gamma}({\tM}))\right)
\end{multline*}
and
\begin{multline*}
  (\overline{\TR}, \overline{\TR}^{del})\colon
  \left(\overline{\Omega}(C^{\infty}(M))+\overline{\Omega}(\Pdo^{-\infty}_{\mathcal{A}\Gamma}({\tM}))\to
    \overline{\Omega}(\Pdo^{0}_{\mathcal{A}\Gamma}({\tM}))\right)\\
  \to \left(\widehat\Omega(\mathcal{A}\Gamma)\to \widehat\Omega^{del}(\mathcal{A}\Gamma)\right)
\end{multline*}
is the natural
morphism of mapping cone complexes induced by $\overline{\TR}$ and $\overline{\TR}^{del}$;
the isomorphism $J_*$ (which is of degree $-1$) is induced by the inclusion $J$ of
$\widehat\Omega^e(\mathcal{A}\Gamma)$ into
$\left(\widehat\Omega(\mathcal{A}\Gamma)\to
  \widehat\Omega^{del}(\mathcal{A}\Gamma)\right)$.
 See the proof of Theorem \ref{commutativity-chern} below 
  for more on $J_*$.
\end{enumerate}
\end{definition}

Observe that we can extend the map of complexes $\overline{\TR}$  to
$\overline{\Omega}(C^{\infty})+
\overline{\Omega}(\Pdo^{-\infty}_{\mathcal{A}\Gamma})$ by setting it to
be zero on $\overline{\Omega}(C^{\infty})$. Indeed, although it is not a
direct sum of DGAs, $\overline{\Omega}(C^{\infty})+
\overline{\Omega}(\Pdo^{-\infty}_{\mathcal{A}\Gamma})$ is a direct sum of
complexes and then this extension makes sense, using Lemma \ref{lem:TR_as_map_of_complexes}.

\smallskip
\noindent
The following theorem is one of the main results of this article.

 \begin{theorem}\label{commutativity-chern}
The following diagram of long exact sequences, with vertical maps as in Definition \ref{ch-c-main}, is  commutative:
\small 	\begin{equation}\label{mktoh}{\tiny
 	\xymatrix{\to K_*(0\to
          \Pdo^{-\infty}_{\mathcal{A}\Gamma}({\tM}))\ar[r]^(.45){i_*}\ar[d]^{\Ch_\Gamma}&
          K_*(C^\infty(M)\xrightarrow{\mathfrak{m}}\Pdo^0_{\mathcal{A}\Gamma}({\tM}))\ar[r]^(.45){\sigma_*}\ar[d]^{\Ch_\Gamma^{del}}&
          K_*(C^{\infty}(M)\xrightarrow{\mathfrak{\pi^*}}
          \Pdo^0_{\mathcal{A}\Gamma}({\tM})/\Pdo^{-\infty}_{\mathcal{A}\Gamma}({\tM}))\ar[d]^{\Ch_\Gamma^{e}}\to
          \\
 		\to H_{[*-1]}(\mathcal{A}\Gamma)\ar[r]& H_{[*-1]}^{del} (\mathcal{A}\Gamma)\ar[r]^{\delta}& H_{[*]}^{e} (\mathcal{A}\Gamma)\to}
}            \end{equation}
            
 \end{theorem}

\begin{proof} 
	In the following, we write $C^\infty$, $\Psi^0_{\mathcal{A}\Gamma}$
        and $\Psi^{-\infty}_{\mathcal{A}\Gamma}$ as shorthands for
        $C^{\infty}(M)$, $\Pdo^0_{\mathcal{A}\Gamma}({\tM})$ and
        $\Pdo^{-\infty}_{\mathcal{A}\Gamma}({\tM})$, respectively.
	By Lemma \ref{iso-mapping-cones}, we have the following isomorphism  of long exact sequences 
	\begin{equation}\label{mapping1}{\tiny
	\xymatrix{\cdots\ar[r]^(.3){\partial'}& K_{*+1}(\Pdo^{-\infty}_{\mathcal{A}\Gamma})\ar[r]^(.4)S\ar[d]^S& K_*(C^\infty\to \Pdo^0_{\mathcal{A}\Gamma})\ar[r]\ar[d]^{\mathrm{id}}& K_*(C^{\infty}+ \Pdo^{-\infty}_{\mathcal{A}\Gamma}\to \Pdo^0_{\mathcal{A}\Gamma})\ar[r]^(.75){\partial'}\ar[d]^{q_*}&\cdots\\
		\cdots\ar[r]^(.3){\partial}& K_*(0\to \Pdo^{-\infty}_{\mathcal{A}\Gamma})\ar[r]& K_*(C^\infty\to\Pdo^0_{\mathcal{A}\Gamma})\ar[r]& K_*(C^{\infty}\to \Pdo^0_{\mathcal{A}\Gamma}({\tM})/\Pdo^{-\infty}_{\mathcal{A}\Gamma}({\tM}))\ar[r]^(.75){\partial}&\cdots}}
	\end{equation}
	and by Lemma \ref{split-excisive-smoothing} by Theorem \ref{functoriality-chern}  there exists the following commutative diagram
		\begin{equation}\label{mapping2}{\tiny
		\xymatrix{ \to K_{*+1}(\Pdo^{-\infty}_{\mathcal{A}\Gamma})\ar[r]^(.4)S\ar[d]^\Ch& K_*(C^\infty\to \Pdo^0_{\mathcal{A}\Gamma})\ar[r]\ar[d]^{\Ch^{rel}}& K_*(C^{\infty}+ \Pdo^{-\infty}_{\mathcal{A}\Gamma}\to \Pdo^0_{\mathcal{A}\Gamma})\to\ar[d]^{\Ch^{rel}}\\
	 H_{[*+1]}(\overline{\Omega}(\Pdo^{-\infty}_{\mathcal{A}\Gamma}))\ar[r]^(.4)S& H_{[*+1]}(\overline{\Omega}(C^\infty)\to \overline{\Omega}(\Pdo^0_{\mathcal{A}\Gamma}))\ar[r]& H_{[*+1]}(\overline{\Omega}(C^{\infty})+ \overline{\Omega}(\Pdo^{-\infty}_{\mathcal{A}\Gamma})\to \overline{\Omega}(\Pdo^0_{\mathcal{A}\Gamma}))\to	}}
    \end{equation}
Furthermore  the following commutative diagram of complexes
	\begin{equation}\label{mapping3}{\tiny
	\xymatrix{0\to (\overline{\Omega}_*(C^\infty)\to \overline{\Omega}_*(\Pdo^0_{\mathcal{A}\Gamma}))\ar[r]\ar[d]^{\overline{\TR}^{del}}&(\overline{\Omega}_*(C^{\infty})+ \overline{\Omega}_*(\Pdo^{-\infty}_{\mathcal{A}\Gamma})\to\overline{\Omega}_*(\Pdo^0_{\mathcal{A}\Gamma}) )\ar[r]\ar[d]^{(\overline{\TR},\overline{\TR}^{del})}& \overline{\Omega}_{*+1}(\Pdo^{-\infty}_{\mathcal{A}\Gamma})\ar[d]^{\overline{\TR}}\to 0\\
		0\to \widehat\Omega^{del}_*(\mathcal{A}\Gamma)\ar[r]& (\widehat\Omega_*(\mathcal{A}\Gamma)\to\widehat\Omega^{del}_*(\mathcal{A}\Gamma))\ar[r]& \widehat\Omega_{*+1}(\mathcal{A}\Gamma)\to 0
		}}
\end{equation}
  induces a map of long exact sequences in homology.
  Finally observe that we have a further isomorphism of exact sequences given by 
  \begin{equation}\label{mapping4}{\small
  \xymatrix{  		\cdots\ar[r]& H_*( \widehat\Omega^{del}(\mathcal{A}\Gamma))\ar[r]^{\delta_\Gamma}\ar[d]^{=}& H_{*+1}(\widehat\Omega^{e}(\mathcal{A}\Gamma))\ar[r]\ar[d]_{\iso}^{J_*}&H_{*+1} (\widehat\Omega(\mathcal{A}\Gamma))\ar[r]\ar[r]\ar[d]^{=}& \cdots\\
  	\cdots\ar[r]& H_*( \widehat\Omega^{del}(\mathcal{A}\Gamma))\ar[r]& H_*(\widehat\Omega(\mathcal{A}\Gamma)\to\widehat\Omega^{del}(\mathcal{A}\Gamma))\ar[r]&H_{*+1} (\widehat\Omega(\mathcal{A}\Gamma))\ar[r]& \cdots
  	}}
 \end{equation}
 where the arrow in the middle is induced by the inclusion of $\widehat\Omega^{e}_*(\mathcal{A}\Gamma)$ into $(\widehat\Omega_*(\mathcal{A}\Gamma)\to \widehat\Omega^{del}_*(\mathcal{A}\Gamma)) $.
 As the mapping of exact sequences in \eqref{mktoh} is given by the composition of the inverse of \eqref{mapping1} followed by \eqref{mapping2},  \eqref{mapping3} and the inverse of \eqref{mapping4}, the result follows.

\end{proof}

By using  the isomorphism of long exact sequence  in \eqref{ases} and passing from C*-algebras to smooth dense subalgebra, thanks to the previous theorem, we  obtain the following commutative diagram 
\begin{equation}\label{main-commutative-diagram}
\xymatrix{
	\cdots\ar[r]& K_{*-1}(C^*_{red}\Gamma)\ar[r]^{s}\ar[d]^{\Ch_\Gamma}& \SG^\Gamma_{*} ({\tM}) \ar[r]^{c}\ar[d]^{\Ch_{\Gamma}^{del}}& K_*^{\Gamma}(\tM)\ar[r]\ar[d]^{\Ch^e_\Gamma}&\cdots\\
	\cdots\ar[r]^(.4){j_*}& H_{[*-1]}(\mathcal{A}\Gamma)\ar[r]& H_{[*-1]}
        ^{del}(\mathcal{A}\Gamma)\ar[r]^{\delta}& H_{[*]}^{e}(\mathcal{A}\Gamma)\ar[r]^{j_*}&\cdots}
\end{equation}
which is the result stated in the Introduction.

\section{The algebraic index and the localized Chern character $\Ch_\Gamma^e$}\label{comparison lott1}

In the next section  we will deal with the delocalized Chern character $\Ch_{\Gamma}^{del}$.
In the current section, we want to recognize the localized  $\Ch_\Gamma^e$ as a well known homomorphism
involving the algebraic index.

\begin{remark}\label{algebraic-k-theory}
	Before proceeding with this section, we need to recall some facts about relative algebraic K-theory, without entering into the details.
	Observe that  the equivalence relation used in the algebraic world between two idempotents $e_0$ and $e_1$ over $A$ is to be stably conjugated by an invertible element $z$ over $A$. Hence the elements in the relative K-group $K_0^{alg}(A\xrightarrow{\varphi}B)$ are given by triples $[e_0,e_1;z]$, where $e_0,e_1$ are idempotents over $A$ and $z$ is an invertible element over $B$ which conjugate $\varphi(e_0)$ into $\varphi(e_1)$.
	If $A$ is a dense holomorphically closed subalgebra of a
	$C^*$-algebra, this definition of relative K-theory is equivalent
	to the one used so far.
	
	In the algebraic setting, given two idempotent $e_0,e_1$ which are
	conjugated by an invertible element $z$, they are homotopic through a
	path $f_t$  defined as in \eqref{classmc}. This is sufficient to define \begin{equation}\label{Ch-algebraic}\Ch^{rel}_{alg}\colon K_0^{alg}(A\to B)\to
	H_{even}(\Omega(A)\to \Omega(B))\end{equation} exactly as in Definition \ref{ch-rel-even}:
	notice indeed that $\int_0^1\overline{\Ch}(f_t)dt$ can be explicitly
	calculated and it is given by a completely algebraic expression in $e_0$,
	$e_1$ and $z$.
\end{remark}

\medskip

We now recall the definition of the algebraic index, see for instance \cite[Definition 5.2]{ConnesMoscovici}.
First we remark that $$\Sigma:=\Pdo^0_{\mathcal{A}\Gamma}(\widetilde{M})/\Pdo^{-\infty}_{\mathcal{A}\Gamma}(\widetilde{M}) \cong
\Pdo^0_{\CC\Gamma}(\widetilde{M})/\Pdo^{-\infty}_{\CC\Gamma}(\widetilde{M}),$$
the algebra of complete symbols. The isomorphism is canonical and it is induced
	by the inclusion $\Pdo^0_{\CC\Gamma}\to \Pdo^0_{\mathcal{A}\Gamma}$. Let us see the details. The map
	is injective, because $\Pdo^0_{\CC\Gamma}(\tM)\cap
	\Pdo^{-\infty}_{\mathcal{A}\Gamma}(\tM) =
	\Pdo^{-\infty}_{\CC\Gamma}(\tM)$, the operators which are at the same
	time smoothing and satisfy the finite support condition of
	$\Pdo_{\CC\Gamma}(\tM)$. For the surjectivity, we use the cutoff argument
	already employed in Definition \ref{def-tr-del} of the delocalized
	trace. Indeed, there we saw that every $\mathbf{T}\in
	\Pdo^0_{\mathcal{A}\Gamma}(\tM)$ can be written as a sum
	$\mathbf{T}=\mathbf{T}_0+\mathbf{T}_\infty$ with $\mathbf{T}_0\in
	\Pdo^0_{\CC\Gamma}$ and $\mathbf{T}_\infty\in
	\Pdo^{-\infty}_{\mathcal{A}\Gamma}$, showing that the inclusion induces a
	surjective map $\Pdo_{\CC\Gamma}^0(\tM)\to
	\Pdo_{\mathcal{A}\Gamma}^0(\tM)/\Pdo^{-\infty}_{\mathcal{A}\Gamma}(\tilde
	M)$, as required.

The algebraic index  
\begin{equation}\label{algind}
	\mathrm{Ind}_{\CC\Gamma}\colon  K_0(C^{\infty}(M)\to \Sigma)\to K_0(\Pdo^{-\infty}_{\CC\Gamma}(\widetilde{M}))
\end{equation}
is defined in the following way. 
A class in $K_0(C^{\infty}(M)\to \Sigma)$ is given by a triple $[e_0, e_1;
\sigma]$, where $e_i\in M_n(C^\infty(M))$ is the projection associated
to the vector bundle $E_i$, and where $\sigma$ is an invertible complete
symbol such that $e_0=\sigma^{-1} e_1\sigma$, here we identify $e_i$
with its image in the complete symbol algebra. Choose two  pseudodifferential operators 
	$P\in \Pdo^0_{\CC\Gamma}(\widetilde{M};\widetilde{E}_0, \widetilde{E}_1)$ and
	$Q\in \Pdo^0_{\CC\Gamma}(\widetilde{M};\widetilde{E}_1, \widetilde{E}_0)$
	corresponding to  $\sigma $ or
	$\sigma^{-1}$, respectively. Concretely, $P=e_1Ae_0$ and $Q=e_0Be_1$ where $A,B\in
\Psi^0_{\complexs\Gamma}(\tM,\complexs^n)$ with symbols
$\sigma,\sigma^{-1}$. Then
\begin{equation*}
	S=1_{E_0}-QP, \quad T=1_{E_1}-PQ \in
	\Pdo^{-\infty}_{\CC\Gamma}(\widetilde{M},\widetilde{E}_0\oplus
	\widetilde{E}_1)
\end{equation*}
and \begin{equation}\label{indCGamma}\mathrm{Ind}_{\CC\Gamma}([e_0, e_1; \sigma]):=[CM(P,Q)]-[diag(0, 1_{E_1})]\in K_0(\Pdo^{-\infty}_{\CC\Gamma}(\widetilde{M}))\end{equation} where 
\begin{equation}\label{matrix-cm}
	CM(P,Q):=\begin{pmatrix}S^2 & S(1_{E_0}+S)Q\\T P & 1_{E_1}- T^2\end{pmatrix}
\end{equation}
as in \cite[Section 2]{ConnesMoscovici}.

Let $\Ch$ denote the first vertical map in \eqref{mapping2}.
\begin{lemma}\label{locality-index}
	The image of
	$\overline{\mathrm{TR}}_*\circ \Ch\circ \mathrm{Ind}_{\CC\Gamma}$ belongs
	to
        \begin{equation*}
          H_{\mathrm{even}}(\Omega^e_*(\CC\Gamma))\subset H_{\mathrm{even}}(\Omega_*(\complexs\Gamma)).
        \end{equation*}
\end{lemma}
\begin{proof}
	This follows immediately from the fact that  we can make our choices so that $CM(P,Q)$ is almost local and that
	the class does not depend on choices.
\end{proof}

Using Lemma \ref{locality-index}, we can give the following definition. 
\begin{definition}
	\begin{equation}\label{ch-alg-e}\Ch^e_{\Gamma,alg}:=\overline{\mathrm{TR}}_*\circ \Ch\circ \mathrm{Ind}_{\CC\Gamma}\colon K_0(C^\infty(M)\to \Sigma)\to H^e_{\mathrm{even}}(\CC\Gamma).\end{equation}
\end{definition}

In fact, we can say much more. Define a Dirac class $[e_0, e_1; \sigma]$ as the class associated to a generalized Dirac operator.
\begin{proposition}\label{loc-lott}
	If $[e_0, e_1; \sigma]\in K_0(C^{\infty}(M)\to \Sigma)$ is a Dirac class, then 
	$$\Ch^e_{\Gamma,alg}([e_0, e_1; \sigma])=\left[ \int_M \operatorname{AS}(M,\sigma)\wedge
	\omega_{{\rm Lott}} \right]\in H^e_{\mathrm{even}}(\CC\Gamma),$$
	where $\operatorname{AS}(M,\sigma)$ is the corresponding Atiyah-Singer integrand and where 
	$\omega_{{\rm Lott}} (M)\in \Omega^* (M)\otimes \Omega_* (\CC\Gamma)$ is Lott's bi-form.\end{proposition}
For more on $\omega_{{\rm Lott}} (M)$, a bi-form which is closed in
both arguments and that appears naturally in Lott's treatment of the Connes-Moscovici higher index theorem,
we refer the reader to \cite{Lott1} and also \cite[Theorem 13.6]{LeichtnamPiazzaMemoires}.
\begin{proof}

	Taking into account the characterization of  $\Ch^e_{\Gamma,alg}$ given in the previous Lemma, the
	result  follows from a sharpening of  Lott's higher index theorem for operators of Dirac type. 
	Indeed, for  this particular statement
	we need to employ  the version of Lott's theorem given in \cite[Section
	5]{Goro-Lott} (just take $B$ equal to a point there), but with the improvement given in
	\cite[Appendix A]{Goro-Lott-adv}, which allows to prove this formula by using almost local parametrices and the Getzler calculus. Notice that this argument avoids heat kernel techniques and rests ultimately on the analysis given by Moscovici
	and Wu in \cite{MoscoviciWu}.  
	In these references the left hand side is paired with a cyclic cocycle; however, as explained in \cite{LP-etale},
	see in particular Definition 11 and Theorem 4 (1) there, 
	one can give a statement directly at the level of non-commutative de Rham
	classes.
\end{proof}
It is straightforward to see that the formula \eqref{indCGamma} for $\mathrm{Ind}_{\CC\Gamma}$ is an
explicit implementation of the composition of $ K_0(C^{\infty}(M)\to
\Sigma)\to K_0(\Pdo^0_{\CC\Gamma}(\widetilde{M})\to \Sigma)$ and the excision
isomorphism
$K_0(\Pdo^{-\infty}_{\CC\Gamma}(\widetilde{M}))\xrightarrow{\iso}
K_0(\Pdo^0_{\CC\Gamma}(\widetilde{M})\to \Sigma)$. Equivalently,
$\mathrm{Ind}_{\complexs\Gamma}$ is the boundary map for the extension of
pairs of algebras
\begin{equation*}
  0 \to (0\to \Pdo^{-\infty}_{\complexs\Gamma}(\tM))
  \to (C^\infty(M) \to \Pdo^0_{\complexs\Gamma}(\tM)) \to
  (C^\infty(M)\to \Sigma)\to 0,
\end{equation*}
composed with the suspension isomorphism.
It is then a direct consequence of Theorem \ref{functoriality-chern} that
the following square is commutative, where we use that the top row is, by Lemma \ref{iso-mapping-cones}, the
map $\partial'$ of Theorem \ref{functoriality-chern}:
\begin{equation}\label{alg-ch}{\small
	\xymatrix{K_0(C^\infty+\Pdo^{-\infty}_{\complexs\Gamma} \to
		\Pdo^0_{\complexs\Gamma}) \ar[r]^{\iso} \ar[d]_{\Ch^{rel}}& K_0(C^{\infty}\to \Sigma)\ar[r]^{\mathrm{Ind}_{\CC\Gamma}}
		&K_0(\Pdo^{-\infty}_{\CC\Gamma})\ar[d]^{\Ch}\\
		H_{\mathrm{odd}}(\overline{\Omega}_*(C^\infty)+\overline{\Omega}_*(\Psi^{-\infty}_{\complexs\Gamma})
		\to \overline{\Omega}_*(\Pdo^0_{\complexs\Gamma}))
		\ar[rr]^{\delta_c'=S\circ pr} &&
		H_{\mathrm{even}}(\overline{\Omega}_*(\Pdo^{-\infty}_{\CC\Gamma})).} }
\end{equation}

\begin{remark}\label{ch-loc-alg}
	Before stating the main proposition of this section, it is worth to notice that, by the very definition of the relative algebraic Chern character, see \eqref{Ch-algebraic}, 
	the following square is commutative:
	\begin{equation}\label{dg-ch-alg-1}
	\xymatrix{
				K_0^{alg}(C^\infty+ \Pdo^{-\infty}_{\CC\Gamma}\to \Pdo^{0}_{\CC\Gamma})\ar[r]^(.4){\Ch^{rel}_{alg}}\ar[d]^{i_*}_\cong& 		H_{\mathrm{odd}}(\overline{\Omega}(C^{\infty})+ \overline{\Omega}(\Pdo^{-\infty}_{\CC\Gamma})\to \overline{\Omega}(\Pdo^{0}_{\CC\Gamma}))\ar[d]	\\K_0(C^\infty+ \Pdo^{-\infty}_{\mathcal{A}\Gamma}\to \Pdo^{0}_{\mathcal{A}\Gamma})
				\ar[r]^(.4){\Ch^{rel}} & H_{\mathrm{odd}}(\overline{\Omega}(C^{\infty})+ \overline{\Omega}(\Pdo^{-\infty}_{\mathcal{A}\Gamma})\to \overline{\Omega}(\Pdo^{0}_{\mathcal{A}\Gamma}))}
	\end{equation}
	where the vertical maps are induced by the inclusions.  
	Moreover, by naturality also the following diagrams commute:
	\begin{equation}\label{dg-ch-alg-2}
		\xymatrix@C=55pt{
			H_{\mathrm{odd}}(\overline{\Omega}(C^{\infty})+ \overline{\Omega}(\Pdo^{-\infty}_{\CC\Gamma})\to \overline{\Omega}(\Pdo^{0}_{\CC\Gamma}))\ar[d]\ar[r]^(.55){(\overline{\TR},\overline{\TR}^{del})_*}&
			H_{\mathrm{odd}}(\Omega(\CC\Gamma)\to\Omega^{del}(\CC\Gamma))\ar[d]\\
			H_{\mathrm{odd}}(\overline{\Omega}(C^{\infty})+ \overline{\Omega}(\Pdo^{-\infty}_{\mathcal{A}\Gamma})\to \overline{\Omega}(\Pdo^{0}_{\mathcal{A}\Gamma}))\ar[r]^(.55){(\overline{\TR},\overline{\TR}^{del})_*}&
			H_{\mathrm{odd}}(\widehat\Omega(\mathcal{A}\Gamma)\to\widehat\Omega^{del}(\mathcal{A}\Gamma))
		}
	\end{equation}
	\begin{equation}\label{dg-ch-alg-2a}
		\xymatrix@C=55pt{		
			H_{\mathrm{odd}}(\Omega(\CC\Gamma)\to\Omega^{del}(\CC\Gamma))\ar[d]& 	H_{\mathrm{even}}(\Omega^e(\CC\Gamma))\ar[d]^{j_*}\ar[l]^(.4){\iso}_(.4){J^{alg}_*}\\
		H_{\mathrm{odd}}(\widehat\Omega(\mathcal{A}\Gamma)\to\widehat\Omega^{del}(\mathcal{A}\Gamma))& 	H_{\mathrm{even}}(\widehat\Omega^e(\mathcal{A}\Gamma))\ar[l]^(.4){\iso}_(.4){J_*}	
		}
	\end{equation}
	where we recall that $J$ is the inclusion of $\Omega^e(\mathcal{A}\Gamma)$ into $(\Omega(\mathcal{A}\Gamma)\to\Omega^{del}(\mathcal{A}\Gamma))$ and $J^{alg}$ is defined similarly. 
	Now, recall that $\Ch_\Gamma^e$ is defined as the composition of the
        bottom rows of \eqref{dg-ch-alg-1}, \eqref{dg-ch-alg-2}, and  \eqref{dg-ch-alg-2a}.
		Then, it follows that we have the following algebraic characterization of  it: $$\Ch_\Gamma^e= j_*\circ (J_*^{alg})^{-1}\circ (\overline{\TR},\overline{\TR}^{del})_*\circ \Ch^{rel}_{alg}\circ (i_*)^{-1}.$$
\end{remark}  
We are finally in the position of proving the following result.

\begin{proposition}\label{prop:alg_ind_descr}
	Let $j\colon
	\Omega^e_*(\CC\Gamma)\hookrightarrow\widehat{\Omega}^e_*(\mathcal{A}\Gamma)$
	be the natural inclusion.
	The localized Chern character factors through the
	algebraic one \eqref{ch-alg-e}, specifically
	\begin{equation}\label{eq:ind_factorization}
		\Ch_\Gamma^e = j_*\circ \Ch^e_{\Gamma,alg}\colon K_0(C^\infty(M)\to \Sigma)\to
		H_{\mathrm{even}}^e(\mathcal{A}\Gamma).
	\end{equation}
	Moreover, if $[e_0, e_1; \sigma]\in K_0(C^\infty(M)\to \Sigma)$ is defined 
	by a Dirac-type operator, then
	\begin{equation}\label{eq:Dirac_ind_formula}
		\Ch_\Gamma^e ([e_0, e_1; \sigma])= 
		j_*   \left[ \int_M {\rm AS}(M,\sigma)\wedge\omega_{{\rm Lott}}
		\right]    
	\end{equation}
	with $\operatorname{AS}(M,\sigma)$ 
	equal to the  Atiyah-Singer integrand corresponding to $\sigma$ and with
	$\omega_{{\rm Lott}} (M)\in \Omega^* (M)\otimes \Omega_* (\CC\Gamma)$ equal to  Lott's bi-form.
\end{proposition}

\begin{proof}
	We only have to prove \eqref{eq:ind_factorization}, given that \eqref{eq:Dirac_ind_formula} follows directly from 
	Proposition \ref{loc-lott}. The main difficulty is that
	$\Ch^e_{\Gamma,alg}$ involves taking the K-theoretic index map, whereas
	$\Ch^e_\Gamma$ doesn't.   
	However, the commutativity of  \eqref{alg-ch} is the
	commutativity of the top rectangle in the following diagram. It follows then
	immediately that the whole diagram below commutes:
	\[
	\xymatrix{
			K_0(C^\infty+ \Pdo^{-\infty}_{\CC\Gamma}\to \Pdo^{0}_{\CC\Gamma})\ar[d]^{\Ch^{rel}_{alg}}\ar[r]^\cong&  	K_0(C^\infty \to \Sigma)
		\ar[r]^{\mathrm{Ind}_{\CC\Gamma}} &K_0(\Pdo^{-\infty}_{\CC\Gamma}) \ar[d]^{\Ch}\\ 	
		  	H_{\mathrm{odd}}(\overline{\Omega}(C^{\infty})+
		\overline{\Omega}(\Pdo^{-\infty}_{\CC\Gamma})\to
		\overline{\Omega}(\Pdo^0_{\CC\Gamma}))\ar[rr]^(.6){\delta_c'=S\circ pr}\ar[d]_{(\overline{\TR},\overline{\TR}^{del})}
		& 	&  H_{\mathrm{even}}(\overline{\Omega}(\Pdo^{-\infty}_{\CC\Gamma}))\ar[d]^{\overline{\TR}}  \\
		H_{\mathrm{odd}}(\Omega(\CC\Gamma)\to\Omega^{del}(\CC\Gamma))
		\ar[rr]^{S\circ pr}&
		&H_{\mathrm{even}}(\Omega(\CC\Gamma))\\
              H_{\mathrm{even}}(\Omega^e(\complexs\Gamma)) \ar[u]^(.4){J_*^{alg}}}
	\]
	Here $\delta_c'$ is defined in \eqref{delta'}, i.e.~indeed is induced by the
	projection onto
	$\overline{\Omega}(\Pdo^{-\infty}_{\complexs\Gamma}\to 0)$ composed with
	suspension to
	$H_{\mathrm{even}}(\overline{\Omega}(\Pdo^{-\infty}_{\complexs\Gamma}))$. Consequently,
	the composition of the lower row is the inclusion
	$H_{\mathrm{even}}(\Omega^e(\complexs\Gamma))\into H_{\mathrm{even}}(\Omega(\complexs\Gamma))$. 
	Equality \eqref{eq:ind_factorization} then follows from Remark \ref{ch-loc-alg}.
\end{proof}

	\begin{remark}
		Since $K_0(C^{\infty}(M)\to \Sigma)$ is generated modulo torsion by Dirac classes and since the range of
		$\overline{\mathrm{TR}}_*\circ \Ch\circ \mathrm{Ind}_{\CC\Gamma}$ is a $\complexs$-vector
		space, Proposition \ref{prop:alg_ind_descr} gives a complete description of
		the homomorphism $\Ch_\Gamma^e$.
		Moreover, we have dealt here with the even dimensional case. We expect that a formula similar to \eqref{eq:Dirac_ind_formula} could be obtained by suspension in  the odd dimensional case.
	\end{remark}

\section{Comparison with Lott's rho form} \label{comparison lott2}

Let $\widetilde{D}_M$ be an $L^2$-invertible $\Gamma$-equivariant Dirac operator  on a 
Galois $\Gamma$-covering $\widetilde{M}\to M$ acting on the sections of a $\Gamma$-equivariant 
complex vector bundle $\widetilde{E}$.
This section is devoted to a comparison of our class $\Ch_\Gamma^{del}([\varrho (\widetilde{D}_M)])$
with the higher rho invariant of Lott. 

Let $\widetilde{D}_M$ be an $L^2$-invertible $\Gamma$-equivariant Dirac operator  on a 
Galois $\Gamma$-covering $\widetilde{M}\to M$ acting on the sections of a $\Gamma$-equivariant 
complex vector bundle $\widetilde{E}$.
This section is devoted to a comparison of our class $\Ch_\Gamma^{del}([\varrho (\widetilde{D}_M)])$
with the higher rho invariant of Lott. 

The higher rho invariant of Lott 
appeared for the first
time in \cite{Lott2}, as the delocalized part of its higher eta invariant. Lott's higher eta invariant and Lott's higher rho invariant were 
originally defined  under the additional assumption that $\Gamma$ is of polynomial growth (see below for the general case).
Lott's higher eta invariant for $\Gamma$ of polynomial growth is an element
$$\eta_{{\rm Lott}} (\widetilde{D}_M)\in \widehat{\Omega}_*(\mathcal{A}\Gamma)_{ab}\,, $$
with $\mathcal{A}\Gamma\subset C^*_{red} \Gamma$ denoting the algebra of rapidly decreasing functions on $\Gamma$. We shall review the definition of Lott's higher eta invariant after this brief introduction.\\
It was conjectured in  \cite{Lott2} that the higher eta invariant was the boundary correction term
in a higher Atiyah-Patodi-Singer index theorem. The conjecture was settled in \cite{LeichtnamPiazzaMemoires}. Later, following
an idea of Lott \cite{Lott3},  the higher
eta invariant and the higher Atiyah-Patodi-Singer index theorem were established for any finitely generated 
discrete group, see \cite{LeichtnamPiazzaBSMF}. In this generality the algebra $\mathcal{A}\Gamma\subset C^*_{red} \Gamma$ can be taken to be
the Connes-Moscovici algebra, using crucially that the latter is the projective limit of involutive Banach algebras with unit. Wahl \cite{Wahl1} extended these results even further, allowing $\mathcal{A}\Gamma\subset C^*_{red} \Gamma$
to be {\it any} projective limit of involutive Banach algebras with unit. Thus, following Wahl, in the sequel we shall assume that we have a projective
system of involutive Banach algebras with units $(\mathcal{A}_{j},\iota_{j+1,j}: \mathcal{A}_{j+1} \to \mathcal{A}_{j})_{j\in\NN}$
satisfying the following  conditions:

(i) $ \mathcal{A}_{0}=C^*_{red} \Gamma$;

(ii) the map $\iota_{j+1,j}: \mathcal{A}_{j+1} \to \mathcal{A}_{j}$ is injective for any $j$;

(iii) the induced map $\iota_j$ from the projective limit $\mathcal{A}_\infty$ to $ \mathcal{A}_{j}$ has dense image;

(iv) for any $j>0$ the algebra  $\mathcal{A}_{j}$ is holomorphically closed in  $ \mathcal{A}_{0}\equiv
C^*_{red} \Gamma$

(v) $\CC\Gamma$ is a subalgebra of each $\mathcal{A}_{j}$.

\noindent
Under these assumptions we can consider $\mathcal{A}\Gamma:=\mathcal{A}_\infty$, an involutive $m$-convex  Fr\'echet algebra 
with unit. As already remarked,  the Connes-Moscovici algebra $\mathcal{B}_\infty$
satisfies these hypothesis. If $\Gamma$ is Gromov hyperbolic then a particular algebra
$\mathcal{A}\Gamma$ defined by  Puschnigg in \cite{Puschnigg} and discussed
in detail in Section \ref{sec:Puschnigg} below  also
satisfies these hypothesis. This ends our brief introduction to Lott's higher eta invariant. Let us move now
to the actual definition.

\subsection*{Lott's higher eta invariant}
We shall now recall the definition of Lott's higher eta invariant, a non-commutative analogue of the Bismut-Cheeger eta form.
We shall follow the sign conventions given in \cite{Lott1}, \cite{Lott2}, based in turn on those of \cite{BGV}.
The definition of higher eta invariants is slightly different depending on the parity of the dimension.
Assume that we have  an even dimensional compact Riemannian manifold $(M,g)$ and
a $\ZZ_2$-graded bundle of hermitian Clifford modules, $E$, endowed with a Clifford unitary connection $\nabla^E$. 
We denote
by $D$ the associated Dirac operator. We also have a Galois $\Gamma$-covering $\widetilde{M}$ and denote
by $\tilde{g}$, $\widetilde{E}$ and $\widetilde{D}$ the corresponding  equivariant lifts. We assume
that $\widetilde{D}$ is $L^2$-invertible.
For endomorphisms of $E=E_+ \oplus E_-$ we have a natural fiberwise supertrace ${\rm Str}$.
Consider now the rescaled superconnection 
$$t \widetilde{D} + \nabla^{{\rm Lott}}\colon \mathcal{E}^{\MF}_{\mathcal A\Gamma}({\tM}, {\tE})\to \mathcal{E}^{\MF}_{\mathcal A\Gamma}({\tM}, {\tE})\oplus \mathcal{E}^{\MF}_{\Omega_1(\mathcal A\Gamma)}({\tM}, {\tE})$$ 
Following \cite{Lott2},  \cite{LeichtnamPiazzaBSMF} and, more generally, \cite{Wahl1}
we can consider the superconnection heat-kernel 
$$ \exp (- (t \widetilde{D} + \nabla^{{\rm Lott}})^2)$$
an element in $\Pdo^{-\infty}_{\widehat{\Omega}_\bullet(\mathcal{A}\Gamma)}({\tM},{\tE})$ and 
$$ \widetilde{D}  \exp (- (t\widetilde{D} + \nabla^{{\rm Lott}})^2)$$
also an element in $\Pdo^{-\infty}_{\widehat{\Omega}_\bullet(\mathcal{A}\Gamma)}({\tM},{\tE})$.
Proceeding as in Section \ref{section6.1} we have a well-defined trace functional
$${\rm STR} : \Pdo^{-\infty}_{\widehat{\Omega}_\bullet(\mathcal{A}\Gamma)}({\tM},{\tE})\to 
\widehat{\Omega}_{*}(\mathcal{A}\Gamma)_{ab}
$$
and the higher eta invariant of Lott is defined as
\begin{equation}
	\eta_{{\rm Lott}} (\widetilde{D}_M) := \frac{2}{\sqrt{\pi}} \int_0^\infty {\rm STR} \left(
	\widetilde{D}  \exp (- (t \widetilde{D} + \nabla^{{\rm
            Lott}})^2)
	\right) dt\quad\in \widehat{\Omega}_{*}(\mathcal{A}\Gamma)_{ab}\,.
\end{equation}
where the existence of the integral is non-trivial both at $0$ and at $+\infty$. The coefficient in front comes from
our choice of rescaling Lott's superconnection with $t$ instead of $t^{\frac{1}{2}}$. We set
$$  \eta_{{\rm Lott}} (\widetilde{D}_M)(t):=  \frac{2}{\sqrt{\pi}}  {\rm STR} \left(
\widetilde{D}  \exp (- (t \widetilde{D} + \nabla^{{\rm Lott}})^2)
\right)$$
so that 
$$ \eta_{{\rm Lott}} (\widetilde{D}_M) =\int_0^\infty  \eta_{{\rm Lott}} (\widetilde{D}_M)(t) dt\,\,\;\in\;\;\,
\widehat{\Omega}_{*}(\mathcal{A}\Gamma)_{ab}$$
One can show easily that $\eta_{{\rm Lott}} (\widetilde{D}_M)\in
\widehat{\Omega}_{{\rm odd}}(\mathcal{A}\Gamma)_{ab}$ if $\widetilde{M}$ is
even-dimensional. See \cite[Ch. 9]{BGV} for the case of families, the argument is exactly the same here.

If $(M,g)$ is odd dimensional and $E$ is a bundle of  hermitian Clifford module endowed with a Clifford unitary connection $\nabla^E$ then we proceed as follows.
Denote by $\CC{\rm l} (1)$
the complexified  Clifford algebra associated to $\RR$ with the standard metric; it is a unital
complex algebra with  generators $1$ and $\sigma$ and the relation $\sigma^2=1$. We consider $E_\sigma:=
E\otimes \CC{\rm l} (1)$. We can identify $E_\sigma$ with $E\oplus E$ and under this identification $\sigma$
defines a bundle endomorphism that acts as $$\begin{pmatrix} 0&1\\1&0 \end{pmatrix}\,.$$
For endomorphisms of $E_\sigma$ as a Clifford module we have a natural fiberwise supertrace:
$${\rm Str}_\sigma (A+\sigma A^\prime):= \Tr  (A^\prime)\,.$$
Consider now the rescaled superconnection 
$$t\sigma \widetilde{D} + \nabla^{{\rm Lott}}\colon \mathcal{E}^{\MF}_{\CC\Gamma}({\tM}, {\tE}_\sigma)\to \mathcal{E}^{\MF}_{\CC\Gamma}({\tM}, {\tE}_\sigma)\oplus \mathcal{E}^{\MF}_{\Omega_1(\CC\Gamma)}({\tM}, {\tE}_\sigma)$$ 
where $\nabla^{{\rm Lott}}$ is a shorthand for  $\nabla^{{\rm Lott}}\otimes {\rm Id}_{\CC^2}\,.$
We then have $$ \exp (- (t\sigma \widetilde{D} + \nabla^{{\rm
    Lott}})^2)\quad\text{and} 
\quad \sigma \widetilde{D}  \exp (- (t\sigma \widetilde{D} + \nabla^{{\rm
    Lott}})^2),$$ 
both elements in $\Pdo^{-\infty}_{\widehat{\Omega}_\bullet(\mathcal{A}\Gamma)}({\tM},{\tE}_\sigma)$.
We have a well-defined trace functional
$${\rm STR}_\sigma : \Pdo^{-\infty}_{\widehat{\Omega}_\bullet^+(\mathcal{A}\Gamma)}({\tM},{\tE}_\sigma)\to 
\widehat{\Omega}_\bullet(\mathcal{A}\Gamma)_{ab}
$$
and the higher eta invariant of Lott is defined in this odd dimensional case as 
\begin{equation}
  \begin{split}
    \eta_{{\rm Lott}} (\widetilde{D}_M) &:= \frac{2}{\sqrt{\pi}} \int_0^\infty {\rm STR}_\sigma \left(
    \sigma \widetilde{D}  \exp (- (t\sigma \widetilde{D} + \nabla^{{\rm Lott}})^2)
                                          \right) dt\\
    &\equiv \int_0^\infty  \eta_{{\rm Lott}} (\widetilde{D})(t) dt \quad\in \widehat{\Omega}_*(\mathcal{A}\Gamma)_{ab}
  \end{split}
\end{equation}
One can show that $\eta_{{\rm Lott}} (\widetilde{D}_M)\in
\widehat{\Omega}_{{\rm even}}(\mathcal{A}\Gamma)_{ab}$ in the odd-dimensional
case. See \cite{BC-adiabatic} for the case of families; once again, the argument is the same here.

\begin{proposition}
	For the invertible  operator  $\widetilde{D}$ the following 
	equality holds:
	$$d  \eta_{{\rm Lott}} (\widetilde{D}_M)=- \int_M {\rm AS} (M)\wedge \omega_{{\rm Lott}} (M) \quad\text{in}\quad  \widehat{\Omega}_*(\mathcal{A}\Gamma)_{ab}\,.$$
	Here $ {\rm AS} (M)$ is the local Atiyah-Singer integrand and  $\omega_{{\rm Lott}} (M)\in \Omega^* (M)\otimes \Omega_* (\CC\Gamma)$ denotes again Lott's bi-form.
\end{proposition}
\begin{proof}
	Assume, for example, that $M$ is even dimensional.
	Then,
	we have the following  variational formula, see \cite{BGV}, \cite{Lott2}:
        \begin{multline*}
          {\rm STR} \left(  \exp (- (T \widetilde{D} + \nabla^{{\rm
                Lott}})^2)
          \right) - {\rm STR} \left( \exp (- (t \widetilde{D} + \nabla^{{\rm Lott}})^2)
          \right)\\ = d \int_t^T \eta_{{\rm Lott}} (\widetilde{D}_M)(t) dt\,.
        \end{multline*}
	Using the invertibility of the operator one can prove, see \cite{Wahl1} for the most general assumptions on $\mathcal A\Gamma$, 
	that $$\lim_{T\to +\infty} {\rm STR} \left(  \exp (- (T \widetilde{D} + \nabla^{{\rm Lott}})^2)
	\right)=0$$
	whereas Lott's theorem on the short time behaviour of the supertrace of Lott's superconnection heat-kernel,
	see
	\cite[Proposition 12]{Lott1} (and  \cite[Theorem 13.6]{LeichtnamPiazzaMemoires}
	for the version employed here), gives us
	$$\lim_{t\to 0^+} {\rm STR} \left(  \exp (- (t \widetilde{D} + \nabla^{\rm Lott})^2)
	\right)=\int_M {\rm AS} (M)\wedge \omega_{{\rm Lott}} (M)\,.$$
	This implies the result in the even dimensional case. The odd dimensional case is similar.
\end{proof}

Following \cite{Wahl2}, we define 
Lott's higher rho invariant $\{\eta_{{\rm Lott}} (\widetilde{D})\}$
as the image of $\eta_{{\rm Lott}} (\widetilde{D})$
in the quotient 
$ 	\widehat{\Omega}_*^{del}(\mathcal{A}\Gamma)_{ab}=\widehat{\Omega}_*(\mathcal{A}\Gamma)_{ab}/\widehat{\Omega}^{e}  _{*}(\mathcal{A}\Gamma)_{ab}$.
Since $\omega_{{\rm Lott}} (M)$ is localized at the identity element by \cite{Lott1}, 
$$d \{\eta_{{\rm Lott}} (\widetilde{D})\} =0 \quad\text{in}\quad  	\widehat{\Omega}^{del}_*(\mathcal{A}\Gamma)_{ab}$$
and we obtain 
$$\varrho_{{\rm Lott}}(\widetilde{D}):=[\{ \eta_{{\rm Lott}}
(\widetilde{D})\} ]  \in H_{*}^{del} (\mathcal{A}\Gamma)$$ 
with $*$ equal to even/odd if the dimension of $M$ is odd/even.
From now on we shall omit  the curly brackets from the notation.

\subsection*{Comparing $\varrho_{{\rm Lott}}(\widetilde{D})$
  and  $\Ch_\Gamma^{del}([\varrho (\widetilde{D})])$}

We want to compare $\varrho_{{\rm Lott}}(\widetilde{D})$
with our class $\Ch_\Gamma^{del}([\varrho (\widetilde{D})])$, which are both elements 
of the group $H^{del}_{*} (\mathcal{A}\Gamma)$. 
 We carry out this program only in the bounding case, where we can use results from higher
APS index theory, and leave the general case to future investigations.\\
We start with our cocompact $\Gamma$-covering 
$\Gamma\to \widetilde{M}\to M$ without boundary, which is endowed with a $\Gamma$-equivariant metric and with a $\Gamma$-equivariant Clifford module $E$.  
We denote as usual  by $\widetilde{D}_M$ the associated Dirac-type operator.
We assume  
that there is a cocompact  $\Gamma$-covering with boundary $\Gamma\to
\widetilde{W}\to W$ 
such that $\partial \widetilde{W}=\widetilde{M}$.
We assume the existence of a  $\Gamma$-equivariant 
metric and of a $\Gamma$-equivariant Clifford module $F$ over $\widetilde{W}$ and
denote by  $\widetilde{D}_W$ the associated Dirac-type operator. If $W$ is even-dimensional
then
$F=F^+\oplus F^-$ and $\widetilde{D}_W$ is $\ZZ_2$-graded and odd:
$$\begin{pmatrix} 0&\widetilde{D}^-_W\\\widetilde{D}^+_W&0 \end{pmatrix}\,.$$
We assume
that all these structures on $\widetilde{W}$ are of product-type near the boundary 
and restrict to the ones of $\widetilde{M}$ on the boundary.
Here we follow the conventions explained in detail in \cite[Section 3]{LP-JFA}. Consequently,
the boundary operator  $\widetilde{D}^\partial_W$  of $\widetilde{D}_W$ is  equal to  $\widetilde{D}_M$; more precisely $\widetilde{D}^{+,\partial}_{W}=\widetilde{D}_M$
if $W$ is even dimensional and $\widetilde{D}^{\partial}_{W}=\widetilde{D}_M$
if $W$ is  odd-dimensional case and with the $\ZZ_2$-grading of $\widetilde{D}_M$
given in terms of Clifford multiplication by the normal vector to the
boundary, see \cite[Section 3.2]{LP-JFA}.

Our strategy is now the following: using the delocalized Chern character, 
we wish to connect  the delocalized APS index theorem in K-theory, proved in \cite{PiazzaSchick_psc},
with the delocalized APS index theorem  in noncommutative de Rham homology. The latter is obtained from the following higher index formula, proved in 
\cite[Theorem 4.1]{LeichtnamPiazzaBSMF}, \cite[(3.1)]{LP-PSC},
\cite[Theorem 9.4, Theorem 11.1]{Wahl1}  and  stated
below for the benefit of the reader:

\begin{theorem}\label{aps-index-lott}
	Under the above assumptions, there exists a well defined APS index class
	$\Ind_{\Gamma,b} (\widetilde{D}_{W})\in  K_* (\mathcal{A}\Gamma)$ such that
	\begin{equation}\label{aps-lp}
		\Ch_\Gamma (\Ind_{\Gamma,b} (\widetilde{D}_{W}))=\left[ \int_W {\rm AS} (W)\wedge \omega_{{\rm Lott}} (W) - \frac{1}{2}\eta_{{\rm Lott}} (\widetilde{D}_M)\right]\quad\in 
		H_{*}(\mathcal{A}\Gamma).
	\end{equation}
	In this formula, $\omega_{\rm Lott}(W)$ is again the explicit bi-form due to Lott, \cite{Lott1} \cite[Theorem 13.6]{LeichtnamPiazzaMemoires}; we recall that $\omega_{\rm Lott}(W)$
		is concentrated at the identity element of the group. \end{theorem}

\begin{remark}
The index class $\Ind_{\Gamma,b} (\widetilde{D}_{W})\in K_* (\mathcal{A}\Gamma)$
is associated to the Dirac operator defined by $\widetilde{D}_{W}$ on the $\Gamma$-covering
with cylindrical ends associated to $\widetilde{W}$. In the framework on Melrose, this is a $b$-manifold, hence the
subscript. There is also an APS-index class $\Ind_{\Gamma,{\rm APS}} (\widetilde{D}_{W})\in K_* (\mathcal{A}\Gamma)$
defined via an APS-boundary value problem and one can prove \cite{LP-JFA}, as in the numeric case treated by Atiyah, Patodi and
Singer, that $\Ind_{\Gamma,b} (\widetilde{D}_{W})=\Ind_{\Gamma,{\rm APS}} (\widetilde{D}_{W})$ in 
$K_* (\mathcal{A}\Gamma)$. With a small and well-established abuse of terminology 
we refer to this index class as a higher APS-index class.
\end{remark}

By applying $q_*\colon H_{*} (\mathcal{A}\Gamma)\to H_{*}^{del}
(\mathcal{A}\Gamma)$ to both sides and using that $ \omega_{{\rm Lott}} (W)$ is concentrated at the identity element, we obtain the

\begin{proposition}[Homological delocalized APS index theorem using superconnections]\label{del-APS-lott}
The following equality holds:
\begin{equation}\label{del-APS-hom}
	q_*(\Ch_\Gamma (\Ind_{\Gamma,b} (\widetilde{D}_{W})))= - \frac{1}{2}\varrho_{{\rm Lott}} (\widetilde{D}_M)\quad\in H_{*}^{del} (\mathcal{A}\Gamma)\,.
\end{equation}
\end{proposition}

Let us go back to our task of comparing  
$\varrho_{{\rm Lott}} ( \widetilde{D}_{M})$ with $\Ch_\Gamma^{del}(\varrho( \widetilde{D}_{M}))$.
If we could prove that
\begin{equation}\label{ch=q*-general}  q_* (\Ch_\Gamma (\Ind_{\Gamma,b} (\widetilde{D}_{W})))=\Ch_\Gamma^{del}(\varrho( \widetilde{D}_{M}))\quad\in H_{*}^{del} (\mathcal{A}\Gamma)
	\end{equation} 
then, clearly, we would obtain the equality
$$\Ch_\Gamma^{del}([\varrho (\widetilde{D}_M)])= -\frac{1}{2}\varrho_{{\rm Lott}}(\widetilde{D}_M)\quad\in H_{*}^{del} (\mathcal{A}\Gamma).$$
We shall explain  this result in the even dimensional:

\begin{proposition}[Homological delocalized APS index theorem using the delocalized Chern character]\label{APS-deloc-ch}
	Assume that $\widetilde{W}$ is even dimensional. Then
	the following formula holds
	\begin{equation}\label{ch=q*}  q_* (\Ch_\Gamma (\Ind_{\Gamma,b} (\widetilde{D}_{W})))=\Ch_\Gamma^{del}(\varrho( \widetilde{D}_{M}))\quad\in H_{{\rm even}}^{del} (\mathcal{A}\Gamma).
	\end{equation} 
\end{proposition}

We postpone the proof of this proposition to the next section.

\smallskip
Summarizing,  from \eqref{del-APS-hom} and \eqref{ch=q*} we 
have obtained  the following result:

\begin{proposition}
	If  $\widetilde{M}$ is odd dimensional, $\widetilde{M}=\partial \widetilde{W}$ and  $\widetilde{D}_M=
	\widetilde{D}^{+,\partial}_W$  as above, then
	$$\Ch_\Gamma^{del}([\varrho (\widetilde{D}_M)])= -\frac{1}{2}\varrho_{{\rm Lott}}(\widetilde{D}_M)\quad\in H_{{\rm even}}^{del} (\mathcal{A}\Gamma).$$
\end{proposition}

\noindent
At least in the bounding case this answers a question raised by Lott in \cite[Remark 4.11.3]{Lott2}. We leave the general non-bounding case
to future investigations; we expect the techniques in  \cite[Section 5]{Goro-Lott} to play a crucial role.

\section{Proof of Proposition \ref{APS-deloc-ch}}
We would like to obtain \eqref{ch=q*}
by employing the delocalized APS index theorem in K-Theory proved in \cite{PiazzaSchick_psc}.
However, since we are using a different description of the analytic surgery sequence, we shall first
need to reformulate  the delocalized APS index theorem in K-theory in our new setting. To this end, we recall and
slightly modify results
due to the third author, see \cite{Zenobi}. We shall freely use Lie groupoids techniques; we refer the reader, for example, to \cite[Section 4, Section 8.1, 8.2]{Piazza-Zenobi} for a quick introduction to the  relevant definitions and basics results.

We start by observing that 
in the above geometric setting we have the following $b$-groupoid,
\[
\mathring{\widetilde{W}}\times_\Gamma\mathring{\widetilde{W}}\cup \widetilde{M}\times_\Gamma\widetilde{M}\times \RR\rightrightarrows W.
\]
The smooth structure of this Lie groupoid is rigorously given by the fact that it is  the blow-up of $\widetilde{M}\times_\Gamma\widetilde{M}$ into $\widetilde{W}\times_\Gamma\widetilde{W}$, see for instance \cite{DebordSkandalis}. 
In order to lighten the notation we will denote by $G(M)$ the groupoid
$\widetilde{M}\times_\Gamma\widetilde{M}$, by $G(W,M)$ the $b$-groupoid,
and by $G(\mathring{W})$ the 
groupoid $\mathring{\widetilde{W}}\times_\Gamma\mathring{\widetilde{W}}$.

We have the following commutative diagram of K-theory groups
\begin{equation}\label{diagram-lott}{\small
\xymatrix{  K_0(C^*_{red}(G(\mathring{W})))\ar[d]^{j_*}\ar[r]^(.4)S_(.4){\iso} &K_1(C^*_{red}(G(\mathring{W}))\otimes C_0(0,1))\ar[d]^{j_*^c} \ar[rd]^{i_*}&\\
	 K_0(\Pdo^0_{\Gamma}(\mathring{\widetilde{W}}))\ar[r]&K_1(C_0(\mathring{W})\to\Pdo^0_{\Gamma}(\mathring{\widetilde{W}})) \ar[r]^\cong&K_1(C^*_{red}(G(\mathring{W})_{ad}^{[0,1)}))\\
	K_1(\Pdo^0_{\Gamma,\RR}(\widetilde{M}\times
        \RR))\ar[u]^{\partial^\Pdo}\ar[r]&
        K_0(C(M)\to\Pdo^0_{\Gamma,\RR}(\widetilde{M}\times
        \RR))\ar[u]^{\partial^c}\ar[r]^\cong& K_0(C^*_{red}((G(M)\times
        \RR)_{ad}^{[0,1)}))\ar[u]^{\partial^{ad}}\\
      K_0(\Pdo^0_\Gamma(\widetilde{M}))\ar[u]^(.4){S^\Pdo} } }
\end{equation}
In this diagram:
\begin{itemize}
	\item $S$ denotes the suspension isomorphism;
	\item  $i$ is the natural inclusion of $C^*_{red}(G(\mathring{W}))\otimes C_0(0,1)$
	into $C^*_{red}(G(\mathring{W})_{ad}^{[0,1)})$; 
	\item $j$ and $j^c$ are the natural inclusions of $C^*_{red}(G(\mathring{W}))$ into $\Pdo^0_{\Gamma}(\mathring{\widetilde{W}})$ and of $C^*_{red}(G(\mathring{W})\otimes C_0(0,1))$ into the mapping cone of $C_0(\mathring{W})\to\Pdo^0_{\Gamma}(\mathring{\widetilde{W}})$ respectively,
the horizontal isomorphisms on the right hand side are given by equation (2.11) in \cite{Zenobi_compare}.
\end{itemize}
Next we define the homomorphisms $\partial^{\Pdo}$, $\partial^c$, $\partial^{ad}$ and $S^\Pdo$, thus completing the
description of the above diagram.

For the  homomorphisms $\partial^{\Pdo}$,  we consider $\Pdo^0_{\Gamma, b}(\widetilde{W})$, the C*-closure of the 0-order $\Gamma$-equivariant $b$-pseudodifferential operators on  $\widetilde{W}$ of $\Gamma$-compact support, which is the same as the C*-closure of the compactly supported 0-order pseudodifferential operators on $G(W,M)$. Then  the restriction to the boundary gives rise to the following  short exact sequences 
\begin{equation}\label{eq:SES}
\xymatrix{0\ar[r]& \Pdo^0_{\Gamma}(\mathring{\widetilde{W}})\ar[r]& \Pdo^0_{\Gamma, b}(\widetilde{W})\ar[r]& \Pdo^0_{\Gamma,\RR}(\widetilde{M}\times \RR)\ar[r]&0,}
\end{equation}
where $\Pdo^0_{\Gamma,\RR}(\widetilde{M}\times \RR)$ are the suspended operators on $\widetilde{M}$, namely operators on $\widetilde{M}\times\RR$ which are translation invariant on $\RR$. Then $\partial^{\Pdo}$ is the boundary morphism associated to \eqref{eq:SES}. Analogously, one defines $\partial^c$ by means of the mapping cone construction.

For the homomorphism $\partial^{ad}$, we consider the boundary map associated to the following short exact sequence
\begin{equation}\label{eq:SESbis}
{ \xymatrix{0\to
   C^*_{red}(G(\mathring{W})_{ad}^{[0,1)})\ar[r]&C^*_{red}(G(W,M)_{ad}^{[0,1)})\ar[r]&
   C^*_{red}((G(M)\times \RR)_{ad}^{[0,1)})\to 0} }
\end{equation}
where we recall that the adiabatic deformation of a Lie groupoid $G\rightrightarrows
G^{(0)}$ is given by $\mathfrak{A}G\times\{0\}\sqcup G\times
(0,1)\rightrightarrows G^{(0)}\times[0,1)$, with $\mathfrak{A}G$ being
the Lie algebroid of $G$,  see \cite{DebordSkandalis,Piazza-Zenobi} where the details  of these constructions are given. Recall that the K-theory of $C^*_{red}(G(M)_{ad}^{[0,1)})$ gives another realization of the structure group $\SG^\Gamma_*(\widetilde M)$, see \cite{Zenobi_compare}.

Finally, the map $S^\Pdo$ in \eqref{diagram-lott} is given by the composition of the suspension isomorphism $S$ followed by the homomorphism $\iota_*$ induced by the inclusion $\iota\colon \Pdo^0_\Gamma(\tM)\otimes C_0(\RR)\to\Pdo^0_{\Gamma,\RR}(\widetilde{M}\times \RR)$ given by Fourier transform in the $\RR$ direction. 

\smallskip
Let us now consider the $\Gamma$-equivariant Dirac $b$-operator
$\widetilde{D}_{W}$ on $\widetilde{W}$. Since $\widetilde D_M$, the operator
on the boundary, is $L^2$-invertible we know that $\widetilde{D}_{W}$ is fully
elliptic and there is a well-defined $b$-index class $\mathrm{Ind}_{\Gamma,
	b}(\widetilde{D}_{W})\in K_0(C^*_{red}(G(\mathring{W})))$. This index class precisely corresponds to the
one appearing in \eqref{aps-lp} through the isomorphisms $K_0(C^*_{red}(G(\mathring{W})))\cong K_0 (C^*_{red} \Gamma)\cong K_0 (\mathcal{A}\Gamma)$.

In this context, the third author of this paper \cite[Equation (3.3)]{Zenobi} has defined the adiabatic rho class $\varrho^{ad}(\widetilde{D}_M)$ as an element in $K_0(C^*_{red}(G(M)_{ad}^{[0,1)}))$.
The delocalized APS index theorem (in K-theory) in the groupoid
framework, \cite[Theorem 3.7]{Zenobi}, states that 
\begin{equation}\label{del-groupoid}
i_*(S(\mathrm{Ind}_{\Gamma, b}(\widetilde{D}_{W})))=\partial^{ad}(\varrho^{ad}(\widetilde{D}_M)).
\end{equation}

We want to use this fundamental equality in order to prove the  identity \eqref{ch=q*}.
To this end we consider the $b$-groupoid $G([0,1), \{0\})$. We have the
following commutative diagram of short exact sequences: 
\begin{equation}{\small
\xymatrix{\Pdo^0_{\Gamma}(\widetilde{M})\otimes C^*_{red}((0,1)\times
  (0,1))\ar@{^{(}->}[r]\ar[d]^{\iota'}
  &\Pdo^0_{\Gamma}(\widetilde{M})\otimes
  C^*_{red}G([0,1), \{0\})\ar[d]^{\iota''}\ar@{->>}[r]&
  \Pdo^0_{\Gamma}(\widetilde{M})\otimes C_0(\RR)\ar[d]^\iota&\\
  \Pdo^0_{\Gamma}(\mathring{\widetilde{W}})\ar@{^{(}->}[r]& \Pdo^0_{\Gamma,
          b}(\widetilde{W})\ar@{->>}[r]& \Pdo^0_{\Gamma,\RR}(\widetilde{M}\times
        \RR)} } 
\end{equation}
with $\iota$ from above and where  $\iota''$ is defined by the fact that
over a collar neighbourhood of the boundary $G(W, M)$ is isomorphic to
$G(M)\times G([0,1), \{0\})$. Finally, $\iota'$ is the restriction of $\iota''$.  
We denote by $\mathrm{id}\otimes \partial$ the K-theory boundary map of the first row. Notice that $ \partial$ is the isomorphism which sends the generator of $K_1(C_0(\RR))$ to the generator of $K_0(C^*_{red}((0,1)\times (0,1)))$, namely the class of a rank one projector $e\in \KK(L^2(0,1))$.
Now, consider $x\in K_0(\Pdo^0_\Gamma(\widetilde{M}))$, then by naturality and by the above remarks
\begin{equation}\label{partial-psi}
\partial^\Pdo(S^\Pdo(x))=\partial^\Pdo(\iota_*(Sx))=\iota'_*(\mathrm{id}\otimes\partial(Sx))= \iota'_*(x\otimes e).
\end{equation}

Consider now the class $[\pi_\geq (\widetilde{D}_M)]=: [\pi_\geq]\in 	K_0(\Pdo^0_\Gamma(\widetilde{M}))$ given by the projection on the positive spectrum of the Dirac operator $\widetilde{D}_M$ on $\widetilde{M}$. We know from \cite[Section 5.3]{Zenobi_compare}  that $[\pi_\geq]$ is sent to  the adiabatic rho class
$\varrho^{ad}(\widetilde{D}_M)\in K_0(C^*_{red}(G(M)_{ad}^{[0,1)}))$ by the composition 
of the homomorphisms appearing in the bottom row in \eqref{diagram-lott}.
By  using the delocalized APS index theorem  in the groupoid
framework, i.e. the equality $i_*(S(\mathrm{Ind}_{\Gamma, b}(\widetilde{D}_{W})))=\partial^{ad}(\varrho^{ad}(\widetilde{D}_M))$, and by a simple diagram chase in \eqref{diagram-lott}, we see that  the difference 
$$j_*(\mathrm{Ind}_{\Gamma, b}(\widetilde{D}_{W}))- \partial
^\Pdo(S^\Pdo([\pi_\geq]))$$ is in the image of $K_0(C_0(\mathring{W}))$ and
is therefore local in the sense that its Chern character lies in the image
of $H^e(\mathcal{A}\Gamma)$.\\
We  then have the following sequence of equalities
\begin{equation}\label{k-aps-del}
\begin{split}
q_* (\Ch_\Gamma(\mathrm{Ind}_{\Gamma, b}(\widetilde{D}_{W})))&= \Ch_\Gamma^{del}(j_*(\mathrm{Ind}_{\Gamma, b}(\widetilde{D}_{W})))\\
&=\Ch_\Gamma^{del}(\partial ^\Pdo(S^\Pdo([\pi_\geq]))) \\
&= \Ch_\Gamma^{del}(\iota^\prime ([\pi_\geq]\otimes e))\\
&= \Ch_\Gamma^{del}([\pi_\geq]\otimes e)\\
&= \Ch_\Gamma^{del}([\pi_\geq])
\end{split}
\end{equation}
where the first equality is given by our main Theorem \ref{main-commutative-diagram};  the second equality employs  the fact that $\Ch_\Gamma^{del}$ is zero on local terms; the third equality is given by \eqref{partial-psi}; the fourth equality
is a direct consequence of the definition of $\iota^\prime$ and finally the last equality is given  by the fact that the trace of the tensor product is the product of the traces and that the trace of $e$ is $1$.
Summarizing:
\begin{equation}\label{k-aps-del-bis}
q_* (\Ch_\Gamma(\mathrm{Ind}_{\Gamma, b}(\widetilde{D}_{W})))=\Ch_\Gamma^{del}([\pi_\geq])
\end{equation}
which is the equality we wanted to show.

\chapter{Higher rho numbers 
  associated to elements in $HC^* (\CC\Gamma,\langle x \rangle)$}
\label{sec:rho_numbers}

  \section{The role of cyclic group cohomology}
\label{sec:intro_cyclic}

  The main result of Section \ref{section6}, and indeed the key achievement of
  this paper, is a compatible Chern character map from the analytic K-theory
  sequence to non-commutative de Rham homology sequences. These de Rham homology sequences are
  obtained for suitable completions $\mathcal{A}\Gamma$ of the group ring
  $\CC\Gamma$ of the underlying group. To make use of this Chern character,
  the next step is then to obtain pairings between this homology and suitable
  cohomology groups. Here, one will typically rely on (periodic) cyclic
  cohomology.

We shall  use the basic fact that the (universal) (reduced) de Rham homology of a Fr\`echet
algebra $A$ is indeed canonically a summand 
of the cyclic homology:
  $\overline{H}_*(A)\into \overline{HC}_*(A)$, \cite[Theorem 2.6.7]{Loday} or
  \cite[Th\'eor\`eme 1.15]{Karoubi}. We will discuss this relation and the
  meaning of ``reduced'' in more
  detail in Section \ref{section7.5}. The cyclic homology can be defined in several
  canonically isomorphic ways, e.g.~as the homology of the cyclic chain
  complex $CC_\bullet(A)$. We will give more details about the cyclic (co)chain
  complexes below.

  By definition, cyclic cohomology pairs with cyclic homology and therefore
  also with non-commutative de Rham homology. In fact,  the
  cyclic cochain complex $CC^\bullet(A)$ is the dual of 
  the cyclic chain complex $CC_\bullet(A)$. Correspondingly, we have the
  pairing $ CC_n(A)\otimes CC^n(A)\to \CC$
  descending to a pairing  $HC_n(A)\otimes HC^n(A)\to \CC$ on the level of
  cyclic (co)homology.
  Of course, when working with Fr\`echet algebras, throughout the topology has
  to be taken into account (topological duals, projective tensor
  products,\ldots).

  For a completion $\mathcal{A}\Gamma$ that one has to use for the Chern character
  in Section \ref{section6}, it will in general be impossible to get direct
  information about $HC^*(\mathcal{A}\Gamma)$. Instead, one would like to use
  the pretty well understood $HC^*(\CC\Gamma)$, i.e.~one would like to use
 (algebraic) cyclic cocycles of the group $\Gamma$. In general, this will not
 be possible. The inclusion $\CC\Gamma\into \mathcal{A}\Gamma$ induces a
 restriction map $CC^\bullet(\mathcal{A}\Gamma)\to CC^\bullet(\CC\Gamma)$. Due to the fact
 that $\CC\Gamma$ is dense in $\mathcal{A}\Gamma$  and that the cochains in
 $CC^\bullet(\mathcal{A}\Gamma)$ are required to be continuous, the restriction map
 is an injection $CC^\bullet(\mathcal{A}\Gamma)\into CC^\bullet(\CC\Gamma)$. A common
 theme is the ``extension'' of cocycles from $\CC\Gamma$ to
 $\mathcal{A}\Gamma$. If this is possible for all of $CC^\bullet(\CC\Gamma)$ it
 simply means that we have an isomorphism
 $CC^\bullet(\mathcal{A}\Gamma)\xrightarrow{\iso} CC^\bullet(\CC\Gamma)$.  If this is
 possible for a subcomplex $E^*\subset CC^\bullet(\CC\Gamma)$, namely we get an isomorphic
 subcomplex $\overline{E}^*\subset CC^\bullet(\mathcal{A}\Gamma)$ of extended
 cochains, then there is a well defined map $H^*(E)\xrightarrow{\iso}
 H^*(\overline{E}^*)\to HC^*(\mathcal{A}\Gamma)$. This fact holds for the reduced (co)homology groups, too. In our situation, $E^*$ will
 be a direct summand of $CC^\bullet(\CC\Gamma)$ and consequently $\overline{H}^*(E^*)$ will be
 a direct summand of $\overline{HC}^*(\CC\Gamma)$ mapping to $\overline{HC}^*(\mathcal{A}\Gamma)$
 and thus pairing with non-commutative de Rham homology. As a
 consequence, one can just
 use the cyclic cohomology of $\CC\Gamma$ (or rather the summand
 $\overline{H}^*(E^*)$) to obtain pairings with $\overline{H}_*(\mathcal{A}\Gamma)$
 in place of the a priori mysterious $\overline{HC}^*(\mathcal{A}\Gamma)$. This is what
 we will achieve for hyperbolic groups (and a convenient completion
 $\mathcal{A}\Gamma$ due to Puschnigg) in  Section \ref{sec:Puschnigg}.
 More directly, this will also be achieved for groups of polynomial
  growth.

\section{Cyclic cohomology of hyperbolic groups}
\label{sec:HC_of_group}

A central role in this article will be played by the following well-known theorem.

\begin{theorem}[Burghelea \cite{Burghelea}]\label{burghelea}
 For any discrete group $\Gamma$ we have the following isomorphisms for the Hochschild and the cyclic cohomology groups, respectively:
 \begin{enumerate}
 	\item $HH_*(\CC\Gamma)\cong \bigoplus_{\langle\gamma\rangle\in \langle\Gamma\rangle}H_*(\Gamma_\gamma;\CC)$,
 	\item $HC_*(\CC\Gamma)\cong \bigoplus_{\langle\gamma\rangle\in \langle\Gamma\rangle^{\infty}}H_*(\Gamma_\gamma/\gamma^\integers;\CC)\oplus \bigoplus_{\langle\gamma\rangle\in \langle\Gamma\rangle^{fin}}H_*(\Gamma_\gamma/\gamma^\integers;\CC)\otimes \CC[z].$
 \end{enumerate}
 Here,
 $\langle\Gamma\rangle$ is the set of all conjugacy classes in $\Gamma$, and
 $\langle \gamma\rangle$ denotes the conjugacy class of $\gamma$;
 $\langle\Gamma\rangle^{\infty}$ (respectively $\langle\Gamma\rangle^{fin}$) denotes the set of conjugacy classes of elements of infinite (respectively finite) order. Moreover $\Gamma_\gamma:=\{g\in\Gamma\mid g\gamma=\gamma g\}$ is the centralizer of
 $\gamma$ in $\Gamma$  and $\gamma^\integers$ denotes the cyclic group generated by $\gamma\in \Gamma$.
 Finally,   $H_*(\Gamma_\gamma;\CC)$ is the usual group cohomology with complex
 coefficients, and the polynomial generator $z$ is of degree $2$.
\end{theorem}

There is an analogous version of this theorem for the Hochschild and the cyclic cohomology groups, where direct sums are replaced by direct products. In what follows we shall see aspects of the proof in order to show that, under suitable assumptions on $\Gamma$, we can bring in polynomial growth conditions. 

\medskip

First, let us recall the definition of the cyclic set
  $Z\Gamma$. For more background on what follows see for instance \cite[Section 9.7]{Weibel}.  Here $Z_n\Gamma:=\Gamma^{n+1}$ and the degeneracies, the face maps, and the cyclic structures are defined in the following way: 
   \begin{equation}\label{cycic-structure}
   \begin{split}
   \partial_i(g_0,\dots, g_n)&:= \left\{\begin{array}{@{}lr@{}}
  (g_0,\dots,g_i g_{i+1},\dots, g_n), & \text{for }i<n\\
   (g_n g_0,g_1,\dots, g_{n-1}), & \text{for }i=n
   \end{array} \right. \\
   \sigma_i(g_0,\dots, g_n)&:= (g_0,\dots, g_i,e,g_{i+1},\dots, g_n),\\
   t(g_0,\dots, g_n)&:= (g_n,g_0,\dots, g_{n-1}).
   \end{split}
   \end{equation}
   Now, for $x\in \Gamma$, let $Z_n(\Gamma,x)$  denote the subset of $Z_n\Gamma$ consisting of all $(g_0,\dots, g_n)$ such that the product $g_0\cdots g_n$ is conjugate to $x$. As $n$ varies, these subsets form a cyclic subset $Z(\Gamma,x)$ of $Z\Gamma$.
   
   Hochschild cohomology $HH^*(\CC\Gamma)$ is given by the cohomology of
   the cochain complex $C^\bullet(Z\Gamma)$ associated to the simplicial  object
   underlying  $Z\Gamma$.   In particular, we have that 
   \begin{equation}\label{hochschild}HH^*(\CC\Gamma)=\prod_{\langle x\rangle\in \langle\Gamma\rangle}HH^*(\CC\Gamma;\langle x\rangle)\end{equation}
 which corresponds by definition to the decomposition
 \begin{equation*}
   H^*(Z\Gamma)=\prod_{\langle
     x\rangle\in \langle\Gamma\rangle}H^*(Z(\Gamma,x)).
 \end{equation*}
   
 \begin{definition} Let $l\colon \Gamma\to[0, +\infty)$ be the word-length  function on $\Gamma$ associated to a symmetric set of generators $S$.
   Define $C^\bullet_{pol}(\CC\Gamma; \langle x\rangle )$ as the subcomplex of the
   Hochschild complex $C^\bullet(\CC\Gamma; \langle x\rangle )$ associated to
   $Z(\Gamma, x)$ whose elements are cochains of polynomial growth,
   i.e.~functions $f\colon Z_n(\Gamma,x)\to\complexs$ with
   $|f(g_0,\dots,g_n)|\le C (1+l(g_0)\cdots l(g_n))^N$ for suitable $C,N$. Denote
   by $HH^*_{pol}(\CC\Gamma;\langle x\rangle)$ its cohomology.
 \end{definition}

 By \cite[Proposition 9.7.4]{Weibel}, for all $x\in
        \Gamma$ the inclusion $\Gamma_x\hookrightarrow\Gamma$ induces a
        homotopy equivalence $\iota\colon Z(\Gamma_x, x)\to Z(\Gamma, x)$ of
        cyclic sets with homotopy inverse $\rho\colon Z(\Gamma, x)\to
        Z(\Gamma_x, x)$, explicitly defined in the following way. First,
          choose representatives of minimal word length for the right
          $\Gamma_x$-coset in 
          $\Gamma$. Given $(g_0,\dots, g_n)\in Z_n(G,x)$, let $y_i$ be the
        chosen coset representative such that
        \begin{equation*}
          y_i(g_{i+1}\dots g_n g_0\dots g_i)y_i^{-1}=x
        \end{equation*}
        and set 
	\begin{equation}\label{homotopy-cyclic}\rho(g_0,\dots, g_n):= (y_ng_0y_0^{-1},y_0g_1y_1^{-1},\dots, y_{n-1}g_ny_n^{-1}).
	\end{equation}  
	Moreover, the simplicial homotopy between the identity map of $Z(\Gamma,x)$ and $\iota\circ \rho$ is given by 
		\begin{equation}\label{chain-homotopy}h_j(g_0,\dots, g_n):= (g_0y_0^{-1},y_0g_1y_1^{-1},\dots,
		y_{j-1}g_jy_j^{-1},y_j, g_j,\dots, g_n);
                \quad j\in\{0,\dots,n\}.
		\end{equation}
\begin{remark}
	   Note that the simplicial cochain complex associated to $ Z(\Gamma_x, x)$ is exactly the standard bar complex which defines the group cohomology $H^*(\Gamma_x;\CC)$. Hence we have that 
	   the group cohomology $H^*(\Gamma_x;\CC)$ is the simplicial set cohomology $H^*(Z(\Gamma_x, x))$, which in turn is isomorphic to 
	   $HH^*(\CC\Gamma; \langle x\rangle ):= H^*(Z(\Gamma, x))$ by means of the pull-back map $\rho^*$. 
	   
\end{remark}

   \medskip

Let us recall the following fundamental definition.
\begin{definition}
A geodesic metric space $(X, d)$ is $\delta$-hyperbolic
if each edge of a geodesic triangle in $X$ is contained in the $\delta$-neighbourhood of the
union of the two other edges. Thus, if $[a,b], [b,c], [c,a]$ are geodesic segments in $X$, then
\begin{equation*}
x\in[a,b]\Rightarrow d_X(x, [b,c]\cup [c,a])\leq\delta.
\end{equation*}
It is simply called hyperbolic if it is $\delta$-hyperbolic for
some $\delta\geq0$. 
\end{definition}

 Let us assume from now on that $\Gamma$ is hyperbolic, namely, by definition, that there exists a $\delta$ such 
   that the Cayley-graph $G(\Gamma,S)$ is $\delta$-hyperbolic. A result of Gromov implies that also $\Gamma_x$ is hyperbolic for all $x\in \Gamma$, see \cite[Section 8.5.M]{Gromov-hyperbolic}.

    For every element $\gamma\in\Gamma$ we choose a
   word $w(\gamma)$ of minimal length representing it and fix an element $\sigma(\gamma)$ of minimal word length $l(\gamma)$ in the conjugacy class $\langle\gamma\rangle$. We recall the statement of \cite[Lemma 4.1]{Puschnigg}.
   \begin{lemma} \label{puschnigg-estimation}
   Let $\Gamma$ be hyperbolic.
    For given $R>0$ there exists a
     constant $C(R)>0$ such that the
   	following holds. Let $\gamma\in \Gamma$ be written as a word
          $w(\gamma)$ in the generators $S$ and
          assume that $\min_{g\in \langle\gamma\rangle}l(g)\le R$. Then some cyclic
   	permutation of $w(\gamma)$ represents an element of length less than $C(R)$ in $\Gamma$.
   \end{lemma}
   As a consequence, we have the following result.
   \begin{lemma} \label{polynomial-growth}
   	If $\Gamma$ is hyperbolic and $x\in\Gamma$, the pull-back of cochains through $\rho$ sends cochains of polynomial growth on $Z(\Gamma_x,x)$ to cochains of polynomial growth on $Z(\Gamma, x)$. The same is true for the simplicial homotopy $\{h_j\}$.
   \end{lemma}
   
   \begin{proof}
      Start with $g_0,\dots,g_n\in \Gamma$ such that $g_0\dots g_n\in\langle
      x\rangle$. Abbreviate $G_i:=g_{i+1}\dots g_ng_0\dots g_i$. The
      construction of $\rho$ and of the simplicial homotopy is based on the
      elements $y_0,\dots,y_n$ which depends on $g_0,\dots,g_n$. We show that
      their length is bounded linearly in the length of the $g_i$, which
      immediately implies the statement.

      First, observe that each $G_i$ is conjugated to $x$. Using Lemma
      \ref{puschnigg-estimation}, there is $C>0$ (depending only on $l(x)$)
      such that we find a subword $a_i$ of $G_i$ and $l(a_i G_ia_i^{-1})\le
      C$. Note that, as a subword of a product of the $g_i$, $l(a_i)\le
      \sum_{k=0}^n l(g_k)$. Secondly, the set of elements of $\langle
      x\rangle$ of length $\le C$ is finite. Therefore, there is $D>0$
      depending only on $C$ and $x$ and $b_i$ with $l(b_i)\le D$ and such that
      $b_ia_iG_ia_i^{-1}b_i^{-1} =x$. Now, also $y_iG_iy_i^{-1}=x$, therefore
      $b_ia_i$ and $y_i$ belong to the same $\Gamma_x$-coset. By minimality,
      \begin{equation*}
        l(y_i)\le l(b_ia_i)\le D+ \sum_{k=0}^n l(g_i).
      \end{equation*}
       Let now $f\colon \Gamma_x^{n+1}\to \CC$ be a cochain of polynomial
      growth, i.e.~$|f(g_0,g_1,\dots,g_n)|\leq C( 1+l(g_0)\cdots l(g_n))^N$ 
      for some $C,N$. It follows now immediately from the definition of $\rho$ that
      also $\rho^*f$ has polynomial growth, as
      \begin{equation*}
        \begin{split}
          |\rho^*f(g_0,\dots,g_n)| &= |f(y_ng_0y_0^{-1},\dots,
                                     y_{n-1}g_ny_n^{-1})|\\
          &\le
          C (1+l(y_ng_0y_0^{-1})\cdots l(y_{n-1}g_ny_n^{-1}))^N \\
          &\le C(1+ (l(y_n)+l(g_0)+l(y_0))\cdots (l(y_{n-1})+l(g_n)+
          l(y_n)))^N
        \end{split}
    \end{equation*}
which, using our preparation, is bounded by a polynomial in
$l(g_0),\dots,l(g_n)$ (even of unchanged degree). The same argument applies to
the chain homotopy and the result follows.
   \end{proof}

\begin{proposition}\label{hochschild-polynomial}
   	Let $\Gamma$ be a hyperbolic group or assume that $\Gamma$ has
          polynomial growth, $x\in\Gamma$.
   	Then the inclusion of complexes $C^\bullet_{pol}(\CC\Gamma; \langle x\rangle )\hookrightarrow C^\bullet(\CC\Gamma; \langle x\rangle )$ induces an isomorphism $HH^*_{pol}(\CC\Gamma;\langle x\rangle)\cong HH^*(\CC\Gamma;\langle x\rangle)$.
   \end{proposition}	
   \begin{proof}
Observe first that  Lemma  \ref{polynomial-growth}  holds also for groups
of polynomial growth, with an easier, immediate proof.
If $\Gamma$ is hyperbolic, the same is true for $\Gamma_x$ by \cite[Section
  8.5.M]{Gromov-hyperbolic}. If $\Gamma$ has polynomial growth, every
  subgroup, in particular $\Gamma_x$
  also has polynomial growth. If $\Gamma$ is hyperbolic it admits a combing
of polynomial growth and by
\cite[Corollary 5.3]{Meyer} the natural inclusion of complexes
$C^\bullet_{pol}(\Gamma_x;\CC)\hookrightarrow C^\bullet(\Gamma_x;\CC)$ induces an
isomorphism between  $H^*_{pol}(\Gamma_x;\CC)$ and $H^*(\Gamma_x;\CC)$; here
$C^\bullet_{pol}(\Gamma_x;\CC)$ is the subcomplex of the bar complex $
C^\bullet(\Gamma_x;\CC)$ whose elements are cochains with polynomial growth on
$\Gamma_x$. The same applies by \cite{Meyer2} if $\Gamma$ has
  polynomial growth.
   	Now, by Lemma \ref{polynomial-growth} $\rho^*$ preserve the polynomial growth of cochains and we have the following commutative square
   	\begin{equation}
   	\xymatrix{
   		H^*_{pol}(\Gamma_x;\CC)\ar[r]^{\cong}\ar[d]^{\rho^*}& H^*(\Gamma_x;\CC)\ar[d]^{\rho^*}\\
   		HH^*_{pol}(\CC\Gamma;\langle x\rangle)\ar[r]& HH^*(\CC\Gamma;\langle x\rangle).
   	}
   	\end{equation}
 Since the simplicial homotopy $\{h_j\}$ preserves polynomial growth,  $\rho^*$ induces an isomorphism on both sides of the square and the result follows.
   \end{proof}

   So far we only treated Hochschild cohomology, which was constructed from the simplicial structure of $Z\Gamma$. Thanks to the cyclic structure \eqref{cycic-structure} of $Z\Gamma$, one also obtains a  bicomplex $CC^\bullet(\CC\Gamma)$ whose total complex gives the cyclic cohomology of $\CC\Gamma$. Moreover the $k$-th column of $CC^\bullet(\CC\Gamma)$, for $k$ even, is equal to the Hochschild complex. For more background about this see for instance \cite[Sections 9.6, 9.7]{Weibel}.
       Let us then consider the following short exact sequence of bicomplexes
       	$$\xymatrix{0\ar[r]&CC_{\{2\}}^\bullet(\CC\Gamma)\ar[r]& CC^\bullet(\CC\Gamma)\ar[r]& CC^\bullet(\CC\Gamma)[2,0]\ar[r]&0}$$ where we denote by $CC_{\{2\}}^\bullet(\CC\Gamma)$ the bicomplex given by the first two columns of $CC^\bullet(\CC\Gamma)$ and $[2,0]$ denotes the 2-degree horizontal shifting  of the bicomplex. It turns out that $CC_{\{2\}}^\bullet(\CC\Gamma)$ is quasi-isomorphic to the Hochschild complex. Now, taking the associated long exact sequence of the total complexes we obtain the Connes periodicity exact sequence  
     \begin{equation}\label{ces}
   \to HH^n(\CC\Gamma)\xrightarrow{B} HC^{n-1}(\CC\Gamma)\xrightarrow{S}
   HC^{n+1}(\CC\Gamma)\xrightarrow{I} HH^{n+1}(\CC\Gamma)\xrightarrow{B}
     \end{equation}
     see \cite[Theorem 2.2.1]{Loday} for the proof of the homological version
     of this fact.
      
    Observe now that, since $Z\Gamma=\bigcup Z(\Gamma, x)$ is  a disjoint union of cyclic subset, we have that also $CC^\bullet(\CC\Gamma)$ decomposes as $\prod_{\langle x\rangle\in \langle\Gamma\rangle}CC^\bullet(\CC\Gamma;\langle x\rangle)$. This decomposition is compatible with the construction of \eqref{ces}.
    
    \begin{theorem}\label{prop:cyclic_pol}
Let the group $\Gamma$ be hyperbolic or of polynomial growth, $x\in\Gamma$.   	Let ${CC}_{pol}^\bullet(\CC\Gamma;\langle x\rangle)$ be the subcomplex of ${CC}^\bullet(\CC\Gamma;\langle x\rangle)$ whose elements are cochains of polynomial growth. Then the natural inclusion induces an isomorphism
    	$$HC_{pol}^*(\CC\Gamma;\langle x\rangle)\cong HC^*(\CC\Gamma;\langle x\rangle).$$
    \end{theorem}
    \begin{proof}
Consider the commutative diagram
    	\[ 
    	\xymatrix{\to HH_{pol}^n(\CC\Gamma;\langle
          x\rangle)\ar[d]\ar[r]^(.5){B}&HC_{pol}^{n-1}(\CC\Gamma;\langle
          x\rangle) \ar[d]\ar[r]^S&HC_{pol}^{n+1}(\CC\Gamma;\langle x\rangle)\ar[d]\ar[r]&\cdots\\
    		\to HH^n(\CC\Gamma;\langle
                x\rangle)\ar[r]^(.5){B}&HC^{n-1}(\CC\Gamma;\langle
                x\rangle)\ar[r]^S&HC^{n+1}(\CC\Gamma;\langle x\rangle)\ar[r] &\cdots}
    	\]
    	
 By induction, starting with the trivial cases $n=-2$ and $n=-1$, the five lemma
 and Proposition \ref{hochschild-polynomial} for $HH$ imply the isomorphism for
 $HC^n$ for all $n\geq -2$. 
\end{proof}

        Finally, for the pairing of cyclic cohomology with operator K-theory,
        we note 
        that this factors through \emph{periodic cyclic cohomology} $HP^*$
        (compare \cite[Chapter 3, Section 3]{Connes}). Recall that the (two
        periodic) $HP^*(\CC\Gamma)$ is the direct limit

        \begin{equation*}
          HP^k(\CC \Gamma) = \dirlim( HC^k(\CC\Gamma)\xrightarrow{S}
          HC^{k+2}(\CC \Gamma)\xrightarrow{S}\cdots) = HC^{k+2*}(\CC\Gamma)\otimes_{\CC[z]}\CC
        \end{equation*}
        In the tensor product, the
        polynomial variable $z$ of degree $2$ acts on 
        $HC^*(\CC\Gamma)$ as the operator $S$ and on $\CC$ as $1$.

        The above
        considerations show
        that we also have a decomposition of $HP^*(\CC\Gamma)$ with one term
        $HP^*(\CC\Gamma;\langle x\rangle)$ for each conjugacy class
        $\langle x\rangle\subset \Gamma$ (compare \cite[Chapter 3, Section
        1]{Connes}). Moreover, we also have polynomial growth versions of
        these periodic cohomology groups.

        \begin{corollary}\label{corol:periodic_pol}
          Let the group $\Gamma$ be hyperbolic or of polynomial growth, $x\in\Gamma$. The inclusion
          induces an isomorphism
          \begin{equation*}
            HP_{pol}^*(\CC\Gamma;\langle x\rangle)\cong HP^*(\CC\Gamma;\langle
            x\rangle). 
          \end{equation*}
          If $x$ has finite order, then we have a functorial (for group
          homomorphisms) isomorphisms
          \begin{equation*}
             HP^*_{pol}(\CC\Gamma;\langle
            x\rangle) \iso \bigoplus_{k\in\ZZ}H^{*+2k}(\Gamma_x;\CC).
          \end{equation*}
         If, on the other hand, $x$ has infinite order and $\Gamma$ is hyperbolic then
          \begin{equation*}
            HC^*_{pol}(\CC\Gamma;\langle x\rangle) \iso
            \begin{cases}
              \CC; & *=0\\ 0; & *>0
            \end{cases}
            ;  \qquad HP^*_{pol}(\CC\Gamma;\langle
            x\rangle) = 0.
          \end{equation*}
          If $x$ has infinite order and $\Gamma$ has polynomial growth
            then at least $HP^*_{pol}(\complexs\Gamma;\langle x\rangle)=0$.
        \end{corollary}
        \begin{proof}
          This is a direct consequence of Theorem \ref{prop:cyclic_pol}
          together with Theorem \ref{burghelea}. From this last result we have, for $x\in\Gamma$ of finite order, that
          $HC^*(\CC\Gamma;\langle x\rangle) = H^*(N_x;\CC)\otimes \CC[z]$,
          where $N_x=\Gamma_x/x^\integers$ and $x^\integers$ is the
          cyclic subgroup generated by $x$. However, the Leray-Serre
          spectral sequence for cohomology with complex
            coefficients of the short exact sequence $1\to x^\integers\to \Gamma_x\to N_x\to 1$ 
          degenerates on the $E_2$ page, as
          \begin{equation*}
            H^*(x^\integers;\complexs)=
            \begin{cases}
              \complexs;& *=0\\ 0; & *\ne 0
            \end{cases}
          \end{equation*}
          for the finite cyclic group $x^{\integers}$
 and therefore the projection map induces an isomorphism
 $H^*(N_x;\complexs)\xrightarrow{\iso} H^*(\Gamma_x;\complexs)$ which we are
 using implicitly. The
          action of $S$ is given by multiplication by $z$ which implies also the
          result for $HP^*(\CC\Gamma;\langle x \rangle)$.

          On the other hand, for $x\in\Gamma$ of infinite order
           $HC^*(\complexs\Gamma;\langle x\rangle)\iso
            H^*(N_x;\complexs)$ where $S$ acts raising degree by $2$. If
            $\Gamma$ is hyperbolic we know that
          $\Gamma_x$ is a \emph{finite} extension of the cyclic subgroup generated by
          $x$, therefore $N_x$ is finite, compare \cite[8.5.M
          (1)]{Gromov-hyperbolic}. It follows that its cohomology vanishes
          above degree $0$ and the operator $S$ of degree $2$ is necessarily
          $0$, so that the periodic cohomology vanishes.

          If $\Gamma$ has polynomial growth or equivalently is virtually
          nilpotent, also $N_x$ is virtually nilpotent and in particular its
          rational cohomological dimension is finite. Consequently, $S^N=0$
          for $N$ sufficiently large and its direct limit
          $HP^*(\complexs\Gamma;\langle x\rangle)=0$.
        \end{proof}

\begin{remark}\label{rmk-finite-order} Let us point out one subtlety when looking at
      the polynomially growing subcomplexes in this section. 
 The results we have established are valid for one conjugacy class 
      or for a finite number of them.
          On the other hand, the {\it full}
    cyclic cohomology groups are the unrestricted direct product of the
      contributions from the different conjugacy classes; consequently the above results have not been 
      established for the full cyclic (co)homology groups
    with a global polynomial growth
      condition.
However, relevant for us is the \emph{periodic} cyclic cohomology, as the
pairing with K-theory factors through the latter. If
     $\Gamma$ is hyperbolic or $\Gamma$ has polynomial growth, then  $\Gamma$
     has only finitely many conjugacy classes of finite order (compare
     e.g.~\cite{Lueck_Eunderbar} and \cite[Corollary 8]{MeintrupSchick}). Only
     those
     conjugacy classes contribute to $HP^*(\complexs\Gamma)$. For finite
     direct sums, polynomial growth conditions for each summand are equivalent
     to a global polynomial growth condition and we get the isomorphism
     $HP^*_{pol}(\complexs\Gamma)\xrightarrow{\iso} HP^*(\complexs\Gamma)$.
\end{remark}

 \section{Extending cocycles for hyperbolic groups}\label{sec:Puschnigg}

Let $\Gamma$ be a Gromov hyperbolic group. We want to use cyclic cohomology of
$\Gamma$ to obtain numerical
primary and secondary invariants of Dirac operators. As explained in
  \ref{sec:intro_cyclic}, for this we have to consider  a suitable completion
  $\mathcal{A}\Gamma$ of $\complexs\Gamma$ which allows to construct the Chern
character of Section \ref{section6}, and then to argue that all cochains in a subcomplex of
the cyclic cochain complex of $\complexs\Gamma$ extend to $\mathcal{A}\Gamma$. This is possible with a dense and holomorphically
closed subalgebra of $C^*_{red}(\Gamma)$, that we will denote by
$\mathcal{A}\Gamma$,  constructed by Puschnigg in \cite{Puschnigg}. 
For this algebra we thus have  that $K_*(\mathcal{A}\Gamma)\cong K_*(C^*_{red}(\Gamma))$.
Puschnigg proves that any delocalized trace $\tau$ on $\CC\Gamma$ which is
supported on a finite number of conjugacy classes extends to
$\mathcal{A}\Gamma$. Note that these are just certain cyclic cocycles of
degree $0$. We will show that the corresponding statement is true also for
higher delocalized cochains, thus implementing the program explained in \ref{sec:intro_cyclic}.
Recall from Theorem \ref{prop:cyclic_pol}  that the summand
  $HC^*(\CC\Gamma;\langle x\rangle)$ supported on $\langle x\rangle$ is computed by $CC^\bullet_{pol}(\CC\Gamma;\langle x\rangle)$, the subcomplex of cyclic cochains of polynomial growth supported on $CC_*(\CC\Gamma;\langle x\rangle)$. 
  
  \begin{definition}
    If $\CC\Gamma$ is densely included in a Fr\'{e}chet algebra
    $\mathcal{A}\Gamma$, then we have an inclusion of
    $CC_\bullet(\CC\Gamma;\langle x\rangle)$ into $CC_\bullet(\mathcal{A}\Gamma)$. We now
    define $CC_\bullet(\mathcal{A}\Gamma;\langle x\rangle)$ as the closure of
    $CC_\bullet(\CC\Gamma;\langle x\rangle)$ inside $CC_\bullet(\mathcal{A}\Gamma)$, and
    correspondingly for $CC^\bullet$.
  \end{definition}

  \begin{proposition} \label{extension}

      Let $\Gamma$ be a hyperbolic group, $x\in\Gamma$. The cochains of the  subcomplex 
      $CC^\bullet_{pol}(\CC\Gamma;\langle x\rangle)$  extend
      continuously to $\mathcal{A}\Gamma$, giving rise to the canonical map
      $$HC^*(\CC\Gamma;\langle x\rangle)\to HC^*(\mathcal{A}\Gamma)$$ which
      again is an inclusion as direct summand.

\end{proposition}
\begin{proof}

  We start with the duality pairing $$c_{alg}\colon CC_\bullet(\CC\Gamma)\otimes
  CC^\bullet_{pol}(\CC\Gamma)\to\CC$$ restricted to the cochains supported on $\innerprod{x}$,
  $$c_{alg}\colon CC_\bullet(\CC\Gamma)\otimes  CC^\bullet_{pol}(\CC\Gamma;\langle
    x\rangle)\to\CC.$$ The assertion about the extension to $\mathcal{A}\Gamma$ means
    that we just have to show that $c_{alg}$ extends continuously to
    $CC_\bullet(\mathcal{A}\Gamma)\otimes  CC^\bullet_{pol}(\CC\Gamma;\langle
    x\rangle)$. For this, we observe that the cochains supported on
    $\langle x\rangle$ pair non-trivially only with the chains supported on
    $\langle x\rangle$, i.e.~$c_{alg}$ factorizes over the projection
    $\pi_{\langle x\rangle}\colon CC_\bullet(\CC\Gamma)\to CC_\bullet(\CC\Gamma;\langle
    x\rangle)$. Next, by the very definition of rapid decay and of polynomial
    growth, we have a first extension
    $$c_{RD}\colon CC_\bullet(l^1_{RD}\Gamma;\langle x\rangle)\otimes
     CC^\bullet_{pol}(\CC\Gamma;\langle x\rangle)\to\CC$$ of the algebraic
    pairing from $\CC\Gamma$ to the rapid decay algebra $l^1_{RD}\Gamma$.

 Next, when localized at $\langle x\rangle$, the  cyclic chains of the $l^1$ rapid
    decay algebra $l^1_{RD}\Gamma$ and of $\mathcal{A}\Gamma$ coincide. More precisely,
    Puschnigg constructs in
    \cite[Proposition 5.6, b)]{Puschnigg} a chain map
    $\iota_*\colon CC_\bullet(l^1_{RD}\Gamma;\langle x\rangle)\xrightarrow{\iso}
    CC_\bullet(\mathcal{A}\Gamma;\langle x\rangle)$  and shows that it is
    an isomorphism. Consequently, we get the pairing $$
    c_{RD}\circ(\iota_*^{-1}\otimes \id)
  \colon CC_\bullet(\mathcal{A}\Gamma;\innerprod{x})\otimes
    CC^\bullet_{pol}(\CC\Gamma;\innerprod{x}) \to \CC. $$
    Finally and crucially, Puschnigg constructs in the proof of \cite[Proposition
5.6, a)]{Puschnigg} a (continuous) projection
$\pi^{\mathcal{A}}_{\langle x\rangle}\colon CC_\bullet(\mathcal{A}\Gamma)\to
CC_\bullet(\mathcal{A}\Gamma;\langle x\rangle)$, extending the corresponding
projection $\pi_{\innerprod{x}}\colon CC_\bullet(\CC\Gamma)\to
CC_\bullet(\CC\Gamma;\innerprod{x})$. Note that he shows that this is well defined
even for the entire (periodic) cyclic bicomplex. The argument works even more
directly for the bicomplex computing usual cyclic homology, where we
degree-wise only have to deal with finite sums.

The composition $ c_{\mathcal{A}\Gamma}:= c_{RD}\circ(\iota_*^{-1}\circ\pi_{\innerprod{x}})\otimes
\id\colon CC_\bullet(\mathcal{A}\Gamma)\otimes
CC^\bullet_{pol}(\CC\Gamma;\innerprod{x})\to \CC$ is  then the desired
extension of $c_{alg}$. The construction is summarized in the following
commutative diagram:

\begin{equation}\label{eq:Puschnigg}
{\tiny \xymatrix{ 
	CC_\bullet(\CC\Gamma)\otimes CC^\bullet_{pol}(\CC\Gamma;\langle x\rangle)\ar[rr]^{\pi_{\langle x\rangle}\otimes \mathrm{id}}\ar[dr]\ar[dd]^{c_{alg}}& &CC_\bullet(\CC\Gamma;\langle x\rangle)\otimes CC^\bullet_{pol}(\CC\Gamma;\langle x\rangle)\ar[d]\\
&CC_\bullet(\mathcal{A}\Gamma)\otimes CC^\bullet_{pol}(\CC\Gamma;\langle x\rangle)\ar[rd]_(.47){\pi^{\mathcal{A}}_{\langle x\rangle}\otimes \mathrm{id}})\ar@{-->}[dl]_(.47){c_{\mathcal{A}\Gamma}} &CC_\bullet(l^1_{RD}\Gamma;\langle x\rangle)\otimes CC^\bullet_{pol}(\CC\Gamma;\langle x\rangle) \ar[d]^{\iota_*\otimes\id}_\cong \\
\CC&&
CC_\bullet(\mathcal{A}\Gamma;\langle x\rangle)\otimes
CC^\bullet_{pol}(\CC\Gamma;\langle x\rangle)\ar[ll]^{c_{RD}\circ(\iota_*^{-1}\otimes\id )}
}}
\end{equation}

\medskip 

  Using the identification
  $HC_{pol}(\CC\Gamma;\innerprod{x})\xrightarrow{\iso}
  HC^*(\CC\Gamma;\innerprod{x})$, $c_{\mathcal{A}\Gamma}$ corresponds dually to
  a homomorphism $HC^*(\CC\Gamma;\innerprod{x}) \to HC^*(\mathcal{A}\Gamma)$
  and the commutativity of the left triangle in \eqref{eq:Puschnigg} means that the
  composition
  \begin{equation*}
    HC^*(\CC\Gamma;\innerprod{x}) \to  HC^*(\mathcal{A}\Gamma) \to  HC^*(\CC\Gamma)
  \end{equation*}
  is the split inclusion $HC^*(\CC\Gamma;\innerprod{x})\to HC^*(\CC\Gamma)$
  and therefore the first map
  \begin{equation*}
    HC^*(\CC\Gamma;\innerprod{x}) \to HC^*(\mathcal{A}\Gamma)
  \end{equation*}
  is split
  injective, too.

\end{proof}

  Let us remark as an addendum the following propositions, the first of which follows
  immediately from the existence of the projection $\pi_{\innerprod{x}}$ of
  \cite[Proposition 5.6]{Puschnigg}.

  \begin{proposition}\label{prop:split_deloc_hyp}
    Let $\Gamma$ be a hyperbolic group, $\mathcal{A}\Gamma$ the Puschnigg
    completion of $\complexs\Gamma$. For each $x\in\Gamma$, we have a split
    \begin{equation*}
      HC_*(\mathcal{A}\Gamma)\xrightarrowdbl{(\pi_{\innerprod{x}})_*}
      HC_*(\mathcal{A}\Gamma;\innerprod{x}) \xleftarrow{\iso} HC_*(l^1_{RD}\Gamma;\innerprod{x})
    \end{equation*}
    of the map $HC_*(\mathcal{A}\Gamma;\innerprod{x})\to
    HC_*(\mathcal{A}\Gamma)$ induced by the (defining) inclusion of chain
    complexes.
  \end{proposition}

  \begin{proposition}\label{prop:ext_hyp} Let $\Gamma$ be a hyperbolic group.  The maps of
    Proposition \ref{extension} for the finitely many conjugacy classes of
    elements of finite order combine to give a canonical map
    \begin{equation*}
      HP^*(\complexs\Gamma)\to HP^*(\mathcal{A}\Gamma)
    \end{equation*}
    which is again an inclusion as direct summand.
  \end{proposition}
  \begin{proof}
    The argument of the proof of Proposition \ref{extension} works equally
    well if we replace
    $CC^\bullet_{pol}(\complexs;\langle x\rangle)$ by a direct sum of finitely many
    terms where the $x$ represent pairwise different conjugacy classes. As
    $\Gamma$ is hyperbolic, there are only finitely many conjugacy classes
    of finite order \cite[Corollary 8]{MeintrupSchick}. We obtain the split injection $\bigoplus_{[x], x\text{ of finite
        order}}HC^*(\complexs\Gamma;\langle x\rangle) \to
    HC(\mathcal{A}\Gamma)$. By naturality, this split injection is compatible
    with the periodicity map $S$. Passing to the limit, we obtain the desired
    split injection
    \begin{equation*}
      HP^*(\complexs\Gamma) = \bigoplus_{[x], x\text{ of finite
        order}}HP^*(\complexs\Gamma;\langle x\rangle) \to HP^*(\mathcal{A}\Gamma).
    \end{equation*}
  \end{proof}

  \section{Extending cocycles for groups of polynomial growth}

  Let $\Gamma$ by a group of polynomial growth. By \cite[3.1.8]{Jolissaint} these
  groups satisfy property (RD). By \cite{Meyer2}, it holds that
  $H_{pol}^*(\Gamma_x)\xrightarrow{\iso}H^*(\Gamma_x)$ (as each subgroup
  $\Gamma_x$ of a group of polynomial growth has itself polynomial growth).

  Consequently, $HP^*_{pol}(\complexs\Gamma)\to HP^*(\complexs\Gamma)$ is an
  isomorphism as this holds for each of the finitely many summands
  $H^*(\Gamma_x)$.
   
  Let us choose for  $\mathcal{A}\Gamma$ the Fr\'{e}chet
  algebra $\mathcal{B}$ defined by Connes and Moscovici in \cite[Section
  6]{ConnesMoscovici}, which is dense and holomorphically closed in
  $C^*_{red}\Gamma$. Note that for groups of polynomial growth,
  $\mathcal{A}\Gamma$ coincides with the ($l^2$)-rapid decay algebra, compare
  \cite[3.1.7]{Jolissaint} and also with the $l^1$-rapid decay algebra
  $l^1_{RD}\Gamma$, compare \cite[Remark 2.6 (2)]{Jolissaint2}. Connes and Moscovici extend to
  $\mathcal{A}\Gamma$ the pairing with $H^*(\Gamma)$, the summand of
  $HP^*(\complexs\Gamma)$ corresponding to $x=e$ (the localized part of cyclic
  cohomology). The proof extends immediately to the full
  $HP^*(\complexs\Gamma)$, using that it works for each of the finitely many summands
  $H^*(\Gamma_x)=H^*_{pol}(\Gamma_x)$.
  
  Equivalently, we have the following proposition.
  \begin{proposition}\label{prop:extend_polygrowth}
    Let $\Gamma$ be a group of polynomial growth. Let $\mathcal{A}\Gamma$ be
    the Connes-Moscovici algebra, a dense and holomorphically closed
    subalgebra of $C^*_{red}\Gamma$ containing $\complexs\Gamma$. Then we have a split injection
    \begin{equation*}
      HP^*(\complexs\Gamma)\to HP^*(\mathcal{A}\Gamma),
    \end{equation*}
   split by the restriction homomorphism.
  \end{proposition}

We also have the splitting of the cyclic complex according to conjugacy
classes, extending the one of the group ring, parallel to Proposition
\ref{prop:split_deloc_hyp}.

\begin{proposition}\label{prop:split_polygrowth}
      Let $\Gamma$ be a group of polynomial growth. Let $\mathcal{A}\Gamma$ be
    the Connes-Moscovici algebra, equal to $l^1_{RD}\Gamma$. For each
    $x\in\Gamma$, we have a split
    \begin{equation*}
      HC_*(\mathcal{A}\Gamma)\xrightarrowdbl{(\pi_{\innerprod{x}})_*} HC_*(\mathcal{A}\Gamma;\innerprod{x})
    \end{equation*}
  of the map $HC_*(\mathcal{A}\Gamma;\innerprod{x})\to
    HC_*(\mathcal{A}\Gamma)$ induced by the (defining) inclusion of chain
    complexes.
\end{proposition}
\begin{proof}
  This is a well known fact, proven in \cite[Corollary 1.4.2]{JiOgleRamsey},
  using that for groups of polynomial growth the subalgebra
  $\mathcal{A}\Gamma$ relevant to us equals the $l^1$-rapid decay algebra
  treated there.
\end{proof}

  \section[Delocalized cyclic group homology]{Delocalized cyclic group homology and non-commutative de Rham homology}
  \label{section7.5}
  
  The considerations of Section \ref{sec:HC_of_group} show, in particular, that the
  cyclic chain complex of a discrete group $\Gamma$ has a ``localized summand'', namely
  $CC^e_*(\CC\Gamma):= CC_*(\CC\Gamma;\innerprod{e})$ with complement the delocalized part
  $CC_*^{del}(\CC\Gamma)$ and short (split) exact sequence
  \begin{equation*}
  0 \to CC_\bullet^{e}(\CC\Gamma) \to CC_\bullet(\CC\Gamma)\to
  CC_\bullet^{del}(\CC\Gamma)\to 0, 
  \end{equation*}
  giving rise to a corresponding short split exact sequence of cyclic homology $$0\to
  HC^{e}_*(\CC\Gamma)\to HC_*(\CC\Gamma)\to HC^{del}_*(\CC\Gamma)\to0.$$
  
  Assume now that $\mathcal{A}\Gamma$ is a Fr\`echet completion of
  $\CC\Gamma$. Let $CC^{e}_*(\mathcal{A}\Gamma)$ denote the subcomplex
  $CC_*(\mathcal{A}\Gamma; \langle e\rangle)$, defined as in Section
  \ref{sec:Puschnigg}, then we  define the delocalized cyclic complex via the
  short exact sequence
  \begin{equation}\label{deloc-cyclic-complex}
  0\to CC^{e}_\bullet(\mathcal{A}\Gamma)\to CC_\bullet(\mathcal{A}\Gamma)\to
  CC^{del}_\bullet(\mathcal{A}\Gamma)\to 0.
  \end{equation}
  It is not guaranteed that this is split exact, but in any case we get an
  associated long exact sequence of localized and delocalized cyclic homology
  \begin{equation*}
  \cdots \to HC^{del}_{*-1}(\mathcal{A}\Gamma)\to HC^{e}_*(\mathcal{A}\Gamma)\to
  HC_*(\mathcal{A}\Gamma)\to HC^{del}_*(\mathcal{A}\Gamma)\to \cdots
  \end{equation*}

  \begin{proposition}\label{prop:homology-deloc-slpit}
     If $\Gamma$ is hyperbolic or of polynomial growth, then the long exact sequence for (de)localized
     cyclic homology splits into short exact sequence with a canonical split
    \begin{equation*}
	\xymatrix{0\ar[r]& HC_*^e(\mathcal{A}\Gamma)\ar[r]&
          HC_*(\mathcal{A}\Gamma)\ar[r]
          \ar@/_1pc/[l]_{(\pi_{\innerprod{e}})_*} &
          HC_*^{del}(\mathcal{A}\Gamma) \ar[r]& 0}.
    \end{equation*}
    The corresponding result holds for cyclic cohomology.
  \end{proposition}
  \begin{proof}
    This is just the case $x=e$ of Propositions \ref{prop:split_deloc_hyp} and
    \ref{prop:split_polygrowth}.
  \end{proof}

\subsection*{Reduced cyclic homology}

  As mentioned in the introduction of this section, we want to relate the
  general results of Sections \ref{ncderham}, \ref{section5}, and
  \ref{section6} to the computations of cyclic (co)homology established in
  the current section. It is well known that this is possible: non-commutative
  de Rham homology canonically embeds in cyclic homology.

  Instead of deriving a full account of this, which also works for the
  relative homology we emphasize in this paper, we just want to cite the
  relevant results in the literature \cite[Th\'eor\`eme
  2.15]{Karoubi}. Unfortunately, there is one detail to take into account: the
  embedding works with \emph{reduced} cyclic homology.\footnote{``Reduced''
      here refers to quotient by terms coming from the ground field
      $\complexs$, corresponding to a base point in topology. This must not be
      confused with notion of reduced homology when chain complexes of
      topological vector spaces are considered and homology is defined by
      taking the quotient by the \emph{closure} of the image of the
      differential. We never use this second meaning in this paper.}

  We omit here the explicit definition \cite[Section 2.9]{Karoubi} and instead
  observe that for the cases relevant to us the reduced homology
  $\overline{HC}_*(\mathcal{A}\Gamma)$ is easily 
  described:

  \begin{proposition}\label{prop:loc_delo_split}
  Let the group $\Gamma$ be hyperbolic or of polynomial growth. As above, let
  $\mathcal{A}\Gamma$ be the Puschnigg algebra if $\Gamma$ is hyperbolic or
  the rapid decay algebra if $\Gamma$ has polynomial growth.

  Then we have a canonical decomposition
  \begin{equation*}
    HC_*(\mathcal{A}\Gamma)= HC_*(\complexs)\oplus
    \overline{HC}_*(\mathcal{A}\Gamma)
  \end{equation*}
  with projection map $\epsilon_*\colon HC_*(\mathcal{A}\Gamma)\to
  HC_*(\complexs)$ splitting the map induced by the inclusion $\complexs\to
  \mathcal{A}\Gamma$.

 \begin{equation*}
   \begin{split}
     HC_*(\mathcal{A}\Gamma)& =
     \underbrace{  \underbrace{HC^{del}_*(\mathcal{A}\Gamma)}_{=\overline{HC}^{del}_*(\mathcal{A}\Gamma)}
     \oplus
     \overline{HC}_*^e(\mathcal{A}\Gamma)}_{=\overline{HC}_*(\mathcal{A}\Gamma)}
                              \oplus HC_*(\complexs);\\
     \overline{HC}^e_*(\mathcal{A}\Gamma)\oplus
     HC_*(\complexs) &= HC^e_*(\mathcal{A}\Gamma)
   \end{split}
 \end{equation*}

\end{proposition}
\begin{proof}
  It suffices to construct the split $\epsilon_*$. For this, we use the split
  \begin{equation*}
    \pi_{\innerprod{e}}\colon HC_*(\mathcal{A}\Gamma)\to
    HC_*(\mathcal{A}\Gamma;\innerprod{e})= HC_*(l^1_{RD}\Gamma;\innerprod{e})
  \end{equation*}
  and compose with the morphism induced by the canonical augmentation map
  $\epsilon  \colon l^1\Gamma\to \complexs; \sum a_\gamma\gamma\mapsto \sum
  a_\gamma$, which is defined because $l^1_{RD}\Gamma\subset l^1\Gamma$.
  
\end{proof}

\begin{corollary}\label{corol:deloc_inclusion}
  Let the group $\Gamma$ be hyperbolic or of polynomial growth and
  $\mathcal{A}\Gamma$ the smooth subalgebras of $C^*_{red}\Gamma$ as
  before. We get a commutative diagram, mapping the localized/delocalized
  non-commutative de Rham homology sequence to the localized/delocalized
  cyclic 
  homology sequence
  \begin{equation}\label{eq:dR_to_HC}{\small
    \begin{CD}
      @>>> H^e_*(\mathcal{A}\Gamma) @>>>
      H_*(\mathcal{A}\Gamma) @>>> H^{del}_*(\mathcal{A}\Gamma) @>>>
      H^e_{*-1}(\mathcal{A}\Gamma)\\
      && @VVV @VVV @VVV\\
      0 @>>> \overline{HC}^e_*(\mathcal{A}\Gamma) @>>>
      \overline{HC}_*(\mathcal{A}\Gamma) @>>>
      HC^{del}_*(\mathcal{A}\Gamma) @>>> 0 
    \end{CD}}
  \end{equation}
\end{corollary}
Note that we have established that the sequence of cyclic homology is split
short exact in our situation in Propositions \ref{prop:split_deloc_hyp} and
\ref{prop:split_polygrowth}. We suspect that the same is true for the
non-commutative de Rham sequence, but we don't make use of this and therefore
don't discuss this issue here. Note, moreover, that such a splitting would
imply immediately that all three arrows are injective (after passage to
reduced de Rham homology), because the middle one is. Again, we don't use this
injectivity 
and therefore don't discuss it further.

\begin{proof}
  This follows from the proof of the inclusion
  $\overline{H}_*(\mathcal{A}\Gamma)\into \overline{HC}_* (\mathcal{A}\Gamma)$ of
  \cite[Section 2.6.4]{Loday} or \cite[Th\'eor\`eme 2.15]{Karoubi} (where we
  note that $\overline{H}_*(\mathcal{A})$ simply is the quotient of
  $H_*(\mathcal{A})$ by the image of $H_*(\complexs)$, where the latter is
  concentrated in degree $0$ and equal to $\complexs$ there. In particular,
  $\overline{H}_n(\mathcal{A}\Gamma)=H_n(\mathcal{A}\Gamma)$ for $n\ge 1$).

  The inclusion ultimately is obtained by mapping a form $a_0da_1\dots da_n$
  to $a_0\tensor\dots \tensor a_n$. It follows immediately that
  $H^e_*(\mathcal{A}\Gamma)$ is mapped to
  $\overline{HC}^{e}_*(\mathcal{A}\Gamma)$, and the left vertical arrow of
  \eqref{eq:dR_to_HC} is defined. Moreover, inspecting the proof then shows
  that also the right vertical map is well defined and of course makes the
  diagram commutative.
\end{proof}

\begin{remark}\label{rmk-homology}
	This fact is key in what follows, since it allows to see elements in
        the image of $\Ch_\Gamma^{del}$, that is elements in
        $H^{del}_*(\mathcal{A}\Gamma)$,  as elements in the cyclic homology of
        $\mathcal{A}\Gamma$ and then to pair them with delocalized cyclic
        cocycles.

        Of course, the fact that we can pair with cyclic cocycles is in the
        end well known, but we preferred to give a formally correct derivation
        of this via the Chern character map, which in our approach a priori
        takes values in non-commutative de Rham homology.
\end{remark}

\begin{remark}
  We believe that the treatment of Corollary \ref{corol:deloc_inclusion} could be
  systematized and improved in the following two ways:
  \begin{itemize}
  \item It is a bit unfortunate that the target of the maps in Corollary
    \ref{corol:deloc_inclusion} is \emph{reduced} cyclic homology, because we know
    that we can  pair K-theory also with unreduced cohomology. This is due to
    the classical definition of de Rham homology defining $d1=0$, which forces
    to implement a similar relation in the target cyclic homology groups,
    leading to the occurrence of $\overline{HC}$.

    A logical remedy, very much in the spirit of the definition of K-theory,
    would be to work with a version of non-commutative de Rham homology which
    first augments the algebra and then passes to the reduced homology of this
    augmented algebra. This works uniformly for unital and non-unital
    algebras, is completely in line with the definition of K-theory, and
    avoids the appearance of reduced cyclic homology in the comparison map.
  \item In the derivation of Corollary \ref{corol:deloc_inclusion}, we tried to be
    as brief as possible, reducing directly to the known result for
    $H_*(\mathcal{A}\Gamma) $. Along the way, we used the splitting of
    Proposition \ref{prop:loc_delo_split}. We believe that a more systematic
    and general direct treatment could be interesting, which should work
    without further splitting assumptions.

  \end{itemize}
\end{remark}

\section{Higher rho numbers}\label{sec:hyperbolic_pairing}

In this section, let the group $\Gamma$ be hyperbolic or of polynomial
growth. 
Using Corollary \ref{corol:periodic_pol}, Proposition \ref{prop:ext_hyp}, Proposition \ref{prop:extend_polygrowth} and Remark \ref{rmk-homology} we
obtain immediately the following result.

\begin{theorem}\label{pairing-hyperbolic}
	Let the group $\Gamma$ be hyperbolic or of polynomial growth and $e\ne x\in\Gamma$ and let ${\tM}$ be a
          Galois covering of a smooth compact manifold $M$ with covering group
        $\Gamma$. Then there exists a well-defined pairing
\[
\SG^\Gamma_*(\tM)\times HP^{*-1}(\complexs\Gamma;\innerprod{x}) =  K_{*}\left(C(M)\to\Pdo^0_\Gamma({\tM})\right)\times HP^{*-1}(\CC\Gamma;\langle x\rangle)\to \CC
\]
given by
\[
(\xi,[\tau])\mapsto\langle\Ch_\Gamma^{del}(\xi), \tau\rangle
\]
for $\xi\in K_{*+1}\left(C(M)\to\Pdo^0_\Gamma({\tM})\right)$ and $\tau\in
CC^{[*-1]}_{pol}(\CC\Gamma; \langle x\rangle)$.
\end{theorem}
\begin{remark}
This pairing is compatible with the usual pairing between $K_{*-1}(C^*_{red}\Gamma)$ and
cyclic cohomology: for $\xi\in K_{*+1}(C^*_{red}\Gamma)$, $[\tau]\in
HC^{[*-1]}(\complexs\Gamma;\innerprod{x})$ we have
$\innerprod{\xi,[\tau]}=\innerprod{s(\xi),[\tau]}$, where $s\colon K_{*-1}(C^*_{red}\Gamma)\to \SG^\Gamma_*(\tM)$ is as in \eqref{main-commutative-diagram}.
\end{remark}

\begin{definition}\label{higher-rho-number}
	Let $\widetilde{D}$ be a generalized Dirac operator which is $\Gamma$-equivariant on the Galois $\Gamma$-covering $\tM$ of a compact smooth manifold $M$ of dimension $n$ without boundary. Suppose that $\widetilde{D}$ is $L^2$-invertible. Then define its higher rho number associated to $[\tau]\in HC^{[n-1]}(\CC\Gamma;\langle x\rangle)$ as
	 	\[
	 	\varrho_\tau(\widetilde{D}):=	\left\langle\Ch_\Gamma^{del}\left(\varrho(\widetilde{D})\right), \tau\right\rangle\in \CC,
	 	\]
	 	where $\varrho(\widetilde{D})$ is defined as in Definition \ref{rho-class}.
\end{definition}

\begin{definition}\label{def:higher_rho_number_cyc}
                Let $g$ be a metric with positive scalar curvature on a spin manifold $M$ 
    of dimension $n$ with fundamental group $\Gamma$. Define the higher rho number associated to $[\tau]\in HC^{[n-1]}(\CC\Gamma;\langle x\rangle)$ as 
\[
\varrho_\tau(g):=	\langle\Ch_\Gamma^{del}\left(\varrho(g)\right), \tau\rangle\in \CC
\]
where $\varrho(g)\in  K_{n}\left(C(M)\to\Pdo^0_\Gamma({\tM})\right)$
 is the rho class of the spin Dirac operator for the metric $g$.
\end{definition}

Recall from \eqref{surjective} that, if the dimension of $M$ is odd, we have a surjective map  $$K_0(\Psi^0_{\mathcal{A}\Gamma}(\tM))\to K_1(C(M)\to \Psi^0_{\mathcal{A}\Gamma}(\tM)),$$given by the composition of the suspension isomorphism and the natural inclusion. Moreover, by Lemma \ref{chern-bott} the triangle
\[
\xymatrix{K_0(\Psi^0_{\mathcal{A}\Gamma}({\tM}))\ar[r]\ar[dr]_{\overline{\TR}^{del}\circ\Ch}& K_1(C(M)\to \Psi^0_{\mathcal{A}\Gamma}(\tM))\ar[d]^{\Ch_\Gamma^{del}}\\
	& H_{\mathrm{even}}^{del}(\mathcal{A}\Gamma)}
\]
is commutative. Therefore, given a metric with positive scalar curvature $g$ on an odd dimensional manifold $M$, we can express $\Ch_\Gamma^{del}(\varrho(g))$ by $\overline{\TR}^{del}(\Ch[\pi_\geq])$, as in Section \ref{rho-classes}.   In order to calculate the  higher rho number associated to an even dimensional cyclic cocycle $\tau$, we then proceed as in the following proposition. 
\begin{proposition}\label{explicit-rho-odd}
	Let $\tau$ be a  cyclic cocycle  of degree $2k$ of polynomial growth
        with support contained in a conjugacy class $\langle x
        \rangle\ne \innerprod{e}$ of $\Gamma$.
	Then the explicit integral formula  is given by the following expression
	\begin{multline}\label{formula-delocalized}
            \varrho_\tau(g)
            =\sum_{g_0\dots g_{2k}\in \langle x
              \rangle}\frac{1}{k!}\int_{{\tM}\times\cdots\times{\tM}}\Tr\left(\chi(x_0)P(x_0g_0,x_1)h(x_1)P(x_1g_1,x_2)\cdots\right.\\
   \left.           \cdots
              h(x_{2k})P(x_{2k}g_{2k},x_0)\right)\,\tau(g_0,\dots,g_{2k})\,dvol_{g}^{2k+1}
	\end{multline}
where $P$ is the Schwartz kernel of $\pi_\geq(\widetilde{D}_g)$ and
where $\chi$ is the characteristic function of a fundamental domain
$\mathcal{F}$ for the action of $\Gamma$ on $\tM$, and where $h$ is
the cutoff function of \eqref{eq:cutoff}.
\end{proposition}
\begin{proof}
       {First observe, using the notation of Section \ref{section6}, as $P$ is
         a projector and $d$ a derivation, that we have
       $dP=d(P^2)=Pd+(dP)P$, or $(dP)P= dP -PdP$. Applying this twice we get
       $PdPdPdPdP=(PdPdP)^2$. Moreover, by the definition of $d$ we have
       $dP=[\nabla^{Lott},P] + PX+XP$ with the auxiliary formal variable $X$
       with $X^2=\Theta$ and $PXP=0$. Substituting, we obtain
     \begin{equation*}
	PdPdP=P[\nabla^{Lott},P][\nabla^{Lott},P]+P\Theta P.
      \end{equation*}}

	Moreover, using that $P[\nabla^{Lott},P]P={P\circ\nabla^{Lott}\circ
          P - P\circ\nabla^{Lott}\circ P=}0$,
	\begin{equation}
	\begin{split}
	P[\nabla^{Lott},P][\nabla^{Lott},P]+P\Theta P&=P[\nabla^{Lott},P][\nabla^{Lott},P]+P(\nabla^{Lott})^2 P+ P[\nabla^{Lott},P]P\nabla^{Lott}\\
	&=P[\nabla^{Lott},P]\nabla^{Lott}P+P(\nabla^{Lott})^2 P\\
	&= P\nabla^{Lott}P\nabla^{Lott}P\\
	&= (P\nabla^{Lott}P)^2.
	\end{split}
	\end{equation}
	Hence we have that $\Ch_{\Gamma,k}^{del}(\pi_\geq)$ is given by
	$$\overline{\TR}^{del}( (P\nabla^{Lott}P)^{2k}).$$
	In order to obtain an explicit formula for this we use the version of $\nabla^{\mathrm{Lott}}$ given by \eqref{lott-version}. A direct calculation then gives \eqref{formula-delocalized}.
\end{proof}

The {corresponding} result applies to any $L^2$-invertible $\Gamma$-equivariant Dirac type operator $\widetilde{D}$   on an odd dimensional manifold $\tM$.
\begin{example}\label{example-deloc-eta}
	Let $\tau$ be a 0-cocycle supported on the conjugacy class $\langle x\rangle$ in  $\Gamma$. Then
	\[
	\varrho_{\tau}(\widetilde{D})=\sum_{\gamma\in \langle x\rangle
        }\int_{\mathcal{F}}\Tr ({\pi_{\geq
   }}(\widetilde{D})(x\gamma,x))
	d\mathrm{vol}(x) \tau(\gamma)
	\]
	which is, up to a normalization factor, exactly the delocalized $\eta$-invariant of Lott obtained in \cite[Section 4.10.1]{Lott2}.
	This example also appears in \cite{ChenWangXieYu}.
\end{example}

  We have a corresponding result for the numeric rho class of a positive
    scalar curvature metric on an even dimensional spin manifold associated to
    an odd degree cyclic cohomology class $[\tau]$, which is not quite as
    explicit and which is worked out in Proposition \ref{prop:rho_formula_even}.

\begin{example}\label{ex:degree_1}
Let us now consider the case of an even dimensional spin Riemannian manifold
$(M,g)$ of positive scalar curvature with fundamental group
  $\Gamma$. We work out explicitly the degree 1 part of the Chern character of
  the 
rho class. Recall that the rho class of the metric $g$ in $K_0(C(M)\to
\Psi^0_{\mathcal{A}\Gamma}(\tM 
))$  is given by the path of projections
\begin{equation*}
p_t=\begin{pmatrix}
\cos^2\left(\frac{\pi}{2}t\right)e_0 & \cos\cdot\sin\left(\frac{\pi}{2}t\right)u^*\\ \cos\cdot\sin\left(\frac{\pi}{2}t\right)u& \sin^2\left(\frac{\pi}{2}t\right)e_1
\end{pmatrix}
\quad\text{on the module}\quad
\Psi^0_{\mathcal{A}\Gamma}(\tM)^n\oplus\Psi^0_{\mathcal{A}\Gamma}(\tM)^n,
\end{equation*}
for some $n\in\naturals$ and $t\in [0,1]$. Here $e_0$ and $e_1$ are the projections over
$C(M)$ associated to the spinor vector bundles $S_+$ and $S_-$ respectively;
whereas $u$ and $u^*$ are the partial isometries $u= e_1\sgn(\widetilde{D})_+e_0$ and
consequently  $u^*= e_0\sgn(\widetilde{D})_-e_1$, with $\tilde D$ the (invertible)
Dirac operator on $\tM$.
 By Proposition \ref{rel-functoriality} and Remark \ref{LMP}, the degree 1 part of $\Ch_{\Gamma}^{del}(\varrho(\widetilde{D}))$ is given by \begin{equation}\label{ch-del-1}\overline{\mathrm{TR}}^{del}\left(\int_{0}^1\Tr\left((2p_t-1)\dot{p}_tdp_t\right)dt\right).\end{equation}
Observe  first that $\dot{p}_t=\begin{pmatrix}
-\pi\cos\cdot \sin\left(\frac{\pi}{2}t\right)e_0 &\frac{\pi}{2} (\cos^2-\sin^2)\left(\frac{\pi}{2}t\right)u^*\\ \frac{\pi}{2} (\cos^2-\sin^2)\left(\frac{\pi}{2}t\right)u& \pi\cos\cdot \sin\left(\frac{\pi}{2}t\right)e_1
\end{pmatrix}$. Moreover recall by Definition \ref{connes-lott} that the differential $d$ involves also the auxiliary variable $X$ and that $\overline{\mathrm{TR}}^{del}$ is zero on the terms in which $X$ appears precisely once. Hence we have that $dp_t$ is given by $\begin{pmatrix}
0& \cos\cdot\sin\left(\frac{\pi}{2}t\right)[\nabla,u^*]\\ \cos\cdot\sin\left(\frac{\pi}{2}t\right)[\nabla,u]& 0
\end{pmatrix}$, up to terms involving $X$. 

Now an elementary but tedious calculation shows that the integrand of \eqref{ch-del-1} is given by 
\begin{equation*}
\begin{split}
  \pi\cos^3&\cdot\sin(\cos^2-\sin^2)\left(\frac{\pi}{2}t\right)e_0u^*[\nabla,u]-\frac{\pi}{2}\cos\cdot\sin\cdot(\cos^2-\sin^2)\left(\frac{\pi}{2}t\right)u^*[\nabla,u]\\
  &+ \pi \cos^3 \cdot\sin^3\left(\frac{\pi}{2}t\right)u^*e_1[\nabla,u]\\
           &-\pi\cos^3\cdot\sin^3\left(\frac{\pi}{2}t\right)ue_0[\nabla,u^*]+\pi\cos\cdot\sin^3\left(\frac{\pi}{2}t\right)e_1u[\nabla,u^*]\\
  &-\frac{\pi}{2}\cos\cdot\sin^3\cdot(\cos^2-\sin^2)\left(\frac{\pi}{2}t\right)u[\nabla,u^*].
\end{split}
\end{equation*}
 Observe that the second and the last term have zero integral. By using that $u^*e_1= e_0u^*$ and that $ue_0=e_1u$ we obtain that the integrand in \eqref{ch-del-1} is given by 
 \begin{equation*}
 \begin{split}
 \pi\cos^3\cdot \sin\left(\frac{\pi}{2}t\right)e_0u^*[\nabla,u]-\pi\cos\cdot\sin^3\left(\frac{\pi}{2}t\right)e_1u[\nabla,u^*]
 \end{split}
 \end{equation*}
 which in turns leads to $\frac{1}{2}\mathrm{TR}^{del}(e_0u^*[\nabla,u]-e_1u[\nabla,u^*])$. Now observe that
 \begin{equation*}
   \begin{split}
     e_0u^*[\nabla,u]&= \sgn(\widetilde{D})_-[\nabla,\sgn(\widetilde{D})_+];\\ 
     e_1u[\nabla,u^*]&= \sgn(\widetilde{D})_+[\nabla,\sgn(\widetilde{D})_-].
   \end{split}
 \end{equation*}
 Thus we finally obtain that the degree 1 part of the delocalized Chern character of $\varrho(g)$ is given by 
 $$ \Ch_\Gamma^{del}( \varrho(g))_{[1]}=\frac{1}{2}\mathrm{STR}^{del}\left(\sgn(\widetilde{D})[\nabla, \sgn(\widetilde{D})]\right)= -\frac{1}{2}\mathrm{STR}^{del}\left([\nabla, \sgn(\widetilde{D})]\sgn(\widetilde{D})\right)$$
 where $\mathrm{STR}^{del}$ denotes the delocalized supertrace with respect to
 the natural grading of the spinor bundle of an even dimensional manifold and
 the minus in the last term appeared because $\sgn(\widetilde{D})$ is odd with
 respect to this grading. 
Using \eqref{eq:nabla_bar}, we can explicitly write  down the Schwartz kernel of $\sgn(\widetilde{D})[ \nabla, \sgn(\widetilde{D})]$,
obtaining
  \begin{multline*}
\mathrm{STR}^{del}(\sgn(\tilde D)[\nabla,\sgn(\tilde D)])\\
  = \sum_{\gamma_0,\gamma_1\not= e} \int_{\tM\times\tM} \mathrm{STr} (\chi(x)\sgn(\widetilde{D})(x\gamma_0,y)h(y)\sgn(\widetilde{D})(y\gamma_1,x))dxdy\; \gamma_0 d\gamma_1 .
    \end{multline*}
with $\chi$ 
  the characteristic function of a fundamental domain of the covering
  projection $\tM\to M$.
\end{example}

  \begin{corollary}
    In the setting of Example \ref{ex:degree_1}, let $\tau$ be a cyclic
    cocycle of degree $1$ of polynomial growth with support the conjugacy
    class $\innerprod{x}\ne \innerprod{e}$ of $\Gamma$. Then
    \begin{multline*}
      \varrho_\tau(g) =\frac{1}{2}
      \sum_{\gamma_0\gamma_1\in\innerprod{x}} \int_{\tM\times\tilde
        M} \mathrm{STr}(\chi(x_0)\sgn(\tilde D)(x_0\gamma_0,x_1)h(x_1)\cdot\\
      \cdot \sgn(\tilde D)(x_1\gamma_1,x_0))\tau(\gamma_0,\gamma_1)\,
      dx_0\,dx_1.
    \end{multline*}
  \end{corollary}

  \begin{corollary}
    In the situation of Example \ref{ex:degree_1}, let $\tau$ be again a
    delocalized cyclic cocycle of degree $1$ supported on a conjugacy class
    $\innerprod{x}\ne \innerprod{e}$. Then the delocalized higher rho
    invariant $\varrho_\tau(g)$ coincides with the one constructed in
    \cite[Section 4.10.2]{Lott2}, i.e.
    \begin{equation*}
      \varrho_\tau(g) =  -\frac{1}{2}\innerprod{\mathrm{STR}^{del}\left([\nabla,
        \widetilde{D}]\widetilde{D}^{-1}\right),\tau} 
    \end{equation*}
  \end{corollary}

\begin{proof}
In the following, let $\sim$ denote equality up
    super-commutators. Then
    \begin{equation}\label{eq:us_to_Lott}
      \begin{split}
        & [\nabla, \sgn(\widetilde{D})]\sgn(\widetilde{D})=\\
        =&[\nabla, \sgn(\widetilde{D})]\sgn(\widetilde{D})+ \frac{1}{2}[\nabla,|\widetilde{D}|]|\widetilde{D}^{-1}|- \frac{1}{2}[\nabla,|\widetilde{D}|]|\widetilde{D}^{-1}|=\\
        =&\frac{1}{2}\left(|\widetilde{D}^{-1}||\widetilde{D}|[\nabla, \sgn(\widetilde{D})]\sgn(\widetilde{D})+[\nabla,|\widetilde{D}|]\sgn(\widetilde{D})\sgn(\widetilde{D})|\widetilde{D}^{-1}|\right)+\\
        &+\frac{1}{2}\left([\nabla, \sgn(\widetilde{D})]|\widetilde{D}||\widetilde{D}^{-1}|\sgn(\widetilde{D})+[\nabla,|\widetilde{D}|]|\widetilde{D}^{-1}|\sgn(\widetilde{D})\sgn(\widetilde{D})\right)\sim\\
        \sim&\frac{1}{2}\left(|\widetilde{D}|[\nabla, \sgn(\widetilde{D})]\sgn(\widetilde{D})|\widetilde{D}^{-1}|+[\nabla,|\widetilde{D}|]\sgn(\widetilde{D})\sgn(\widetilde{D})|\widetilde{D}^{-1}|\right)+\\
        &+\frac{1}{2}\left([\nabla, \sgn(\widetilde{D})]|\widetilde{D}||\widetilde{D}^{-1}|\sgn(\widetilde{D})+\sgn(\widetilde{D})[\nabla,|\widetilde{D}|]|\widetilde{D}^{-1}|\sgn(\widetilde{D})\right)=\\
        =&\frac{1}{2}[\nabla, |\widetilde{D}|\sgn(\widetilde{D})]\sgn(\widetilde{D})|\widetilde{D}^{-1}|+ \frac{1}{2}[\nabla, \sgn(\widetilde{D})|\widetilde{D}|]|\widetilde{D}^{-1}|\sgn(\widetilde{D})=\\
        =& [\nabla, \widetilde{D}]\widetilde{D}^{-1}.
      \end{split}
    \end{equation}
    where in the second to last line we used the Leibniz formula for
    commutators. Note that we apply super-commutators in the algebra of
      all pseudodifferential operators, not only operators of degree
      $0$. However, because our \emph{delocalized} supertrace has the trace
      property on this larger algebra by Proposition
      \ref{prop:trace_on_all_pseudos}, Equation \eqref{eq:us_to_Lott} implies
      that the degree 1 part of the delocalized Chern character of
    $\varrho(g)$ in the even dimensional case is given by
 $$ \Ch_\Gamma^{del}(\varrho(g))_{[1]}=
   -\frac{1}{2}\mathrm{STR}^{del}\left([\nabla,\sgn(\tilde D)]\sgn(\tilde
     D)\right) = -\frac{1}{2}\mathrm{STR}^{del}\left([\nabla,
   \widetilde{D}]\widetilde{D}^{-1}\right) $$ which, after paired with a cyclic
 cocycle of degree 1, gives exactly the expression appearing in \cite[Section
 4.10.2]{Lott2}.

\end{proof}

 We now pass to the general case of cyclic cocycles of odd degree, say $2k+1$. 
 	First of all, we need to identify the $2k+1$-degree part of  $\Ch_{\Gamma}^{del}(\varrho(\widetilde{D}))$; recall that this
	is given by 
 	\begin{equation}\label{ch-del-2k+1}\overline{\mathrm{TR}}^{del}\left(\int_{0}^1\Tr\left((2p_t-1)\dot{p}_t(dp_t)^{2k+1}\right)dt\right).\end{equation}
 Observe that $$dp_t=[\nabla, p_t]+ Xp_t+p_tX$$ and then that $$(dp_t)^2= [\nabla, p_t]^2 + [\nabla, p_t]p_tX +Xp_tX+Xp_t[\nabla, p_t]+p_t\Theta p_t.$$
 	Taking into account that $\textbf{T}_1X\textbf{T}_2=0$  and that $\overline{\mathrm{TR}}^{del}$ is zero on operators of the form $\textbf{T}_1X$ or $X\textbf{T}_2$, we can observe the following  fact:
 $(2p_t-1)\dot{p}_tdp_t= (2p_t-1)\dot{p}_t([\nabla, p_t]+p_tX)$, hence, when we calculate $dp_t(dp_t)^{2k}$ inside the integrand of \eqref{ch-del-2k+1}, we are reduced to consider 
 		$$([\nabla, p_t]+p_tX) ([\nabla, p_t]^2 + [\nabla, p_t]p_tX +Xp_tX+Xp_t[\nabla, p_t]+p_t\Theta p_t)^{2k}.$$
Thus, once we apply the delocalized trace $\overline{\mathrm{TR}}^{del}$ the only products that survive are the following ones: 
\begin{itemize}
	\item if they start with $[\nabla, p_t]$ then they contain after this factor, in all permutations, elements of the form $[\nabla, p_t]^{2j}$, $[\nabla, p_t]p_tX (Xp_tX)^hXp_t[\nabla, p_t]=[\nabla, p_t](p_t\Theta p_t)^{h+1}[\nabla, p_t]$ or $(p_t\Theta p_t)^l$, so that the total degree is $2k+1$;
	\item if they start with $p_tX$, then, bearing in mind that $\textbf{T}_1X\textbf{T}_2=0$, such a term needs to be multiplied by an element of the form $(Xp_tX)^hXp_t[\nabla, p_t]$. We then obtain in the first place $(p_t\Theta p_t)^{h+1}[\nabla, p_t]$, which will be followed by element of the form as in the previous point, in all permutations, so that the total degree is $2k+1$.
\end{itemize}
Summarizing $(2p_t-1)\dot{p}_t(dp_t)^{2k+1}$ is the sum of products starting with $(2p_t-1)\dot{p}_t$, followed by all the permutations of $2h+1$ copies of $[\nabla, p_t]$ and $l$ copies of $p_t\Theta p_t$, for all $h,l\in \NN$ such that $2l+2h+1=2k+1$. 
Let us denote by $Q_t$ the Schwartz kernel of $(2p_t-1)\dot{p}_t$ and by $P_t$ the Schwartz kernel of $p_t$. Then we know that the Schwartz kernel of $[\nabla, p_t]$ is given by 
\begin{equation}\label{commutator}\sum_{\gamma\in \Gamma}(h(x)-R^*_{\gamma^{-1}}h(y)) R^*_{(\gamma,e)}P_t(x,y) d\gamma\end{equation}
whereas the Schwartz kernel of $p_t\Theta p_t$ is given by 
\begin{equation}\label{pthetap}\begin{split}&\sum_{\gamma_1,\gamma_2\in \Gamma} \int_{\tM}P_t(x,z)h(z) R^*_{\gamma_1}h(z)R^*_{(\gamma_1\gamma_2,e)}P_t(z,y) dvol_g(z)\otimes d\gamma_1d\gamma_2=\\
=&\sum_{\gamma_1,\gamma_2\in \Gamma}
\int_{\tM}R^*_{(\gamma_1,e)}P_t(x,z)R^*_{\gamma^{-1}_1}h(z) h(z)R^*_{(e,
  \gamma^{-1}_2)}P_t(z,y) dvol_g(z)\otimes d\gamma_1d\gamma_2
\end{split}
\end{equation}

Without working out the precise result, we obtain the following abstract
  formula for \eqref{ch-del-2k+1}.
\begin{proposition}\label{prop:rho_formula_even}
    \begin{multline}
\Ch_{\Gamma}^{del}(\varrho(\widetilde{D}))_{2k+1}=
\sum_{\gamma_0\dots\gamma_{2k+1}\neq e}\int_0^1dt
\int_{\tM\times\dots\times\tM}\Tr\,\Big(\chi(x_0)R^*_{\gamma_0}Q_t(x_0,
  x_1)
\\\left.\sum R^*_{(\alpha^i_0,\dots,\alpha^i_{2k+1})}C^i_t(x_1, \dots,x_{2k+1},x_0)\right)dvol_{g}^{2k+1}\gamma_0d\gamma_1\dots d\gamma_{2k+1}
    \end{multline}
    where $\alpha^i_j$ is a product of elements in
    $\{\gamma_0,\gamma_1, \dots,\gamma_{2k+1}\}$ and the $C^i_t$ are the
    convolution of all the permutation of $2h_i+1$ copies of
    \eqref{commutator} and $l_i$ copies of \eqref{pthetap}, for all
    $h_i,l_i\in \NN$ such that $l_i+h_i=k$ (the central sum is over all these
    terms). 
  \end{proposition}
  Now it is immediate to obtain an expression like the one in
    formula \eqref{formula-delocalized} when we pair with $\tau$, a cyclic
    cocycle of degree $2k+1$ of polynomial growth with support contained in a
    conjugacy class $\langle x \rangle$ of $\Gamma$.

 \chapter[Higer rho numbers from relative cohomology]{Higher rho numbers from relative cohomology $H^*(M\to B\Gamma)$ }\label{section8}

When looking at the Higson-Roe sequence $\cdots\to K_*(M)\to
K_*(C^*\Gamma)\to \SG_{*-1}^\Gamma ({\tM}) \to K_{*-1}(M)\to\cdots$ we realize that, at least morally, $K(C^*\Gamma)$
is closely related to $K_*(B\Gamma)$. Consequently, there should be a second
source of information on the relative term (the structure set  $\SG_{*-1}^\Gamma ({\tM})$) in terms of  the difference between $M$ and $B\Gamma$. In this section,
we make this idea precise: we construct a pairing between the analytic structure set $\SG_{*}^\Gamma ({\tM})$
and the relative cohomology $H^*(M\to B\Gamma;\reals)$, at least for suitable
groups. To achieve this and to arrive at explicit formulae, we start with a convenient cocycle model for $H^*(M\to B\Gamma;\reals)$, using
suitable Alexander-Spanier cochains on $\tM$.

\section{Alexander-Spanier cochains and singular cohomology}
\label{sec:AS_cochains}

\begin{definition}
	Let $C^q_{AS}(\tM)^{\Gamma}$ be the vector space of all (not
          necessarily continuous) functions
	$\varphi\colon \tM^{q+1}\to \RR$ which are invariant with respect to the
	diagonal action of $\Gamma$ on $\tM^{q+1}$.  Define
	$\delta\colon C^q_{AS}(\tM)^{\Gamma}\to C^{q+1}_{AS}(\tM)^{\Gamma}$ by
        the standard formula
	\[
	(\delta\varphi)(x_0,\dots,x_{q+1})=\sum_{i=0}^{q+1}(-1)^{i}\varphi(x_0,\dots,\hat{x}_i,\dots,
	x_{q+1})
	\]
	which makes $\{C^\bullet_{AS}(\tM)^{\Gamma},\delta\}$ a cochain complex.
	
	An important subcomplex of $C^\bullet_{AS}(\tM)^{\Gamma}$ is given by
	locally zero cochains.  An element
	$\varphi\in C^\bullet_{AS}(\tM)^{\Gamma}$ is said to be locally zero if
	there is an open neighborhood of the diagonal $\tM\subset \tM^{q+1}$ on which $\varphi$ vanishes.
	We denote by $C^q_{AS,0}(\tM)^{\Gamma}$ the vector
	space of all such functions, forming a subcomplex of
	$C^\bullet_{AS}(\tM)^{\Gamma}$.
\end{definition}
We thus obtain the following exact sequence of cochain complexes
\begin{equation}\label{AScomplexes}
\xymatrix{0\ar[r]& C^\bullet_{AS,0}(\tM)^{\Gamma}\ar[r]& C^\bullet_{AS}(\tM)^{\Gamma}\ar[r]&\overline{C}^\bullet_{AS}(\tM)^{\Gamma}\ar[r]& 0}
\end{equation}
where $\overline{C}^\bullet_{AS}(\tM)^{\Gamma}$ is the quotient
complex. Let us denote in the following way the associated
cohomology groups, forming the associated long exact sequence,
\begin{equation}\label{AScohomologies}
\xymatrix{\dots\ar[r]& H^\bullet_{AS,0,\Gamma}(\tM)\ar[r]& H^\bullet_{AS,\Gamma}(\tM)\ar[r]&\overline{H}^\bullet_{AS,\Gamma}(\tM)\ar[r]& \dots}.
\end{equation}

We first identify $H^\bullet_{AS,\Gamma}(\tM)$ with $H^\bullet(\Gamma)$. In order to do that fix a point $z$ in $\tM$ and consider the inclusion
$i\colon\Gamma \hookrightarrow \tM$ given by $i(\gamma)=z\gamma$. Observe that $i$ is $\Gamma$-equivariant.

\begin{proposition}\label{isoas}
	The map $i^*\colon C^\bullet_{AS}(\tM)^\Gamma\to
	C^\bullet_{AS}(\Gamma)^\Gamma$	 induces an isomorphism in cohomology and in fact
	is a chain homotopy equivalence.
\end{proposition}
\begin{proof}
	Let us fix a bounded measurable fundamental domain $\mathcal{F}$ for the action of $\Gamma$ on $\tM$ such that $z\in \mathcal{F}$.
	Let $\mu\colon \tM\to \Gamma$ be the map with $\mu(x)=\gamma$ for $x\in \mathcal{F}\gamma$, and notice that $\mu$ is $\Gamma$-equivariant.
	Observe that $\mu\circ i=\id_\Gamma$ and $r:=i\circ \mu$
	is the map which sends $x\in \mathcal{F}\gamma$ to $z\gamma$.
	To prove that $i^*$ induces an isomorphism in cohomology we construct a
	cochain homotopy $K\colon C^k_{AS}(\tM)^\Gamma \to
	C^{k-1}_{AS}(\tM)^\Gamma$ between $\id$ and $r^*$, setting
	\[
	K(\varphi)(x_0,\dots, x_{k-1}):= \sum _{j=0}^{k-1}(-1)^{j+1}\varphi(r(x_0),\dots,r(x_j),x_j,\dots,x_{k-1})
	\]
	for $\varphi\in C^k_{AS}(\tM)^\Gamma $. Of course, since $r$ is $\Gamma$-equivariant, so is $K$.
		We have 
	\begin{equation}
	\begin{split}
	\delta K(\varphi)(x_0,\dots,x_k)=&\sum_{i=0}^k(-1)^i K(\varphi)(x_0,\dots,\hat{x}_i,\dots,x_k)=\\
	=&\sum_{i=0}^k(-1)^i\left(\sum
	_{j=0}^{i-1}(-1)^{j+1}\varphi(r(x_0),\dots,r(x_j),x_j,\dots,\hat{x}_i,\dots,x_{k})
	\right.\\
	&+\left. \sum _{j=i+1}^{k}(-1)^{j}\varphi(r(x_0),\dots,\widehat{r(x_i)},\dots,r(x_j),x_j,\dots,x_{k})\right),
	\end{split}
	\end{equation}
	\begin{equation}
	\begin{split}
	K(\delta\varphi)(x_0,\dots,x_k)=& \sum _{j=0}^{k}(-1)^{j+1}\delta\varphi(r(x_0),\dots,r(x_j),x_j,\dots, x_k)=\\
	=&\sum _{j=0}^{k}(-1)^{j+1}\left(\sum_{i=0}^j(-1)^i\varphi(r(x_0),\dots, \widehat{r(x_i)},\dots, r(x_j), x_j\dots, x_k)\right.\\
	&+ \left. \sum_{i=j}^k(-1)^{i+1}\varphi(r(x_0),\dots,r(x_j),x_j,\dots,\hat{x}_i,\dots, x_k)\right).
	\end{split}
	\end{equation}
	Comparing the terms of the two expressions we see that they cancel out
	perfectly, except for the term for $i,j=0$ from the second sum which gives
	$-\varphi(x_0,\dots,x_k)$ 
	and the term for  $i,j=k$ from the second sum which gives
	$\varphi(r(x_0),\dots,r(x_k))$, which means that $\delta K+K\delta=  r^*-\id$. The proposition follows.
\end{proof}

\begin{remark}
	We have the following variant of the cochain homotopy $K$ in the proof of Proposition \ref{isoas}, which
	preserves continuous or smooth Alexander-Spanier cochains.
	Let $h\colon\tM\to [0,1]$ be  a compactly supported smooth cut-off function as in Section \ref{lifting-delocalized}.
	Then define
	\begin{multline*}
          K\varphi(x_0,\dots,x_k)\\
          :=\sum_{\gamma_0,\dots,\gamma_k\in\Gamma^{k+1}}
	R^*_{\gamma_1}h(x_0)\cdots R^*_{\gamma_k}h(x_k) \sum_{j=0}^k (-1)^{j+1} \varphi(z\gamma_0,\dots,z\gamma_j,x_j,\dots,x_k)
	\end{multline*}
	The construction in \ref{isoas} is the same construction with $h$ the
	characteristic function of the fundamental domain $\mathcal{F}$. The above argument
	again applies, using that $\sum_{\gamma\in\Gamma} R^*_{\gamma}h(x)=1$ for each $x\in \tM$.
	
\end{remark}

\begin{remark}
	Notice that $C^\bullet_{AS}(\Gamma)^\Gamma$ is precisely the standard cobar
	complex used to define $H^\bullet(\Gamma;\reals)$.
\end{remark}

Next we identify the remaining cohomology groups in \eqref{AScohomologies} as
certain singular cohomology groups.
Let $B\Gamma$ be a classifying space for $\Gamma$, $E\Gamma\to B\Gamma$ a
universal covering (a universal principal $\Gamma$-bundle). Let $\mu\colon
M\to B\Gamma$ the map classifying the $\Gamma$-covering $\tM\to M$, with $\Gamma$-equivariant lift $\tilde \mu\colon \tM\to E\Gamma$.
Using naturality of the constructions, we get a commutative diagram of cochain complexes
\begin{equation*}
\begin{CD}
C^\bullet_{AS}(\Gamma)^\Gamma @>=>> C^\bullet_{AS}(\Gamma)^\Gamma\\
@AA{\sim}A @AA{\sim}A\\
C^\bullet_{AS}(E\Gamma)^\Gamma @>{\tilde\mu^*}>{\sim}> C^\bullet_{AS}(\tM)^\Gamma\\
@VV{f_E}V @VV{f_M}V\\
\overline{C}_{AS}^\bullet(E\Gamma)^\Gamma @>{\overline{\tilde\mu}^*}>>
\overline{C}^\bullet_{AS}(\tM)^\Gamma
\end{CD}
\end{equation*}
The first vertical maps are induced by the inclusions $\Gamma\to \tM$
 and $\Gamma\to E\Gamma$ given by $g\mapsto zg$ and  $ g\mapsto \tilde\mu(z)g$
 respectively. They are both
chain homotopy equivalences by Proposition \ref{isoas}. Therefore
also $\tilde\mu^*$ is a chain homotopy equivalence.
For the quotients by the locally zero cochains, we note in addition that
the pull-back defines canonical vertical isomorphisms
\begin{equation}\label{eq:AS_up_down_iso}
\begin{CD}
\overline{C}_{AS}^\bullet(B\Gamma) @>{\overline{\mu}^*}>>
\overline{C}^\bullet_{AS}(M)\\
@VV{\iso}V @VV{\iso}V\\
\overline{C}_{AS}^\bullet(E\Gamma)^\Gamma @>{\overline{\tilde\mu}^*}>>
\overline{C}^\bullet_{AS}(\tM)^\Gamma
\end{CD}
\end{equation}
Now, $f_M$ is surjective with kernel $C^\bullet_{AS,0}(\tM)^\Gamma$. We
get a sequence of cochain maps
\begin{multline}\label{eq:chain_of_cones}
C^\bullet_{AS,0}(\tM)^\Gamma \xrightarrow{\sim} C^\bullet(f_M)
\xleftarrow[\tilde\mu^*]{\sim} C^\bullet(f_M\circ \tilde \mu^*) = C^\bullet
(\overline{\tilde \mu}^*\circ f_E)\\
\xrightarrow{f_E} C^\bullet(\overline{\tilde
	\mu}^*) \iso C^\bullet(\overline\mu) =: C^\bullet_{AS}(M\to B\Gamma).
\end{multline}
The first map is induced by the inclusion and is the standard chain homotopy equivalence
between the mapping cone of a surjection and the kernel. If the extension of
chain complexes is $0\to A^\bullet\xrightarrow{j} B^\bullet\xrightarrow{q} C^\bullet\to 0$, it sends
$a\in A^n$ to $(j(a),0)\in B^n\oplus C^{n-1}= C^n(q)$, which is the mapping cone complex.
The second map is the standard map induced by $g$ from the mapping cone of a
composition $A^\bullet \xrightarrow{f} B^\bullet \xrightarrow{g} C^\bullet$ to
the one of $f$: $C^n(g\circ f)= A^n\oplus C^{n-1} \ni (a,c)\mapsto (f(a),c)\in
B^n\oplus C^{n-1} = C^n(g)$. It is standard and easy to see that this is a
chain homotopy equivalence or homology isomorphism if and only if $f\colon
A^\bullet\to B^\bullet$ is. Because of Proposition \ref{isoas}, the map
induced by $\tilde \mu^*$ here is a chain homotopy equivalence (and
implicitly, we have in mind to use a chain homotopy inverse of it). 
The next map, $f_E$, is obtained the same way.
The isomorphism at the end implements the isomorphism
\eqref{eq:AS_up_down_iso}. Finally, we use the definition of the relative
cohomology of $M\to B\Gamma$ 
defined via Alexander-Spanier cochains. Because $M$ is a manifold and
$B\Gamma$ a CW-complex, this is canonically isomorphic to the relative
singular cohomology.

\begin{proposition}\label{Higson-lemma}(Higson's Lemma)
	The sequence of cochain maps of \eqref{eq:chain_of_cones} defines a
	canonical isomorphism
	\begin{equation*}
	H^\bullet_{AS,0,\Gamma}(\tM)^\Gamma\xrightarrow{\iso}
	H^\bullet(M\xrightarrow{u} B\Gamma).
	\end{equation*}
\end{proposition}
\begin{proof}
It only remains to prove that $f_E\colon
        C^\bullet_{AS}(E\Gamma)^\Gamma\to
	\overline{C}^\bullet_{AS}(E\Gamma)^\Gamma$ is a cohomology isomorphism. Note
	that the cohomology in both cases is the cohomology of $\Gamma$ in a
	canonical way, and therefore there already is a canonical isomorphism
	between the two cohomology groups. We will explain this, and will show that
	$f_E$ induces precisely this canonical isomorphism. This also makes the
	identification asserted in the proposition truly canonical.
	
	We start with the free simplicial set generated by the vertices $\Gamma$, considered as
	a (discrete) set. Being free, this is contractible, and the associated chain
	complex $\dots \integers[\Gamma^3]\to \integers[\Gamma^2]\to \integers[\Gamma]\to
	\integers$ is acyclic. Note that it consists of free $\integers[\Gamma]$-modules,
	therefore is a free resolution of the trivial $\integers[\Gamma]$-module
	$\integers$. 
	Note that our cochain equivalence with the bar complex
	$C^\bullet_{AS}(\Gamma)^\Gamma$ is obtained from a chain equivalence of the
	underlying free resolutions, and also $\tilde\mu^*$ is such a cochain
	equivalence.
	
	To get the canonical identification of
	$\overline{H}_{AS}^\bullet(E\Gamma)^\Gamma$ with $H^\bullet(\Gamma)$,
        we use singular homology as intermediate step. Let $\Delta^\bullet(X)$
        be the
	singular cochain complex of a space $X$. Let $\Delta^\bullet_{0}(X)$
        be the
	subcomplex of cochains $\phi$ which are locally $0$, i.e.~by definition for which an
	open covering $\{U_i\}$ of $X$ exists such that $\phi$ vanishes on all singular simplices $\sigma\colon
	\Delta^\bullet\to X$ whose image is contained in one of the $U_i$. Define
	$\overline{\Delta}^\bullet(X):=\Delta^\bullet(X)/\Delta_0^\bullet(X)$. 
	
	There is a canonical map $C^\bullet_{AS}(X)\to \Delta^\bullet(X)$ dual to
	the map between the singular chain complex $\Delta_\bullet(X)$ and the
	Alexander-Spanier chain complex $\integers[X^{\bullet+1}]$ which sends a
	singular simplex to the tuple of its endpoints. It induces
	$\overline{C}^\bullet_{AS}(X)\to \overline{\Delta}^\bullet(X)$.
	
	It is a standard result, compare \cite[Section 8.8]{Massey}, that for spaces with a nice local topology, in
	particular for CW-spaces like our $B\Gamma$ or $M$, these maps are chain
	homotopy equivalences, which we combine here with the projection induced
	isomorphisms with equivariant cochains on $E\Gamma$
	\begin{equation}\label{eq:sing_AS_comp}
	\begin{CD}
	\overline{C}^\bullet_{AS}(B\Gamma) @>{\sim}>>
	\overline{\Delta}^\bullet(B\Gamma) @<{\sim}<< \Delta^\bullet(B\Gamma)\\
	@VV{\iso}V @VV{\iso}V @VV{\iso}V\\
	\overline{C}^\bullet_{AS}(E\Gamma)^\Gamma @>{\sim}>>
	\overline{\Delta}^\bullet(E\Gamma)^\Gamma @<{\sim}<<
	\Delta^\bullet(E\Gamma)^\Gamma\\
	\end{CD}
	\end{equation}
	
	Finally, as $E\Gamma$ is contractible, the singular chain complex of
	$E\Gamma$ provides another resolution of $\integers$ (by free
	$\integers\Gamma$-modules), and the canonical map to the Alexander-Spanier
	chain complex (which is a $\integers[\Gamma]$-module map) is automatically
	a chain equivalence. After taking $\Hom_{\integers[\Gamma]}(\cdot,\reals)$
	this induces, once again, the standard cohomology isomorphism for
	different ways to compute group cohomology. Combining this information
	with \eqref{eq:sing_AS_comp} we obtain the diagram
	\begin{equation*}
	\begin{CD}
	C^\bullet_{AS}(E\Gamma)^\Gamma @>{\sim}>>
	\Delta^\bullet(E\Gamma)^\Gamma\\
	@VV{f_E}V  @VV{\sim}V\\
	\overline{C}^\bullet_{AS}(E\Gamma)^\Gamma @>{\sim}>> \overline{\Delta}^\bullet(E\Gamma)^\Gamma.
	\end{CD}
	\end{equation*}
	It follows that our map $f_E$ is a chain equivalence, and, because all the
	other maps induce the canonical isomorphisms, that it induces the
        canonical isomorphism between different ways to compute group
        cohomology. 
\end{proof}

\begin{corollary}\label{AStoRel}
 There exists a canonical and natural commutative diagram 
 \begin{equation}
	\xymatrix{\cdots\ar[r]& H^\bullet_{AS,0,\Gamma}(\tM)\ar[r]\ar[d]& H^\bullet_{AS,\Gamma}(\tM)\ar[r]\ar[d]&\overline{H}^\bullet_{AS,\Gamma}(\tM)\ar[r]\ar[d]& \cdots\\
		\cdots\ar[r]& H^\bullet(M\xrightarrow{u} B\Gamma)\ar[r]& H^\bullet(B\Gamma)\ar[r]&H^\bullet(M)\ar[r]& \cdots}
 \end{equation}
 where the vertical lines are isomorphisms.
\end{corollary}

\begin{proposition}\label{prop:smooth_sym}
	In \eqref{AScomplexes}, the inclusion of the subcomplexes of smooth and
	antisymmetric cochains induce cohomology isomorphism. Here, a cochain $\varphi$ is
	antisymmetric if $\varphi(x_0,\dots,
	x_n)=\sgn(\sigma)\varphi(x_{\sigma(0)},\dots, x_{\sigma(n)})$, for all the
	permutations $\sigma\in \mathfrak{S}_{n+1}$. 
\end{proposition}
\begin{proof}
	By \cite[Lemma 1.1 and Lemma 1.4]{ConnesMoscovici}, for calculating
	$\overline{H}^\bullet_{AS,\Gamma}({\tM})$ one can use the subcomplex
	of smooth antisymmetric Alexander-Spanier cochains $\varphi$.
	Analogously $H^\bullet(\Gamma)$ can be calculated by means of the subcomplex
	of antisymmetric group cochains $c$. This implies by the five lemma the same
	for  $H^{\bullet}_{AS,0, \Gamma}({\tM})$. 
\end{proof}

\begin{definition}
Define the subcomplex
	$C^{\bullet}_{AS,0,pol}({\tM})^\Gamma$ of
	$C^{\bullet}_{AS,0}({\tM})^\Gamma$ consisting of smooth skew-symmetric cochains  of polynomial growth, namely cochains $\varphi$ such that 
	$$|\varphi(x_0,\dots,x_n)|\leq K \prod_i(1+d(x_i,x_{i+1}))^k$$
	for some $K>0$ and $k\in \NN$.
      \end{definition}
      
\begin{proposition}\label{AS-pol}
In our situation, the canonical inclusion induced map
\begin{equation}\label{eq:rel_pol_map}
    H^*(C^\bullet_{AS,0,pol}({\tM})^\Gamma) \to  H^*(C^\bullet_{AS,0}({\tM})^\Gamma)
\end{equation}
is an isomorphism if and only if $H^*_{pol}(\Gamma)\to H^*(\Gamma)$ is an
isomorphism. 
In particular, \eqref{eq:rel_pol_map} is an isomorphism if $\Gamma$ is a group with a polynomial combing, e.g.~a hyperbolic
	group or more generally an automatic group, or if $\Gamma$ has polynomial growth.
\end{proposition}
\begin{proof}
First, our proof of Proposition \ref{isoas} works also
	for the subcomplexes of polynomial growth. This implies that
	$C^\bullet_{AS,pol}(\tM)^\Gamma\to C^\bullet_{AS}(\tM)^\Gamma$ is a
	cohomology isomorphism if and only if
        $C^\bullet_{AS,pol}(\Gamma)^\Gamma\to C^\bullet_{AS}(\Gamma)^\Gamma$
        is one, i.e.~if and only if $H^*_{pol}(\Gamma)\to H^*(\Gamma)$ is an
        isomorphism. Finally, we get a diagram with pullback isomorphisms
	\begin{equation*}
	\begin{CD}
	\overline{C}^\bullet_{AS}(M)  @>{\iso}>>
	\overline{C}^\bullet_{AS,pol}(\tM)^\Gamma\\
	@VV{=}V @VV{i_{pol}}V\\
	\overline{C}^\bullet_{AS}(M) @>{\iso}>> \overline{C}^\bullet_{AS}(\tM)^\Gamma,
	\end{CD}
	\end{equation*}
	because in $\overline{C}$ only equivalence classes of cochains defined locally
	near the diagonal
	play a role, and locally the polynomial growth condition becomes a boundedness
	condition, but due to the compactness of $M$ locally every smooth cochain
	automatically is bounded, so that indeed
        $\overline{C}^\bullet_{AS,pol}(\tilde 
	M)^\Gamma$ is isomorphic to $\overline{C}^\bullet_{AS}(M)$.
	The first statement then follows from the five lemma.

 By  \cite[Corollary
	5.3]{Meyer}, if $\Gamma$ admits a combing of polynomial growth, in
        particular if $\Gamma$ is hyperbolic, then $H^*_{pol}(\Gamma)\to
        H^*(\Gamma)$ is an isomorphism. The same holds for
 groups of
        polynomial growth by \cite{Meyer2}. This concludes the proof.
\end{proof}

\section{Lifting $H^*(M\to B\Gamma)$ to  relative cyclic cocycles}

We continue to fix a smooth compact manifold $M$ and a $\Gamma$-covering
$\tM\to M$.

\begin{lemma}\label{lem:tauX}
	Assume that $\chi\in C^k_{{\rm AS},0}(\tM)^\Gamma$ smooth and antisymmetric and $A_0,\dots,
	A_k \in \Pdo^0_{\Gamma,c}(\tM,{\tE})$ are given. Then the (formal)
	expression
	\begin{multline}\label{eq:def_tauX}
	\tau_\chi(A_0,\dots,A_k)=\int_{\mathcal{F}}\Tr\left(\int_{\tM^k}A_0(x_0,x_1)A_1(x_1,x_2)\dots
          A_k(x_k,x_{0})\chi(x_0,\dots,x_k)\right.\\
     \left.    dvol_g^k(x_1,\dots,x_k)\right)\,dvol_g(x_0)
	\end{multline}
	where $A_j(x,y)$ is the Schwartz kernel of $A_j$ and $\mathcal{F}$ is
	a measurable fundamental domain for $\tM\to M$
	\begin{enumerate}
		\item makes sense and gives a well defined value in $\CC$
		(independent of $\mathcal{F}$),
	\item vanishes if all the $A_j$ are local operators, i.e.~the support
		of the Schwartz kernel is supported on the diagonal (e.g.~if the
		$A_j$ are the operators of multiplication with a smooth function),
		\item is antisymmetric under cyclic permutations of the $A_j$,
		\item satisfies $\tau_{\boundary \chi}= b\tau_\chi$ where $\boundary$ is the
		Alexander-Spanier differential and $b$ the cyclic differential. 
	\end{enumerate}
\end{lemma}
\begin{proof}

To establish (1), by assumption, there exists $\delta>0$
		such that
                \begin{equation*}
                  \chi(x_0,\dots,x_k)=0 \text{ if }d(x_i,x_j)<\delta,\;i,j=0,1,\dots,k.
                \end{equation*}
                Choose a smooth $\Gamma$-equivariant
		function $\alpha\colon 
		\tM\times \tM\to [0,1]$ which is $1$ in a
		neighborhood of the 
		diagonal and $0$ outside the $\delta/2$-neighborhood of the
		diagonal. Set $\beta:=1-\alpha$. To simplify the notation we
		assume $k=1$, it will 
		be clear how the argument generalizes.
		
		We observe that $B_0(x,y):=A_0(x,y)\cdot\beta(x,y)$ is
		smooth and 
		therefore a smoothing operator. The composition of $B_0$ with
		$A_1$ therefore is well defined and a smoothing operator
		with smooth kernel
                $$(B_0\circ A_1)(x_0,x_2)= \int_{\tM}
		A_0(x_0,x_1)A_1(x_1,x_2)\beta(x_0,x_1)\,dx_1.$$
                Note that the
		composition makes sense and the integral is over a compact
		subset of $\tM$ due to the $\Gamma$-compactness of the
		supports of $A_0$ and $A_1$. If we
		multiply with $f(x_0)g(x_1)$ for smooth functions $f,g$, the
		kernel
                \begin{equation*}
                  (x_0,x_2)\mapsto\int_{\tM}A_0(x_0,x_1)A_1(x_1,x_2)\beta(x_0,x_1)f(x_0)g(x_1)\,dx_1
                \end{equation*}
		is (by definition) the Schwartz kernel of the composition of
		$m_{f},B_0,m_g,A_1$ which again is smoothing and indeed
                depends 
		continuously on the function $(x,y)\mapsto f(x)g(y)$. Note
		that by the $\Gamma$-compact support condition, for fixed
		$x_0,x_2$ only the values of $f(x) g(y)$ on a compact subset
		of $\tM\times \tM$ enter. The span of functions of the
form $(x,y)\mapsto f(x)g(y)$ is dense in all continuous
functions of two variables and similarly for three variables. By continuity we
therefore also
		define $\int_{\tM} A_0(x_0,x_1)A_1(x_1,x_2)\beta(x_0,x_1)
		\chi(x_0,x_1,x_2)\,dx_1$.
		The operator $C_0:=A_0-B_0$ has Schwartz kernel
		$A_0(x,y)\alpha(x,y)$ with support in the
		$\delta/2$-neighborhood of the diagonal. Define
		$B_1(x,y):=A_1(x,y)\beta(x,y)$, again a smoothing
kernel. As before, the function $(x_0,x_2)\mapsto \int_{\tM}
		C_0(x_0,x_1)A_1(x_1,x_2)\beta(x_1,x_2)\chi(x_0,x_1,x_2)\,dx_1$
		is  well defined and smooth, obtained as Schwartz
		kernel of a composition which involves the smoothing
		operator $B_1$.
		Finally, observe that
		$\alpha(x_0,x_1)\alpha(x_1,x_2)\chi(x_0,x_1,x_2)$ is
		identically $0$, as the product of the first two factors has
		support on the $\delta/2$-neighborhood of the diagonal where
		$\chi$ vanishes identically. 
		Using
\begin{equation*}
\begin{split} &\beta(x_0,x_1)-\beta(x_0,x_1)\beta(x_1,x_2)+\alpha(x_0,x_1)\alpha(x_1,x_2)+\beta(x_1,x_2)\\
                  = &\beta(x_0,x_1)\alpha(x_1,x_2)
                  +\alpha(x_0,x_1)\alpha(x_1,x_2) +\beta(x_1,x_2)=
                  \alpha(x_1,x_2)+\beta(x_1,x_2)=1,
                  \end{split}
                  \end{equation*}
                  
              we therefore define 
		\begin{equation*}
                  \begin{split}
                    \int_{\tM}
                    A_0(x_0,x_1)&A_1(x_1,x_2)\chi(x_0,x_1,x_2)\,dx_1 \\
                    :=&
                    \int_{\tM}
                        A_0(x_0,x_1)A_1(x_1,x_2)\beta(x_0,x_1)\chi(x_0,x_1,x_2)\,dx_1\\
                    &+
                    \int_{\tM}                   A_0(x_0,x_1)A_1(x_1,x_2)\beta(x_1,x_2)\chi(x_0,x_1,x_2)\,dx_1\\
 &    - \int_{\tM}      A_0(x_0,x_1)A_1(x_1,x_2)\beta(x_0,x_2)\beta(x_1,x_2)\chi(x_0,x_1,x_2)\,dx_1.
                  \end{split}
		\end{equation*}
		The support condition on $\chi$ implies that this expression
		is independent of the chosen function $\alpha$. By
		$\Gamma$-invariance, the final integral appearing in \eqref{eq:def_tauX} is independent of
		$\mathcal{F}$ (indeed, the function descends to a smooth
		function of $x_0\in M$).

Concerning (2) the formula clearly gives $0$ if all the Schwartz
		kernels have support on the diagonal.

To prove (3), the reduction as in the proof of (1) shows that the
expression which defines $\tau_\chi$ reduces to the
		$L^2$-trace (the integral over $\mathcal{F}$ over the diagonal) of a
		composition of a smoothing $\Gamma$-equivariant operator with bounded
		$\Gamma$-invariant operators. We can then use the trace property of the
		$L^2$-trace \cite{Atiyah}  and get
		\begin{equation*}
		\begin{split}
                  &\tau_\chi(A_k, A_0,\dots,A_{k-1}) \\
                  &=
		\int_{\mathcal{F}}\Tr\left(\int_{\tM^k}A_k(x_0,x_1)A_0(x_1,x_2)\dots
		A_{k-1}(x_k,x_{0})\chi(x_0,\dots,x_k) dx_1\cdots dx_k\right)\,dx_0\\
		&= \int_{\mathcal{F}\times
                  \tM^k}\Tr\left(A_0(x_0,x_1)A_1(x_1,x_2)\dots
                  A_k(x_k,x_0)\chi(x_k,x_0\dots,x_{k-1})
                \right)\,dx_0\cdots dx_k
		\end{split}
		\end{equation*}
		where we note that commuting the operators involves commuting the
		contribution of $\chi$. Using the skew-adjointness
		$\chi(x_k,x_0,\dots,x_{k-1})=(-1)^k \chi(x_0,\dots,x_k)$ gives the desired formula.
 The last property (4) is checked by a standard
straightforward calculation, compare \cite[Lemma 2.1]{ConnesMoscovici}.
\end{proof}

\begin{corollary}\label{corol:AS_to_cyc}
	Write $\mathfrak{m}_{c}\colon C^\infty(M)\to \Pdo_{\Gamma,c}^0(\tM)$ for the
	inclusion as $\Gamma$-equivariant multiplication operators.
	Formula \eqref{eq:def_tauX} defines a homomorphism
	\begin{equation*}
	H^\bullet(M\to B\Gamma)\xrightarrow{\iso}  H^\bullet_{AS,0,\Gamma}(\tilde
	M)\to HC^\bullet(C^\infty(M)\xrightarrow{\mathfrak{m}_c}
	\Pdo_{\Gamma,c}^0(\tM)) 
	\end{equation*}
	given by $[\chi]\mapsto [(0,\tau_\chi)]$.
\end{corollary}
\begin{proof}
	By Lemma \ref{lem:tauX} (2) and (4), $\mathfrak{m}_c^*(\tau_\chi)=0$,
	i.e.~$(0,\tau_\chi)$ indeed is a relative cyclic cocycle, provided
	$\chi$ is a locally zero Alexander-Spanier cocycle, and
	$(0,\tau_{\boundary\chi})=b(0,\tau_\chi)$. It follows that the
	assignment gives a well
	defined map between the homology groups.
\end{proof}

Similarly, we have an analogous homomorphism for non locally zero Alexander-Spanier cocycles. Indeed, by using exactly the same formula  \eqref{eq:def_tauX}, we obtain the following.
\begin{corollary}
	There exists an homomorphism 
	$$H^\bullet_{AS,\Gamma}(\tM) \xrightarrow{\tau}HC^\bullet(0\to
	\Pdo_{\Gamma,c}^{-\infty}(\tM))$$ given by 
	$[\chi]\mapsto[(0,\tau_\chi)]$ with $\tau_\chi$ as in \eqref{eq:def_tauX}.
\end{corollary}

Following Connes and Moscovici, let us recall, in an alternative fashion, the map from the Alexander-Spanier cohomology of $M$ to the relative cyclic cohomology associated to $\pi^*$,
$\overline{H}^\bullet_{AS}(M)\to
HC^\bullet(C^\infty(M)\xrightarrow{\pi^*}\Sigma)$. Here $\Sigma$ is the
  quotient of the 0-order pseudodifferential operator algebra by the ideal of
  smoothing operators (notice that we use the same notation for the quotient
  on $\tM$ and the quotient on $M$ because, as we have already seen, they are isomorphic).
Let $\overline{\varphi}$ be an antisymmetric cocycle in  $\overline{C}^k_{AS}(M)$ and let $\varphi$ be any representative in $C^k_{AS}(M)$. If $A_0,A_1, \dots, A_k$ are such that at least one of them is a  smoothing operator,
then formula  \eqref{eq:def_tauX} gives cyclic cochain over $\Psi^{-\infty}(M)$, see \cite[Lemma2.1]{ConnesMoscovici}.
Consider then the following two compositions
\begin{equation}\label{index-HC}
\xymatrix{\partial^{HC}\colon  HC^\bullet\left(\Psi^{-\infty}(M)\right)\ar[r]& HC^\bullet\left(\Psi^0(M)\xrightarrow{\sigma_{pr}}\Sigma\right)\ar[r]^(.5){(\mathfrak{m},id)^*}& HC^\bullet\left(C^\infty(M)\xrightarrow{\pi^*}\Sigma\right)}
\end{equation}
where $\mathfrak{m}\colon C(M)\to \Psi^0(M)$ is the obvious inclusion, and
\begin{equation}{\small
\partial_\Gamma^{HC}\colon  HC^\bullet\left(\Psi_{\Gamma,c}^{-\infty}(\tM)\right)\to HC^\bullet\left(\Psi_{\Gamma,c}^0(\tM)\xrightarrow{\sigma_{pr}}\Sigma\right)\xrightarrow{(\mathfrak{m}_c,id)^*} HC^\bullet\left(C^\infty(M)\xrightarrow{\pi^*}\Sigma\right)}
% \xymatrix{\partial_\Gamma^{HC}\colon \, HC^\bullet\left(\Psi_{\Gamma,c}^{-\infty}(\tM)\right)\ar[r]& HC^\bullet\left(\Psi_{\Gamma,c}^0(\tM)\xrightarrow{\sigma_{pr}}\Sigma\right)\ar[r]^(.53){(\mathfrak{m}_c,id)^*}& HC^\bullet\left(C^\infty(M)\xrightarrow{\pi^*}\Sigma\right)}.
\end{equation}
The first arrows in both compositions are given by excision, obtained as the inverse of the inclusion of smoothing operators into 0-order pseudodifferential operators. 
\begin{definition}
	The Connes-Moscovici map 
        \begin{equation*}
          \overline{\tau}\colon \overline{H}^\bullet_{AS}(M)\to HC^\bullet\left(C^\infty(M)\xrightarrow{\pi^*} \Sigma\right)
        \end{equation*}
        is given by
	$$[\overline{\varphi}]\mapsto \partial^{HC}(\tau_{\varphi}),$$
where $\tau_\varphi$ is formally defined as in \eqref{eq:def_tauX}.
\end{definition}

For later use we point out the following compatibility result.

\begin{theorem}\label{squareAS}\label{AStoHC}
	The following diagram is commutative.
	\begin{equation}
	\xymatrix{\ar[r]& H^\bullet_{AS,\Gamma}(\tM)\ar[r]\ar[d]^\tau&\overline{H}^\bullet_{AS}(M)\ar[d]^{\overline{\tau}}\ar[r]^{\partial^{AS}}& H^{\bullet+1}_{AS,0,\Gamma}(\tM)\ar[r]\ar[d]^\tau& \\
		\ar[r]&HC^{\bullet}(\Psi^{-\infty}_{\Gamma,c}(\tM))\ar[r]^(.65){\partial^{HC}_\Gamma}&HC^{\bullet}(\pi^*)\ar[r]^{(1,\sigma_{pr})^*}&HC^{\bullet}(\mathfrak{m}_c)\ar[r]&	}
	\end{equation}
	
\end{theorem}

\begin{proof}
	The commutativity of the first square is immediate: indeed, up to suspension of the target, the map $HC^{\bullet-1}(\mathfrak{m}_c)\to HC^{\bullet}(\Psi^{-\infty}_{\Gamma,c}(\tM))$ is given by restriction along the inclusion of pairs $0\to \Psi^{-\infty}_{\Gamma,c}(\tM)$ into $C^\infty(M)\xrightarrow{\mathfrak{m}_c}\Psi^0_{\Gamma, c}(\tM)$ and the commutativity follows directly by using explicit formulae. 
	
	The commutativity of the  second square follows from the proof of \cite[Theorem 5.2]{ConnesMoscovici}.
	
	Finally, let us prove the commutativity of the third square. 
	Consider a cocycle $\overline{\varphi}$ whose cohomology class belongs to $\overline{H}^*_{AS, \Gamma}(\tM)$.
	Let us recall the elementary construction of $\partial^{AS}[\overline{\varphi}]$ given by the snake lemma: 
	\begin{itemize}
		\item by exactness of \eqref{AScomplexes} there exists a cochain $\widetilde{\varphi}\in C^\bullet_{AS}(\tM)^\Gamma$ whose image in the quotient is $\overline{\varphi}$; 
		\item consider then $\delta\widetilde{\varphi}\in C^{*+1}_{AS}(\tM)^\Gamma$: since $\overline{\varphi}$ is a cocycle it follows that the image of $\delta(\widetilde{\varphi})$ into the quotient complex is zero;
		\item it follow that $\delta\widetilde{\varphi}$ is the image of a locally zero cocycle $\psi\in C^{*+1}_{AS,0}(\tM)^\Gamma$, whose class represents $\partial^{AS}[\overline{\varphi}]$.
	\end{itemize}
	Let us now pick a relative cyclic cycle $[x,y]\in HC_*(\mathfrak{m}_c)$. In the following, with a small abuse of notation, we shall denote with $\alpha$ (instead of $\alpha_*$) the map induced between cyclic complexes by an algebra homomorphism.  Then we have that $$\langle(0,\tau_{\partial^{AS}\overline{\varphi}}),(x,y)\rangle=\langle\tau_\psi,y\rangle=\langle\tau_{\delta\widetilde{\varphi}},y\rangle= \langle\tau_{\widetilde{\varphi}},\mathfrak{d}y\rangle$$
	where $\mathfrak{d}$ denotes here the cyclic differential.
	Since $(x,y)$ is a relative cycle, it follows that $\mathfrak{d}y=\mathfrak{m}_c(x)$ so 
	$$\langle\tau_{\widetilde{\varphi}},\mathfrak{d}y\rangle= \langle\tau_{\widetilde{\varphi}},\mathfrak{m}_c(x)\rangle.$$
	Observe that $\widetilde{\varphi}$ is any lift of $\overline{\varphi}$, so we can mod out locally zero cochain in the left member of the pairing, transferring the pairing on $M$ and not on $\tM$, namely 
	$\langle\tau_{\widetilde{\varphi}},\mathfrak{m}_c(x)\rangle= \langle\tau_{\varphi},\mathfrak{m}(x)\rangle$
	where $\varphi$ is a lift of $\overline{\varphi}$ in $\overline{C}^*_{AS}(M)$. 
	Since $\sigma_{pr}\circ\mathfrak{m}(x)= \sigma_{pr}\circ\mathfrak{m}_c(x)=\sigma_{pr}(\mathfrak{d}y)=\mathfrak{d}(\sigma_{pr}(y))$, it follows that $(\mathfrak{m}(x),\sigma_{pr}(y))$ is a cycle for the pair $C(M)\xrightarrow{\mathfrak{m}}\Psi^0(M)$. Therefore we have that 
\begin{equation}\begin{split}\langle\tau_{\varphi},\mathfrak{m}(x)\rangle&= \langle(\tau_{\varphi},0),(\mathfrak{m}(x),\sigma_{pr}(y))\rangle=\langle\tau_\varphi,\partial_{HC}(x, \sigma_{pr}(y))\rangle= \\
&=\langle\partial^{HC}(\tau_{\varphi}),(x, \sigma_{pr}(y))\rangle=\langle(1, \sigma_{pr})^*\circ\partial^{HC}(\tau_{\varphi}),(x, y)\rangle, \end{split}
\end{equation}
	where $\partial_{HC}$ is the map in cyclic homology dual to \eqref{index-HC} and the second equality is obtained by excision. Summarizing, we have proved that $$(1, \sigma_{pr})^*\overline{\tau}_{[\overline{\varphi}]}= \tau_{\partial^{AS}[\overline{\varphi}]},$$
	namely the commutativity of the third square. 
\end{proof}

\section{Rapid decay}\label{subsect:RD}
We next need a different kind of smooth subalgebra of  pseudodifferential operators.
In this section we shall assume that $\Gamma$ has property (RD). We recall
the definition below. Consider the Lie groupoid $G:={\tM}\times_\Gamma{\tM}\rightrightarrows M$.
Let us recall some basic definitions.

\begin{definition}
	A length function on $G$ is a proper and continuous function $l\colon G\to \RR_+$ such that 
	$l(\gamma\gamma')\leq l(\gamma)+l(\gamma')$,  $l(\gamma)=l(\gamma^{-1})$, and $l(x)=0$
	for all $x\in M$ and $\gamma,\gamma'\in G$ composable.
\end{definition}

\begin{example}
	Let $d$ be any proper $\Gamma$-invariant distance function on ${\tM}$.
	Let $[\widetilde{x},\widetilde{y}]$ be an element of $G$. Then
        $l([\widetilde{x},\widetilde{y}]):=d(\widetilde{x},\widetilde{y})$ is
        a proper length function on $G$.
\end{example}
Let us fix a non-zero length function $l$ on $G$.
Define for $m\in\NN$ the following norms on $C_c^{\infty}({\tM}\times_\Gamma{\tM})$:
\begin{multline}\label{rd-seminorm}
\|f\|^2_{m,l}=\max\left\{\sup_{x\in M}\int_{s^{-1}(x)}|f(\tilde{x},\tilde{y})|^{2}|1+l(\tilde{x},\tilde{y})|^{2m}\right.\\, \left.\sup_{y\in M}\int_{r^{-1}(y)}|f(\tilde{x},\tilde{y})|^{2}|1+l(\tilde{x},\tilde{y})|^{2m}\right\}.
\end{multline}
\begin{definition}
	Set $S_l^2({\tM}\times_\Gamma{\tM}):=\invlim_{m\in\NN} L_{m,l}$, where $L_{m,l}$ is the completion of $C_c^{\infty}({\tM}\times_\Gamma{\tM})$ with respect to the norm $	\|\cdot\|_{m,l}$. 
	We say that $G$ has property (RD) with respect to $l$ if there exists $m\in\NN$ and $C>0$ such that
	$$\|f\|_{red}\leq C\|f\|_{l,m} \quad \forall f\in  C_c^{\infty}({\tM}\times_\Gamma{\tM}),$$
	where $\|\cdot\|_{red}$ denotes the norm of the reduced groupoid
	C*-algebra.
	
	Recall that, by definition, the finitely generated group $\Gamma$ itself has property
	(RD) if it has property (RD) as a groupoid over $\{*\}$, i.e.~if 
	$G=\Gamma\times_\Gamma\Gamma\rightrightarrows \{*\}$ has. In other
	words, there must exist $C>0$ and $m\in\NN$ such that $\|f\|^2_{red}\le
	C \sum_{g\in\Gamma} |f(g)|^2(1+l(g))^{2m}$ for all
	$f\in\complexs[\Gamma]$, where $l(\cdot)$ is the word length function
	associated to a finite set of generators.
\end{definition}
\begin{proposition}
	If the finitely generated group $\Gamma$ has property (RD) then
	${\tM}\times_\Gamma{\tM}$ is also (RD) with respect to
	the length function $l\colon [\widetilde{x},\widetilde{y}]\mapsto
	d(\widetilde{x},\widetilde{y})$, where the metric $d$ on $
	\tM$ is induced from a Riemannian metric on $M$ via pullback.
\end{proposition}
\begin{proof}
	Let $\mathcal{F}$ be a bounded measurable fundamental domain for the
	action of $\Gamma$ on ${\tM}$.
	Fix $u\in L^2(\tM)$ and $f\in C_c^\infty(\tM\times_\Gamma\tM)$. We also use the letter $f$ for its
	$\Gamma$-equivariant pullback to $\tM\times \tM$. Then
        
	\begin{equation*}
	\begin{split}
	\abs{f*u}^2_{L^2(\tM)} & = \sum_{g\in\Gamma} \int_{\mathcal{F}g}
	\abs{ \sum_{h\in\Gamma} \int_{\mathcal{F}h} f(\tilde x,\tilde y)
		u(\tilde y)\,d\tilde y }^2\,d\tilde x\\
	&\le \sum_{g\in\Gamma} \int_{\mathcal{F}}
	\abs{ \sum_{h\in\Gamma} \int_{\mathcal{F}} \abs{f(\tilde x
			g,\tilde y h)}
		\abs{u(\tilde y h)}\,d\tilde y }^2\,d\tilde x\\
	&\le \int_{\mathcal{F}} \sum_{g\in\Gamma}
|\sum_{h\in\Gamma}
	\underbrace{ \left(\int_{\mathcal{F}}\abs{\underbrace{f(\tilde x g,\tilde y
				h)}_{=f(\tilde x,\tilde y hg^{-1})}}^2\,d\tilde y
		\right)^{1/2}}_{=: vf(\tilde x
		,gh^{-1})}\underbrace{\left(\int_{\mathcal{F}}\abs{u(\tilde y 
			h)}^2\,d\tilde y\right)^{1/2} }_{=: vu(h)} |^2\,d\tilde
	x\\
	&= \int_{\mathcal{F}} \sum_{g\in\Gamma} \abs{ vf(\tilde x,\cdot)*
		vu(\cdot)(g)}^2\\
	&\le \int_{\mathcal{F}} \| vf(\tilde x,\cdot)\|_{red}^2 \cdot
	\sum_{h\in\Gamma} \abs{vu(h)}^2 \,d\tilde x\\
	&\le C\int_{\mathcal{F}} \sum_{g\in\Gamma} \abs{vf(\tilde
		x,g)}^2(1+l(g))^{2m} \cdot \left(\sum_{h\in\Gamma} \int_{\mathcal{F}}
	\abs{u(\tilde y h)}^2\,d\tilde y\right)\\
	&\le CC' \underbrace{\int_{\mathcal{F}} \sum_{g\in\Gamma}
		\int_{\mathcal{F}}\abs{f(\tilde x,\tilde y g)}^2 (1+d(\tilde x,\tilde y g))^{2m}\,d\tilde
		y\,d\tilde x }_{=\int_{x\in M} \int_{s^{-1}(x)}
		\abs{f(\tilde x,\tilde y)}^2 (1+d(\tilde x,\tilde y))^{2m} } 
	\cdot \abs{u}^2_{L^2(\tM)}\\
	&\le  CC'\vol(M)\underbrace{\sup_{x\in M}
		\int_{s^{-1}(x)}\abs{f(\tilde x,\tilde y)}^2(1+d(\tilde
		x,\tilde y))^{2m}}_{\norm{f}^2_{m,l}}\cdot \abs{u}^2_{L^2(\tM)}
	\end{split}
	\end{equation*}
	The first inequality uses the elementary fact that the $L^2$-norm is
	unconditional (i.e.~the norm of a convolution does never decrease if the
	functions are replaced by their absolute value functions). The second
	inequality is the Cauchy-Schwartz inequality for the integral over
	$\mathcal{F}$.
	
	Then we use that
	\begin{equation*}
          \begin{split}
            h\mapsto vu(h) &=\int_{\mathcal{F}} \left(\abs{u(\tilde
            y h)}^2\,d\tilde y\right)^{1/2}\in l^2(\Gamma);\\
            g\mapsto
            vf(\tilde x,g) &= \left(\int_{\mathcal{F}} \abs{f(\tilde x,\tilde y
            g^{-1})}^2\,d\tilde y\right)^{1/2}\in \CC[\Gamma]. 
          \end{split}
	\end{equation*}
	The forth inequality is property (RD) for
	$\Gamma$. Finally, we use the Milnor-Svarc result that the map $\Gamma\to
	\tM$  given by $g\mapsto \tilde x g$ is a quasi-isometry, with constants
	uniform for $x$ in the compact subset $\overline{\mathcal{F}}$, to obtain the
	fifth inequality. 
	As $C,C'$ do not depend on $f$ and $u$, we have proved that $\tM\times_\Gamma\tM$ has property (RD).
\end{proof}

\begin{proposition}
	The *-algebra	$S_l^2({\tM}\times_\Gamma{\tM})$ is dense and holomorphically closed in $C^*_{red}({\tM}\times_\Gamma{\tM}) $.
\end{proposition}
\begin{proof}
	The proof of the corresponding result 
	\cite[Theorem 4.2]{Hou} for $r$-discrete groupoids also works in this
	case, replacing sums over the $r$-fibers by integrals.
	This follows an idea which already appeared in \cite{ConnesMoscovici},
	and more concretely in  \cite[Section 4]{Wu}.
\end{proof}

\begin{definition}\label{rd-smooth}
	Set \[S_l^{2,\infty}:=\{f\in S_l^{2}\,:\, f\, \mbox{is smooth and } \|f\|_{m,l,\alpha,\beta}<\infty \quad \forall
	m,\alpha,\beta\}, \]
	where $\|f\|^2_{m,l,\alpha,\beta}$ is defined as
        \begin{multline*}
          \max\left\{\sup_{x\in
              M}\int_{s^{-1}(x)}|\nabla_x^\alpha\nabla_y^\beta
            f(\tilde{x},\tilde{y})|^{2}|1+l(\tilde{x},\tilde{y})|^{2m},\right.\\ \left. \sup_{y\in M}\int_{r^{-1}(y)}|\nabla_x^\alpha\nabla_y^\beta f(\tilde{x},\tilde{y})|^{2}|1+l(\tilde{x},\tilde{y})|^{2m}\right\}.
        \end{multline*}
\end{definition}

\begin{proposition}
	The *-algebra $S_l^{2,\infty}$ is dense and holomorphically closed in $C^*_{red}({\tM}\times_\Gamma{\tM})$.
\end{proposition}
\begin{proof}
	It is immediate that
	$S_l^{2,\infty}$ is a dense semi-ideal in $S_l^{2}$.  Then we use
	\cite[Prop 3.3]{LMN} to conclude. 
\end{proof}

\begin{proposition}\label{prop:pol_smooth}
	Denote by $\Pdo_{\Gamma,rd}^0({\tM})$ the sum $\Pdo_{\Gamma,c}^0({\tM})+S_l^{2,\infty}({\tM}\times_\Gamma{\tM})$.
	It is a *-algebra and it is dense and holomorphically closed in $\Pdo_{\Gamma}^0({\tM})$.
\end{proposition}

\begin{proof}
The proof of this proposition follows verbatim the proof of \cite[Corollary
7.9]{LMN}. It is enough to replace $\mathcal{G}$, $\Pdo^0(\mathcal{G})$,
$\mathcal{S}(\mathcal{G},\phi)$ and $\Pdo^0_s(\mathcal{G})$ there with $\widetilde{M}\times_\Gamma\widetilde{M}$, $\Pdo_{\Gamma,c}^0({\tM})$, $S_l^{2,\infty}({\tM}\times_\Gamma{\tM})$ and $\Pdo_{\Gamma,rd}^0({\tM})$, respectively.

\end{proof}

\section{Extending cyclic cocycles associated to classes in $H_{pol}^* (M\xrightarrow{u}B\Gamma)$}

Let $\chi\colon \tM^{k+1}\to \reals$ be a smooth anti-symmetric
$k$-dimensional  delocalized Alexander-Spanier cocycle in $C^\bullet_{AS,0}(\tM)^\Gamma$  of polynomial
growth. Take the corresponding cyclic cocycle $\tau_\chi$ as in Lemma
\ref{lem:tauX} and Corollary \ref{corol:AS_to_cyc}, defined on the algebra
$\Pdo^0_{\Gamma,c}({\tM})$. In 
this section we are going to prove that $\chi$ extends continuously to
the algebra $\Pdo^0_{\Gamma,rd}({\tM})$. 
Recall that this algebra is defined as the sum of $\Pdo^0_{\Gamma,c}({\tM})$ and  $S^{2,\infty}_l({\tM}\times_\Gamma{\tM})$ of Definition \ref{rd-smooth}. Hence it is sufficient to prove the following Lemma.

\begin{lemma}\label{lem:cocyc_ext}
	The cyclic cocycle $\tau_\chi$ extends continuously to the algebra
        \begin{equation*}
          S^{2,\infty}_l({\tM}\times_\Gamma{\tM}).
        \end{equation*}
\end{lemma}
\begin{proof}
	Indeed it is sufficient to prove that it extends to $S^2_l({\tM}\times_\Gamma{\tM})$.
	To this end we have to show that $\tau_\chi$ is continuous with respect to some seminorm of the type \eqref{rd-seminorm}.
	We shall apply the method used in \cite[Proposition
        6.5]{ConnesMoscovici}.
	By the polynomial growth of $\chi$ we know that there exists a
        constant $C>0$ and an integer $n\geq0$ such that
        
	\begin{equation}\label{pg}
	|\chi(x_0,x_1,\dots,x_k)|\leq C (1+d(x_0,x_1))^{2n}\dots(1+d(x_{k-1},x_k))^{2n}
	\end{equation}
	for all $x_0,x_1,\dots,x_k\in {\tM}$.
	
	Let $A_0,\dots, A_k$ be elements in $C^\infty_c({\tM}\times_\Gamma{\tM})$ and put  $B_i(x,y):=|A_i(x,y)|(1+d(x,y))^{2n}$ for $i=0,\dots, k-1$ and $B_k(x,y):=|A_k(x,y)|$. Then
	\begin{equation}
	\begin{split}
          |\tau_\chi(A_0,&\dots, A_k)|\\
                     & = \abs{\int_{\mathcal{F}} Tr
		\left(\int_{\tM^k} A_0(x_0,x_1) \cdots
		A_x(x_k,x_0)\chi(x_0,\dots,x_k)\,dx_1\dots dx_k\right)\,dx_0}\\
	& \leq C\int_{\mathcal{F}}\left(\int_{{\tM}^k}B_0(x_0,x_1)B_1(x_1,x_2)\dots B_k(x_k,x_{0})dx_1\dots dx_k\right)dx_0\\
	&= C\int_{\mathcal{F}}(B_0*B_1*\dots *B_k)(x_0,x_0)dx_0\\
	&\leq C' \sup_{x\in\mathcal{F}}\left|B_0*B_1*\dots *B_k(x,x)\right|\\
	&\leq C'' \|B_0*B_1*\dots* B_k\|_{red}\\
	&\leq C'' \|B_0\|_{red}\|B_1\|_{red}\dots\| B_k\|_{red}
	\end{split}
	\end{equation}
	where in the last but one step we use the fact that 
	the  map $C^*_{red}(G)\to C_0(G^{(0)})$, given by the restriction to the
	units, is a continuous conditional expectation of C*-algebras.
	
	\noindent
	Then, since the groupoid has property (RD), we have that
	\begin{equation}
	\begin{split}
	|\tau_\chi(A_0,\dots, A_k)|&\leq C'''\|B_0\|_{l,m}\|B_1\|_{l,m}\dots\|A_k\|_{red}\\
	&\leq C'''\|A_0\|_{l,m+n}\|A_1\|_{l,m+n}\dots\|A_k\|_{m}
	\end{split}
	\end{equation}
	which gives the desired result.

\end{proof}

Propositions \ref{Higson-lemma} and \ref{AS-pol}, Lemma
\ref{lem:cocyc_ext}, and Corollary \ref{corol:AS_to_cyc}  now give the
following result.

\begin{theorem}\label{pairing-relative}
	Let $\Gamma$ be a discrete group with property (RD). Let ${\tM}\to M$ be a $\Gamma$-covering of a compact manifold $M$. Then there is a well-defined homomorphism
	\begin{equation*}
\Xi\colon	H_{pol}^*(M\to B\Gamma)
	\to HC^*(C^\infty(M)\xrightarrow{\mathfrak{m}}
	\Psi^0_{\Gamma,rd}({\tM})).
	\end{equation*}
Using Proposition \ref{prop:pol_smooth} we then get a well defined pairing
	\[
	H_{pol}^* (M\to B\Gamma)\times K_{*+1}\left(C(M)\to \Pdo^0_\Gamma({\tM})\right)\to \CC,
	\]
explicitly given by associating to $(\alpha,x)$ the number $\langle \Xi (\alpha), x\rangle$, where the
pairing between cyclic cohomology and K-theory has been used.

If, in addition, $H^*_{pol}(\Gamma)\to H^*(\Gamma)$ is an isomorphism, we
  can replace $H^*_{pol}(M\to B\Gamma)$ by the usual relative cohomology
  $H^*(M\to   B\Gamma)$. Examples of groups which satisfy both conditions are
  hyperbolic groups or groups of polynomial growth \cite{Meyer,Meyer2}.
\end{theorem}

\begin{corollary}
	There exists a commutative diagram of the following shape: 
	\begin{equation}\label{eq:rel_cohom_to_HC}
  	\xymatrix{\cdots\ar[r]& H_{pol}^*(B\Gamma)\ar[r]\ar[d]&H^*(M)\ar[d]\ar[r]^(.4){\delta}& H_{pol}^{*+1}(M\to B\Gamma)\ar[r]\ar[d]& \\
  		\cdots\ar[r]&HC^{*}(\Psi^{-\infty}_{\Gamma, rd}(\tM))\ar[r]^(.65){\partial^{HC}_\Gamma}&HC^{*}(\pi^*)\ar[r]^{(1,\sigma_{pr})^*}&HC^{*}(\mathfrak{m}_{rd})\ar[r]&	}
            \end{equation}
Here, $\mathfrak{m}_{rd}$ is the inclusion as multiplication operators
$C^\infty(M)\to \Psi^0_{\Gamma,rd}(\tM)$ and $\pi^*$ is the composition of $\mathfrak{m}_{rd}$ and the quotient map to $\Sigma$.	
\end{corollary}

\begin{definition}\label{def:relative_higher_rho}\label{higher-rho-number-AS}
   Let $\Gamma$ be a discrete group with property (RD).
	Let $\widetilde{D}$ be a generalized Dirac operator which is $\Gamma$-equivariant on the Galois $\Gamma$-covering $\tM$ of a compact smooth manifold $M$. Suppose that $\widetilde{D}$ is $L^2$-invertible. Then define its higher rho number associated to $[\alpha]\in H_{pol}^{[*-1]}(M\to B\Gamma)$ as
	\[
	\varrho_\alpha(\widetilde{D}):=	\left\langle\Xi(\alpha), \varrho(\widetilde{D})\right\rangle\in \CC,
	\]
	where $\varrho(\widetilde{D})$ is defined as in Definition \ref{rho-class}.
\end{definition}
\begin{definition}\label{def:higher_rho_of_g_rel}
    Let $g$ be a metric with positive scalar curvature on a closed spin
    manifold $M$  
    of dimension $n$ with fundamental group $\Gamma$. We define  the higher rho number associated to  $g$
    and to 
    $\alpha\in H_{pol}^{[*-1]}(M\to B\Gamma)$ as 
\[
\varrho_\alpha (g):=	\langle \Xi (\alpha), \varrho(g)\rangle\in \CC,
\]
where $\varrho(g)\in  K_{n}\left(C(M)\to\Pdo^0_\Gamma({\tM})\right)$
is the rho class of the spin Dirac operator for the metric $g$.\end{definition}

\chapter{Delocalized  higher Atiyah-Patodi-Singer indices}\label{sect:higher-aps}

Let us consider as  in Section \ref{comparison lott2} a cocompact even-dimensional  $\Gamma$-covering with boundary $\widetilde{W}\to W$
with product metric near the boundary, and a $\Gamma$-equivariant 
 Dirac operator  $\widetilde{D}_{W}$
on $\widetilde{W}$ such that $\pa \widetilde{W}=\widetilde{M}$ and such that the boundary operator of 
$\widetilde{D}_{W}$ is equal to $\widetilde{D}_M$. Moreover, let us  assume
that $\Gamma$ is Gromov hyperbolic or of polynomial growth. 
In this section we discuss how the higher rho numbers we have defined in the previous sections enter into 
index formulae for delocalized higher APS indices.

\section{Higher APS indices associated to elements in $HC^* (\CC\Gamma,\langle x \rangle)$}

Let $\mathrm{Ind}_{\Gamma,b}(\widetilde{D}_{W})$ be the class in
$K_* (\mathcal{A}\Gamma)$ from Theorem \ref{aps-index-lott}. Let us fix a
delocalized cyclic cocycle $\tau$ of polynomial growth  such that $[\tau]\in HP^{even}_{del}(\CC\Gamma)$. Then define 
\begin{equation}
\mathrm{Ind}^\tau_{\Gamma, b}(\widetilde{D}_{W}):=\langle\Ch_\Gamma (\mathrm{Ind}_{\Gamma,b}(\widetilde{D}_{W})), \tau\rangle\in \CC
\end{equation}
where the pairing is defined in the following way.
\begin{itemize}
	\item $\Ch_\Gamma (\mathrm{Ind}_{\Gamma,b}(\widetilde{D}_{W}))$ is an element in $H_{*}(\mathcal{A}\Gamma)$.
	\item Consider its image in $\overline{HC}_{*}(\mathcal{A}\Gamma)$
	using Corollary \ref{corol:deloc_inclusion}, see \eqref{eq:dR_to_HC}.
With a small abuse of notation we keep the notation $\Ch_\Gamma (\mathrm{Ind}_{\Gamma,b}(\widetilde{D}_{W}))$ for this element of  $\overline{HC}_*(\mathcal{A}\Gamma)$.
	\item Now, by Proposition \ref{prop:loc_delo_split},
            $\overline{HC}_{*}(\mathcal{A}\Gamma)$ splits as
            $\overline{HC}^e_{*}(\mathcal{A}\Gamma)\oplus
            HC^{del}_{*}(\mathcal{A}\Gamma)$
          and then, following \eqref{aps-lp}, we have that 
	$$ \Ch_\Gamma (\Ind_{\Gamma,b} (\widetilde{D}_{W}))= \pi^e\left(\Ch_\Gamma (\Ind_{\Gamma,b} (\widetilde{D}_{W}))\right)\oplus\pi^{del}\left(\Ch_\Gamma (\Ind_{\Gamma,b} (\widetilde{D}_{W}))\right) ,$$ 
	where the first summand is an element of the reduced cyclic homology
        group localized at the identity and the second summand lies in
        unreduced cyclic homology group $HC^{del}_{*}(\mathcal{A}\Gamma)$.
\end{itemize} 
We \emph{define}
\begin{equation*}
  \langle\Ch_\Gamma (\mathrm{Ind}_{\Gamma,b}(\widetilde{D}_{W})), \tau\rangle=:
  \mathrm{Ind}^\tau_{\Gamma, b}(\widetilde{D}_{W})
\end{equation*}
	as the pairing of
	$\pi^{del}\left(\Ch_\Gamma (\Ind_{\Gamma,b} (\widetilde{D}_{W}))\right)\in
HC^{del}_{*}(\mathcal{A}\Gamma)$
	with  $[\tau]\in HC^{*}_{del}(\CC\Gamma)$.

\medskip
Thanks to \eqref{aps-lp} and \eqref{del-APS-hom}, we obtain the following equality
$$\mathrm{Ind}^\tau_{\Gamma, b}(\widetilde{D}_{W})=-\frac{1}{2}
\innerprod{\varrho_{Lott}(\widetilde{D}_M),\tau}
  \in \CC,$$
where $\varrho_{Lott}(\widetilde{D}_M)$ is the higher rho class of Lott, which
we pair with $\tau$.\\
In the even dimensional case we also have, by Proposition
\ref{APS-deloc-ch}, $$\pi^{del}\left(\Ch_\Gamma (\Ind_{\Gamma,b}
  (\widetilde{D}_{W}))\right)=
\Ch_\Gamma^{del}\left(\varrho(\widetilde{D}_{M})\right)\in
H_{even}^{del}(\mathcal{A}\Gamma), $$ 
so that we obtain the following formula for the higher delocalized APS indices
$$\mathrm{Ind}^\tau_{\Gamma, b}(\widetilde{D}_{W})= \varrho_\tau(\widetilde{D}_M)\in \CC.$$
Summarizing, in the even dimensional case, for the higher delocalized APS indices we have the following double equality of complex numbers:
\begin{equation}\label{double-aps-numbers}
- \frac{1}{2}\innerprod{\varrho_{Lott}(\widetilde{D}_M),\tau}=\mathrm{Ind}^\tau_{\Gamma, b}(\widetilde{D}_{W})=
\varrho_\tau(\widetilde{D}_M)
\end{equation}
If we apply all this to the delocalized trace associated to $\langle x \rangle$, that is
	$$\tau_{\langle x \rangle} (\sum \alpha_\gamma \gamma):= \sum_{\gamma\in \langle x \rangle}  \alpha_\gamma \gamma\,,$$
	we obtain from Example \ref{example-deloc-eta} the following formula:
	\begin{equation}\label{aps-deloc-lott-0} 
	\mathrm{Ind}^{\tau_{\langle x \rangle}}_{\Gamma, b}(\widetilde{D}_{W})=- \frac{1}{2}\eta_{\langle x \rangle} (\widetilde{D}_M)
	\end{equation}
	with $ \eta_{\langle x \rangle} (\widetilde{D}_M)$ denoting the 
	delocalized eta invariant of Lott. 
	
	\begin{remark}
Notice that equality \eqref{aps-deloc-lott-0} 
	has been known for quite some time and it also holds for a general group $\Gamma$
	once we assume that $\langle x \rangle$ has polynomial growth: it is a  direct consequence of the specialization
	to 0-degree of the higher APS
	index theorem in non-commutative de Rham homology (Theorem \ref{aps-index-lott}). This is discussed in
	detail in \cite{PiazzaSchick_BCrho}  for arbitrary $\Gamma$ and $\langle x \rangle$
	of polynomial growth, but the argument given there also applies to hyperbolic groups, once we use \emph{directly}
	the results of 
	Puschnigg \cite{Puschnigg}.
\end{remark}

\section{Higher APS indices associated to elements in $H^* (M\to B\Gamma)$}
Let $\chi$ be a cocycle defining a class in $H^\Gamma_{AS,0}(\widetilde{W})$. Consider the image of $[\chi]$ through the map
$$H^{even}_{AS,0,\Gamma}(\widetilde{W})\xrightarrow{\alpha}H^{even}_{AS, \Gamma}(\widetilde{W}).$$ 
By Theorem \ref{squareAS}, we have that 
$$\Ind_{\Gamma,b}^{\alpha(\chi)} (\widetilde{D}_{W}):=\langle\tau_{\alpha(\chi)},\Ind_{\Gamma,b} (\widetilde{D}_{W})\rangle= \langle\tau_{\chi},j_*\Ind_{\Gamma,b} (\widetilde{D}_{W})\rangle,$$ where $j_*\colon K_*(\mathcal{S}^{2,\infty}_{l}(\mathring{\widetilde{W}}\times_\Gamma\mathring{\widetilde{W})})\to K_*( \mathrm{\Psi}^0_{\Gamma, rd}(\mathring{\widetilde{W}}))$.
Since $\chi$ is a delocalized cocycle, reasoning as in \eqref{k-aps-del}, we have that 
\begin{equation}\label{dual-formula}	
\langle\tau_{\chi},j_*\Ind_{\Gamma,b} (\widetilde{D}_{W})\rangle= \langle\tau_{\chi},[\pi_\geq]\otimes e\rangle \end{equation}
where with a small abuse of notation we identify $[\pi_\geq]\otimes e$ with
its image in the group
	$K_*( \mathrm{\Psi}^0_{\Gamma, rd}(\mathring{\widetilde{W}}))$.
Now, since $\pi_\geq\otimes e$ is supported in a open collar neighbourhood $\partial\widetilde{W}\times (0,1)$ of the boundary, we can easily see, from the explicit expression \eqref{eq:def_tauX}, that  the right hand side 
	of \eqref{dual-formula} is equal to $\langle\tau_{\iota^*\chi},[\pi_\geq]\otimes e\rangle$
	with  $\iota$  the natural inclusion of $\partial\widetilde{W}\times (0,1)$ into $\widetilde{W}$. 
	Summarizing,
	$$\langle\tau_{\alpha(\chi)},\Ind_{\Gamma,b} (\widetilde{D}_{W})\rangle=\langle\tau_{\iota^*\chi},[\pi_\geq]\otimes e\rangle\,.$$
By the K\"{u}nneth theorem, we can assume that $\iota^*\chi$ is the product of the restriction  $\iota^*_\partial\chi$ of the cocycle to $\partial \widetilde{W}$ and the generator of the cohomology of the interval. 
Then, we finally obtain the following delocalized APS-type equality 

\begin{equation}
\Ind_{\Gamma,b}^{\alpha(\chi)} (\widetilde{D}_{W})=\varrho_{\iota_\partial^*\chi}(\widetilde{D}_{M})\;\;\text{in}\;\; \CC.
\end{equation}

\section{Filling of positive scalar curvature metrics}

From the above formulae we immediately obtain the following result.

Let $M$ be a closed compact n-dimensional spin manifold with a classifying map $f:M\to B\Gamma$ defining
a Galois $\Gamma$-covering $\widetilde{M}$. 
Let $g_M$ be a metric of positive scalar curvature on $M$. 
Assume now that $(M,f)$ is null-bordant in $\Omega^{\spin}_{n}(B\Gamma)$; thus there
exists a compact spin manifold $W$ and $F\colon W\to B\Gamma$
such that $\partial W=M$ and $F|_{\partial}=f$. 
\begin{proposition}
If $g_M$ extends to a positive scalar curvature metric $g_W$, product-type
near the boundary, then all higher rho numbers on $\Gamma\to \widetilde{M}\to
M$ vanish.

Of course, we make the necessary assumptions on $\Gamma$ so that these higher rho numbers are indeed well-defined.
\end{proposition}

\chapter[Positive scalar curvature and higher rho invariants]{The space of metrics of positive scalar curvature and higher
	rho invariants}
\label{sec:moduli_space}

The techniques developed in this
paper have interesting geometric applications. More specifically, we can
use them to study important aspects of the topology of positive scalar
curvature. 
The basic
question we want to answer is: what can one say about the space of  Riemannian metrics of positive scalar
curvature on $M$?
Let us start with a definition to fix some notation. 

\begin{definition}
	Let $M$ be a fixed  {smooth compact manifold without boundary} of dimension $n\ge 5$.
	\begin{itemize}
		\item	 Let
		\emph{$\Riem^+(M)$ be the space of Riemannian metrics on $M$ with positive scalar
			curvature} (with its usual $C^\infty$-topology).
		
		\item Let us denote its set of path components by $\pi_0(\Riem^+(M))$. 
		\item The set \emph{${\PosConc}(M)$ denotes the set of concordance classes of metrics of
			positive scalar curvature}. Two metrics $g_0,g_1\in\Riem^+(M)$
		are defined to be \emph{concordant} if  there is a metric $G$ on
		$M\times [0,1]$ which has positive scalar curvature, is of product structure
		near the boundary, and restricts to $g_0$ on $M\times \{0\}$ and to $g_1$ on
		$M\times \{1\}$.
	\end{itemize}
	
\end{definition}
It is a standard fact that two metrics in the same path
component of $\Riem^+(M)$ are also concordant, i.e.~${\PosConc}(M)$ is a
quotient of $\pi_0(\Riem^+(M))$.

A central tool for the study of positive scalar curvature metrics is Stolz'
positive scalar curvature exact sequence (where $\Gamma=\pi_1(M)$)
\begin{equation}\label{ses}
\xymatrix{\cdots\ar[r]& \Omega^{\spin}_{n+1}(M)\ar[r]
	&\mathrm{R}^{\spin}_{n+1}(\Gamma)\ar[r]^\partial&\mathrm{Pos}^{\spin}_{n}(M)\ar[r]&
	\Omega^{\spin}_n(M)\ar[r]&\cdots}
\end{equation}
Again, to fix notation let us recall the definition of the terms in the Stolz
positive scalar curvature exact sequence.

	The three groups appearing in the Stolz sequence are bordism groups.
	In each case, cycles of the $k$-th group are $k$-dimensional manifolds $X$
	with extra structure and with a map $f\colon X\to M$.
		For $\Omega^{\spin}_k$, the manifolds $X$ are closed and the extra structure
	is a spin structure. For $\Pos^{\spin}_k$, the cycles are closed manifolds and the
	extra structure is a spin structure together with a metric $g$ of positive
	scalar curvature on $X$. Finally, for $\mathrm{R}^{\spin}_k$, the cycles are
	compact manifolds with boundary, the extra structure is a spin structure and
	a metric of positive scalar curvature on the boundary.
For all of them the group structure is given by disjoint union.

\begin{remark}
	The type of cycles we have given a priori defines a group
	$\mathrm{R}^{\spin}_k(M)$. Post-composing with a map $u\colon M\to
	B\Gamma$ classifying a universal covering gives a canonical isomorphism
	$\mathrm{R}^{\spin}_{k}(M)\to \mathrm{R}^{\spin}_k(B\Gamma)=:
	\mathrm{R}^{\spin}_k(\Gamma)$ which we use to identify all these groups with
	a group which only depends on $\Gamma=\pi_1(M)$. For details compare
	\cite{Stolz}. 
\end{remark}

The groups in the Stolz exact sequence have direct relevance also for the set
of concordance classes due to the following result of Stolz \cite[Proof of
Theorem 5.4]{Stolz}.

\begin{proposition}\label{prop:Stolz_trans_act}
	There is a free and transitive action of the abelian group
	$\mathrm{R}^{\spin}_{n+1}(\Gamma)$ on $\PosConc(M)$. The inverse of the
	corresponding action bijection $\mathrm{R}^{\spin}_{n+1}(\Gamma)\times
	\PosConc(M)\to \PosConc(M)\times\PosConc(M); (a,x)\mapsto (a\cdot x,x)$ is
	given by the map
	\begin{equation*}
          \begin{split}
            \PosConc(M)\times\PosConc(M) &\to     \mathrm{R}^{\spin}_{n+1}(\Gamma)\times
                                           \PosConc(M)\\
            ([g_0],[g_1])& \mapsto((M\times [0,1]\xrightarrow{pr_M} M,
            g_0\disjointunion g_1),[g_0])
          \end{split}
	\end{equation*}
here we put the two given metrics $g_0,g_1$ on the two ends of the cylinder
	$W=M\times [0,1]$. 
\end{proposition}

The Stolz positive scalar curvature exact sequence is related to the
theme of the present paper (pairing the Higson-Roe exact sequence with
cohomology) via  \cite{PiazzaSchick_psc}, where a mapping from the
positive scalar curvature exact sequence \eqref{ses} to the Higson-Roe
analytic exact sequence is constructed:
\begin{equation}\label{HRses}{\small
\xymatrix{\ar[r]& 
	\Omega^{\spin}_{n+1}(M)\ar[r]\ar[d]^\beta &\mathrm{R}^{\spin}_{n+1}(\Gamma)\ar[r]^\partial\ar[d]^{{\Ind}^{\Gamma}}&\mathrm{Pos}^{\spin}_{n}(M)\ar[r]\ar[d]^{\varrho}& \Omega^{\spin}_n(M)\ar[r]\ar[d]^\beta&\\
\ar[r]
&K_{n+1}(M)\ar[r]^{\quad\mu^\Gamma_M\qquad}&K_{n+1}(C_{red}^*\Gamma)\ar[r]^{s}&\SG^\Gamma_n(\widetilde{M})\ar[r]^{c}&
K_n(M)\ar[r]^{\quad\mu^\Gamma_M}& }}
\end{equation}
This results was established in 
	\cite{PiazzaSchick_psc} in the even dimensional 
	case (the odd dimensional case was later
	treated by suspension by Zenobi in \cite{zenobi-JTA}). 
	Subsequent contributions of
	Xie and Yu and of Zeidler, see 
	\cite{XY-advances}
	\cite{zeidler-JTOP},
	gave alternative treatments of
	the result of Piazza and Schick, valid in 
	every dimension.
In \eqref{HRses} $\beta$ is the well-known map of generalized homology theory, $\Ind^\Gamma$ is a higher APS-index map, while $\varrho$ associates to a cycle $(f\colon X\to M, g)$ the push-forward of the rho class associated to $\widetilde{D}_g$ on $f^*\tM$.

The general pattern to study metrics of positive scalar curvature is now the
following: a metric $g$ of positive scalar curvature on $M$ defines a class in
$\mathrm{Pos}^{\spin}_n(M)$, which is independent of the concordance class of
$M$, thus there is a well-defined map from $\PosConc(M)$ to
$\mathrm{Pos}^{\spin}_n(M)$ which associates to the concordance class $[g]$
the bordism class  $[M\xrightarrow{\mathrm{id}}M, g]$. Therefore, the
following proposition is relevant (which is consistent with the notation of
\eqref{HRses}).
Recall that the compatibility result of \cite[Section 5.3]{Zenobi_compare} gives the following result.
\begin{proposition}
	For a metric $g$ of positive scalar curvature on $M$ the definition in
	\eqref{HRses} of
	$\varrho(g):= \varrho([M\xrightarrow{\id}M,g])$ as given in
	\cite{PiazzaSchick_psc} coincides with the one we give here in Definition
	\ref{rho-class}: 
	\begin{equation*}
	\varrho(g)= \varrho(\tilde D_g) \in \SG^\Gamma_n(\tM),
	\end{equation*}
	as the class  of the Dirac
	operator $\tilde D$ on the universal covering of $M$, through the identification of $\SG^\Gamma_n(\tM)$ with the relative K-theory group $K_{n-1}(C(M)\to \Psi^0_\Gamma(\tM))$.
\end{proposition}

As a consequence, the higher rho numbers defined on
$\mathcal{S}^\Gamma_n(\tM)$ can be applied to $\varrho(g)$ and can be used
to distinguish elements in ${\PosConc}(M)$ and in particular in $\pi_0(\Riem^+(M))$.

Typically, it is impossible to explicitly compute $\varrho(g)$ or the resulting
higher rho numbers. Instead, one uses commutativity of the diagram
\eqref{HRses} and changes $g$ by elements of
$\mathrm{R}^{\spin}_{n+1}(\Gamma)$ and controls how the rho class and the higher rho
numbers \emph{change}. Classes in $\mathrm{R}^{\spin}_{n+1}(\Gamma)$, in turn,
often have topological origin. For example, we have a factorization
$\Omega^{\spin}_{n+1}(M)\to \Omega^{\spin}_{n+1}(B\Gamma)\to
\mathrm{R}^{\spin}_{n+1}(\Gamma)$ so that knowledge about
$\Omega^{\spin}_{n+1}(B\Gamma)$ can be used to obtain non-trivial
elements in $\mathrm{R}^{\spin}_{n+1}(\Gamma)$ and eventually many different
classes of metrics of positive scalar curvature on $M$. This route is
exploited e.g.~in \cite{SchickZenobi}. Equivariant generalized homology
classes are used in \cite{XieYuZeidler} to construct a different type of
classes in $\mathrm{R}^{\spin}_{n+1}(\Gamma)$. In both cases, their
non-triviality is detected directly in the K-groups of the Higson-Roe analytic
exact sequence. We will see that these classes pair non-trivially with our
higher cocycles and therefore provide examples of interesting higher rho
numbers. 
A main advantage of these higher rho numbers is the following: in favorable
situations our numeric rho 
invariants obtained from delocalized cyclic cocycles of the fundamental
group of a manifold $M$ or from the relative cohomology of $M\to B\Gamma$ can
be used not only to distinguish path 
components of the space of metrics of positive scalar curvature (or
concordance classes), but also classes in the moduli space
quotient by the action of the diffeomorphism group. It is precisely at this point that the use of these numeric rho invariants is particularly powerful: for
these purposes it is much harder to work directly with the rho classes in the
analytic structure group $\SG_n^\Gamma(\tM)$.

Philosophically, a reason why this has better chances to work is that the
large K-theory group $\SG_n^\Gamma(\tM)$ is in a certain sense
stratified into pieces of different
homological nature and degree, and we know a priori that the induced action
of the diffeomorphism group will preserve these pieces.

\chapter{Naturality for the action by the diffeomorphism group}
\label{sec:naturality}

It is well known and already established by Higson and Roe, that the
Higson-Roe analytic exact sequence is natural: it can be defined generally for
$\Gamma$-coverings $\tilde X\to X$ (where $X$ is a compact metrizable space),
and it is natural for maps of $\Gamma$-coverings $(\tilde X\to X)
\xrightarrow{f} (\tilde Y\to Y)$. In this paper, we deal entirely with
manifolds $X=M$ and describe the Higson-Roe exact sequence based on algebras
of pseudodifferential operators, developed and shown to be equivalent to the
original one in \cite{Zenobi_compare}.

In this section, we describe, as special case of naturality, the action of the
diffeomorphism group of $M$ in our (pseudodifferential) model of the
Higson-Roe exact sequence, and also on the (non-commutative) homology and
cohomology groups we are using. We finally show that our maps to homology,
respectively our pairings with cohomology, are compatible with the actions.

Let us fix a connected spin manifold $M$ of dimension $n$ with fundamental
group $\pi_1(M,m)\cong \Gamma$.  
To be precise, 
if we denote by $\Gamma$ exactly  $\pi_1(B\Gamma, *)$, then the previous
isomorphism, which will be denoted by $\pi_1(u)$, is induced by the map of
pointed spaces $u\colon (M,m) \to (B\Gamma, *)$ which classifies the universal
covering $\tM\to M$ associated to the basepoint $m\in M$. 

\emph{Throughout, we assume that $M$ is path-connected.}

\section{Actions of diffeomorphisms}

Let us denote by $\mathrm{Diffeo}(M)$ the group of self-diffeomorphisms of $M$. In this section we are going to describe how this group acts on the K-theoretical and (co)homological objects attached to $M$. 

\subsection{Action on the fundamental group}
Fix an element $\psi\in \mathrm{Diffeo}(M)$.
The diffeomorphism $\psi$ can be seen as a map of pointed spaces, hence
it  induces a group isomorphism $\pi_1(\psi)\colon \pi_1(M,\psi^{-1}(m))\to
\pi_1(M,m)$. Consequently, we observe that only the subgroup of $\Diffeo(M)$
which fixes the basepoint $m$ really acts on $\pi_1(M,m)=\Gamma$.

However, choosing a path $\gamma$ from $\psi(m)$ to $m$ and concatenating
loops with this path and its inverse gives a standard isomorphism
$\pi_1(M,m)\to \pi_1(M,\psi^{-1}(m))$. Of course, this isomorphism is
non-canonical: it depends on the choice of $\gamma$. If we choose a second
path $\gamma'$ we obtain another isomorphism. Their ``difference'', i.e.~the
composition of the first such isomorphism with the inverse of the second, is
given by conjugation with the loop obtained by concatenating $\gamma$ with the
inverse of $\gamma'$, i.e.~is given by an inner automorphism of $\pi_1(M,m)$.

Therefore, we can identify $\pi_1(M,\psi^{-1}(m))$ with $\pi_1(M,m)$ with an
isomorphism which is determined up to post-composition with an inner
automorphism of $\pi_1(M,m)$. After this identification, $\pi_1(\psi)$ becomes
an automorphism of $\pi_1(M)$, but only well defined up to precomposition with
an inner automorphism. This defines a map
\begin{equation}\label{eq:diffeo_to_out}
\alpha\colon \Diffeo(M)\to \Out(\pi_1(M))=\Aut(\pi_1(M,m))/\Inn(\pi_1(M,m))
\end{equation}
which is a group homomorphism. Note that, as a consequence of what has just
been explained, the outer automorphism groups of $\pi_1(M,m)$ and $\pi_1(M,m_2)$
for two basepoints $m,m_2\in M$ are canonically isomorphic and therefore it
makes sense to write $\Out(\pi_1(M))$, dropping the basepoint.

\subsection{Action on the universal covering}
 Let the
$\Gamma$-covering $\tM \to M$ be a universal covering.
We can look at the set of all possible lifts of
diffeomorphisms of $M$. This defines a new group $\widetilde\Diffeo(M)$ acting
by diffeomorphisms on $\tM$ and fitting into a group extension
\begin{equation}\label{eq:diff_tilde}
1\to \Gamma\to \widetilde\Diffeo(M)\to\Diffeo(M)\to 1.
\end{equation}
The second map assigns to a lifted diffeomorphism the map it covers, and by
definition the deck transformation group, identified with $\Gamma$ via the
principal action of $\Gamma$ on $\tM$, is the kernel of this map:
the transformations which cover the identity. The map
$\widetilde\Diffeo(M)\to\Diffeo(M)$ indeed is surjective by the universal
lifting property of coverings, the lifting condition is automatically
satisfied because the fundamental group of $\tM$ is trivial.
\begin{remark}
	Note that, in general, this exact sequence does not split, i.e.~we do not have
	an action of $\Diffeo(M)$ on $\tM$. An easy example of the phenomenon is
	provided by the circle $M=U(1)$: if we look at the subgroup of diffeomorphisms
	given by $U(1)$ itself (acting by group multiplication, i.e.~by rotations by a
	fixed angle), then the resulting group extension is $0\to\integers\to\reals\to
	U(1)\to 0$, which does not split.
\end{remark}

Let us fix a diffeomorphism $\psi\colon M\to M$ and choose a lift 
\begin{equation}\label{iso-principal-bundle}
  F\colon   \tM\to \tM,  \end{equation}
(a diffeomorphism covering $\psi$). This
diffeomorphism is unique up to pre- or post-composition with a deck-transformation,
i.e.~an element of $\Gamma$.

\begin{definition}
  The map $F$ defines a map $\alpha^F\colon\Gamma\to \Gamma$ as follows: for
  $x\in \tM$, observe that the image of the $\Gamma$-orbit of $x$ under $F$ is
  the $\Gamma$-orbit of $F(x)$, $F(x\cdot
  \Gamma)=F(x)\cdot\Gamma$. Consequently, for $\gamma\in\Gamma$ there is a
  unique $\alpha^F(\gamma)\in\Gamma$ 
  with
  \begin{equation}\label{twisted-equivariance}F(x\cdot \gamma )=F(x)\cdot\alpha^F(\gamma).
  \end{equation} 
\end{definition}
\begin{proposition}\label{prop:change_Gamma_structure}
	The map $\alpha^F$ is an automorphism of $\Gamma$ and its construction
        is independent of $x$.
The map $F\colon \tM\to \tM$ is $\alpha^F$-equivariant,
        i.e.~$F(x\gamma)=F(x)\alpha^F(\gamma)$ for all $x\in \tM$, $\gamma\in\Gamma$.
        Finally, if we change $F$ by pre-composition with the deck transformation $g\in \Gamma$,
	the corresponding automorphism $\alpha^{F\circ g}$ is obtained from $\alpha^F$ by pre-composition
	with the inner automorphism of conjugation with $g$, $\alpha^{F\circ
          g}(\gamma) = \alpha^F(g^{-1}\gamma g)$ $\forall \gamma\in\Gamma$. If
        we change $F$ by post-composition with $g$, we have to post-compose
        $\alpha^F$ correspondingly and get $\alpha^{g\circ
          F}(\gamma)=g^{-1} \alpha^F(\gamma) g$ for all $\gamma\in\Gamma$.
\end{proposition}
\begin{proof}
	First, we show that $\alpha^F(\gamma)$ is independent of $x\in\tM $. If
	$y\in \tM $, choose a path $p$ connecting $x$ to $y$, which exists because
	$M$ is connected. Then $p \gamma$ is a path connecting $x\gamma$ to
	$y\gamma$, and $F(p\gamma) )$ is a path starting at $F(x)\alpha^F(\gamma)$. On
	the other hand, $F(p)$ is a path connecting $F(x)$ to $F(y)$ and
	$F(p)\alpha^F(\gamma)$ also is a path starting at $F(x)\alpha^F(\gamma)$.
	Both paths $F(p)\alpha^F(\gamma)$ and $F(p\gamma)$ have the same image in 
	$M$ and start at the same point in $\tM $, therefore they coincide
        and indeed $F(y)\alpha^F(\gamma) = F(y\gamma)$ as endpoint of this one
        path.
	
	If now $\gamma_2\in\Gamma$ is a second element, then
	\begin{equation*}
	F(x)\cdot(\alpha^F(\gamma_2\gamma))=F(x\cdot
        \gamma_2\gamma)=F(x\cdot \gamma_2)\cdot\alpha^F( \gamma)=
	F(x)\cdot \alpha^F(\gamma_2)\alpha^F(\gamma)
	\end{equation*}
	so $\alpha^F$ is indeed a homomorphism. As $F$ maps the orbit $x\Gamma$
	bijectively to $F(x)\Gamma$, $\alpha^F$ is bijective and hence
	$\alpha^F\in\Aut(\Gamma)$.

        The statement that $F\colon \tM\to \tM$ is
        $\alpha$-equivariant is a reformulation of the fact that the definition
        of $\alpha^F(\gamma)$ is independent of $x\in\tM$.
	Another direct computation shows that $\alpha^{F\circ
          g}(\gamma) = \alpha^F(g^{-1}\gamma g)$ and $\alpha^{g\circ F}(\gamma)=g^{-1}\alpha(\gamma)g$ for all $g,\gamma\in\Gamma$.
\end{proof}

\begin{lemma}\label{lem:alpha_in_Out}
  The outer automorphism $\alpha(\psi)$ assigned to $\psi$ by
  \eqref{eq:diffeo_to_out} coincides with the outer
  automorphism induced by $\alpha^F$ of 
  Proposition \ref{prop:change_Gamma_structure}. Note that, as outer
  isomorphism, by Proposition
  \ref{prop:change_Gamma_structure} $\alpha^F$ is independent of the
  particular lift $F$ of $\psi$.
\end{lemma}
\begin{proof}
  By definition, the (outer) automorphism $\alpha^F$ is determined by 
  $F (\tilde x\gamma) = F(\tilde x)\alpha^F(\gamma)$ for an
  arbitrary $\tilde x\in\tM$ and for $\gamma\in\Gamma$.

  Let $\gamma\in\Gamma$ be
  represented by a loop starting at
  $x=p(\tilde x)\in M$, and lift it to a path $\tilde\gamma$ from $\tilde x$
  to $\tilde x \gamma$. Then $F(\tilde x\gamma)=F(x)\cdot \alpha^F(\gamma)$ is the endpoint of the
  path $F(\tilde \gamma)$ starting at $F(\tilde x)$. Its image
  under the covering projection $p\colon\tM\to M$ is the image under $\psi$
  of the initial loop $p(\tilde \gamma)$ representing $\gamma$, i.e.~is
  $\pi_1(\psi)(\gamma)$, so that $\alpha(\psi)(\gamma)=\pi_1(\psi)(
  \gamma)=\alpha^F(\gamma)$. Note that, modulo inner automorphisms, we can ignore
  base points here and everything is well defined, and the claim is proved.
\end{proof}

The following observation is straightforward.
\begin{lemma}
  Assume we have a vector bundle $E$ over $M$ and a vector bundle isomorphism
  $\Psi\colon E\to E$ covering the diffeomorphism $\psi\colon M\to M$. Denote by ${\tE}$ the
  (equivariant) lift of $E$ to $\tM$.

  Then $\Psi$ pulls back to a vector bundle isomorphism $\tilde\Psi\colon\tE\to
  \tE$ covering every lift $F\colon\tM\to\tM$ of $f$.
\end{lemma}

\subsection{Action on pseudodifferential operators}
\label{sec:action_on_pseudos}
It is a standard fact, see for instance
Theorem 4.1 in Shubin's book \cite[Formula 4.2]{Shubin}, 
that if $X$ is a smooth manifold, then $\Diffeo(X)$ acts on properly supported pseudodifferential operators by conjugation through pull-back operators.
More precisely, let $f\in \Diffeo(X)$ and $K$ a properly supported pseudodifferential operator with distributional kernel $K(x,y)$. Then
$f_*K$ is obtained as $(f^{-1})^*\circ K\circ f^*$, with distributional kernel
$$f_*K(x,y):=K(f^{-1}(x), f^{-1}(y)).$$

	Now let us concretely describe the action of $f\in \Diffeo(M)$ on the pseudodifferential
	operator algebras showing up in our description of the Higson-Roe sequence. Let $K(x,y)$ be the kernel of an element $K$ in $\Pdo_{\Gamma,c}^{0}(\tM)$,
	i.e.~$K$ is a $\Gamma$-equivariant distribution on $\tM\times \tM$ for
	the diagonal action of $\Gamma$.
	Let $\widetilde{f}\in\widetilde{\Diffeo}(M)$ be a lift of $f$.
	Then, since $\widetilde{f}$ enjoys property
        \eqref{twisted-equivariance}, it is immediate that $\widetilde{f}_*K$
        is $\Gamma$-equivariant and that this action does not depend on the
        choice of the lift $\widetilde{f}$.
	Moreover, this action preserves both the ideal of smoothing operators and the subalgebra $C(M)$, where the action reduces
	to the usual action of $f\in \Diffeo(M)$ on $C(M)$ by precomposition
        with $f^{-1}$. 
By continuity, this extends to an action of $\Diffeo(M)$ by automorphisms of C*-algebras.

Finally, let us consider the case where our pseudodifferential operators act
on the sections of a given vector bundle $E\to M$.

\begin{definition}\label{def:Diffeo_E}
  We now need a slight refinement of \eqref{eq:diff_tilde}.  This is given by a subgroup $\Diffeo_0(M)$ of
  $\Diffeo(M)$ and an extension
  $1\to K\to \Diffeo(E)\xrightarrow{\pi} \Diffeo_0(M)\to 1$, together with a
  lift of the action of $\Diffeo_0(M)$ on $M$ to an action of $\Diffeo(E)$ on
  $E$ through bundle automorphisms such that $\psi\in\Diffeo(E)$ covers
  $\pi(\psi)\in\Diffeo(M)$. Then all the constructions work in a same way and
  give us an action of $\Diffeo(E)$ by automorphisms of
  C*-algebras. Notationally, we will ignore these details and write
  $\Diffeo(M)$ throughout, but we should have in mind that, strictly speaking,
  it has to be replaced by $\Diffeo(E)$ where appropriate.
\end{definition}
\begin{remark}
  The situation we have in mind here is the one where $E$ is the spinor bundle
  of a spin structure. Then $\Diffeo_0(M)$ is the subgroup of $\Diffeo(M)$ of
  diffeomorphisms which are covered by a map of the spin principal
  bundle. There are always two such lifts, and $\Diffeo(E)$ constitutes the
  resulting bundle maps, a $2$-fold covering of $\Diffeo_0(M)$. In this case,
  $K=\integers/2$. 
\end{remark}

\subsection{Action on  the Higson-Roe exact sequence }

Let $M$ be a closed manifold and $E$ a vector bundle over $M$ such that we are
in the situation of Definition \ref{def:Diffeo_E}. Then,
by what we have just seen, $\Diffeo(E)$ naturally acts on the Higson-Roe exact
sequence 
\begin{multline*}
  \xrightarrow{\partial} K_*(0\to
  \Pdo^{-\infty}_{\mathcal{A}\Gamma}({\tM},\tE))\xrightarrow{i_*}K_*(C^\infty(M)\xrightarrow{\mathfrak{m}}\Pdo^0_{\mathcal{A}\Gamma}({\tM},{\tE}))\\
  \xrightarrow{\sigma_*}K_*(C^{\infty}(M)\to \Pdo^0_{\mathcal{A}\Gamma}({\tM},{\tE})/\Pdo^{-\infty}_{\mathcal{A}\Gamma}({\tM},{\tE}))\xrightarrow{\partial}
\end{multline*}

   \subsection{Action on the Stolz positive scalar curvature exact
     sequence}
\label{sec:action_on_Stolz}

  The group $\Diffeo(M)$, or more generally any subgroup $U$ of the
  homeomorphism group of $M$ acts on the whole Stolz exact sequence, simply by
  post-composition of a cycle, which is represented by some object together
  with a map to $M$, with the homeomorphism. It is clear that this is
  compatible with the maps in the Stolz exact sequence. The canonical isomorphism
  $\mathrm{R}^{\spin}_{*}(M)\to \mathrm{R}^{\spin}_*(B\Gamma)$ implies that
  the action on this term of the Stolz exact sequences factors through
  $\Out(\pi_1(M))$.

  If a cycle is represented by a metric $g$ on $M$ itself, we have a competing
  option to let $\psi\in\Diffeo(M)$ act, namely by pulling back $g$ with
  $\psi$. Fortunately, these two options coincide:
  \begin{proposition}\label{prop:pullback_and_postcomposition_compatible}
    Let $(f\colon M\to M,g)$ represent a class in $\Pos^{\spin}_n(M)$ and
    let $\psi$ be a spin structure preserving diffeomorphism of $M$. Then
    \begin{equation*}
      [(\psi\circ f\colon M\to M,g)] = [(f\colon M\to M, \psi_*g)] \in \Pos^{\spin}_n(M).
    \end{equation*}

    Similarly, let $(f\colon M\times [0,1]\to M,g_0\disjointunion g_1)$ with $g_0,g_1$
    Riemannian metrics of positive scalar curvature on $M\times \{0\}$ and
    $M\times \{1\}$ represent a class in $\mathrm{R}_{n+1}^{\spin}(M)$. Then
    \begin{equation*}
      [(\psi\circ f\colon M\times [0,1]\to M, g_0\disjointunion g_1)] =
      [(f\colon M\times [0,1]\to M, \psi_*g_0\disjointunion \psi_*g_1)] \in
      \mathrm{R}_{n+1}^{\spin}(M). 
    \end{equation*}

  \end{proposition}
  \begin{proof}
    We construct the spin bordism $W= M\times [0,1/2]\cup_\psi M\times [1/2,1]$
    where we glue $x$ in the left copy to $\psi(x)$ in the right copy. Note
    that $W$ inherits a spin structure from the action of $\psi$ on the spin
    principal bundle which is given because we assume that $\psi$ is spin
    structure preserving. Then
    $f\circ\psi\times \id_{[0,1/2]}\cup f_M\times \id_{[1/2,1]}$ is a well defined
    continuous map $W\to M$. Moreover, the metric $g+dt^2$ on $M\times
    [0,1/2]$ glues isometrically to $\psi_*g+dt^2$ on $M\times [1/2,1]$,
    defining a metric $g_W$ of positive scalar curvature on $W$.
We obtain the desired bordism between $(\psi\circ f\colon M\to M, g)$ and
    $(f\colon M\to M,\psi_*g)$.

    The proof for $\mathrm{R}^{\spin}_{n+1}(M)$ is identical.
  \end{proof}

  \begin{corollary}\label{corol:R_P_diffeo_compatible}
    The free and transitive action of $\mathrm{R}^{\spin}_{n+1}(M)$ on
    $\PosConc(M)$ is compatible with the action of $\Diffeo(M)$ on both sets.
  \end{corollary}
  \begin{proof}
    This is a direct consequence of the defining formula in Proposition
    \ref{prop:Stolz_trans_act} and Proposition
    \ref{prop:pullback_and_postcomposition_compatible}. 
  \end{proof}

\subsection{Action on the cyclic (co)homology of group algebras}

Recall that we have a well
defined homomorphism (independent of issues with basepoints)
$\alpha\colon \Diffeo(M)\to \Out(\pi_1(M))$.

Fortunately, the invariants we care about are of (co)homological nature
for the fundamental group, and inner automorphisms act trivially on
(co)homology. Moreover, the (co)homological constructions we apply are clearly
natural. In particular, we have the following consequence:
\begin{proposition}\label{prop:Out_on_homol}
	There are well defined natural actions of $\Out(\Gamma)$ and via the
	homomorphism $\alpha$ to $\Out(\Gamma)$ also of $\Diffeo(M)$ on
	\begin{equation*}
	H_*(\mathcal{A}\Gamma) \into HC_*(\mathcal{A}\Gamma) \to HP_*(\mathcal{A}\Gamma)
	\end{equation*}
	and dually on
	\begin{equation*}
	HC^*(\complexs\Gamma)\to HP^*(\complexs\Gamma)\to HP^*(\mathcal{A}\Gamma) \to
	HP_*(\mathcal{A}\Gamma)' \to H_*(\mathcal{A}\Gamma)'
      \end{equation*}
 here $'$ denotes the dual space, and cohomology maps to the dual space of homology
 via the usual pairing of cohomology and homology.

 The corresponding statements hold for localized and delocalized (co)homology
 like $H^{del}_{*}(\mathcal{A}\Gamma)$.
\end{proposition}
\begin{proof}
  By naturality of the constructions, it is evident that the true automorphism
  group $\Aut(\Gamma)$ acts on the whole sequence and the maps are
  equivariant.

  Moreover, it is a standard fact that inner automorphisms induce the trivial map on
  group (co)homology. For the (co)homology groups relevant here this already
	holds on the level of the defining cochain complexes. Recall that they
        are obtained from the universal DGA
	over $\mathcal{A}\Gamma$ taking the quotient by the subspace generated
        by commutators. For an inner
	automorphism $\alpha(\gamma)=g\gamma g^{-1}$, and a non-commutative form
	$\omega$ we directly have that $\alpha(\omega)-\omega$ is a commutator.
	For example, $g\gamma g^{-1}-\gamma = g (\gamma g^{-1}) - (\gamma g^{-1})g$
	is the commutator of $g$ and $\gamma g^{-1}$. Similarly,
	$d(g \gamma g^{-1}) -d\gamma = dg \cdot (\gamma g^{-1}) +g d(\gamma g^{-1})
	- d(\gamma g^{-1})\cdot g -( \gamma g^{-1}) dg$ is the sum of the
	commutators of $g$ and $d(\gamma g^{-1})$ and of $dg$ and $(\gamma
	g^{-1})$. The proof for general $\omega$ is left to the reader.
\end{proof}

\subsection{Naturality of the Chern character}
\label{sec:naturality_of_Chern}

We now show that the action of $\Diffeo(M)$ (or rather $\Diffeo(E)$ in the
case of the additional data of a bundle $E\to M$ as in Definition
\ref{def:Diffeo_E}) extends to a canonical and compatible action on the whole
Chern character sequence \eqref{mktoh}, the main construction of this
paper. In fact, the action on the homological part is given through the
homomorphism $\Diffeo(E)\to\Diffeo(M)\to\Out(\Gamma)$ and is the canonical
action of the outer automorphism group on group (co)homology.

More specifically, we will prove the following result.

\begin{theorem}\label{theo:action_on_Ch}
	Let $M$ be connected with fundamental group $\Gamma$, $\tM\to M$ a
	universal covering, considered as $\Gamma$-principal bundle. Fix
	$\psi\in\Diffeo(M)$. Fix an additional bundle $E\to M$ and assume we
        haven an extension $\Diffeo(E)\to\Diffeo_0(M)\to 1$ as in Definition \ref{def:Diffeo_E}.

        For every $\psi\in\Diffeo(E)$ we obtain a commutative diagram
	\begin{equation*}
	\xymatrix{ K_*(C^\infty(M)\xrightarrow{\mathfrak{m}}\Pdo^0_{\mathcal{A}\Gamma}({\tM,\tE}))\ar[r]^{\psi_*}_\cong\ar[d]^{\Ch_{\Gamma}^{del}}& K_*(C^\infty(M)\xrightarrow{\mathfrak{m}}\Pdo^0_{\mathcal{A}\Gamma}({\tM,\tE}))\ar[d]^{{\Ch_{\Gamma}^{del}}}\\
		H^{del}_{[*-1]}(\mathcal{A}\Gamma) \ar[r]^{\alpha(\psi)_*}_\cong&
		H^{del}_{[*-1]}(\mathcal{A}\Gamma)
	}
	\end{equation*}
	and correspondingly for the other terms of the sequence \eqref{mktoh},
	compatible with the maps in \eqref{mktoh}.
	
	The whole construction is natural in the diffeomorphism $\psi$, i.e.~we
	obtain an natural action of $\Diffeo(M)$ on the diagram \eqref{mktoh},
	where the action on cohomology is through the homomorphism
	$\alpha\colon \Diffeo(E)\to  \Diffeo(M)\to \Out(\Gamma)$.
\end{theorem}

We claim that, to a large extent, the assertion of Theorem
\ref{theo:action_on_Ch} is tautological, given that our constructions in
Sections \ref{section5} and \ref{section6} are natural and canonical.

To observe this, we formulate two technical propositions which displays this
fact.

\begin{proposition}\label{prop:naturality_of_Secs_5_6}
  Let $\tM\to M$ be a $\Gamma$-principal bundles over a
  smooth manifold $M$, and $\tN\to N$ a $\Gamma_2$-principal bundle over the
  smooth manifold $N$. Let $\alpha\colon \Gamma_2\to \Gamma$ be a group
  isomorphism, $f\colon N\to M$ be a diffeomorphism, covered
  by an $\alpha$-equivariant diffeomorphism $F\colon \ N\to
  \tM$, i.e.~$F(x\cdot\gamma)=F(x)\cdot\alpha(\gamma)$ for all $x\in\tilde N$,
  $\gamma\in\Gamma_2$. Let $E_M\to M$ and $E_N\to N$ be vector bundles with a bundle
  isomorphism $\Psi\colon E_N\to E_M$ covering the diffeomorphism $\psi\colon
  N\to M$.

  Then the maps $(F,\Psi,\alpha)$ induce tautological isomorphisms making the following diagrams commutative
	\begin{equation*}%\label{hrdiffeo}
	\xymatrix{
		\cdots\ar[r] & K_*(0\to
		\Pdo^{-\infty}_{\mathcal{A}\Gamma_2}({\tN,\tE_N}))\ar[r]^(.45){i_*}\ar[d]_{\iso}^{(F,\Psi,\alpha)_*}&
		 K_*(C^\infty(N)\xrightarrow{\mathfrak{m}}\Pdo^0_{\mathcal{A}\Gamma_2}({\tN,\tE_N}))\ar[d]_{\iso}^{(F,\Psi,\alpha)_*}\\
		\cdots\ar[r]& K_*(0\to \Pdo^{-\infty}_{\mathcal{A}\Gamma}({\tM,\tE_M}))\ar[r]^(.45){i_*}& K_*(C^\infty(M)\xrightarrow{\mathfrak{m}}\Pdo^0_{\mathcal{A}\Gamma}({\tM,\tE_M}))}
            \end{equation*}
	\begin{equation*}%\label{hrdiffeo-a}
{\small	\xymatrix{
		K_*(C^\infty(N)\xrightarrow{\mathfrak{m}}\Pdo^0_{\mathcal{A}\Gamma_2}({\tN,\tE_N}))\ar[r]^(.4){\sigma_*}\ar[d]_{\iso}^{(F,\Psi,\alpha)_*}& K_*(C^{\infty}(N)\xrightarrow{\mathfrak{\pi^*}}
		\Pdo^0_{\mathcal{A}\Gamma_2}({\tN,\tE_N})/\Pdo^{-\infty}_{\mathcal{A}\Gamma_2}({\tN,\tE_N}))
                \ar[d]_{\iso}^{(F,\Psi,\alpha)_*}\\
	 K_*(C^\infty(M)\xrightarrow{\mathfrak{m}}\Pdo^0_{\mathcal{A}\Gamma}({\tM,\tE_M}))\ar[r]^(.4){\sigma_*}&
         K_*(C^{\infty}(M)\xrightarrow{\mathfrak{\pi^*}}
         \Pdo^0_{\mathcal{A}\Gamma}({\tM,\tE_M})/\Pdo^{-\infty}_{\mathcal{A}\Gamma}({\tM,\tE_M}))}
       }
            \end{equation*}
          
  All the vertical isomorphism  are compatible with the Chern character maps of Theorem
  \ref{commutativity-chern} if we map the group homology of $\Gamma_2$ to the
  one of $\Gamma$ via the map induced by the group isomorphism
  $\alpha$. Explicitly, the diagram
	\begin{equation*}
	\xymatrix{
          K_*(C^\infty(N)\xrightarrow{\mathfrak{m}}\Pdo^0_{\mathcal{A}\Gamma_2}({\tN,\tE_N}))\ar[r]^{(F,\psi,\alpha)_*}_\cong\ar[d]^{\Ch_{\Gamma_2}^{del}}& K_*(C^\infty(M)\xrightarrow{\mathfrak{m}}\Pdo^0_{\mathcal{A}\Gamma}({\tM,\tE_M}))\ar[d]^{{\Ch_{\Gamma}^{del}}}\\
		H^{del}_{[*-1]}(\mathcal{A}\Gamma_2) \ar[r]^{\alpha_*}_\cong&
		H^{del}_{[*-1]}(\mathcal{A}\Gamma)
	}
      \end{equation*}
      is commutative, as well as the corresponding other diagrams obtained
      from the sequence \eqref{mktoh}.
 Moreover, the tautological isomorphisms we are using here are natural
      (in the triples $(F,\psi,\alpha)$).
\end{proposition}
\begin{proof}
  The effect of the maps $\alpha$ $f$, $F$, $\Psi$ is the one of
  renaming the elements of $\Gamma$ and the
  points of $M,\tM,E$, preserving all bits of the structure which is used
  in the constructions of Sections \ref{section5} and \ref{section6}. This
  implies that the maps induce natural isomorphisms of the algebras and DGAs
  constructed, compatible with all the maps we apply. The statement for the
  induced K-theory and the compatibility with the Chern character follows, and
  all of this is tautological.
\end{proof}

\begin{proof}[Proof of Theorem \ref{theo:action_on_Ch}]
  The proof of the main result of this section, the fact that we have an
  action of the diffeomorphism group on the whole diagram \eqref{mktoh} (the
  Chern character map from the Higson-Roe analytic sequence to group homology)
  follows directly from the
  abstract compatibility Proposition \ref{prop:naturality_of_Secs_5_6}.

  Let us discuss the details of this derivation:
  It is a direct consequence of Proposition \ref{prop:naturality_of_Secs_5_6}
  and Proposition \ref{prop:change_Gamma_structure} that we have an action of
  $\widetilde{\Diffeo}(E)$, the extension of $\Diffeo(E)$ constructed like 
 $\widetilde{\Diffeo}(M)$ in \eqref{eq:diff_tilde}, and that the corresponding
 action on group homology is via $\alpha^F\in\Aut(\Gamma)$ for
 $F\in\widetilde{\Diffeo}(E)$. 

 It remains to check that the action factors through $\Diffeo(E)$ and through
 the image of $\alpha^F\in\Out(\Gamma)$, which by Lemma \ref{lem:alpha_in_Out}
 is indeed $\alpha(\psi)$ if $F$ is a lift of $\psi\in\Diffeo(E)$.

  The statement about group homology is precisely Proposition
  \ref{prop:Out_on_homol}.

  For the action on the Higson-Roe exact sequence, we observe that the terms
  are given as (relative) K-theory groups of algebras of $\Gamma$-equivariant
  pseudodifferential operators on $\tM$. We have already discussed in
  Section \ref{sec:action_on_pseudos} that this action, which is by pull-back of
  the distributional kernel on $\tM\times \tM$ does not depend on
  the particular lift $F$ of a diffeomorphism $\psi\in\Diffeo(E)$. The claim follows.
\end{proof}

\begin{corollary}\label{corol:natural_pairing}
	Assume that $M$ is a connected manifold with universal covering the
	$\Gamma$-principal bundle $\tM\to M$. Fix $\psi\in\Diffeo(M)$ with
	$\alpha=\pi_1(\psi)\in\Out(\Gamma)$, $x\in \SG^\Gamma_*(\tM)$, and
	$\tau \in HC^{[*-1]}_{del}(\mathcal{A}\Gamma)$.
	
	Then we have naturality of the pairing between K-theory and cohomology
	\begin{equation*}
	\innerprod{\psi_*x,\tau}:=  \innerprod{\Ch_\Gamma^{del}(\psi_* x),\tau} \stackrel{(a)}{=}
	\innerprod{\alpha_*\Ch_\Gamma^{del}(x),\tau} \stackrel{(b)}{=}
	\innerprod{\Ch_\Gamma^{del}(x),\alpha^*\tau} =\innerprod{x,\alpha^*\tau}.
	\end{equation*}
\end{corollary}
\begin{proof}
	We define the pairing between K-theory and cohomology through the Chern
	character map from K-theory to homology and then the standard pairing
	between homology and cohomology. Equation (a) is just the naturality
	assertion of Theorem \ref{theo:action_on_Ch}. Equation (b) is the
	naturality of the pairing between homology and cohomology, here for the
	algebra automorphism induced by $\alpha$ (or rather for any automorphism of
	$\Gamma$ representing the outer automorphism $\alpha$), as detailed in
        Proposition \ref{prop:Out_on_homol}.
\end{proof}

\begin{remark}
	In Section \ref{sec:hyperbolic_pairing} we give rather explicit integral formulae
	for the pairings in Corollary \ref{corol:natural_pairing}. From these
	explicit formulae, one can then also see directly that the pairing is
	natural as stated in Corollary \ref{corol:natural_pairing}. This
        constitutes an alternative proof of Corollary
        \ref{corol:natural_pairing} which also implies Theorem
        \ref{theo:action_on_Ch} due to the fact that the pairing between
        homology and cohomology is sufficiently non-degenerate.
\end{remark}

\begin{proposition}\label{diffeo-rho}
	Let $g$ be a metric with positive scalar curvature on $M$ and $\psi\in
	\mathrm{Diffeo}(M)$ with a lift to the spin principle bundle
	(i.e.~$\psi$ ``is spin structure preserving'') then 
\begin{equation}
	\psi_*(\varrho(\psi^*g))=\varrho(g) \in
        K_*(C(M)\to\Pdo^0_\Gamma(\widetilde{M})).\label{eq:nat_of_rho_g}
\end{equation}
\end{proposition}
\begin{proof}
  Let $\tilde\psi$ be a lift of $\psi$ to $\tM$.
	As in the proof of \cite[Proposition 2.10]{PiazzaSchick_PJM} one can
        prove that $\tilde D_{\psi^*g}$ is obtained from $\tilde D_g$ by
        unitary conjugation with $\tilde{\psi}^*$, or in other words
        $\tilde\psi_*(\tilde D_{\psi^*g})=\tilde D_g$.
	By naturality of functional calculus, this is also true for the sign
        (and other functions) of these operators. Then,
        by Definition \ref{rho-class}, $\varrho(\psi^*g)$ is an explicitly given
        function of $\sgn(D_{\psi^*g})$ and hence \eqref{eq:nat_of_rho_g} follows.
\end{proof}

\begin{corollary}
	Let $g$ be a metric with positive scalar curvature on $M$, $\psi\in
		\mathrm{Diffeo}(M)$ with a lift to the spin principle bundle and 	$\tau \in HC^{[*-1]}_{del}(\mathcal{A}\Gamma)$, then 
	\begin{equation}\label{diffeo-pairing}
	\innerprod{\varrho(g),\tau} =\innerprod{\varrho(\psi^*g),\alpha^*\tau}
	\end{equation}
	with $\alpha=\pi_1 (\psi)\in\Out(\Gamma)$.
\end{corollary}

\subsection{Action on the long exact sequence in cohomology of the pair
	$(B\Gamma, M)$}

The action of $\psi\in\Diffeo(M)$ on cocycles is obviously given by
pull-back. Moreover it
induces a natural automorphism of the whole long exact sequence of cohomology
groups.
  \begin{equation}\label{diffeo-rel-cohom}
{\small  \begin{CD}
  @>>> H^{*-1}(M) @>{\delta}>> H^*(M\to B\Gamma) @>>> H^*(B\Gamma) @>{u^*}>>
  H^*(M) @>>>\\
  &&   @VV{\psi^*}V @VV{(\psi,B\alpha_\psi)^*}V @VV{B\alpha_\psi^*}V @VV{\psi^*}V\\
  @>>>  H^{*-1}(M) @>{\delta}>> H^*(M\to B\Gamma) @>>> H^*(B\Gamma) @>{u^*}>>
  H^*(M) @>>>
  \end{CD} }
  \end{equation}

This works also in the model with Alexander-Spanier cochains on
$\tM$ in Section \ref{sec:AS_cochains}, via a lift $\tilde\psi$ of $\psi$
to $\tM$. As before, the
action on cohomology is independent of the lift. In particular, no issues with
base points (in the cohomology of $B\Gamma$) occur. With little effort, it
follows that the pairings of Theorem \ref{pairing-relative}, Definition
\ref{higher-rho-number-AS} and the whole diagram \eqref{eq:rel_cohom_to_HC}
are compatible with the action of $\Diffeo(M)$.

\section{Moduli spaces of metrics of positive scalar
	curvature}

Given a smooth manifold $M$, the space of metrics of positive scalar curvature
arguably is not the most relevant object: should the many different but
isometric Riemannian metrics really be all distinguished? This is why we are interested in the
\emph{moduli space of Riemannian metrics of positive scalar curvature}
    on $M$
    \begin{equation*}
    \Riem^+(M)/\Diffeo(M),
\end{equation*}
where we divide by the action of the diffeomorphism group, a
diffeomorphism mapping one metric to the other being an isometry between the
two metrics.

However, in geometry there are important metric/topological concepts
which are not preserved by this concept.
One of these, studied intensely in hyperbolic geometry, is the marked length
spectrum.

\begin{definition}
	Let $(M,g)$ be a closed Riemannian manifold. Its \emph{marked length
		spectrum}
	is the map which assigns to each free homotopy class of closed loops (free
	means that the loops don't have to respect a basepoint) the length of the shortest
	closed geodesic in its class.
	
	Less information contains the \textbf{unmarked length spectrum}, which just is
	the subset of the real line consisting of the numbers we obtain this way
	(listed with multiplicity).  
\end{definition}

In light of this important invariant, we propose to study not only the usual
moduli space, but also the \emph{marked moduli space}. This is obtained as the
quotient by the \emph{marked diffeomorphism group}
\begin{equation*}
\Diffeo_m(M):=\{\phi\in
\Diffeo(M)\mid \phi_*=1 \in \Out(\pi_1(M))\},
\end{equation*}
those diffeomorphisms whose
induced outer automorphism of the fundamental group is trivial.

Note that the set of free homotopy classes of loops in $M$ is canonically
in bijection to the set of conjugacy classes in $\pi_1(M)$ (and the latter set
does not depend on the basepoint). This means that the marked diffeomorphisms
are precisely those which act trivially on the set of free homotopy classes of
loops. As a consequence, a \emph{marked isometry} between two metrics
$g_1,g_2$ on $M$, i.e.~an isometry induced by a marked diffeomorphisms,
preserves the marked length spectrum.

We therefore propose to study besides the usual moduli space also the following object.
\begin{definition}
We define the
	\emph{marked moduli space} of metrics of positive scalar curvature on $M$ as the space $\Riem^+(M)/\Diffeo_m(M)$.
\end{definition}

\begin{proposition}\label{prop:detect_moduli_in_Pos}
	Let $M$ be a connected spin manifold of dimension $n$. Assume that $(M,g_1)$
	and $(M,g_2)$ are two metrics of positive scalar curvature 
	which represent two different classes in $\Pos^{\spin}_n(X)$ for a fixed
	reference map $u\colon M\to X$.

	Then there is no diffeomorphism $\phi\colon M\to M$ such that $u\circ \phi$
	is homotopic to $u$ and such that $g_2=\phi^*g_1$.
	In the special case that $X=B\pi_1(M)$ and $u\colon M\to B\pi_1(M)$ induces
	the identity map of fundamental groups, the condition $u\circ\phi\simeq
	u$ means precisely that $\phi$ induces the identity outer automorphism of
	$\pi_1(M)$. 
\end{proposition}
\begin{proof}
	Given a diffeomorphism $\phi$ with $\phi^*g_1=g_2$ and $u\simeq u\circ
	\phi$ we can produce a psc-bordism $W\to B\Gamma$ as follows: $W=M\times
	[0,1/2]\cup_{\phi} M\times [1/2,1]$, i.e.~we glue to cylinders over $M$
	using the diffeomorphism $\phi$. On the first half, we put the metric
	$g_1+dt^2$ and on the second half we put $g_2+dt^2$. The fact that $\phi$ is
	an isometry between $g_1$ and $g_2$ means that this is a well defined smooth
	metric on $W$. As reference map to $B\Gamma$ we use $u\circ pr_M$ on the
	first half of the cylinder, and the homotopy between $u\circ \phi$ and $u$
	on the second half of the cylinder.
	
	It is a standard fact that two maps to
	$B\Gamma$ are homotopic relative to a
	basepoint if and only if they induce the same map on fundamental groups, and
	two maps are homotopic not necessarily preserving the basepoint if and only
	if they induce the same map up to an inner automorphism. As
	$u_*=\id_{\pi_1(M)}$ this is satisfied for the maps $u$ and $u\circ \phi$
	precisely if $\phi$ induces the trivial outer automorphism of $\pi_1(M)$. 
\end{proof}

The following theorem is a direct corollary of Proposition \ref{prop:detect_moduli_in_Pos}.
\begin{theorem}\label{theo:detect_in_Pos}
	Assume that we are in the situation of Proposition
	\ref{prop:detect_moduli_in_Pos}: we have metrics $g_1$ and $g_2$ of positive
	scalar curvature in $M$ which represent different elements in
	$\Pos^{\spin}_n(B\pi_1(M))$.
	Then $g_1$ and $g_2$ lie in different
	components of the marked moduli space $\Riem^+(M)/\Diffeo_m(M)$.
	
	In other words: non-triviality of $\Pos^{\spin}_n(B\Gamma)$ 
	immediately 
	gives non-triviality of the set of components of the marked moduli space $\pi_0(\Riem^+(M)/\Diffeo_m(M))$.
\end{theorem}

\begin{remark}
	It is much more challenging to prove in the context of Theorem
	\ref{theo:detect_in_Pos} that the metrics also lie in different components
	of the full moduli space $\Riem^+(M)/\Diffeo(M)$. Here, our higher rho
	numbers can be used.
\end{remark}

\section{Defining the size of the moduli space}
\label{sec:size_moduli}

We are interested in measuring the ``size'' of the moduli space of metrics of
structure. More specifically, recall that we have a surjective map
$\pi_0(\Riem^+(M))\to {\PosConc}(M)$, where ${\PosConc}(M)$ is the set of concordance
classes of metrics of positive scalar curvature on $M$.

Next, as we have already seen in Proposition \ref{prop:Stolz_trans_act}, there is a free and transitive action of the abelian group $\mathrm{R}^{\spin}_{n+1}(M)$
on ${\PosConc}(M)$ (if $M$ is a spin manifold of dimension $n\ge 5$) with fundamental
group $\Gamma$, compare the
proof of \cite[Theorem 5.4]{Stolz}. Occasionally, one defines a
(non-canonical) group structure on ${\PosConc}(M)$ by choosing a basepoint
$g_0\in \PosConc(M)$ and then identifying $\PosConc(M)$ with the group
$\mathrm{R}^{\spin}_{n+1}(M)$ via the action of $\mathrm{R}^{\spin}_{n+1}(M)$
on $g_0$. In any event, it is customary to measure the size of
${\PosConc}(M)$ as the rank of the group $\mathrm{R}^{\spin}_{n+1}(M)$ which acts
freely and transitively on it.

Because $\pi_0(\Riem^+(M))$ surjects onto ${\PosConc}(M)$, we can then say that the ``rank of
$\pi_0(\Riem^+(M))$'' in this sense is bounded below by the rank of
$\mathrm{R}^{\spin}_{n+1}(M)$, even though $\pi_0(\Riem^+(M))$ in
  general has no abelian group structure.

  Apart from the concordance relation, we often use the bordism relation for
  metric of positive scalar curvature.  It is now a bit more tricky to define
  ranks for the moduli space of metrics of positive scalar curvature. The
  standard way how this is done, interpreted slightly differently from the
  literature, is discussed in the following.

  We have the canonically defined map
  \begin{equation*}
  \pi_0(\Riem^+(M))\to  \PosConc(M) \to \Pos^{\spin}_n(M); [g]\mapsto (M,g,\id),
  \end{equation*}
  where a metrics of positive scalar curvature is mapped to the bordism class of $\id\colon M\to M$,
  with metric $g$ on the source $M$. Note that the image of $\PosConc(M)$
  is not a subgroup, but a coset of the image of
  $\partial\colon \mathrm{R}^{\spin}_{n+1}(\Gamma)\to \Pos^{\spin}_n(M)$; 
 indeed, it is the inverse image  of the fundamental class
  $[M\xrightarrow{\id} M]\in \Omega^{\spin}_n(M)$ under the map  $\Pos^{\spin}_n(M)\to \Omega_n^{\spin} (M)$ of the Stolz
  exact sequence \eqref{HRses}
  and thus a coset of the kernel of this map.

  \begin{definition}
    We define the set  $\Omega\PosConc(M)$ of bordism classes of metrics
     of positive scalar curvature on $M$ as
      \begin{equation*}
    \Omega\PosConc(M):= \im(\PosConc(M)\to \Pos^{\spin}_n(M)).
  \end{equation*}
 Concretely, this means that
    we look at the space of metrics of positive scalar curvature on $M$ but
    impose the even stronger spin bordism relation instead of the concordance
    relation.
    
    We define the \emph{rank of the set of bordism classes of metrics
     of positive scalar curvature on $M$} as the rank of
    the coset $\im(\PosConc(M)\to \Pos^{\spin}_n(M)$);
this is, by definition,  the 
   rank of the corresponding group
    $\im(\mathrm{R}^{\spin}_{n+1}(M)\to \Pos^{\spin}_n(M))$.
  \end{definition}

    Now we consider the action of the group $\Diffeo(M)$ on the whole
    diagram. By Proposition \ref{prop:pullback_and_postcomposition_compatible}
    and Corollary \ref{corol:natural_pairing}, all the maps we consider are
    equivariant for this action. Our goal is to study the quotients by this
    action. When it comes to the abelian groups $\mathrm{R}_*^{\spin}(M)$ and
    $\Pos^{\spin}_*(M)$, the set of orbits will of course not have a group
    structure. Instead, one passes to the coinvariant group.
    Recall that if a group $G$ acts on an abelian group $A$, the coinvariant
    group $A_G$ is defined as the quotient of $A$ by the subgroup generated by
    the set $\{a-ag\mid a\in A, g\in G\}$.

    In our situation, for any subgroup $U\subset \Diffeo(M)$ we then have
for each choice of a basepoint $g_0\in \PosConc(M)$ a surjective map
     \begin{equation*}
       \PosConc(M)/U \to \mathrm{R}^+_{n+1}(M)_U
     \end{equation*}
     induced      by the inverse of the action map, and correspondingly
     a well-defined chain of maps
    \begin{equation*}
      \pi_0(\Riem^+(M)/U)=\pi_0(\Riem^+(M))/U \onto \PosConc(M)/U \to
      \Pos^{\spin}_n(M)_U 
    \end{equation*}
     where the image in $\Pos^{\spin}_n(M)_U$ is a coset of the image group
     $\mathrm{R}^{\spin}_{n+1}(M)_U\to \Pos^{\spin}_n(M)_U$.

 To measure the ``size'' of the moduli space of Riemannian metrics of positive
 scalar curvature on $M$, or the concordance moduli space, it is now customary
 to use the rank of the coinvariant groups $\mathrm{R}^{\spin}_{n+1}(M)_U$ or of its
 image in $\Pos^{\spin}_n(M)_U$, where the latter describes the size of a
 psc-bordism moduli space.

\begin{remark}
 However, there is one subtlety here: when passing from a group $U$ which acts
 on an abelian  group $A$
 to a subgroup $U_2$ of finite index, the rank of the coinvariant group can
 change. Take, as an example $A=\integers$ and $U=\{\integers/2\}$ acting by
 reflection. Then $A_U$ is finite, whereas of course $A_{\{1\}}=A$ so that the
 rank of the coinvariant group increases from $0$ to $1$ if we replace the
 acting group $U$ by the trivial group, which has index $2$ in $U$.
 \end{remark}

 We propose that we get a better idea of the size of the moduli space if we
 consider all subgroups of finite index in $\Diffeo(M)$, because it might just
 happen that some elements of $\Diffeo(M)$ have some ``folding''
 action on $\mathrm{R}^{\spin}_{n+1}(M)$ or $\Pos^{\spin}_n(M)$ which artificially reduce
 the rank of the coinvariant space.

 \begin{definition}
   The \emph{virtual rank of the concordance moduli space $\widetilde\PosConc(M)$} of Riemannian metrics of
   positive scalar curvature on $M$ is defined as
   \begin{multline*}  
  \vrank(\widetilde\PosConc(M)):=
  \vrank(``\PosConc(M)/\Diffeo(M)\text{\strut''})\\
  := \sup\{
     \rank(\mathrm{R}^{\spin}(M)_U)\mid U\subgroup \Diffeo(M)\text{ of finite index}\}.
   \end{multline*}
   Similarly, we define the \emph{virtual rank of the bordism moduli space $\widetilde{\Omega\PosConc}$} of
   Riemannian metrics of positive scalar curvature on $M$ as
   \begin{multline*}
     \vrank(\widetilde{\Omega\PosConc}(M))    \\
     :=      \sup\{ \rank(\im(\mathrm{R}_{n+1}^{\spin}(M)_U\to \Pos^{\spin}_{n}(M)_U))\mid
     U\subgroup \Diffeo(M)\text{ of finite index}\}.
   \end{multline*}
\end{definition}

  \begin{remark}
    Obviously, whenever
    $\vrank(\widetilde\PosConc(M))>0$ then
    \begin{equation*}
\abs{\pi_0(\Riem^+(M)/\Diffeo(M))}= \infty.
\end{equation*}

    This follows because there is a finite index subgroup $U\subset\Diffeo(M)$
    such that $\pi_0(\Riem^+(M))/U$ surjects onto $\PosConc(M)/U$ which maps
    surjectively onto an abelian group $\mathrm{R}^{\spin}_{n+1}(M)_U$ of positive
    rank. In particular, all the sets involved are infinite.

    As $U$ has finite index in $\Diffeo(M)$, also the full quotient
    \begin{equation*}
     \pi_0( \Riem^+(M)/\Diffeo(M)) = \pi_0(\Riem^+(M))/\Diffeo(M)
    \end{equation*}
    then has to be infinite.
  \end{remark}

\begin{remark}
  After choices of basepoints, the above maps define non-natural abelian group
  structures on the quotients  $\PosConc(M)/U$, by identifying them with
  $\mathrm{R}^{\spin}_{n+1}(M)_U$ or with a subgroup of $\Pos^{\spin}_n(M)_U$. In \cite[Definition
  5.7]{Xie-Yu-moduli} the set $\widetilde{\PosConc(M)}$ is defined for the first time
  (using the full group $\Diffeo(M)$), together with a
  non-canonical abelian group structure depending on the choice of a
  basepoint. This
  group structure coincide with the one we are using,
  as shown in \cite[Proposition 4.1]{XieYuZeidler}.

  Although not stated there explicitly, also \cite{Xie-Yu-moduli} employs the
  concept of $\vrank$: there one has to switch from the full diffeomorphism group
  $\Diffeo(M)$ to the spin structure preserving subgroup $U$ to obtain
  compatibility of the rho invariant with the action and then get a well
  defined map $\tilde \varrho\colon \widetilde\PosConc(M)\to
  \mathrm{S}^\Gamma_n(\tM)_U$.
\end{remark}
   
We will apply our techniques to obtain lower bounds on the virtual concordance
moduli space and bordism moduli space of Riemannian metrics of positive scalar
curvature on suitable manifolds $M$.

\chapter{The role of real K-theory}
\label{sec:appl_real_K}

In this paper, for simplicity we have worked throughout with complex
K-theory. To get more precise information on the scope of the applicability to
the classification of positive scalar curvature metrics, we must observe that
the Stolz positive scalar curvature exact sequence actually maps to a real
version of the Higson-Roe analytic surgery exact sequence \cite[Theorem 3.1.13]{Zeidler_Diss}, which we write as
\begin{equation}\label{HRreal}
  \dots \to KO_{*+1}(C^*_{\reals,red}\Gamma)\xrightarrow{\iota} \SG O^\Gamma_{*}(\tM)\to
  KO_*(M)\xrightarrow{\mu_\reals} KO_*(C^*_{\reals,red}\Gamma)\to \dots
\end{equation}

To understand the structure group $\SG O^\Gamma_*(\tM)$ which is our main
object of study, in light of the exact sequence it is important to understand
the KO-theory of the group $C^*$-algebra.

The Baum-Connes conjecture asserts that the assembly map $\mu$ is an isomorphism,
known for large classes of discrete groups, e.g.~for hyperbolic groups:
\begin{equation}\label{eq:real_BC}
  KO^\Gamma_*(\underline{E}\Gamma)\xrightarrow{\mu} KO_*(C^*_{\reals,red}\Gamma).
\end{equation}

Real and complex K-theory are connected by a long exact sequence which
canonically splits
after inverting $2$ (compare e.g.~\cite[Theorem 2.1]{SchickRC}), giving in
particular the split exact sequence
\begin{equation}\label{eq:K_via_KO}
  0 \to KO_p(C^*_{\reals,red}\Gamma)\tensor\rationals \xrightarrow{c}
  K_p(C^*_{red}\Gamma)\tensor\rationals \xrightarrow{\delta}
  KO_{p-2}(C^*_{\reals,red}\Gamma)\tensor\rationals \to 0.
\end{equation}
Here, $c$ is complexification and $\delta$ is the composition of the Bott
periodicity isomorphism with ``forgetting the complex structure''. In other
words, $K_k(C^*\Gamma)\otimes\QQ$ is canonically the direct sum of real
K-theory in degree $k$ and in degree $k-2$.

  After complexification, we have a Chern character isomorphism
  \begin{equation}\label{eq:K_Ch}
    K^\Gamma_p(\underline{E}\Gamma)\tensor\complexs\xrightarrow[\iso]{\Ch} 
    \bigoplus_{k\in\integers} H_{p+2k}(\Gamma;F\Gamma),
  \end{equation}
  worked out in great detail in \cite[Theorem 1.4]{Matthey} or with a
  different model of K-homology in \cite[Section 4,
  5]{CarrilloWangWang}.   Here,
  $F\Gamma:=\{f\colon \Gamma_{fin}\to\complexs\mid \abs{\supp(f)}<\infty\}$,
  where $\Gamma_{fin}$ is the set of elements of finite order in $\Gamma$. It
  becomes a $\Gamma$-module by conjugation. Note that $F\Gamma$ is a direct
  sum of permutation modules $F_{[\gamma]}\Gamma:=\{f\in F\Gamma\mid \supp(f)\subset\langle\gamma\rangle\}$, one for each conjugacy
  class $[\gamma]$ of elements of finite order $\gamma\in\Gamma$.  The
  centralizer $\Gamma_\gamma$ is the stabilizer group of $\gamma$ under the
  conjugation action, and the $\Gamma$-module $F_{[\gamma]}\Gamma$ is
  precisely the trivial $\Gamma_\gamma$-module $\complexs$ induced up to
  $\Gamma$, $F_{[\gamma]}\Gamma=\ind_{\Gamma_\gamma}^\Gamma\complexs$. It
  follows that we get
  \begin{equation}\label{eq:FGamma_decomp}
    F\Gamma = \bigoplus_{[\gamma], \gamma\in\Gamma\text{ of finite order}}
    F_{[\gamma]}\Gamma = \bigoplus_{[\gamma], \gamma\in\Gamma\text{ of finite
          order}} \ind_{\Gamma_\gamma}^\Gamma\complexs.
  \end{equation}

  Recall that group homology is compatible with induction: for a subgroup
  $H\subset\Gamma$ we have a canonical isomorphism
  $H_*(\Gamma; \ind_H^\Gamma F)\xrightarrow{\iso} H_*(H;F)$.  In the present
  case we therefore get, combining this and the decomposition \eqref{eq:FGamma_decomp}
  \begin{equation}
    \begin{split}
      H_*(\Gamma;F_{[\gamma]\Gamma}) &=H_*(\Gamma,\ind_{\Gamma_\gamma}^\Gamma\complexs)\xrightarrow{\iso}
      H_*(\Gamma_\gamma;\complexs),\\
      H_*(\Gamma;F\Gamma) &= \bigoplus_{[\gamma], \gamma\in\Gamma\text{ of finite
          order}} H_*(\Gamma_\gamma;\complexs).\label{eq:deco_complex_Chern_RHS}
    \end{split}
  \end{equation}
See also \cite[(15.11)]{BC-fete}.
  Note that these are precisely the summands in the computation of periodic
  cyclic homology as derived by Burghelea \cite{Burghelea}, as explained in
  the proof of Corollary \ref{corol:periodic_pol}. 
    Even more is true, there is a canonical inclusion
  $H_*(\Gamma;F\Gamma)\into HP_*(\complexs\Gamma)$ which is an isomorphism for
  all groups where conjugacy classes of elements of infinite order don't
  contribute to $HP_*(\complexs\Gamma)$, e.g.~for hyperbolic groups or groups
  of polynomial growth.

  With this canonical map, we get a well known commutative diagram
  connecting the Chern character on the topological side of the Baum-Connes
  assembly map to the Chern character on the operator algebra side
  \begin{equation}\label{eq:CBC_Ch_compl}
    \xymatrix{   K_p^\Gamma(\underline{E}\Gamma) \ar[rr] \ar[d]^{\Ch} &&
      K_p(C^*_{red}\Gamma)\ar[d]^{\Ch} \\
      \bigoplus_{k\in\integers} H_{p+2k}(\Gamma;F\Gamma) \ar[r] &
      \bigoplus_{k\in\integers} HP_{p+2k}(\complexs\Gamma) \ar[r] &
      \bigoplus_{k\in\integers} HP_{p+2k}(\mathcal{A}\Gamma)
    }
  \end{equation}
  where as usual $\mathcal{A}\Gamma$ is a dense and holomorphically closed
  subalgebra of $\Gamma$, and $HP_*(\mathcal{A}\Gamma)$ shows up as explained
  in Section \ref{sec:intro_cyclic}.

  Let us finally define $F_{del}\Gamma:=\{f\in F\Gamma\mid \sum
  f(g)=0\}$. This is a sub-representation of $\Gamma F$ which is complementary
  to the sub-representation of functions supported on $e$. The latter is of
  course the trivial representation and the permutation module
  $F_{[e]}\Gamma$. $F_{del}\Gamma$ therefore is the sum of all permutation
  modules for the conjugacy classes of non-trivial elements.

  Under the monomorphism $H_*(\Gamma;F\Gamma)\to HP_*(\complexs\Gamma)$, the
  image of $H_*(\Gamma; F_{[e]}\Gamma)$ is precisely the localized periodic
  cyclic homology $HP_*^{e}(\complexs\Gamma)$ and $H_*(\Gamma;F_{del}\Gamma)$
  maps to delocalized periodic cyclic homology $HP_*^{del}(\complexs\Gamma)$,
  even isomorphically, if no conjugacy class of an element of infinite order
  contributes non-trivially to $HP_*(\complexs\Gamma)$.

The splitting of complex K-theory into a sum of two (shifted) copies of real
  K-theory of \eqref{eq:K_via_KO} also causes the Chern character for complex
  K-theory \eqref{eq:K_Ch} to split into a Chern character for the real
  K-theory summands, sometimes called
Pontryagin character. This is worked out in \cite[Proposition
2.2]{BarcenasZeidler}  and gives
  \begin{equation}\label{eq:KO-chern}
  KO^\Gamma_p(\underline{E}\Gamma)\tensor\complexs \xrightarrow[\iso]{\Ch}
  \bigoplus_{k\in\integers} H_{p+4k}(\Gamma; F^0\Gamma) \oplus
  H_{p+2+4k}(\Gamma;F^1\Gamma).
\end{equation}
Here, we need a somewhat more elaborate splitting of the homology. Observe that $F\Gamma$ is
decomposed as the direct sum $F\Gamma=F^0\Gamma\oplus F^1\Gamma$ with $F^r\Gamma=\{f\in
F\Gamma\mid f(\gamma)=(-1)^r f(\gamma^{-1})\}$, the $\pm 1$-eigenspaces for the
involution $\tau\colon F\Gamma\to F\Gamma$ induced by sending $\gamma$ to
$\gamma^{-1}$. We want to combine this with the decomposition of $F\Gamma$
into permutation modules $F_{[\gamma]}\Gamma$. Unfortunately, this is a bit
more involved as not all the $F_{[\gamma]}\Gamma$ are preserved by $\tau$.
To sort this out we make the following observations and definitions.

  \begin{lemma}\label{lem:classify_tau_gamma_parts}
    For $\gamma\in\Gamma$ of finite order, let
      $F_{\{[\gamma],[\gamma^{-1}]\}}\Gamma$ be the smallest $\complexs\Gamma$
      submodule containing $F_{[\gamma]}\Gamma$ closed under $\tau$. We then
      get the following structure result as $\complexs\Gamma$ modules.
      \begin{itemize}
      \item If $\gamma\in\Gamma$ is of finite order and not conjugate to its
        inverse, then
        \begin{equation*}
          F_{\{[\gamma],[\gamma^{-1}]\}}= F_{[\gamma]}\Gamma\oplus
          F_{[\gamma^{-1}]}\Gamma
        \end{equation*}
        split under the involution $\tau$ into an
        even and an odd summand, both isomorphic to $F_{[\gamma]}\Gamma$ as
        $\Gamma$-modules.
      \item If $\gamma=\gamma^{-1}$, i.e.~$\gamma^2=1$, then $\tau$ acts
        trivially on $F_{\{[\gamma],[\gamma^{-1}]\}}\Gamma=F_{[\gamma]}\Gamma$
        and this module belongs entirely to $F^0\Gamma$.
      \item If $\gamma\ne \gamma^{-1}$ but the two elements are conjugate, let
        $\Gamma_{\{\gamma,\gamma^{-1}\}}$ be the subgroup of $\Gamma$
        conjugating the set $\{\gamma,\gamma^{-1}\}$ to itself. It contains
        $\Gamma_\gamma$ as normal subgroup of index $2$. The $\Gamma$-module
        $F_{\{[\gamma],[\gamma^{-1}]\}}\Gamma = F_{[\gamma]}\Gamma$ can be
        considered 
        as induced up from the trivial $\Gamma_\gamma$-module or,
        equivalently, from the $\Gamma_{\{\gamma,\gamma^{-1}\}}$-module
        $\complexs[\Gamma_{\{\gamma,\gamma^{-1}\}}/\Gamma_\gamma]$. The latter
        decomposes into $\tau$-eigenspaces as the trivial representation of
        $\Gamma_{\{\gamma,\gamma^{-1}\}}$ for the even part and the sign
        representation $\complexs^{\sgn}$ (via projection to the 2-element
        group $\Gamma_{\{\gamma,\gamma^{-1}\}}/\Gamma_\gamma$) as the odd
        part. Therefore, it contributes a summand of the form
        $H_*(\Gamma_{\{\gamma,\gamma^{-1}\}};\complexs)$ to
        $H_*(\Gamma;F^0\Gamma)$ and one of the form
        $H_*(\Gamma_{\{\gamma,\gamma^{-1}\}};\complexs^{\sgn})$ to
        $H_*(\Gamma;F^1\Gamma)$.
      \end{itemize}
    \end{lemma}
    \begin{proof}
      These results are certainly well known and easy exercises. For the
      convenience of the reader and in lack of an appropriate reference, we
      give the arguments.

      If $\gamma$ is not conjugate to its inverse, then
      $F_{[\gamma]}\Gamma\cap F_{[\gamma^{-1}]}=\{0\}$ and we therefore have
      the direct sum decomposition
      $F_{\{[\gamma],[\gamma^{=1}]\}}=F_{[\gamma]}\Gamma \oplus
      F_{[\gamma^{-1}]}\Gamma$ and $\tau$ just exchanges the two summands. We
      have then $\complexs\Gamma$-module isomorphisms from
      $F_{[\gamma]}\Gamma$ to the $\pm1$-eigenspaces of the action of $\tau$
      on $F_{[\gamma],[\gamma^{-1}]}\Gamma$ given explicitly by
      \begin{equation*}
        \begin{split}
          F_{[\gamma]}\Gamma \xrightarrow{\iso}
          F^0_{\{[\gamma],[\gamma^{=1}]\}}\Gamma;& x\mapsto
          \frac{1}{2}(x,x)\\
          F_{[\gamma]}\Gamma \xrightarrow{\iso}
          F^1_{\{[\gamma],[\gamma^{=1}]\}}\Gamma;& x\mapsto
          \frac{1}{2}(x,-x)\\
        \end{split}
      \end{equation*}

      If $\gamma=\gamma^{-1}$ then $\tau(g\gamma g^{-1})=(g\gamma
      g^{-1})^{-1}=g\gamma g^{-1}$ i.e.~$\tau$ acts trivially on
      $F_{[\gamma]}$, which is therefore also equal to
      $F_{\{[\gamma],[\gamma^{-1}]\}}$ and belongs entirely to the $+1$-eigenspace
      $F^0\Gamma$.

      Finally, if $\gamma\ne\gamma^{-1}$ but the two elements are conjugate,
      then $\tau$ preserves $F_{[\gamma]\Gamma}$ and therefore
        $F_{\{[\gamma],[\gamma^{-1}]\}}\Gamma=F_{[\gamma]}\Gamma$. We need to
        identify the $\pm1$-eigenspaces for the action of $\tau$ on
        $F_{[\gamma]}\Gamma$ as $\complexs\Gamma$-modules.

        Note that the subgroup $\Gamma_{\{\gamma,\gamma^{-1}\}}$ acts
        transitively on $\{\gamma,\gamma^{-1}\}$ because we assume that
        $\gamma$ and $\gamma^{-1}$ are conjugate and therefore surjects onto
        the permutation group of this set, which is the $2$-element group. The
        kernel of this surjection consists of those group elements which conjugate $\gamma$ to
        itself, and then automatically also $\gamma^{-1}$ to $\gamma^{-1}$,
        i.e.~of 
        the centralizer $\Gamma_\gamma$. We get the desired extension
        \begin{equation}\label{eq:Gamma_ext}
          1\to \Gamma_{\gamma}\to \Gamma_{\{\gamma,\gamma^{-1}\}}\to
          \integers/2\integers\to 1.
        \end{equation}
    Because $F_{[\gamma]}\Gamma$ is a permutation module via the transitive
    action of $\Gamma$ on $[\gamma]$ (by conjugation) and with stabilizer
    group $\Gamma_\gamma$ of the element $\gamma\in[\gamma]$, we have the
    standard identification as induced representation (the $=$ is an
    isomorphism which depends on the choice of the basepoint
    $\gamma\in[\gamma]$) 
       $F_{[\gamma]}\Gamma = \ind_{\Gamma_\gamma}^\Gamma\complexs$. By
       transitivity of induction of representation, we can write this as
       \begin{equation*}
         F_{[\gamma]}\Gamma =
         \ind^\Gamma_{\Gamma_{\{\gamma,\gamma^{-1}\}}}(\ind^{\Gamma_{\{\gamma,\gamma^{-1}\}}}_{\Gamma_\gamma}\complexs). 
         \end{equation*}
         By the explicit description of induced representations,
         $\ind^{\Gamma_{\{\gamma,\gamma^{-1}\}}}_{\Gamma_\gamma}\complexs$ is
           the $\Gamma_{\{\gamma,\gamma^{-1}\}}$-permutation module
           $\complexs[\Gamma_{\{\gamma,\gamma^{-1}\}}/\Gamma_{\gamma}]$ where
           the group $\Gamma_{\{\gamma,\gamma^{-1}\}}$ acts by left
           multiplication on the coset ``space''
           $\Gamma_{\{\gamma,\gamma^{-1}\}}/\Gamma_{\gamma}$. Note that this
           action is the permutation action on the $2$-element set via the
           projection $\Gamma_{\{\gamma,\gamma^{-1}\}}\to
           \integers/2\integers$ of the extension \eqref{eq:Gamma_ext}.

     To determine the action of $\tau$ in this situation, choose $g\in
     \Gamma_{\{\gamma,\gamma^{-1}\}}\setminus \Gamma_\gamma$, so that we have
     the right coset decomposition $\Gamma_{\{\gamma,\gamma^{-1}\}}=
     \Gamma_\gamma\disjointunion g\Gamma_\gamma$. Note that $g$ acts
     non-trivially on $\{\gamma,\gamma^{-1}\}$ by conjugation, so
     $\gamma^{-1}=g\gamma g^{-1}$. Moreover, the
     $\pm1$-eigenspaces of $F_{[\gamma]}\Gamma$ are generated as
     $\complexs\Gamma$-modules by $\gamma\pm \gamma^{-1} = \gamma\pm g\gamma
     g^{-1}$. Comparing with the definition of the induced representation,
     this implies that the sub-representation $\Gamma^0_{[\gamma]}$ is induced
     up from the trivial sub-representation $\complexs$ of
     $\Gamma_{\{\gamma,\gamma^{-1}\}}$, while $\Gamma^1_{[\gamma]}$ is induced
     up from the sign representation $\complexs^{\sgn}$, where the sign
     representation is defined via the projection
     $\Gamma_{\{\gamma,\gamma^{-1}\}}\to \integers/2$ of
     \eqref{eq:Gamma_ext}.

     Finally, the standard isomorphism of group homology with coefficients
     in induced representations gives in the present case
     \begin{equation*}
       \begin{split}
         H_*(\Gamma; F^0_{[\gamma]}\Gamma) & =
         H_*(\Gamma;\ind_{\Gamma_{\{\gamma,\gamma^{-1}\}}}^\Gamma\complexs) =
         H_*(\Gamma_{\{\gamma,\gamma^{-1}\}};\complexs),\\
         H_*(\Gamma; F^1_{[\gamma]}\Gamma) & =
         H_*(\Gamma;\ind_{\Gamma_{\{\gamma,\gamma^{-1}\}}}^\Gamma\complexs^{\sgn}) =
         H_*(\Gamma_{\{\gamma,\gamma^{-1}\}};\complexs^{\sgn}).         
       \end{split}
     \end{equation*}
   \end{proof}

Observe that in the cases $\gamma$ not conjugate to $\gamma^{-1}$ or
$\gamma=\gamma^{-1}$ of Lemma \ref{lem:classify_tau_gamma_parts} the resulting
modules $F^0_{[\gamma]}\Gamma$ and $F^1_{[\gamma]}\Gamma$ are obviously
induced  up from the trivial (or zero) $\Gamma_\gamma$-module. We therefore
have a complete and explicit decomposition of the target of the Pontryagin character for
real K-theory, refining \eqref{eq:deco_complex_Chern_RHS}.

\begin{corollary}\label{cor:pontryagin}
  The Pontryagin character \eqref{eq:KO-chern} can be written as an
  isomorphism
  \begin{equation}\label{eq:del_Chern}
{    \begin{split}
      K&O^\Gamma_p(\underline{E}\Gamma)\tensor\complexs \xrightarrow[\iso]{\Ch}\\
      &\bigoplus_{k\in\integers} \left(\bigoplus_{[\gamma]\ne [\gamma^{-1}]\text{ or
          }\gamma=\gamma^{-1}} H_{p+4k}(\Gamma_\gamma;\complexs) \oplus
        \bigoplus_{[\gamma]=[\gamma^{-1}]\text{ and } \gamma\ne \gamma^{-1}}
        H_{p+4k}(\Gamma_{\{\gamma,\gamma^{-1}\}};\complexs)\right)\\
      \oplus & \bigoplus_{k\in\integers} \left( 
        \bigoplus_{[\gamma]\ne [\gamma^{-1}]}
        H_{p+2+4k}(\Gamma_\gamma;\complexs) \oplus 
        \bigoplus_{[\gamma]=[\gamma^{-1}]\text{ and } \gamma\ne \gamma^{-1}}
        H_{p+2+4k}(\Gamma_{\{\gamma,\gamma^{-1}\}};\complexs^{\sgn}) \right)
    \end{split}}
  \end{equation}
 Here, the sums are over all conjugacy classes $[\gamma]$ of elements
 $\gamma\in\Gamma$ 
 of finite order (satisfying the additional conditions as stated in the
 formula).

   The dual cohomology groups decompose accordingly (as direct products
   instead of direct sums), and the homology/cohomology pairing respects this
   decomposition, i.e.~each homology summand pairs non-trivially only with the
   corresponding cohomology summand and these individual pairings are
   non-degenerate. For later reference, let us also spell this out explicitly:
   \begin{equation}
     \begin{split}
       \label{eq:Ch-target-cohom-decomp}
       \bigoplus_{k\in\integers} &H^{p+4k} (\Gamma;F^0\Gamma)\oplus
       H^{p+2+4k}(\Gamma;F^1\Gamma)  \\
      \iso &\bigoplus_{k\in\integers} \left(\prod_{[\gamma]\ne [\gamma^{-1}]\text{
             or }\gamma=\gamma^{-1}} H^{p+4k}(\Gamma_\gamma;\complexs) \oplus
          \prod_{[\gamma]=[\gamma^{-1}]\text{ and } \gamma\ne \gamma^{-1}}
          H^{p+4k}(\Gamma_{\{\gamma,\gamma^{-1}\}};\complexs)
       \right)\\
       \oplus &\bigoplus_{k\in\integers} \left( \prod_{[\gamma]\ne
           [\gamma^{-1}]} H^{p+2+4k}(\Gamma_\gamma;\complexs) \oplus
         \prod_{[\gamma]=[\gamma^{-1}]\text{ and } \gamma\ne \gamma^{-1}}
         H^{p+2+4k}(\Gamma_{\{\gamma,\gamma^{-1}\}};\complexs^{\sgn}) \right)
     \end{split} 
   \end{equation}

\end{corollary}

Under the decomposition of complex K-theory as direct sum of two shifted
copies of real K-theory $K_*\otimes\complexs\iso KO_*\tensor\complexs\oplus
KO_{*-2}\tensor\complexs$, the
complex Baum-Connes assembly map splits as the direct sum of two shifted
copies of the real Baum-Connes assembly map \eqref{eq:real_BC}, and also the
group cohomology $H^*(\Gamma;F\Gamma)$ 
which pairs with $K_*^\Gamma(\underline E\Gamma)$  splits compatibly, as described in \eqref{eq:del_Chern}.

Note that this description is compatible with the computation of periodic
cyclic group cohomology in Corollary \ref{corol:periodic_pol} of
\cite{Burghelea}.
Specifically, we have a real version of the diagram \eqref{eq:CBC_Ch_compl} as
follows
  \begin{equation}\label{homology-compatibility}
    \xymatrix{   KO_p^\Gamma(\underline{E}\Gamma) \ar[r] \ar[d]^{\Ch} &
      KO_p(C^*_{\reals,red}\Gamma)\ar[d]^{\Ch} \\
      \bigoplus_{k\in\integers} H_{p+4k}(\Gamma;F^0\Gamma) \oplus
      H_{p+2+4k}(\Gamma;F^1\Gamma) \ar[r] &
 HP_{p}(\mathcal{A}\Gamma)^0\oplus  HP_{p}(\mathcal{A}\Gamma)^1}
  \end{equation}
where the superscripts $0$ and $1$ of the periodic cyclic homology groups denote
the $\pm1$-eigenspaces under the endomorphism of cyclic homology induced by
mapping $\gamma\in\Gamma$ to $\gamma^{-1}$.

The pairings with the summands $H^*(\Gamma;F^{0/1}_{\{[\gamma],[\gamma^{-1}]\}}\Gamma)$ of
$HP^*(\complexs\Gamma)$ on the right hand side are compatible with the pairings
of the same summands of
$H^*(\Gamma; F\Gamma)$ with the left hand side.

Summarizing, if $\Gamma$ is hyperbolic or of polynomial growth, the dual of the  bottom line in \eqref{homology-compatibility} tells us that we have an isomorphism
\begin{equation}\label{HPtoH}
HP^*_{del}(\CC\Gamma)\to H^*(\Gamma, F_{del}\Gamma)
\end{equation}  
and that, for $\gamma\in \Gamma$ of finite order, $H^*(\Gamma_{\gamma};\CC)$ is isomorphic to a summand, in a compatible way, of both members of \eqref{HPtoH}. 
We obtain in this way the following commutative diagram 
\begin{equation}\label{eq:pairing_nat}
\xymatrix{
	&H^{*+4k+2\epsilon}(\Gamma_\gamma;\CC)\ar[rd]\ar[d]\ar[ld]&\\
	HP^{*}_{del}(\CC\Gamma)\ar[r]^=\ar[d]&	HP^{*}_{del}(\CC\Gamma)\ar[r]\ar[d]&H^*(\Gamma, F_{del}\Gamma)\ar[d]\\
	\SG O_{*-1}^\Gamma(\tM)' \ar[r]& KO_*(C^*_{\reals,red}\Gamma)'\ar[r]&KO_*^\Gamma(\underline{E}\Gamma)'
	}
\end{equation}
for  $\gamma\in\Gamma$  of
finite order and $[\gamma]\ne [\gamma^{-1}]$,
$k\in\integers$ and $\epsilon\in\{0,1\}$;
a  similar diagram can be written  for the other $\gamma\ne e$ summands of
\eqref{eq:del_Chern}.

\chapter{Geometric applications of delocalized cyclic cohomology}\label{sect:application-via-cyclic}

We are now in the situation to formulate a meta theorem which describes one of the
main geometric application of our higher rho numbers, which help to
distinguish metrics of positive scalar curvature.

First, observe that for $\psi\in\Aut(\Gamma)$ the $\Gamma$-module $\psi^*F\Gamma$ can be canonically
identified with $F\Gamma$, and therefore indeed $\Out(\Gamma)$ acts on
$H^*(\Gamma;F\Gamma)$, the action also preserves the decomposition in the even
and odd part relevant in the Chern character for real K-theory.

  \begin{theorem}\label{theo:geometric_meta}
    Let $M$ be a closed $n$-dimensional spin manifold with a Riemannian metric
    $g$ of positive scalar curvature. Set $\Gamma=\pi_1(M)$. We have the
    $KO$-theoretic rho invariant $\varrho(g)\in \SG O^\Gamma_*(\tM)$.
    Assume there is a
    subspace $H\subset H^*(\Gamma;F_{del}\Gamma)$of dimension $k>0$ which is
    pointwise 
      fixed by a finite index subgroup $V$ of $\Out(\Gamma)$.

    Assume that there is a dense and holomorphically closed Fr\'echet
      subalgebra $\mathcal{A}\Gamma$ of $C^*_{red}\Gamma$ which contains
      $\complexs\Gamma$ and to which all $\lambda\in H$ extend. By the
    construction of Sections \ref{section6} and \ref{sec:hyperbolic_pairing},
    specifically by Definition \ref{higher-rho-number}, the subgroup
      $H$ then pairs with $\SG O^\Gamma_n(\tM)$.

    Assume that we can find a subgroup
      $B\subset KO_{n+1}(C^*_{\reals,red}\Gamma)$ such that the pairing
      between $B$ and $H$ is non-degenerate (in particular, $\rank(B)\ge k$)
      and assume that we can construct metrics $g_\beta$ on $M$ of positive
      scalar curvature for all $\beta\in B$ such that
      \begin{equation}\label{eq:change_rho}
        \varrho(g_\beta) = \varrho(g) + \iota( \beta)
      \end{equation}
      with $\iota$ as in \eqref{HRreal}.  Then the metrics
      $\{g_\beta\mid \beta \in B\}$ represent infinitely many different
      components of the moduli space of metrics of positive scalar curvature
      of $M$. More precisely, their images in $\Pos^{\spin}_n(M)$ span a coset
      of rank $k$ whose image in the coinvariant group $\Pos^{\spin}_n(M)_U$
      for a suitable finite index subgroup $U$ of $\Diffeo(M)$ still has rank
      $k$.

    In particular, $\vrank(\tilde\PosConc(M))\ge k$ and even
      $\vrank(\widetilde{\Omega\PosConc}(M))\ge k$, the virtual rank of the
      concordance moduli space and even of the bordism moduli space of metrics
      of positive scalar curvature is $\ge k$.
  \end{theorem}
  \begin{proof}
    By naturality of the pairings \eqref{eq:pairing_nat} and linearity, we
    have
    \begin{equation}\label{eq:rho_versus_K_pairing}
        \innerprod{\varrho(g_\beta),\lambda}-\innerprod{\varrho(g_{\beta'}),\lambda} =  
        \innerprod{\beta-\beta',\lambda}\qquad\forall \beta\in B,\lambda\in H. 
    \end{equation}
    Recall that the subgroup $\Diffeo^{\spin}(M)$ of spin structure
      preserving diffeomorphisms of $M$ has finite index in $\Diffeo(M)$. Let
      $U$ be the intersection of $\Diffeo^{\spin}(M)$ and the inverse image of
      $V$ under $\alpha\colon \Diffeo(M)\to\Out(\Gamma)$, which is then also a
      finite index subgroup of $\Diffeo(M)$.

    Because $H$ is invariant under $U$ and $U$ consists of spin structure
    preserving diffeomorphisms, its pairing with $\Pos^{\spin}_n(M)$ descends
    to a pairing with the coinvariants
    \begin{equation*}
      H\times \Pos^{\spin}_n(M)_U\to \complexs.
    \end{equation*}
    Indeed,  for a diffeomorphism $\psi$ and a metric $g$ of positive scalar
    curvature on $M$ we have
    \begin{equation}\label{eq:diffeo_comp}
      \innerprod{\varrho(\psi^*g_\beta),\lambda} =
      \innerprod{\varrho(g_\beta),(\alpha_\psi^{-1})^*\lambda}=\innerprod{\varrho(g_\beta),\lambda}\quad\forall\beta\in
      B,\lambda\in H.
    \end{equation}
    Here we use \eqref{diffeo-pairing} for the first equality and for the
    second one we use our invariance assumption on $H$.

    By \eqref{eq:rho_versus_K_pairing}, on the affine subgroup spanned by the
    $g_\beta$ for $\beta\in B$, the pairing with $H$ is at least as
    non-degenerate as the pairing of $B$ with $H$ and therefore the rank of
    this affine subgroup is at least $k$.

      Finally, because $U\subset \Diffeo(M)$ has finite index,
      $\Riem^+(M)/\Diffeo(M)$ is a further finite quotient of the infinite set
      $\Riem^+(M)/U$
      and therefore the classes of the $g_\beta$ still lie in infinitely many
      different components.
    \end{proof}

We now give specific examples for the meta theorem. We start with auxiliary
results which guarantee the existence of cocycles with the desired invariance
property.

  The most obvious one is that the outer automorphism group is
  finite. In this case, $H$ in Theorem \ref{theo:geometric_meta} 
  can be taken to be equal to $H^*(\Gamma;F_{del}\Gamma)$. Generically, this is the case for hyperbolic groups, as proven by
  Bestvina and Feighn, refining a result of Paulin:

\begin{proposition}[{\cite[Corollary 1.3]{BestvinaFeighn}}]\label{prop:BestvinaFeighn}
  Let $\Gamma$ be a hyperbolic group. Then $\abs{\Out(\Gamma)}=\infty$ if and
  only if $\Gamma$ does split over a virtually cyclic group
\end{proposition}

Recall here that a group $\Gamma$ is defined to \emph{split over a subgroup
  $\Lambda$} if $\Gamma$ is a non-trivial amalgamated product, amalgamated over $\Lambda$,
or an HNN extension with base group $\Lambda$.

\begin{example}
  Clearly, non-abelian free groups split over the trivial group, and
  hyperbolic surface groups split over infinite cyclic groups, and their outer
  automorphism groups are infinite.

  On the other hand, hyperbolic fundamental groups of aspherical manifolds
  don't split over virtually cyclic subgroups (and have finite outer
  automorphism group), compare \cite[\S5, 5.4.A]{Gromov-hyperbolic}. These
  groups are torsion free, but then also any finite extension $\Gamma$ of such
  a group doesn't split and therefore $\abs{\Out(\Gamma)}<\infty$.
\end{example}

Our second class of examples makes use of the specific structure of the
delocalized cyclic group cohomology to identify subspaces which necessarily
are fixed under a finite index subgroup of the outer automorphism group,
without making any a priori assumption on the latter.
For this, it is important to observe that the action of $\Diffeo(M)$ on
$HP^*(\complexs\Gamma)$ (which factors through the action of $\Out(\Gamma)$)
can be refined. Recall the splitting of $HP^*(\complexs\Gamma)$ given in \eqref{eq:Ch-target-cohom-decomp}. 
The naturality of the isomorphism implies that the action of $\Out(\Gamma)$ on
$HP^*(\complexs\Gamma)$
becomes the evident action on this direct product:
$\alpha\in\Aut(\Gamma)$ maps $\Gamma_\gamma$ to $\Gamma_{\alpha(\gamma)}$ and
this map induces a well defined action of $\Out(\Gamma)$
on the cohomology groups forming the summands of the decomposition \eqref{eq:Ch-target-cohom-decomp}. 
In other words, $\alpha_\psi\in\Out(\Gamma)$ acts by permuting
the summands corresponding to different conjugacy classes the same way
$\alpha_\psi$ permutes the conjugacy classes.

\begin{lemma}\label{lem:dim_1}

    Assume that, in the decomposition \eqref{eq:Ch-target-cohom-decomp}
    dual to the right hand side of the Pontryagin character,
    one of the summands $H$, e.g.~the group 
    cohomology of one of the centralizers in one fixed degree, satisfies
    $\dim(H)=1$.
  
    Assume that $\abs{\Out(\Gamma)\cdot [\gamma]}<\infty$,
    i.e.~applying all automorphisms of $\Gamma$ to $\gamma$ yields
    only finitely many new conjugacy classes.
Then a subgroup of finite index of $\Out(\Gamma)$ acts trivially on the
    summand $H$.

\end{lemma}
\begin{proof}
  By assumption, the $\Out(\Gamma)$-orbit of
  the conjugacy class $[\gamma]$ is finite, therefore its stabilizer
  $\Out(\Gamma)_{[\gamma]}\subset \Out(\Gamma)$ has finite index in
  $\Out(\Gamma)$.

  Finally, the summand $H$ of~\eqref{eq:del_Chern}, which is a fixed  
  cohomology group of a particular subgroup of $\Gamma$, contains 
 the image of the cohomology with $\integers$-coefficients in cohomology
  with $\complexs$-coefficients as an integer lattice. This lattice is isomorphic to $\integers$, because
  $\dim(H)=1$ by assumption. Moreover, the action of $\Out(\Gamma)_{[\gamma]}$
  on 
  $H$ factors through the automorphisms of this integral lattice, i.e.~through
  $\integers/2=\Aut(\integers)$. It follows that a subgroup of index $1$ or
  $2$ acts trivially on $H$. Therefore also a subgroup of index $1$ or $2$
  of $\Diffeo(M)_{[\gamma]}$ acts trivially on $H$. Its intersection with the
  spin structure preserving diffeomorphisms is the desired subgroup $U$ of
  finite index.
\end{proof}

\begin{lemma}\label{lem:diag_action}
  If $\Gamma$ is hyperbolic or of polynomial growth then $\Gamma$ has only
  finitely many conjugacy classes of elements of finite order, in particular,
  for each $\gamma\in \Gamma$ of finite order, $\abs{\Out(\Gamma)\cdot [\gamma]}<\infty$.
\end{lemma}
\begin{proof}
  It is well known that hyperbolic groups have only finitely many conjugacy
  classes of elements of finite order, compare e.g.~\cite[Corollary
  8]{MeintrupSchick}.

  Groups of polynomial growth contain a subgroup of finite index which is
  finitely generated nilpotent and therefore such groups are in particular
  virtually poly-cyclic. The latter groups are known to have a finite model for
  the universal space for proper actions and in particular have only finitely
  many conjugacy classes of finite order \cite{Lueck_Eunderbar}.
\end{proof}

  \begin{corollary}
    Assume that $\Gamma$ is hyperbolic. Then the direct sum of all summands in
    the decomposition dual to \eqref{eq:del_Chern} which are $1$-dimensional
    forms a subspace $H\subset H^*(\Gamma;F_{del}\Gamma)$ which is pointwise
    fixed by a finite index subgroup of $\Out(\Gamma)$.
  \end{corollary}

\begin{example}
  Let $\Gamma$ be a hyperbolic group or a group of polynomial growth.
  Then the direct sum of all summands of homological degree $0$ in the decomposition of
  cyclic group cohomology \eqref{eq:del_Chern} forms a subspace $H$ which is
  pointwise fixed by a finite index subgroup of $\Out(\Gamma)$. Let  $p$ be as in
  Corollary \ref{cor:pontryagin}. For $p\equiv
  0\pmod 4$ its rank is equal to the number of conjugacy classes of
  elements of finite order. For $p=2$ its rank
  is equal to the number of conjugacy classes of elements which are not
  conjugate to their inverse.
  \end{example}

\begin{example}\label{ex:hyp_times_fin}
  Assume that $\Gamma_0$ is a hyperbolic group with unit $e_0$. Assume for
$0=k_1<k_2<\dots<k_r$ that $\dim H^{k_i}(\Gamma_0;\complexs)=1$. As 
  a specific
  example, let $N$ be a closed connected orientable manifold of dimension
$k_2$
  with negative sectional curvature and set $\Gamma_0:=\pi_1(N)$. Let $F$ be a
  non-trivial finite group and set $\Gamma:=\Gamma_0\times F$.

  Then, for each $e\ne \gamma\in F$, its centralizer $\Gamma_{(e_0,\gamma)}$
  surjects onto $\Gamma_0$ with a finite kernel (a subgroup of $F$). In particular,
  $H^*(\Gamma_{(e_0,\gamma)};\complexs) \iso H^*(\Gamma_0;\complexs)$, and therefore
  $\dim H^{k_i}(\Gamma_{(e_0,\gamma)};\complexs) =1$ by the assumption on $\Gamma$.

  Therefore, the $H^{k_i}(\Gamma_{(e_0,\gamma)};\complexs)$ give rise summands $H$ of
  the decomposition \eqref{eq:del_Chern} as in Lemma \ref{lem:dim_1}, and we
  obtain a subspace of delocalized cyclic group cohomology $H\subset
  H^*(\Gamma;F_{del}\Gamma)$ of dimension $rd$ which is fixed by a finite index
  subgroup of $\Out(\Gamma)$, with $d$ the number of non-trivial conjugacy
  classes in $F$.
\end{example}

These results provide an ample supply of examples as required in the
meta 
theorem \ref{theo:geometric_meta}. It remains to discuss the existence of the
metrics $g_k$ which can be distinguished from each other as in Equation
\eqref{eq:change_rho}. This is achieved in \cite[Corollary 1.2]{XieYuZeidler},
or more precisely in the proof of \cite[Theorem 1.1]{XieYuZeidler}, compare also
\cite[Lemma 2.1]{SchickZenobi}. From these references, we obtain the following
theorem.

\begin{theorem}\label{theo:XYZ}
  Let $M$ be a closed connected  Riemannian spin manifold of dimension $n\ge 5$ with a
  Riemannian metric $g$ of positive scalar curvature. Assume that
  $\Gamma=\pi_1(M)$ satisfies the rational strong Novikov conjecture, i.e.~the
  Baum-Connes assembly map $\mu_\reals\colon
  KO_*^\Gamma(\underline{E}\Gamma)\to KO_*(C^*_{\reals,red}\Gamma)$ is
  rationally injective.

  Then for each subgroup $H \subset \bigoplus_{k>0,p\in\{0,1\}} H^{n+1-4k-2p}(\Gamma;F_{del}^p\Gamma)$ (which can be translated to group cohomology of
  centralizers of elements of finite order as in \eqref{eq:del_Chern})
such that the classes in $H$ are represented by  cocycles which extend to a smooth subalgebra
  $\mathcal{A}\Gamma$ of $C^*_{red}\Gamma$, there is a subgroup $B\subset
  KO_{n+1}(C^*_{\reals,red}\Gamma)$ and for all $\beta\in B$ there are metrics
  $g_\beta$ of positive scalar curvature on $M$ such that 
  \begin{equation*}
\varrho(g_\beta)= \varrho(g)+\iota(\beta);\qquad\forall \beta\in B
  \end{equation*}
and such that the pairing between $H$ and $B$ is non-degenerate.
\end{theorem}

Finally, recall that for every finitely generated group and every $n\ge 5$
there are closed connected spin manifolds $M$ of dimension $n$ with
$\pi_1(M)=\Gamma$ and which admit a Riemannian metric of positive scalar
curvature, compare e.g.~\cite[Example 3.5]{SchickZenobi}.

Combined, this gives plenty of examples, e.g.~with fundamental groups as in
Example \ref{ex:hyp_times_fin} with infinitely many components in the moduli
space of Riemannian metrics of positive scalar curvature.

  \begin{corollary}
    Let $\Gamma$ be a hyperbolic group with finite outer automorphism group
    and let $M$ be a closed connected spin manifold of dimension $n\ge 5$ with
    $\pi_1(M)=\Gamma$. Assume that $M$ admits a Riemannian metric of positive
    scalar curvature. Then the virtual rank of the concordance moduli space
    and even of the bordism moduli space of is bounded below by the rank of the
    full delocalized cyclic homology in degrees below $n$, in the appropriate
    degrees mod $4$:
    \begin{equation*}
      \vrank(\widetilde\PosConc(M))\ge \vrank(\widetilde{\Omega\PosConc}(M)) \ge
      \sum_{k>0,p\in\{0,1\}} \rank(H^{n+1-4k-2p}(\Gamma;F^p_{del}\Gamma)). 
    \end{equation*}

    Note that the condition $\abs{\Out(\Gamma)}<\infty$ is in some sense
    generic for hyperbolic groups, as discussed above.
  \end{corollary}
  \begin{proof}
    This is a direct consequence of Theorem \ref{theo:XYZ} and Theorem
    \ref{theo:geometric_meta}. 
  \end{proof}

  \begin{example}\label{example:comparison}
    Let the group $\Gamma$ be hyperbolic or of polynomial growth. Let $M$ be a
    closed connected spin manifold of odd dimension $n\ge 5$ with
    $\pi_1(M)=\Gamma$ which admits a metric of positive scalar curvature.

    Define the effective number of conjugacy classes of elements of finite
    order as
    \begin{equation*}
      d :=
      \begin{cases}
        \abs{\{[\gamma]\subset\Gamma\mid \ord(\gamma)<\infty\}}; & n\equiv -1\pmod 4\\
        \abs{\{[\gamma]\subset\Gamma\mid [\gamma]\ne[\gamma^{-1}]\}}; &
        n\equiv 1\pmod 4 
      \end{cases}
    \end{equation*}
     Then
     \begin{equation*}
       \vrank(\tilde\PosConc(M))\ge \vrank(\widetilde{\Omega\PosConc}(M)) \ge d.
     \end{equation*}

     Note that elements of different finite order are never conjugate to each
     other, so
     \begin{equation*}
d\ge
\begin{cases}
  \abs{\{ k>1\mid \exists \gamma\in\Gamma\text{ with
    }\ord(\gamma)=k\}}; n\equiv -1\pmod 4\\
 \abs{\{k>1\mid \exists\gamma\in\Gamma\text{ with }\ord(\gamma)=k\text{
      and }[\gamma]\ne [\gamma^{-1}]\}}; & n\equiv 1\pmod 4.
\end{cases}
\end{equation*}
This way, the example strengthens for our special
     class of groups $\Gamma$ the main result
       \cite[Theorem 5.8]{Xie-Yu-moduli}, where $\vrank(\tilde\PosConc(M))$ is
       bounded below by the number of orders of elements of finite order in
       $\Gamma$  (note that the assertion there for $\dim(M)\equiv 1\pmod 4$
       is not correct, one has to rule out $\gamma$ with
       $[\gamma]=[\gamma^{-1}]$, as corrected also in \cite{XieYuZeidler}).
   \end{example}

\begin{remark}
  We started with a closed $n$-dimensional connected spin manifold $M$ with
  fundamental group $\Gamma$.  
 Our approach here to distinguish elements of $\pi_0(\Riem^+(M)/\Diffeo(M))$ is
 as 
 follows: we map $\pi_0(\Riem^+(M))$ to the bordism group $\Pos_n^{\spin}(M)$ and
 then use the delocalized higher rho number as an invariant defined on
 $\Pos^{\spin}_n(M)$ which descends to the quotient and has infinite image on
 it. 

 Here, one might wonder why one does not map to the set ${\PosConc}(M)$ of
 concordance classes of metrics of positive scalar curvature on $M$ and then
 use the identification ${\PosConc}(M)\iso R^{\spin}_{n+1}(\Gamma)$ and then a
 not necessarily delocalized higher rho number obtained as composition
 \begin{equation*}
 \pi_0(\Riem^+(M))\to {\PosConc}(M)\to
 \mathrm{R}^{\spin}_{n+1}(\Gamma)\xrightarrow{\Ind^\Gamma} K_{n+1}(C^*\Gamma)
 \xrightarrow{\varrho} \complexs.
\end{equation*}
  The reason is that the identification $\mathrm{R}^{\spin}_{n+1}(\Gamma)\iso {\PosConc}(M)$
  is too non-canonical and therefore not controlled enough: in reality, we can
  only change a metric $g_0$ to a second metric $g_1$ by ``adding'' an element
  of $R^{\spin}_{n+1}(M)$, so the map gives a well-defined higher rho invariant only
  for the difference of two metrics, not for an individual one. There seems no
  way to control the action of the diffeomorphism group on this. The
  identification of  ${\PosConc}(M)$ with $\mathrm{R}^{\spin}_{n+1}(\Gamma)$ depends on the
  choice of a base point $g_0$ in ${\PosConc}(M)$. We would like to choose $g_0$
  invariant under the action of $\Diffeo(M)$, which is impossible except for
  degenerate cases.

Hence, $\Pos^{\spin}_n(M)$ is much more convenient, because a metric $g_0$ on
  $M$ defines canonically an element here, and therefore we can assign in a
  canonical way the delocalized higher rho invariants. We still have to fight
  with the fact that in general we can't compute the rho invariant explicitly,
  but have only ways to control how it changes, which leads to the conditions
  in the meta Theorem \ref{theo:geometric_meta}.
\end{remark}

\begin{remark}
  The infinitely many classes in $\pi_0(\Riem^+(M)/\Diffeo(M))$ detected in
 \cite{Xie-Yu-moduli} are of the same type as ours. However, Xie and Yu focus on
 degree $0$ cyclic cocycles, for which the extension to dense and
 holomorphically closed subalgebras seems a bit easier. Therefore weaker
 conditions on  the group
 than ours, namely ``finite embeddability'' suffice in \cite{Xie-Yu-moduli}. To make the cocycles
 invariant, implicitly some averaging process is used there: instead of using
 the 
 cocycle concentrated on one conjugacy class, it is transported to all
 conjugacy classes obtained from the given one by the action of the
 diffeomorphism group (or rather: any finitely generated subgroup). Because of
 this, the condition of ``finite embeddability'' has to be strengthened a bit
 in \cite{Xie-Yu-moduli}.

 Prior to \cite{Xie-Yu-moduli}, there is \cite{PiazzaSchick_PJM} which proves
 that the moduli space is infinite whenever the fundamental group $\Gamma$ of
 $M$ contains an element $e\ne\gamma\in\Gamma$ of finite order, and
 $\dim(M)\equiv 3\pmod{4}$. This uses the $L^2$-rho invariant which, again, is
 just a cyclic cocycle of degree $0$. Here, this even extends to a
 $C^*$-completion of $\complexs\Gamma$ (the completion associated to the
 direct sum of the regular representation and the trivial representation). For
 K-amenable $\Gamma$, one can work 
 with the reduced completion, and the result of \cite{PiazzaSchick_PJM} is
 just a special case of the work done here.

 For arbitrary groups, one should observe that it is not difficult to work out a
 generalized version of what we develop in the paper at hand which works for
 the relevant completions of the group $C^*$-algebra. There are two approaches:
 \begin{itemize}
 \item Instead of working on the universal covering, one can stick
   systematically to operators acting on the Mishchenko bundle. Then one can
   use arbitrary $C^*$-completions, but leaves a bit the very direct geometric
   description given here (operators on $\tM$ equivariant for the action
   of $\Gamma$)
 \item One can also work simultaneously with several coverings and ``coupled''
   operators on them (for the case at hand used in \cite{PiazzaSchick_PJM},
  it suffices to work with operators just on $\tM$ and $M$). This allows
  to stick to the classical equivariant
   description given here, the details of such an approach are explained in
   \cite[Section 3]{SchickSeyedhosseini}. 
 \end{itemize}
\end{remark}

\chapter{Geometric applications of relative cohomology}
\label{sec:rel_cohom_geometric}

  Assume that $\Gamma$ has property (RD) and $H^*_{pol}(\Gamma)\to
  H^*(\Gamma)$ is an 
  isomorphism. Then
   from Proposition \ref{AStoHC} and Corollary \ref{AStoRel}, we obtain the following commutative diagram
  	\begin{equation}
{\small  	\xymatrix{ H^*(M\to B\Gamma)\ar[r]\ar[d]& H^*(B\Gamma)\ar[r]\ar[d]&H^*(M)\ar[d]\ar[r]^(.4){\delta}& H^{*+1}(M\to B\Gamma)\ar[r]\ar[d]& \\
  		 HC^{*-1}(\mathfrak{m})\ar[r]&HC^{*}(\Psi^{-\infty}_{\mathcal{A}\Gamma}(\tM))\ar[r]^(.65){\partial^{HC}_\Gamma}&HC^{*}(\pi^*)\ar[r]^{(1,\sigma_{pr})^*}&HC^{*}(\mathfrak{m})\ar[r]&
               }
               }
  	\end{equation}
  and  then the pairing of Theorem \ref{pairing-relative} fits in a
  map from the pair sequence in cohomology
  \begin{equation*}
  \cdots\to  H^*(M\to B\Gamma)\xrightarrow{\iota^*} H^*(B\Gamma)\xrightarrow{u^*} H^*(M)\xrightarrow{\partial} H^{*+1}(M\to B\Gamma)\to\cdots
  \end{equation*}
  to the dual of the Higson-Roe analytic surgery exact sequence
  \begin{equation*}
    \cdots \to \SG^\Gamma_{*-1}(\tM)'\xrightarrow{s'} K_{*}(C_{red}^*\Gamma)'\xrightarrow{\mu_M'}
    K_{*}(M)' \xrightarrow{c'} \SG^\Gamma_{*}(\tM)'\to\cdots
  \end{equation*}
  Here, the map $H^*(M)\to K_*(M)'$ --or equivalently the pairing between
  $K_*(M)$ and $H^*(M)$-- is given by the Chern character, and the pairing
  between $H^*(\Gamma)$ and $K_*(C_{red}^*\Gamma)$ is the Connes-Moscovici pairing.

  In other words, if $\tilde\alpha\in H^*(M\to B\Gamma)$ is mapped to
  $\alpha\in H^*(B\Gamma)$ in the pair sequence and if $x\in K_*(C_{red}^*\Gamma)$,
  then
  \begin{equation}\label{compatibility1}
    \langle s(x),\tilde\alpha\rangle = \langle x,\alpha\rangle.
  \end{equation}
  Similarly, if $y\in \SG^\Gamma_*(\tM)$ and $\beta\in H^*(M)$,
  then
  \begin{equation}\label{compatibility2}
    \langle c(y),\beta\rangle = \langle y,\partial(\beta)\rangle.
  \end{equation}

\begin{remark}
  As we can deduce from diagram \eqref{diffeo-rel-cohom}, the pairings and maps above are
  natural for the action of the diffeomorphism group of $M$.
  In particular, $\Diffeo(M)$ preserves the kernel of $u^*\colon
  H^*(B\Gamma)\to H^*(M)$, and of course also the integral cohomology lattice
  inside this kernel.

  If this kernel is $1$-dimensional in degree $j$, then the integral lattice
  is infinite cyclic and its automorphism group is $\integers/2$. Therefore, a
  finite index subgroup of $\Diffeo(M)$ then acts trivially on
  $\ker(u^*)\colon H^j(B\Gamma)\to H^j(M)$.
\end{remark}

\begin{proposition}\label{prop:rel_pairing}
  Let $M$ be a spin manifold of $\dim(M)=n\ge 5$ with a metric $g$ of positive
  scalar curvature. Let $u\colon M\to B\Gamma$ be the classifying map for a
  universal covering of $M$; in particular, $\Gamma\iso\pi_1(M)$ and  assume that $\Gamma$ has property (RD) and $H^*_{pol}(\Gamma)\to H^*(\Gamma)$
  is an isomorphism.
   For example, $\Gamma$ could be hyperbolic or of polynomial
  growth.

Assume that the subspace $H\subset H^{*}(M\to
    B\Gamma;\rationals)$ is pointwise fixed by the action of a finite index
    subgroup  $U\subset \Diffeo^{\spin}(M)$ and is mapped bijectively to
    $H'\subset H^*(B\Gamma;\rationals)$ under the restriction map $r\colon
    H^*(M\to B\Gamma;\rationals)\to H^*(B\Gamma;\rationals)$. Assume
    $\dim(H)=d>0$.

Then the virtual rank of the concordance moduli space
    and even of the bordism moduli space  is bounded below by $k$:
    \begin{equation*}
      \vrank(\widetilde\PosConc(M))\ge \vrank(\widetilde{\Omega\PosConc}(M)) \ge d. 
    \end{equation*}
\end{proposition}
\begin{proof}
    Because $\Gamma$ has property (RD) and polynomial cohomology, it also
    satisfies the strong Novikov conjecture.  By \cite[Theorem 1.5]{SchickZenobi} and its proof (which is based on
  \cite{XieYuZeidler}), there is a subgroup $B\subset
  KO_{n+1}(C^*_{\reals,red}\Gamma)$ and for all $\beta\in B$ there are metrics
  $g_\beta$ of positive scalar curvature on $M$ such that
  \begin{equation*}
    \varrho(g_\beta)= \varrho(g) + \iota(\beta)\quad\forall \beta \in B
  \end{equation*}
  and such that the pairing between $H'$ and $B$ is non-degenerate. Then for
  each $\alpha\in H$ and $\beta\in B$
    \begin{equation*}
      \innerprod{\varrho(g_\beta),\alpha} = \innerprod{\varrho(g), \alpha}
      +  \innerprod{\iota(\beta),\alpha} =
      \innerprod{\varrho(g),\alpha} + \innerprod{\beta,r(\alpha)}
    \end{equation*}
    As the pairing between $H'$ and $B$ is non-degenerate and $\dim(H')=d$,
    this implies that the $(g_\beta)_{\beta\in B}$ span an affine subgroup of
    $\Pos^{\spin}_n(M)$ of rank $\ge d$. Due to the fact that the $\alpha$
    are fixed by the action of $U\subset\Diffeo^{\spin}(M)$, as in the proof
    of Theorem 
    \ref{theo:geometric_meta} this remains true for their images in the
    coinvariant group $\Pos^{\spin}_n(M)_U$ and the assertion about the
    virtual rank of the bordism moduli space follows.
  \end{proof}

  \begin{example}
    In the situation of Proposition \ref{prop:rel_pairing}, assume
    specifically that
    \begin{equation*}
    \dim(\ker(u^*\colon H^{n+1-4k_j}(B\Gamma)\to H^{n+1-4k_j}(M))) =1 \quad\text{for }k_1<k_2<\dots<k_d.
  \end{equation*}
Additionally, assume that the
    degree $4k_j$-component of the $\hat{A}$-genus of $M$ vanishes:
    $0=\hat{A}(M)_{[4k_j]}\in H^{4j}(M;\rationals)$ for $j=1,\dots,d$.  This
    is the case in particular if also $H^{n+1-4k_j}(M\to B\Gamma; \rationals)$
    is $1$-dimensional.

    Then the virtual rank of the concordance moduli space and even of the
    bordism moduli space of is bounded below by $k$:
    \begin{equation*}
      \vrank(\tilde\PosConc(M))\ge \vrank(\widetilde{\Omega\PosConc}(M)) \ge d. 
    \end{equation*}
  \end{example}
  \begin{proof}
    For a diffeomorphism $\Psi$ of $M$ we have the commutative diagram of pair
    sequences in cohomology
    \begin{equation}\label{eq:psi_on_cohom}
{\small      \begin{CD}
        @>>> H^{*-1}(M) @>{\delta}>> H^*(M\to B\Gamma) @>>> H^*(B\Gamma)
        @>{u^*}>>
        H^*(M) @>>>\\
        &&   @VV{\psi^*}V @VV{(\psi,B\alpha_\psi)^*}V @VV{\B\alpha_\psi*}V @VV{\psi_*}V\\
        @>>> H^{*-1}(M) @>{\delta}>> H^*(M\to B\Gamma) @>>> H^*(B\Gamma)
        @>{u^*}>> H^*(M) @>>>
      \end{CD} }
    \end{equation}

    In particular, the kernel of $u^*$ inside $H^*(B\Gamma)$ is preserved by
    every diffeomorphism $\psi$ of $M$. Because in degree $n+1-4k_j$ this kernel
    is $1$-dimensional by assumption for $j=1,\dots,d$, the subgroup $U$ of
    spin structure 
    preserving diffeomorphism which acts trivially on all these kernels has
    finite 
    index in $\Diffeo(M)$ (the argument is parallel to the one of
    Section \ref{sec:appl_real_K}). Let $H\subset H^*(M\to B\Gamma)$ be the
    inverse image of $\bigoplus_{j=1}^d \ker( u^*\colon
    H^{n+1-4k_j}(B\Gamma)\to H^{n+1-4k_j}(M))$. By the exactness of the pair
    sequence, $\dim(H)\ge d$. By commutativity of
    \eqref{eq:psi_on_cohom}, if $\psi\in U$ and $\alpha\in H$ then
    \begin{equation*}
      \psi^* \alpha =  \alpha+ \delta(\lambda_\alpha)\qquad\text{for some }\lambda_\alpha\in H^{n-4j}(M)
    \end{equation*}
    We then get for an arbitrary metric $g'$ of positive scalar curvature on $M$, by using \eqref{compatibility2} and \eqref{eq:nat_of_rho_g},
    that
    \begin{equation*}
      \innerprod{\varrho(\psi^*g'), \alpha} =
      \innerprod{\varrho(g'),\psi^*\alpha} = \innerprod{\varrho(g'),\alpha}
      + \innerprod{\varrho(g'),\delta (\lambda_\alpha)} = \innerprod{\varrho(g'),\alpha} + \innerprod{c(\varrho(g')),\lambda_\alpha}.
    \end{equation*}

    By the commutativity of the third square in \eqref{HRses},
    $c(\varrho(g'))=\beta([M])\in K_n(M)$ is the K-homology fundamental class. It
    is a standard fact that the K-homology fundamental class is mapped to the
    dual of the total $\hat A$-genus under the homological Pontryagin
    character. Consequently, for the relevant pairing we get
    \begin{equation*}
      \innerprod{c(\varrho(g')),\lambda_\alpha} = \innerprod{\beta([M]),\lambda_\alpha} = \int_M \hat A(M)\cup
      \lambda_\alpha = 0.
    \end{equation*}
    Here, we use our assumption that $\hat{A}(M)_{[4k_j]}=0$ and
    $\lambda_\alpha\in H^{n-4k_j}(M)$.  Therefore we get finally
    \begin{equation*}
    \varrho_{\tilde \alpha}(g') = \varrho_{\tilde\alpha}(\psi^*g')\qquad \forall g\in U.
  \end{equation*}

  It follows that we have a subspace $H\subset H^*(M\to B\Gamma;\rationals)$
  with $\dim(H)\ge d$ and such that the pairing of $\varrho(g')$ with elements of
  $H$ is fixed under the action of $U\subset \Diffeo^{\spin}(M)$ on
  $\Riem^+(M)$. Exactly as in the proof of Proposition \ref{prop:rel_pairing}
  this implies the desired lower bound on the virtual rank of the bordism
  moduli space.

  Let us finally prove that $\hat
    A(M)_{[k_j]}=0$ if also $\dim(H^{n+1-4k_j}(M\to B\Gamma))=1$. This is
    equivalent to vanishing of the Poincar\'e dual
    $\Ph([M])_{n-4j}\in H_{n-4j}(M;\rationals)$. Because $M$ has positive
    scalar curvature and $\Gamma$ satisfies the strong Novikov conjecture, the
    image of the K-homology fundamental class of $M$ in
    $K_*(B\Gamma)\otimes\QQ$ vanishes. Dually, this means that
    $\Ph([M])_{n-4k_j}$ has trivial pairing with the image of the cohomology of
    $B\Gamma$.

    Now consider the long exact cohomology sequence
    \begin{multline*}
      \to H^{n-4k_j}(B\Gamma)\xrightarrow{u^*} H^{n-4k_j}(M)\xrightarrow{\delta} H^{n+1-4k_j}(M\to
      B\Gamma)\\
      \to H^{n+1-4k_j}(B\Gamma)\xrightarrow{H^{n+1-4k_j}(u)} H^{n+1-4k_j}(M)\to
    \end{multline*}
    By assumption, $\dim \ker(H^{n+1-4k_j}(u))=1=\dim H^{n+1-4k_j}(M\to
    B\Gamma)$. Therefore, $H^{n+1-4k_j}(M\to B\Gamma)\to H^{n+1-4k_j}(B\Gamma)$ is
    injective, $\delta$ is the zero map and $H^{n-4k_j}(u)$ is surjective.  By
    the reasoning above this means that $\Ph([M])_{n-4k_j}$ pairs trivially with
    $H^{n-4k_j}(M)$, i.e.~$\Ph([M])_{n-4k_j}=0$ and consequently also its
    Poincar\'e dual $\hat A(M)_{[4k_j]}=0$.
  \end{proof}

\begin{remark}\label{rem:Carr_and_bordism}
  The (higher) rho invariants based on group cohomology, as developed in
  Section \ref{sec:appl_real_K}, ultimately relies on the existence of torsion
  in the fundamental group. This is explicitly explained in
  \cite{XieYuZeidler}, but was evident from the beginning of the use of
  rho invariants to distinguish metrics of positive scalar curvature, see for example \cite{BotvinnikGilkey_eta,PiazzaSchick_PJM}.

  On the other hand, the use of the relative cohomology $H^*(M\to B\Gamma)$ as
  in the present 
  section is {\it independent on the existence of torsion in $\Gamma$.}

  Notice that if $\Gamma$ is torsion free and the Baum-Connes map is an isomorphism then
  the "lower" numeric rho invariants, that is, the Atiyah-Patodi-Singer rho
  invariant, the Cheeger-Gromov rho invariant
  and Lott's delocalized eta invariant, are all equal to zero and therefore of no use in distinguishing
  metrics of positive scalar curvature. See \cite{Keswani1}, \cite{Keswani2}, \cite{PiazzaSchick_BCrho}, \cite{HigsonRoe4}, \cite{BenameurRoy}.
\end{remark}

\begin{remark}
  The metrics which are distinguished in Proposition \ref{prop:rel_pairing}
  not only lie in different components of the moduli space
  $\Riem^+(M)/\Diffeo(M)$, they are also (by the way they are detected) not
   bordant to each other (with reference map to $M$). This distinguishes them
   from families of metrics obtained by Carr \cite{Carr}, which are shown to
   be bordant in \cite[Remark 3.4]{SchickZenobi}. Moreover, Carr gives
   infinitely many non-concordant metrics of positive scalar curvature on any
   spin manifold $M$ of dimension $4d-1$. But these are only known to represent
  infinitely many different components of the moduli space
  $\Riem^+(M)/\Diffeo(M)$  for very special $M$, e.g.~if $M$ is a
  homology sphere or if $M$ is stably parallelizable, compare \cite[Theorem
  1.11]{HankeSchickSteimle}.
\end{remark}

It is clear that there are plenty of situations where Proposition
\ref{prop:rel_pairing} can be applied. The following example gives a very
general method to construct some.

\begin{example}
  Let $\Gamma$ be a discrete group with a classifying space $B\Gamma$ with
  finite $k$-skeleton $B\Gamma^{(k)}$ and assume that
  $H^{k}(\Gamma;\rationals)\ne \{0\}$ for some $k\ge 2$.

  Construct a sub-CW-complex $X$ of $B\Gamma^{(k)}$ as follows. Start with
  $Y=B\Gamma^{(k)}$ and leave out one $k$-cell after the other to obtain
  $Y\superset Y_1\superset Y_2\superset\cdots$. We get a corresponding
  filtration of the $k$-cellular chain group $$C_k(B\Gamma)\tensor\rationals =
  C_k(Y)\tensor \rationals \superset C_k(Y_1)\tensor\rationals\superset
  C_k(Y_2)\tensor\rationals\cdots \superset\{0\}.$$

  We have a surjective map $\ker(\partial_k)\to
  H_k(B\Gamma)=\ker(\partial_k)/\im(\partial_{k+1})$. Restricting $\partial_k$
  to the decreasing chain of subspace $C_k(Y_j)\tensor\rationals$ we get a
  decreasing sequence of kernels starting with $\ker(\partial_k)$ and ending
  in $\{0\}$. Therefore, there will be a subspace $Y_j$ such that
  $\ker(\partial_k|_{C_k(Y_j)})\to H_k(B\Gamma)$ has codimension $1$. Set $X:=
  Y_j$. Because $X$ is an $k$-dimensional CW-complex, $H_k(X) =
  \ker(\partial_k|_{C_k(X)})$ which maps to $H_k(B\Gamma)$ with
  $1$-dimensional cokernel, whereas $H_j(X)\to H_j(B\Gamma)$ is an isomorphism
  for $j<k$ and also an isomorphism on $\pi_1$, as
  $X^{(k-1)}=B\Gamma^{(k-1)}$.

  Choose now $n\ge 2k$ and embed $X$ into $\reals^{2n+1}$. Take a thickening
  $W$ of the embedding: a compact ``tubular neighbourhood'' manifold with
  boundary such that the   inclusion $X\to W$ is a homotopy equivalence. Set
  $M:= \boundary W$. By general position, the inclusion $M\to W$ is an
  $n-k\ge k$-equivalence (an isomorphism on $\pi_j$ for $j<n-k$ and an
  epimorphism for $j=n-k$). In particular, $\pi_1(M)=\Gamma$ and $$H_k(M)\to
  H_k(W)=H_k(X)\to H_k(B\Gamma)$$ has still $1$-dimensional cokernel.

  By construction, $W$ is a spin null-bordism of $M$ with reference map to
  $B\Gamma$ or, after cutting out a disk, a $B\Gamma$ spin bordism from $S^n$
  to $M$. By the Gromov-Lawson surgery theory \cite{GromovLawson} (compare
  e.g.~\cite[Section 3]{SchickZenobi} and \cite{EbertFrenck}), $M$
  admits a Riemannian metric of positive scalar curvature.
\end{example}

Let us give another example of slightly different nature.
\begin{example}
	Let us take $\Gamma$ equal to $\ZZ^{k}$ for $k\ge 3$, or more generally
        $\Gamma$ finitely presented with $$\dim(H_{k}(\Gamma))=1.$$  Observe that $B\ZZ^{k}$ is realized by the $k$-dimensional torus $\mathbb{T}^{k}$.
	
	Let us construct a suitable manifold which fits with the situation in
        Proposition \ref{prop:rel_pairing}.
        	Let us pick a finite presentation $\Gamma=\langle x_1,\dots,\xi_{g}\mid r_1,\dots, r_h \rangle$.
	Take then the wedge of $g$ circles and, for each relation $r_l$,
        attach a two cell. Denote by $X$ this 2-dimensional CW-complex.

        Choose $n=k-1+4j$ for some $j\in\naturals_{>0}$ and embed $X$ into
        $\RR^{n+1}$, which is possible if $n\ge 5$. Now, consider a tubular neighbourhood $\mathcal{N}$ of $X$, hence its boundary $M:=\partial\mathcal{N}$ is an $n$-dimensional spin manifold with fundamental group $\Gamma$.
	Moreover, since $\mathcal{N}$ is a spin null-bordism for $M$ with
        reference map  to $B\Gamma$, it follows that $M$ admits a metric with
        positive scalar curvature. The manifold $W$ has a trivial tangent
        bundle because its interior is an open subset of $\reals^{n+1}$, and
        therefore the tangent bundle of $M$ is stably parallelizable and $\hat
        A(M)=1$ is trivial.
	
The canonical map $M\to B\Gamma$ factors as $M\to X\to B\Gamma$. Now $X$ is a
$2$-dimensional CW-complex and therefore $H_k(X)=0$, as $k\ge 3$.

If $\Gamma$ satisfies the technical conditions that it has property (RD) and
that $H^*_{pol}(\Gamma)\to H^*(\Gamma)$ is an isomorphism, which is certainly
the case for $\Gamma=\integers^k$, then all conditions of Proposition
\ref{prop:rel_pairing} are fulfilled and if follows that the moduli space
$\Riem^+(M)/\Diffeo(M)$ has infinitely many components.
\end{example}

\chapter{Geometric applications: a topological approach}

There is a different approach to the
definition of higher rho numbers associated to elements in $H^* (M\to B\Gamma)$,
 based heavily on the Baum-Connes 
map being an isomorphism and  due to  Weinberger, Xie and Yu, see \cite[Section 7]{WXY}.
In this Section we want to compare the two approaches and also elaborate on  how the 
applications  we have
given involving the diffeomorphism
group can be extended to the context treated in  \cite[Section 7]{WXY}.

\medskip
The approach of this paper for understanding the Higson-Roe analytic exact
sequence via non-commutative cohomology is a classical one, employed initially
to get explicit understanding of the term $K_*(C^*_{red}\Gamma)$ in that
sequence. The main result of this paper is to show that this approach extends
systematically to the whole sequence, in particular to the structure group
$\SG^\Gamma_*(\tM)$.

We give specific examples where this is particularly well controlled, in
particular for groups $\Gamma$ which are hyperbolic or of polynomial growth,
and we explain specific geometric applications of our approach (which of
course work best when the homology side is well understood, e.g.~again for
groups $\Gamma$ which are hyperbolic or of polynomial growth).

To understand $K_*(C^*_{red}\Gamma)$ there is of course the Baum-Connes
conjecture and the Baum-Connes assembly map
$K^\Gamma_*(\underline{E}\Gamma)\xrightarrow{\mu} K_*(C^*_{red}\Gamma)$. It is pretty clear that there should be a natural map from the long exact pair sequence in equivariant K-homology to the whole
Higson-Roe analytic exact sequence:
\begin{equation}\label{eq:WXY}
{\small  \begin{CD}
    @>>> K^\Gamma_*(\tM) @>>> K^\Gamma_{*}(\underline{E}\Gamma) @>>>
      K_*^\Gamma(\tM\to \underline{E}\Gamma) @>>> K^\Gamma_{*-1}(\tilde
      M)\to\\
      && @VV{=}V @VV{\mu}V @VV{\Lambda}V @VV{=}V\\
     @>>> K^\Gamma_*(\tM) @>>> K_*(C^*_{red}\Gamma) @>>>
\SG_{*-1}^\Gamma(\tM) @>>>  K_{*-1}^\Gamma(\tM)\to
  \end{CD} }
\end{equation}
This is (briefly) stated in \cite[diagram (13)]{WXY}.
Using the diagram \eqref{eq:WXY}, the 5-lemma immediately gives the following corollary.
\begin{corollary}
  Assume that the Baum-Connes assembly map $\mu\colon
  K_*(\underline{E}\Gamma)\to K_*(C^*_{red}\Gamma)$ is an isomorphism
  and that 
 a natural commutative diagram such as \eqref{eq:WXY} has been constructed. Then   we get a natural isomorphism
  \begin{equation*}
    \Lambda\colon K_*^\Gamma(\tM\to \underline{E}\Gamma)
    \xrightarrow{\iso} \SG^\Gamma_{*-1}(\tM).
  \end{equation*}
So, in this case, we have full understanding of the structure group
  $\SG^\Gamma_{*}(\tM)$.
\end{corollary}
Note that the Baum-Connes conjecture predicts that $\mu\colon
K_*^\Gamma(\underline{E}\Gamma)\to K_*(C^*_{red}\Gamma)$ is an isomorphism for
every group $\Gamma$. This conjecture is known for large classes of groups, in
particular for Gromov hyperbolic groups \cite{MineyevYu} and for groups of
polynomial growth, 
which are our main examples. No counterexamples to the conjecture are known to
this date.

\medskip
Let us have a closer look at the term $K_*^\Gamma(\tM\to
\underline{E}\Gamma)$ which conjecturally computes
$\SG_{*-1}^\Gamma(\tM)$.

\begin{proposition}\label{prop:rel_Chern}
  The Chern character provides a natural isomorphism
  \begin{equation*}
    \Ch\colon    K_*^\Gamma(\tM\to \underline{E}\Gamma)\tensor\complexs \xrightarrow{\iso}
    \bigoplus_{k\in\integers} H_{*+2k}(M\to B\Gamma;\complexs) \oplus
    \bigoplus_{k\in\integers} H_{*+2k}(\Gamma, F_{del}\Gamma).
  \end{equation*}
  Here, the $\Gamma$-module $F_{del}\Gamma$ is the delocalized summand of
  $F\Gamma$ of Section \ref{sec:appl_real_K}, i.e.
  \begin{equation*}
    F_{del}\Gamma=\{f\colon
    \Gamma_{fin}\setminus\{e\}\to\complexs\mid \abs{\supp(f)}<\infty\}
  \end{equation*}
  with
    conjugation action, where $\Gamma_{fin}$ is the set of elements $e\ne\gamma\in\Gamma$ of finite
    order. We have the obvious decomposition
    $F\Gamma=F_{del}\Gamma\oplus\complexs$ with $\complexs=\{f\colon
    \{e\}\to\complexs\}$ the trivial $\Gamma$-representation.

    As in Section \ref{sec:appl_real_K}, we have the splitting over all
    conjugacy classes $[\gamma]$ of elements of finite order
    \begin{equation*}
      H_*(\Gamma;F_{del}\Gamma) \iso \bigoplus_{[\gamma], \gamma\ne e}
      H_*(\Gamma_\gamma;\complexs). 
    \end{equation*}

  All this is compatible with the maps of the pair sequence in equivariant
  K-homology and the natural Chern character isomorphism
  \begin{equation*}
    \begin{split}
      K^\Gamma_*(\tilde
      M)\tensor\complexs &\xrightarrow{\iso} \bigoplus_{k\in\integers}
                           H_{*+2k}(M;\complexs);\\
      \qquad K^\Gamma_*(\underline{E}\Gamma)\tensor\complexs
      &\xrightarrow{\iso} \bigoplus_{k\in\integers} H_{*+2k}(\Gamma;\complexs)
      \oplus \bigoplus_{k\in\integers} H_{*+2k}(\Gamma; F_{del}\Gamma),
    \end{split}
  \end{equation*}
 where of
  course $H_*(\Gamma;\complexs)=H_*(B\Gamma;\complexs)$.
\end{proposition}
\begin{proof}
  This follows directly by comparison with the Chern character
  for the pair sequence $$\cdots\to K_*(M)\to K_*(B\Gamma)\to K_*(M\to
  B\Gamma)\to\cdots$$ mapping to the Chern character of the pair sequence $$\cdots\to
  K^\Gamma_*(\tM)\to K^\Gamma_*(\underline{E}\Gamma)\to K^\Gamma_*(\tilde
  M\to
  \underline{E}\Gamma)\to\cdots$$  via the
  natural isomorphisms $K^\Gamma_*(\tilde X)\to K_*(X)$ for a space $\tilde X$
  with free $\Gamma$-action and with quotient $X=\tilde X/\Gamma$.
  In the terms for $\underline{E}\Gamma$ then $H_*(B\Gamma;\complexs)$ injects to $H_*(\Gamma; F\Gamma)$
  as the summand localized at the identity, while in the terms for $\tM$,
  $H_*(M;\complexs)$ is the full image of the Chern character.

  Consequently, $H_*(M\to B\Gamma;\complexs)$ injects as a direct summand into
  the image of the Chern character isomorphism for $(\tM\to
  \underline{E}\Gamma)$ with remaining summand coming from
  $\underline{E}\Gamma$, where it is precisely the delocalized homology
  $H_*(\Gamma;F_{del}\Gamma)$. 
\end{proof}

  \begin{corollary}\label{corol:topol_pairing}
    Assume that the Baum-Connes map is an isomorphism. Then this induces
    natural pairings of $S^\Gamma_{*-1}(\tM)$ with
    $H_{*+2k}(M\to B\Gamma;\complexs)$ and with
    $H_{*+2k}(\Gamma;F_{del}\Gamma)$ for all $k\in\integers$. These are
    compatible with the pairings  of $H_{*+2k}(\Gamma;F_{del}\Gamma)$ and
    $H^*(\Gamma;\complexs)$ with $K_*(C^*_{red}\Gamma)$ and the pairing of
    $H^*(M;\complexs)$ with $K_*(M)$.
  \end{corollary}
  \begin{proof}
    This just a consequence of the map to homology, where we can then pair
    with cohomology as usual.
  \end{proof}

\begin{remark}
  Throughout, we can pass to real K-theory and obtain a diagram corresponding
  to \eqref{eq:WXY} with the real K-homology and structure group throughout.

  The Chern character isomorphism of Proposition \ref{prop:rel_Chern} then has
  the analog for real K-homology using the splitting into even and odd part of
  $F_{del}\Gamma$ as in Section \ref{sec:appl_real_K}, in particular in
  \eqref{eq:del_Chern}.
\end{remark}

We shall now briefly treat the action of the diffeomorphism group
in this context, using the ideas we have explained in the previous sections.
Naturality of the constructions implies that we have the expected action of
$\Diffeo(M)$ on the whole diagram \eqref{eq:WXY}. Under the Chern character
isomorphism of Proposition \ref{prop:rel_Chern} the action on the group
cohomology terms factors through the homomorphism $\Diffeo(M)\to
\Out(\Gamma)$. The methods we have developed in Section \ref{sec:appl_real_K} and Section
\ref{sec:rel_cohom_geometric} apply without change here and we obtain the following
corollary:

\begin{corollary}\label{corol:BC_geom_cons}
  Assume that $\Gamma$ satisfies the Baum-Connes conjecture 
  so that there
  exists a Baum-Connes isomorphism $K_{*+1}^\Gamma(\tM\to
  \underline{E}\Gamma)\xrightarrow{\iso} \SG^\Gamma_{*}(\tM)$
 (coming from the commutative diagram \eqref{eq:WXY}).
  Assume that $M$ is a spin manifold of dimension $n\ge 5$ with positive scalar
  curvature. Assume that a class $0\ne \alpha\in
  H_{n+1-4j}(\tM\to \underline{E}\Gamma;\complexs)$ is preserved by the
    action of $\Diffeo(M)$ for some $j>0$.
  Then the moduli space $\Riem^+(M)/\Diffeo(M)$ has infinitely many
  components.
  
For example, we obtain such a homology class $\alpha$ whenever a delocalized group homology summand as in
    \eqref{eq:del_Chern} of degree $n+1-2p-4j$ or a summand of relative homology
    $H_{n+1-4j}(M\to B\Gamma)$  for some $j>0$ is
    $1$-dimensional, as in Sections \ref{sec:appl_real_K} or
    \ref{sec:rel_cohom_geometric}. 
\end{corollary}

\begin{remark}
  Assume, more generally than in Corollary \ref{corol:BC_geom_cons}, that the
  group $\Gamma$ satisfies that the Baum-Connes assembly map $\mu\colon
  K^\Gamma_*(\underline{E}\Gamma)\to K_*(C^*_{red}\Gamma)$ is rationally split
  injective  with a split $\alpha\colon
  K_*(C^*_{red}\Gamma)\tensor\rationals\to
  K_*^\Gamma(\underline{E}\Gamma)\tensor\rationals$ which is natural for the
  obvious action of $\Out(\Gamma)$. Then we can state and prove the analogue of Corollary \ref{corol:BC_geom_cons}.
  
\end{remark}
\begin{proof}
  By routine diagram chases, as spelled out in \cite[Section 7]{WXY}, the
  assumptions give a natural split injection of the pair sequence into the
  Higson-Roe exact analytic sequence (where we omit the tensor factor
  $\tensor\rationals$ for better readability):
  \begin{equation*}
      \begin{CD}
        @>>> K^\Gamma_{n}(\tM) @>>> K^\Gamma_{n}(\underline{E}\Gamma)
        @>>> K^\Gamma_n(\tM\to \underline{E}\Gamma) @>>>\\
        && @VV=V @VVV @VVV \\
        @>>> K_n^\Gamma(\tM) @>>> K_n^\Gamma(\underline{E}\Gamma)\oplus
        \mathcal{E} @>>> K^\Gamma_n(\tilde
          M\to\underline{E}\Gamma)\oplus\mathcal{E} @>>>\\
          && @VV=V @V{\iso}V{\mu\oplus \iota}V @V{\iso}V{\Lambda\oplus \iota}V\\
       @>>> K_n^\Gamma(\tM) @>>> K_n(C^*_{red}\Gamma) @>>>
       \SG^\Gamma_{n-1}(\tM) @>>>
      \end{CD}
  \end{equation*}
  Here, $\mathcal{E}$ is the natural summand complementary to
  $\mu(K_n^\Gamma(\underline{E}\Gamma)\tensor\rationals)$ of
  $K_n(C^*_{red}\Gamma)\tensor\rationals$ determined by the split $\alpha$,
  which by a diagram chase then also becomes the natural summand
  complementary to $\Lambda(K^\Gamma_n(\tilde
  M\to\underline{E}\Gamma)\tensor\rationals)$ in $\SG^\Gamma(\tilde
  M)$.

  The same arguments as in the case where the Baum-Connes map is an
  isomorphism can then be used, ignoring throughout the (unknown) summand
  $\mathcal{E}$. 
\end{proof}

  \begin{remark}
    Using Corollary \ref{corol:BC_geom_cons}, if the Baum-Connes map is an
    isomorphism or more generally, if it is split injective, we can define via
    this topological route for each cohomology class
    $[\tau]\in H^{*+2k}(M\to B\Gamma;\complexs)$ or
    $[\tau]\in H^{*+2k}(\Gamma;F_{del}\Gamma)$ topologically defined numeric
    rho invariants (say, of a positive scalar curvature metric $g$ on a spin
    manifold $M$ via its invertible Dirac operator $\tilde D$)
    \begin{equation*}
      \varrho^{BC}_{\tau}(g) := \innerprod{ \Ch(\Lambda^{-1}(\varrho(g))),[\tau]}
    \end{equation*}
    (if the Baum-Connes map only has a split $\alpha$, we have to replace
    $\Lambda^{-1}$ by this split $\alpha$). In \cite{WXY}, $\alpha(\varrho(g))$
    is called $\varrho^{Nov}(g)$.
    The results of Corollary \ref{corol:BC_geom_cons} apply equivalently
    to these numeric rho invariants.

    A priori, it is not clear if these numbers coincide with the numeric rho
    invariants we defined earlier in Definitions
    \ref{def:higher_rho_number_cyc} and \ref{def:higher_rho_of_g_rel}. There,
    the definition is via the Chern character to non-commutative
    de Rham homology and the paring of the latter with cyclic or relative
    cohomology.

    This depends on the compatibility of our Chern character map with the
    Baum-Connes version, and depends in particular on the construction of
    $\Lambda$.

    Nonetheless, all our computations in Sections
    \ref{sect:application-via-cyclic} and \ref{sec:rel_cohom_geometric} only
    compute \emph{differences} of higher rho numbers, which are obtained as
    values of the pairing with $K_*(C^*_{red}\Gamma)$. For the latter, the
    compatibility of the Baum-Connes Chern character and the cyclic homology
    Chern character is well known. Therefore, the \emph{differences} of higher
    rho numbers we obtain in Sections   \ref{sect:application-via-cyclic} and
    \ref{sec:rel_cohom_geometric} coincide with the differences of the
    corresponding $\varrho^{BC}_{\tau}(g)$.
  \end{remark}

\chapter{Concluding remarks}
We end this article by stressing why we believe that the homological methods 
developed here (and in contributions by other authors) are 
of interest, both from a general point of view and also when compared to the more K-theoretic arguments of the last section based on the Baum-Connes 
conjecture:

  \begin{itemize}
  \item  the pairings we construct in Section \ref{sec:rho_numbers}
are of  interest in a very general context; they have been applied here to rho invariants but it is conceivable 
 that  other invariants to which these methods apply will be defined in the future;
   \item we obtain very explicit integral formulas for the higher rho numbers,
    which a priori might be used for calculations in entirely new situations;
  \item we give, for hyperbolic groups and groups of polynomial growth, a
    proof of the geometric applications which does not rely on the proof of
    the Baum-Connes conjecture;
   \item our results are a priori more general than just working for groups
    where the Baum-Connes conjecture is known: whenever one has appropriate
    control of the homology of a smooth subalgebra $\mathcal{A}\Gamma$, one
    gets explicit formulas and the desired geometric consequences. There is no reason why this
    should only work if the Baum-Connes conjecture holds.
  \item we feel that the (cyclic) homological method has potential to be used in other
    geometric situations, e.g.~when analyzing foliated manifolds for the study
    of metrics with leafwise positive scalar curvature, where it is not clear
    how to define the  \emph{coarse geometric} target of $\Lambda$ and $\Lambda$
    itself in that context;
  \item  another geometric situation of relevance for our methods is
    the one where a reductive Lie group acts properly and isometrically on a
    manifold. The (cyclic)  homological approach has been applied very successfully
    to 
    primary invariants of Dirac operators   and there is in fact a large literature on the subject.
    For secondary invariants new results through the homological methods adopted here
    are emerging recently, compare e.g.~\cite{HST,SongTang}.
  \end{itemize}

%%% Local Variables:
%%% mode: LaTeX
%%% TeX-master: "numeric-rho_memoir_style"
%%% End:

%% file: numeric_rho_biblio.tex
%%% Local Variables:
%%% mode: LaTeX
%%% TeX-master: "numeric-rho_memoir_style"
%%% End: